\pdfoutput=1

\documentclass{amsart}
\usepackage[utf8]{inputenc}

\usepackage{geometry}
\usepackage{stmaryrd}
\usepackage{accents}
\usepackage{upgreek}
\usepackage[normalem]{ulem}

\usepackage[dvipsnames]{xcolor}

\newcommand\myshade{85}
\colorlet{mylinkcolor}{Red}
\colorlet{mycitecolor}{Green}
\colorlet{myurlcolor}{Plum}

\usepackage{epsdice}

\usepackage{enumitem}

\usepackage[titletoc]{appendix}

\usepackage[OT2,T1]{fontenc}

\usepackage[final]{microtype}

\usepackage{tikz}
\usepackage{tikz-cd}

\tikzcdset{
  cells={font=\everymath\expandafter{\the\everymath\displaystyle}},
}

\makeatletter
\tikzcdset{
  eq node/.style={
    commutative diagrams/math mode=false, anchor=center},
  eq/.style={
    phantom, description,
    /tikz/every to/.append style={
      edge node={node[commutative diagrams/eq node]
        {\@eqnswtrue\make@display@tag\ltx@label{{#1}}}}}}}
        
\tikzcdset{scale cd/.style={every label/.append style={scale=#1},
    cells={nodes={scale=#1}}}}
\makeatother

\tikzset{
  trim node/.default=1cm,
  trim node/.style={
    overlay,
    append after command={%
      ([xshift={+#1}]\tikzlastnode.north west)
      ([xshift={+-#1}]\tikzlastnode.south east)}},
  down and trim/.default=1cm,
  down and trim/.style={
    yshift=-(\pgfmatrixcurrentcolumn-1)*1.5\baselineskip,
    trim node={#1}},
  downup and trim/.default=1cm,
  downup and trim/.style={
    yshift=iseven(\pgfmatrixcurrentcolumn) ? -1.5\baselineskip : 0pt,
    trim node={#1}},
  -|/.style={to path={-|(\tikztotarget)\tikztonodes}},
  |-/.style={to path={|-(\tikztotarget)\tikztonodes}},
  -| sl/.style={-|, xslant=-1},
  |- sl/.style={|-, xslant= 1},
  center picture/.style={
    trim left=(current bounding box.center),
    trim right=(current bounding box.center)}}

\newcommand{\multi}[3][A]{\arrow[#2, phantom, "#3"{name=#1, inner sep=1ex}]}

\usepackage{rotating}
\usepackage{floatpag}

\usepackage[draft=false,linktocpage=true,pdfusetitle]{hyperref}

\hypersetup{
    linkcolor = mylinkcolor!\myshade!black,
    citecolor = mycitecolor!\myshade!black,
    urlcolor = myurlcolor!\myshade!black,
    colorlinks = true
}
\usepackage[capitalise]{cleveref}
\crefformat{equation}{\ensuremath{(#2#1#3)}}
\crefmultiformat{equation}{\ensuremath{(#2#1#3)}}{ and~\ensuremath{(#2#1#3)}}{, \ensuremath{(#2#1#3)}}{, and~\ensuremath{(#2#1#3)}}
\crefname{enumi}{\unskip}{\unskip}

\AtBeginEnvironment{appendices}{\crefalias{section}{appendix}}

\AtBeginDocument{
    \def\MR#1{}
}

\usepackage{colonequals}
\usepackage{mathrsfs}

\usepackage{relsize}
\usepackage[bbgreekl]{mathbbol}
\usepackage{amsfonts}
\DeclareSymbolFontAlphabet{\mathbb}{AMSb} %
\DeclareSymbolFontAlphabet{\mathbbl}{bbold}

\mathchardef\mhyphen="2D

\renewcommand{\bf}{\mathbf}
\newcommand{\cal}{\mathcal}
\newcommand{\ov}{\overline}
\renewcommand{\rm}{\mathrm}
\newcommand{\ud}{\underline}
\newcommand{\w}{\widetilde}
\newcommand{\wdh}{\widehat}
\newcommand{\xr}{\xrightarrow}

\renewcommand{\AA}{\mathbf{A}}
\newcommand{\DD}{\mathbf{D}}
\newcommand{\EE}{\mathbf{E}}
\newcommand{\FF}{\mathbf{F}}
\newcommand{\GG}{\mathbf{G}}
\newcommand{\LL}{\mathbf{L}}
\newcommand{\NN}{\mathbf{N}}
\newcommand{\PP}{\mathbf{P}}
\newcommand{\QQ}{\mathbf{Q}}

\newcommand{\ZZ}{\mathbf{Z}}
\newcommand{\Z}{\mathbf{Z}}

\newcommand{\cE}{\mathcal{E}}
\newcommand{\cF}{\mathcal{F}}
\newcommand{\F}{\mathcal{F}}
\newcommand{\cG}{\mathcal{G}}
\newcommand{\G}{\mathcal{G}}
\newcommand{\cL}{\mathscr{L}}

\newcommand{\cO}{\mathcal{O}}
\renewcommand{\O}{\mathcal{O}}
\newcommand{\cU}{\mathscr{U}}

\newcommand{\cX}{\mathscr{X}}
\newcommand{\cY}{\mathscr{Y}}

\newcommand{\fm}{\mathfrak{m}}
\newcommand{\m}{\mathfrak{m}}
\newcommand{\fn}{\mathfrak{n}}
\newcommand{\p}{\mathfrak{p}}
\newcommand{\fS}{\mathfrak{S}}

\newcommand{\ra}{\mathrm{a}}
\newcommand{\rR}{\mathrm{R}}

\newcommand{\LLambda}{\underline{\Lambda}}

\newcommand{\alg}{\mathrm{alg}}
\newcommand{\an}{\mathrm{an}}
\newcommand{\BH}{\mathrm{BH}}
\newcommand{\cHom}{\mathscr{H}\kern-.5pt om}
\newcommand{\coh}{\mathrm{coh}}
\newcommand{\et}{\mathrm{\acute{e}t}}

\newcommand{\eval}{\mathrm{eval}}
\newcommand{\fet}{\mathrm{f\acute{e}t}}
\newcommand{\gen}{\mathrm{gen}}
\newcommand{\h}{\mathrm{h}}
\newcommand{\IC}{\mathrm{IC}}
\newcommand{\IH}{\mathrm{IH}}
\newcommand{\lisse}{\mathrm{lis}}
\newcommand{\op}{\mathrm{op}}
\newcommand{\red}{\mathrm{red}}
\newcommand{\sm}{\mathrm{sm}}
\renewcommand{\sp}{\mathrm{sp}}
\newcommand{\zc}{\mathrm{zc}}

\DeclareMathOperator{\Adj}{Adj}
\DeclareMathOperator{\BC}{BC}
\DeclareMathOperator{\cl}{cl}
\DeclareMathOperator{\charac}{char}
\DeclareMathOperator{\coker}{coker}
\DeclareMathOperator*{\colim}{colim}
\DeclareMathOperator{\Et}{\mathrm{\acute{E}t}}
\DeclareMathOperator{\Ev}{Ev}
\DeclareMathOperator{\Ex}{Ex}
\DeclareMathOperator{\Ext}{Ext}
\DeclareMathOperator{\Frac}{Frac}
\DeclareMathOperator{\GL}{GL}

\DeclareMathOperator{\Hh}{H}
\DeclareMathOperator*{\hocolim}{hocolim}
\DeclareMathOperator{\Hom}{Hom}
\DeclareMathOperator{\id}{id}
\DeclareMathOperator{\KM}{KM}

\DeclareMathOperator{\ord}{ord}
\DeclareMathOperator{\PD}{PD}
\DeclareMathOperator{\PF}{PF}
\DeclareMathOperator{\Pic}{Pic}
\DeclareMathOperator{\pr}{pr}
\DeclareMathOperator{\res}{res}
\DeclareMathOperator{\rk}{rk}
\DeclareMathOperator{\Shv}{Shv}
\DeclareMathOperator{\SL}{SL}
\DeclareMathOperator{\Spa}{Spa}
\DeclareMathOperator{\Spec}{Spec}
\DeclareMathOperator{\spec}{sp}
\DeclareMathOperator{\Spf}{Spf}
\DeclareMathOperator{\supp}{supp}

\DeclareMathOperator{\tr}{Tr}
\DeclareMathOperator{\ttr}{tr}

\newcommand{\abs}[1]{\lvert#1\rvert}
\newcommand{\suchthat}{\;\ifnum\currentgrouptype=16 \middle\fi\vert\;}
\newcommand\restr[2]{{\left.\kern-\nulldelimiterspace#1\vphantom{\big|}\right|_{#2}}}
\newcommand{\blank}{{-}}
\newcommand*\lon{
        \nobreak
        \mskip6mu plus1mu
        \mathpunct{}
        \nonscript
        \mkern-\thinmuskip
        {:}
        \mskip2mu
        \relax
}

\usepackage{mathtools}
\usepackage{amsmath}

\usepackage{amsthm,amssymb}
\numberwithin{equation}{subsection}

\theoremstyle{plain}
\newtheorem{theorem}[equation]{Theorem}

\newtheorem{proposition}[equation]{Proposition}
\newtheorem{lemma}[equation]{Lemma}
\newtheorem{claim}[equation]{Claim}
\newtheorem{corollary}[equation]{Corollary}

\theoremstyle{definition}
\newtheorem{definition}[equation]{Definition}
\newtheorem{notation}[equation]{Notation}
\newtheorem{convention}[equation]{Convention}
\newtheorem{construction}[equation]{Construction}
\newtheorem{variant}[equation]{Variant}

\newtheorem{example}[equation]{Example}

\newtheorem{remark}[equation]{Remark}

\newtheorem{warning}[equation]{Warning}
\newtheorem{setup}[equation]{Setup}

\title{Relative Poincar\'{e} duality in nonarchimedean geometry}

\author{Shizhang Li}
\author{Emanuel Reinecke}
\author{Bogdan Zavyalov}

\address[Shizhang Li]{Morningside Center of Mathematics and Hua Loo-Keng Key Laboratory of Mathematics,
Academy of Mathematics and Systems Science, Chinese Academy of Sciences, Beijing 100190, China}
\email{lishizhang@amss.ac.cn}

\address[Emanuel Reinecke]{Institut des Hautes \'Etudes Scientifiques, 35 route de Chartres, 91440 Bures-sur-Yvette, France}
\email{reinecke@ihes.fr}

\address[Bogdan Zavyalov]{\parbox{\linewidth}{Princeton University, 304 Washington Road, Princeton, NJ 08540, USA \\
Institute for Advanced Study, 1 Einstein Drive, Princeton, NJ 08540, USA}} 
\email{bogd.zavyalov@gmail.com}

\begin{document}

\begin{abstract} 
We prove a conjecture of Bhatt--Hansen that derived pushforwards along proper morphisms of rigid-analytic spaces commute with Verdier duality on Zariski-constructible complexes.
In particular, this yields duality statements for the intersection cohomology of proper rigid-analytic spaces.
In our argument, we construct cycle classes in analytic geometry as well as trace maps for morphisms that are either smooth or proper or finite flat, with appropriate coefficients.
As an application of our methods, we obtain new, significantly simplified proofs of $p$-adic Poincar\'e duality and the preservation of $\bf{F}_p$-local systems under smooth proper higher direct images. 
\end{abstract}

\maketitle

\tableofcontents

\section{Introduction}

\addtocontents{toc}{\protect\setcounter{tocdepth}{1}}

\subsection{Main results}

Let $K$ be a nonarchimedean field of characteristic $0$ and residue characteristic $p\geq 0$.
Set $\Lambda = \Z/n\Z$ for some integer $n>0$. When $n$ is coprime to $p$, Berkovich and Huber independently developed in \cite{Berkovich} and \cite{Huber-etale} (see also \cite{dJ-vdP}) a robust theory of \'etale cohomology of rigid-analytic spaces over $K$ which shares most of the nice properties of the algebraic theory of \'etale cohomology developed in \cite{SGA4,SGA41/2,SGA5}. 

However, things become significantly more complicated when $n=p$. Many of the most basic properties fail:
for example, finiteness of (compactly) supported cohomology (see \cref{distance and cycle class} and \cref{rmk:infinite-cohomology}), proper base change, etc.
On the other hand, Scholze recently proved in his seminal paper \cite{Scholze-Hodge} that the $\bf{F}_p$-cohomology groups of smooth  \emph{proper} rigid-analytic spaces are finite dimensional.\footnote{
The smoothness assumption was later removed in \cite[Th.~3.17]{Scholze-CDM}.}
This led to a significant interest in studying $p$-adic \'etale cohomology groups in $p$-adic analytic geometry.
For instance, \cite{Z-thesis} and \cite{Mann-thesis} showed that smooth and proper rigid-analytic spaces satisfy Poincar\'e duality for $\bf{F}_p$-local systems (see also \cite{LLZ} and \cite{CGN} for rational variants of duality), while \cite{BH} developed a robust theory of Zariski-constructible sheaves, dualizing complexes, Verdier duality, and perverse $t$-structures.

Despite all these advances, one question that has remained open is whether there is a relative version of Poincar\'e duality with coefficients (more general than local systems).
In \cite{BH}, Bhatt and Hansen put forward a conjecture that relates the behavior of derived proper pushforward and dualizing complexes (in the sense of \cite[Th.~3.21]{BH}). Our main result is the proof of this conjecture: 

\begin{theorem}[Bhatt--Hansen's conjecture, {\cref{general Poincare duality}}]
\label{Intro: proper PD}
Let $f\colon X \to Y$ be a proper morphism of rigid-analytic spaces over $K$, and let $\omega_X$ and $\omega_Y$ be the dualizing complexes on $X$ and $Y$ respectively.
Then there is a canonical trace morphism $\tr_f \colon \rR f_* \omega_X \to \omega_Y$ such that the induced duality morphism
\[ \PD_f \colon  \rR f_*\rR \cHom (\F, \omega_X) \xlongrightarrow{\Ev_f} \rR\cHom(\rR f_*\F,\rR f_* \omega_X) \xlongrightarrow{\tr_f \circ \blank} \rR\cHom(\rR f_*\F, \omega_Y)
\]
is an equivalence for any $\F \in D_\zc(X_\et; \Lambda)$. In other words, derived pushforward along a proper morphism commutes with Verdier duality on Zariski-constructible complexes. 
\end{theorem}

First, we want to note that, if $n$ is coprime to $p$, then \cref{Intro: proper PD} follows from \cite[Th.~3.21]{BH}. In fact, \textit{loc.\ cit.} implies that \cref{Intro: proper PD} admits a compactly supported version for an arbitrary taut separated $f$ and arbitrary coefficients. However, a version of \cref{Intro: proper PD} for non-proper $f$ (or proper $f$ and non-Zariski-constructible $\F$) fails miserably when $n=p$;
see \cref{rmk:no-Poincare-duality-closed-unit-disk}.
This lack of a local version makes the proof of \cref{Intro: proper PD} quite difficult for two (somewhat related) reasons: one cannot run standard arguments to reduce to the case of a relative affine (or projective) line and, more importantly, the definition of the dualizing complex $\omega_X$ is \emph{local} on $X$, so it is not well-adapted for proving global results like \cref{Intro: proper PD}. For these reasons, our approach to \cref{Intro: proper PD} is completely different from \cite[Th.~3.21]{BH} and from the classical approach in algebraic geometry.
Moreover, it works uniformly for all $n$, divisible by $p$ or not.
It seems likely that for $n=p$, the 6-functor formalism and Poincar\'e duality statement developed in \cite{Mann-thesis} can also be applied toward a different proof of \cref{Intro: proper PD};
see \cref{potential different proof remark}.

Another issue we want to point out is that, in order to formulate \cref{Intro: proper PD} precisely, one first needs to construct a trace morphism. In fact, constructing a trace map satisfying some sufficiently nice properties is one of the key steps in the proof of \cref{Intro: proper PD}. To do this, we first develop a robust theory of trace morphisms for smooth (but not necessarily proper) morphisms and revisit Poincar\'e duality for smooth proper morphisms by giving a new easy, uniform in $n$, and essentially diagrammatic proof. 

Before we discuss these results in more detail in the next subsection, we want to mention several immediate corollaries of \cref{Intro: proper PD} which look more similar to the classical Poincar\'e duality results. First, we note that \cref{Intro: proper PD} implies a version of Poincar\'e duality for some class of smooth non-proper rigid-analytic spaces: 

\begin{corollary}[{\cref{cohomology duality}}]\label{cor:intro-open-duality}
Let $\overline{X}$ be a proper rigid-analytic space over $K$ and $X \subset \overline{X}$ be a smooth Zariski-open
rigid-analytic subspace of equidimension $d$.
Let $\bf{L}$ be a local system of finite free $\Lambda$-modules on $X_\et$ and let $\bf{L}^{\vee}$ be its $\Lambda$-linear dual.
Set $C\coloneqq \wdh{\overline{K}}$. 
Then the groups $\Hh^i_c(X_C, \LL)$ and $\Hh^{2d-i}(X_C, \LL^\vee)$ are finite and there is a Galois-equivariant isomorphism
\[
\Hh^i_c(X_C, \LL)^{\vee} \simeq \Hh^{2d-i}(X_C, \LL^\vee)(d)
\]
which is functorial in $\LL$.
\end{corollary}

As a second application of \cref{Intro: proper PD}, we prove another conjecture of Bhatt and Hansen predicting duality of intersection cohomology on proper rigid-analytic spaces (see \cite[Paragraph after Th.~4.13]{BH}). In fact, we show a slightly stronger statement: 

\begin{corollary}[{Bhatt--Hansen's conjecture, \cref{intersection cohomology duality}}]
Let $\ov{X}$ be a proper rigid-analytic space over $K$ and $U\subset X\subset \ov{X}$ be two rigid-analytic subspaces which are both Zariski-open in $\ov{X}$.
Assume that $U$ is smooth of equidimension $d$.
Let $\LL$ be a local system of finite free $\Lambda$-modules on $U_\et$ and let $\LL^{\vee}$ be its $\Lambda$-linear dual. 
Set $C\coloneqq \wdh{\overline{K}}$. Then the groups $\IH^{i}_c(X_C, \bf{L})$ and $\IH^{-i}(X_C, \bf{L}^\vee)$ are finite and there is a Galois-equivariant isomorphism
\[
\IH^{i}_c(X_C, \bf{L})^\vee \simeq \IH^{-i}(X_C, \bf{L}^\vee)(d)
\]
which is functorial in $\bf{L}$.
\end{corollary}

\subsection{Trace and duality for smooth maps}

In this subsection, we discuss our main duality results for smooth morphisms. Unlike in the previous subsection, the results of this subsection hold for arbitrary locally noetherian\footnote{
We have to impose the locally noetherian assumption only because \cite{Huber-etale} works out general theories of smooth morphisms and \'etale cohomology of analytic adic spaces under the locally noetherian assumption. We never use this noetherianness assumption in any serious way.}
analytic adic spaces. For this subsection, we fix an integer $n>0$ and put $\Lambda = \Z/n\Z$. 

We start by discussing the construction of trace maps for separated taut smooth morphisms.\footnote{We also impose the taut separated assumption on $f$ simply because the $\rR f_!$-functor has been defined only for such morphisms in \cite{Huber-etale}. \cref{Intro: smooth trace} can be formally extended to all smooth morphism of equidimension $d$ once one works out a robust theory of $\rR f_!$ for a general morphism $f$ of finite type (see \cite[Th.~9.4]{adic-notes} for the case when $n\in (\O_Y^+)^\times$).}

\begin{theorem}[{\cref{smooth-trace-constant}, \cref{Compatibility with algebraic geometry}, \cref{cor:compatibility-berkovich-trace-2}, \cref{lemma:comparison-LLZ}}]
\label{Intro: smooth trace}
There is a unique way to assign to any separated taut smooth of equidimension $d$
morphism $f \colon X \to Y$ of locally noetherian analytic adic spaces with $n \in \cO^\times_Y$ a trace map $\ttr_f \colon \rR f_! \LLambda_X(d)[2d] \to \LLambda_Y$
in $D(Y_\et; \Lambda)$ such that:
\begin{enumerate}[label=\upshape{(\arabic*)}]
\item $\ttr$ is compatible with compositions;
\item $\ttr$ is compatible with pullbacks;
\item if $f$ is \'etale, then $\ttr_f$ is given by the counit
\[ \rR f_! \LLambda_X \simeq \rR f_! f^* \LLambda_X \to \LLambda_Y \]
of the adjunction between $\rR f_!$ and $f^*$; and
\item If $f$ is the analytification of the structure morphism $\PP^1_C \to \Spec C$ for some complete, algebraically closed nonarchimedean field $C$, then $\ttr_f$ is the analytification of the algebraic trace.
\end{enumerate}
Furthermore, these trace maps satisfy the following compatibilities: 
\begin{enumerate}
    \item whenever $A$ is a strongly noetherian Tate affinoid algebra and $f$ is the analytification of a finite type separated smooth of equidimension $d$ morphism of locally finite type $A$-schemes, our $\ttr_f$ is the analytification of the algebraic trace map;
    \item whenever $K$ is a nonarchimedean field and $f\colon X \to Y$ is a partially proper smooth morphism of equidimension $d$ between rigid-analytic spaces, our $\ttr_f$ coincides with the Berkovich trace $t_f$ from \cite[Th.~7.2.1]{Berkovich} (see also \cite[Th.~5.3.3]{Z-thesis} for the translation into the language of adic spaces);
    \item $\ttr_f$ is compatible with the trace of Lan--Liu--Zhu from \cite[Th.~1.3]{LLZ} whenever the latter is defined. 
\end{enumerate}
\end{theorem}

When $n$ is invertible in $\O_Y^+$, the trace map was previously constructed by Huber in \cite[Th.~7.3.4]{Huber-etale}. When $n$ is only invertible in $\O_Y$ and $f \colon X \to Y$ is a \textit{partially proper} smooth morphism of rigid-analytic spaces over a non-archimedean field $K$, the trace map $f$ was constructed by Berkovich in \cite[Th.~7.2.1]{Berkovich}.
Our construction of the trace map is independent of either of these constructions\footnote{In fact, the construction of Huber is very specific to the case of $n$ being invertible in $(\O^+_Y)^\times$, while the construction of Berkovich is very specific to the partially proper case.} and, as we explain after \cref{rmk:why-general-trace}, it crucially uses the techniques of universal compactifications and the existence of higher rank points, both of which are only available in Huber's formalism of adic spaces. 

\begin{remark}\label{rmk:why-general-trace} Our main motivation for developing a general theory of smooth trace maps comes from the needs of our proof of \cref{Intro: proper PD}. Indeed, to prove \cref{Intro: proper PD}, we need to construct proper trace maps with coefficients in dualizing complexes, for which the full strength of \cref{Intro: smooth trace} is used. Namely, even though we define trace maps with coefficients in dualizing complexes only for proper maps, it is indispensable for the construction to have smooth trace available for non-partially proper morphisms. We elaborate on this more in \cref{section:intro-proper}. 
\end{remark}

We now explain the main i   deas behind the construction of $\ttr_f$ in \cref{Intro: smooth trace}. A d\'evissage similar to the one in algebraic geometry \cite[Exp.~XVII]{SGA4} allows us to reduce the construction of smooth trace maps in general to the situation when $f \colon X \to Y = \Spa(C,\cO_C)$ is a smooth connected affinoid curve over an algebraically closed nonarchimedean field $C$.
In this case, we crucially use the geometry of adic spaces:
The complement of $X$ inside its universal compactification $X^c$ is a pseudo-adic space that consists of finitely many points corresponding to valuations of rank $2$.
The second components of these valuations give rise to a map $\Hh^1(X^c \smallsetminus X, \mu_n) \to \ZZ/n$.
We then show that this map descends to an analytic trace map $\ttr_X \colon \Hh^2_c(X, \mu_n) \to \ZZ/n$ via the exact excision sequence
\[ \Hh^1(X, \mu_n) \to \Hh^1(X^c \smallsetminus X, \mu_n) \to \Hh^2_c(X, \mu_n) \to 0 \]
and that the thus constructed analytic trace maps are compatible with \'etale morphisms and (algebraic) trace maps for algebraic curves. The verification of these claims is extremely subtle and occupies most of \cref{section:analytic-trace}.

We note that it is somewhat surprising that the trace map exists when $n$ is not invertible in $\O_Y^+$ due to the observation that, for a smooth connected affinoid rigid-analytic curve over an algebraically closed nonarchimedean field $C$, the top degree compactly supported cohomology group $\Hh^2_c(X, \mu_p)$ behaves pretty wildly. In fact, the group $\Hh^2_c(\DD^1_C, \mu_p)$ is already quite pathological due to the following observations: 

\begin{remark} Even though the truncated smooth trace $\ttr_{\DD^1_C} \colon \Hh^2_c(\DD^1_C, \mu_p) \to \bf{F}_p$ is still an epimorphism (see \cref{tr-epi}), it is certainly not an isomorphism in contrast to the situation in algebraic geometry (or when $n\in (\O_Y^+)^\times$). Furthermore, the trace map does not induce any kind of ``weak'' Poincar\'e duality in general (see \cref{rmk:no-Poincare-duality-closed-unit-disk}), the group $\Hh^2_c(\DD^1_C, \mu_p)$ is infinite (see \cref{distance and cycle class}), and depends on the choice of the algebraically closed ground field $C$ (see \cref{example:no-proper-base-change}). Moreover, different points $x, y\in \DD^1_C$ might have different cycle classes in $\Hh^2_c(\DD^1_C, \mu_p)$ (see \cref{distance and cycle class}) and 
$\Hh^2_c(\DD^1_C, \mu_p)$ is \emph{not} generated by cycle classes of points (see \cref{cycle class of points do not generate}).
As a consequence, many of the usual tricks familiar from algebraic geometry cannot be applied anymore.
\end{remark}

Once we have a trace morphism at hand, we can give an easy and essentially formal proof of Poincar\'e duality for smooth proper morphisms and locally constant coefficients;
many cases were treated before in \cite[Th.~7.3.1]{Berkovich}, \cite[Cor.~7.5.5]{Huber-etale}, \cite[Th.~1.1.2]{Z-thesis}, and \cite[Cor.~3.10.22]{Mann-thesis}: 

\begin{theorem}[{\Cref{Poincare dualizability theorem}, \cref{general smooth Poincare duality}}]
\label{Intro: smooth PD}
Let $f \colon X \to Y$ be a smooth proper morphism of equidimension $d$ between locally noetherian analytic adic spaces such that $n \in \cO^\times_Y$ and let $\cal{E} \in D_\lisse(X_\et; \Lambda)$.
Then the duality morphism
\[ \PD_f \colon \rR f_*\rR \cHom (\cE, \ud{\Lambda}_X(d)[2d]) \xlongrightarrow{\Ev_f} \rR\cHom(\rR f_*\cE,\rR f_* \ud{\Lambda}_X(d)[2d]) \xlongrightarrow{\ttr_f \circ \blank} \rR\cHom(\rR f_*\cE, \ud{\Lambda}_Y)
\]
is an isomorphism. %
\end{theorem}

When $n$ is invertible in $\O_Y^+$, \cref{Intro: smooth PD} was first proven by Berkovich in \cite[Th.~7.3.1]{Berkovich} and by Huber in \cite[Cor.~7.5.5]{Huber-etale} independently (and it was later revisited in \cite[Th.~1.3.2]{Z-revised}). When $X$ and $Y$ are rigid-analytic spaces over a nonarchimedean field $K$ of mixed characteristic $(0, p)$ and $n=p$, \cref{Intro: smooth PD} was proven in \cite{Z-thesis} and \cite{Mann-thesis} independently.
Both proofs crucially rely on the theory of perfectoid spaces, $\O^+/p$-cohomology groups, and the Grothendieck--Serre duality in characteristic $p$. In particular, the previous proofs only apply either when $(n,p) = 1$ or when $n = p$ and the strategies in the two cases are completely different. 

In contrast, our proof of \cref{Intro: smooth PD} is different from any of the four proofs mentioned above (instead, it is somewhat motivated by the proof presented in \cite[Th.~1.3.2]{Z-revised}).
It uses a bare minumum of the perfectoid theory, works uniformly for any integer $n \in \O_Y^\times$, and is essentially diagrammatic once the trace map is constructed. 

These methods are quite formal and, thus, they can be used in different cohomological setups. For instance, in an ongoing joint project with Nizio\l{}, we expect to generalize the techniques developed in this paper to prove a version of Poincar\'e duality for pro-\'etale $\QQ_p$-local systems on smooth proper $X$. 

We now explain the main ideas behind our proof. We first reduce the question to showing that, for a dualizable $\cal{E}\in D(X_\et; \Lambda)$, the complex $\rR f_* \cal{E}$ is dualizable in $D(Y_\et; \Lambda)$ with the dual given by $\rR f_*\cal{E}^\vee(d)[2d]$. Then we need to construct the evaluation and coevaluation morphisms and check that certain compositions are the identity.
The evaluation map essentially comes from the trace map constructed in \cref{Intro: smooth trace}, while the coevaluation map essentially comes from the K\"unneth isomorphism established in \cref{Kunneth formula}\footnote{
The proof of \cref{Kunneth formula} is the only place which needs, via \cite[Lem.~3.25]{BH}, the perfectoid machinery; see \cref{use-perfectoid}.}
and the cycle class map of the diagonal (see \cref{cycle classes section} and \cref{evaluation and coevaluation maps}).
The verification that certain compositions are equal to the identity boil down to the computation that $\ttr_{\pr}(c\ell_\Delta)=\id$ for a projection $\pr \colon X \times_Y X \to X$ and the diagonal morphism $\Delta \colon X \hookrightarrow X\times_Y X$ and to the fact that the braiding morphism on $\bigl(\ud{\Lambda}_{X\times_Y X}(d)[2d]\bigr)^{\otimes 2}$ is the identity morphism. 

As a formal consequence of our proof of \cref{general smooth Poincare duality}, we also get that derived pushforwards along smooth and proper morphisms preserve locally constant sheaves: 

\begin{corollary}[\cref{cor:preservation-lisse-sheaves}]\label{intro:preservation-lisse-sheaves}
Let $f \colon X \to Y$ be a smooth proper morphism of locally noetherian analytic adic spaces with $n \in \cO^\times_Y$. 
Let $\mathcal{E} \in D_\lisse(X_\et; \Lambda)$ be a lisse complex. Then $\rR f_*\mathcal{E}$ lies in $D_\lisse(Y_\et; \Lambda)$.
If $\mathcal{E}$ is locally bounded (resp.\ perfect), then so is $\rR f_*\mathcal{E}$.
\end{corollary}

If $n\in (\O_Y^+)^\times$, \cref{intro:preservation-lisse-sheaves} was first shown in \cite[Cor.~6.2.3]{Huber-etale} under some extra assumptions and also recently revisited in \cite[App.~1.3.4~(4)]{Z-revised} in full generality. If $Y$ is a rigid-analytic space over $\Spa(K, \O_K)$ and $n=p$ is equal to the characteristic of the residue field of $\O_K$, this result was shown in \cite[Th.~10.5.1]{Berkeley}, using the full strength of the perfectoid and diamond machinery. 
In contrast to these two proofs, our proof is uniform in $n$, is essentially formal, and remains largely in the world of locally noetherian analytic adic spaces.

\subsection{Trace and duality for proper maps}\label{section:intro-proper}

In this section, we go back to the discussion of \cref{Intro: proper PD} and explain the main ideas behind its proof. 
We fix a nonarchimedean field $K$ of characteristic $0$ and residue characteristic $p \ge 0$.
Set $\Lambda = \Z/n\Z$ for some integer $n>0$. 

For general (not necessarily smooth) proper morphisms of rigid-analytic spaces over $K$, one cannot expect to have trace maps with constant coefficients as in \cref{Intro: smooth trace}. Instead, it is more reasonable to expect trace maps with coefficients in the dualizing complexes $\omega_X \in D^b_\zc(X; \Lambda)$;
see \cite[Th.~3.21]{BH} for their definition.
One first main result is the actual construction of such morphisms: 

\begin{theorem}[{Digest of \cref{general proper trace subsection}}]
\label{Intro: proper trace}
For any proper morphism $f\colon X \to Y$ of rigid-analytic spaces over $K$, the complex $\rR \cHom(\rR f_*\omega_X, \omega_Y)$ lies in $D^{\geq 0}(Y_\et; \Lambda)$. Furthermore, one can assign to every such proper morphism $f \colon X \to Y$ a trace map $\tr_f \colon \rR f_*\omega_X \to \omega_Y$ such that:
\begin{enumerate}[label=\upshape{(\arabic*)}]
\item\label{Intro: proper trace-1} $\tr$ is compatible with compositions;
\item\label{Intro: proper trace-2} $\tr_f$ is \'etale local on $Y$; 
\item\label{Intro: proper trace-3} if $f$ is a closed immersion, then $\tr_f$ is adjoint to the natural isomorphism $\omega_X \xrightarrow{\sim} \rR f^! \omega_Y$ from \cite[Th.~3.21(1)]{BH} (see \cref{trace for closed immersion});
\item\label{Intro: proper trace-4} if $f$ is smooth and proper, $\tr_f$ is compatible with the smooth trace map from \cref{Intro: smooth trace} (via \cref{constructing smooth trace for omega});
\item\label{Intro: proper trace-5} $\tr$ is compatible with extensions of base fields.
\end{enumerate}
\end{theorem}

Just like in \cref{Intro: smooth trace}, we can modify the list of compatibility axioms in \cref{Intro: proper trace} to characterize our proper traces uniquely;
see \cref{most general proper trace!!}. We do not do so here in order to avoid complicated notations. 

We point out that the proof of \cref{Intro: proper PD} is not so difficult given the construction of trace maps from \cref{Intro: proper trace}. Indeed, \cref{Intro: proper trace}\cref{Intro: proper trace-3} eventually allows us to reduce to the case $Y=\Spa(K, \O_K)$. In this case, we use resolution of singularities, explicit generators in $D_\zc(X_\et; \Lambda)$, and the compatibilities of $\tr_f$ from \cref{Intro: proper trace} again to eventually reduce the question to the (very special case of) \cref{Intro: smooth PD}.  

\begin{remark} \Cref{Intro: smooth trace} suggests that it is reasonable to expect that one can assign a reasonable trace morphism $\tr_f\colon \rR f_!\, \omega_X \to \omega_Y$ to any taut separated morphism $f\colon X \to Y$ of rigid-analytic spaces over $K$. The main obstacle to apply our method for such $f$ is that we do not know whether $\rR\cHom(\rR f_!\, \omega_X, \omega_Y)$ lies in $D^{\geq 0}(Y_\et; \Lambda)$ in such generality (see \cref{rmk:general-coconnectivity} for more detail). 
It would certainly be interesting to have trace maps constructed in greater generality,
even though the analog of \cref{Intro: proper PD} cannot hold for nonproper $f$. Therefore, we decided not to pursue this direction in this paper. 
\end{remark}

Now we mention the main ideas that go into the proof of \cref{Intro: proper trace}. We first construct the trace map when $X$ is smooth and $Y$ is quasicompact and separated. In this case, we factor $f$ through its graph $\Gamma_f$ as 
\[ X \xhookrightarrow{\Gamma_f} X \times Y \xrightarrow{\pr_2} Y. \]
Since $\Gamma_f$ is a closed embedding by our separatedness assumption, it has a closed trace map thanks to \cite[Th.~3.21~(1)]{BH} (cf.\ \cref{trace for closed immersion}). On the other hand, $\pr_2$ is smooth taut separated, hence it has a trace map thanks to \cref{Intro: smooth PD} (cf.\ \cref{constructing smooth trace for omega}). The composition of these traces gives a trace map for $f$; we call it the \textit{smooth-source trace}. 

\begin{remark} We want to emphasize that the morphism $\pr_2$ is not (partially) proper unless $X$ is (partially) proper. For this reason, \cref{Intro: smooth trace} for partially proper morphisms is inadequate for the purpose of proving \cref{Intro: proper trace} (at least via our methods). 
\end{remark}

Then we use resolution of singularities, the constructed above smooth-source trace, and some d\'evissage to reduce the claim that $\rR \cHom(\rR f_*\omega_X, \omega_Y)$ lies in $D^{\geq 0}(Y_\et; \Lambda)$ to the case when $Y=\Spa(K, \O_K)$ and $X$ is smooth. In this case, this coconnectivity claim boils down to the classical estimates on the cohomological dimension of $X$ (see \cite[Prop.~5.5.8]{Huber-etale}). 

Finally, we use the coconnectivity established above to reduce the question to the case of quasicompact separated $Y$. Then we use resolution of singularities and the smooth-source trace again to reduce to the case when, in addition, $X$ is smooth of equidimension $d$. In this case, the smooth-source trace does the job again. After that, we have to figure out all the desired compatilibities which requires quite delicate arguments and diagram chases; this occupies the most of the proof of \cref{most general proper trace!!}. 

\subsection*{Organization of the paper}
In \cref{preliminaries section}, \cref{cycle classes section}, and \cref{section:curves}, we provide some technical background on finite morphisms, trace maps for finite flat morphisms, cycle classes, and curves in the setting of analytic adic geometry, for which we often could not find suitable references in the literature. We recommend the reader to skip these sections on the first reading. 
\Cref{section:analytic-trace} is one of the key sections of the paper, in which we construct analytic trace map for smooth affinoid rigid-analytic curves and verify its properties. The first half of \cref{smooth-traces} extends the construction of smooth traces to an arbitrary separated smooth taut morphism of equidimension $d$, while the second half of this section gives our ``diagrammatic'' proof of \cref{Intro: smooth PD}. This is another key novelty of this paper (an impatient reader may start reading there and take for granted the existence of smooth trace map). 
\Cref{proper-trace} discusses the construction of proper traces and general Poincar\'e duality for proper morphisms, leading up to the proofs of our main results \cref{Intro: proper trace} and \cref{Intro: proper PD}. 

\subsection*{Acknowledgments}
We would like to thank Bhargav Bhatt, David Hansen, Sasha Petrov, Peter Scholze, and Weizhe Zheng for useful conversations.

This material is based upon work supported by the National Science Foundation under Grant No.~DMS-1926686.
In addition, S.L.~was also supported by the National Natural Science Foundation of China (No.~12288201) and during a stay at the Institute for Advanced Study by the Ky Fan and Yu-Fen Fan Endowed Fund.
We are grateful to the following institutions for their financial support and for providing excellent working conditions while we completed parts of this project:
the School of Mathematics of the Institute for Advanced Study (S.L., E.R., B.Z.), the University of Michigan (S.L.), the Morningside Center of Mathematics (S.L.), the Max-Planck-Institut f\"ur Mathematik in Bonn (E.R., B.Z.), the Institut des Hautes \'Etudes Scientifiques (E.R.), and Princeton University (B.Z.).

\subsection*{Notation and conventions}
In this paper, a topological field $K$ is called \emph{nonarchimedean} if its topology is induced by a valuation of rank $1$ on $K$ and if $K$ is complete with respect to this topology;
beware that while the completeness assumption is somewhat standard, it is not imposed in \cite[Def.~1.1.3]{Huber-etale}.
We usually denote the ring of integers of $K$ by $\cO_K$, its maximal ideal by $\fm_K$, and its residue field by $k \colonequals \cO_K/\fm_K$.

A \emph{rigid-analytic space} over a nonarchimedean field $K$ is always understood to be an adic space which is locally of finite type over  $\Spa(K,\cO_K)$;
by \cite[Prop.~4.5]{Huber-2}, the resulting category of quasiseparated rigid-analytic $K$-spaces is equivalent to the category of quasiseparated rigid-analytic $K$-varieties in the classical sense as in, say, \cite[Def.~9.3.1/4]{BGR}.
An analytic adic space is \emph{locally noetherian} if every $x \in X$ is contained in an open affinoid subspace $U \subset X$ for which $\cO(U)$ is a strongly noetherian Tate ring.
An \emph{affinoid field} means a Huber pair $(k,k^+)$ such that $k$ is a field and $k^+\subset k$ is an open and bounded microbial valuation ring;
it is not assumed to be complete.
Note that this definition is slightly narrower than \cite[Def.~1.1.5]{Huber-etale}, which also allows $k$ to be discrete.

An \emph{admissible} formal $\cO_K$-scheme is a flat, locally finitely presented formal $\cO_K$-scheme $\cX$.
To any formal scheme $\cX$ which is locally of finite presentation over $\Spf(\cO_K)$, we attach the \emph{special fiber} $\cX_s \colonequals \cX \times_{\Spf(\cO_K)} \Spec(k)$, which is a locally finitely presented $k$-scheme, and the \emph{rigid-analytic generic fiber} $\cX_\eta$ in the sense of \cite[Prop.~1.9.1]{Huber-etale}, which is a quasiseparated rigid-analytic space over $\Spa(K,\cO_K)$.
Given a rigid-analytic space $X$ over $\Spa(K,\cO_K)$, any admissible $\cX$ over $\Spf(\cO_K)$ with $\cX_\eta \simeq X$ is called a \emph{formal model} of $X$. 

We usually denote the points of an adic space $X$ by $x$, $y$, etc.\ and valuations in their equivalence class by $v_x$, $v_y$, etc.
Following Huber, we use multiplicative notation for valuations and value groups.

Given a locally noetherian analytic adic space $X$, the $d$-dimensional \emph{(closed) unit disk} over $X$ is defined as $\DD^d_X \colonequals \Spa(\ZZ[T_1,\dotsc,T_d],\ZZ[T_1,\dotsc,T_d]) \times_{\Spa(\ZZ,\ZZ)} X$.
Alternatively, one can set $\DD^d_X \colonequals \Spa(A\langle T \rangle ,A\langle T \rangle^+)$ when $X = \Spa(A,A^+)$ is affinoid;
since this construction is functorial in $X$, this gives $\DD^d_X$ for general $X$ via gluing.
Likewise, the $d$-dimensional \emph{open unit disk} over $X$ is $\accentset{\circ}{\DD}^d_X \colonequals \Spa(\ZZ\llbracket T_1,\dotsc,T_d \rrbracket,\ZZ\llbracket T_1,\dotsc,T_d \rrbracket) \times_{\Spa(\ZZ,\ZZ)} X$ and the $d$-dimensional \emph{affine space} over $X$ is $\AA^{d,\an}_X \colonequals \Spa(\ZZ[T_1,\dotsc,T_d],\ZZ) \times_{\Spa(\ZZ,\ZZ)} X$.
When the base is understood from the context (mainly over an algebraically closed field), we drop it from the notation and simply write $\DD^d$, $\accentset{\circ}{\DD}^d$, $\AA^{d,\an}$, etc.

Given a locally noetherian analytic adic space $X$ and a finite commutative ring $\Lambda$, we denote the (triangulated) derived category of \'etale sheaves of $\Lambda$-modules on $X$ by $D(X_\et;\Lambda)$.
We write $D^{(b)}(X_\et;\Lambda)$ for the full subcategory spanned by locally bounded complexes and $D_\lisse(X_\et;\Lambda)$ for the full subcategory spanned by complexes with lisse cohomology sheaves.
When $X$ is a rigid-analytic space over $K$, we also consider the full subcategory $D_\zc(X_\et;\Lambda)$ spanned by complexes with Zariski-constructible cohomology sheaves.
If $f \colon X \to Y$ is a morphism of locally noetherian analytic adic spaces and $f_*$ (for $f$ finite) or $f_!$ (for $f$ \'etale) is exact, we often drop the ``$\rR$'' from the notation of the associated derived functors $\rR f_*$ or $\rR f_!$, respectively.

We denote the unit of an adjunction of functors $F \colon \mathcal{C} \rightleftarrows \mathcal{D} \lon G$ by $\eta \colon \id \to G \circ F$ and its counit by $\epsilon \colon F \circ G \to \id$.
In the special case when $f \colon X \to Y$ is a morphism of locally noetherian analytic adic spaces and $(F,G) = (f^*,\rR f_*)$ or (for $f$ \'etale) $(F,G) = (\rR f_!,f^*)$ or (for $f$ finite) $(F,G) = (f_*,\rR f^!)$, we also use $\eta_f$ and $\epsilon_f$.
We denote \emph{canonical} isomorphisms by $\simeq$ and \emph{noncanonical} isomorphisms by $\cong$.

\addtocontents{toc}{\protect\setcounter{tocdepth}{2}}

\section{Preliminaries}\label{preliminaries section}

\numberwithin{equation}{subsection}

\subsection{Valuation rings}

In this subsection, we collect some facts about valuation rings that we will use throughout the paper. We assume that most results of this subsection are well-known to the experts, but it seems difficult to extract them from the existing literature.  

We start by recalling the following classical definition: 

\begin{definition}\label{defn:residue-fields-1} Let $x\in X$ be a point of an analytic adic space $(X, \O_X, \{v_x\})$. Then 
\begin{enumerate}[label=\upshape{(\roman*)}]
    \item the \emph{residue field} $k(x)$ is the residue field of the local ring $\O_{X, x}$ and $k(x)^+\subset k(x)$ is a valuation ring associated to the valuation $v_x$ of $\O_{X, x}$. Then $\bigl(k(x), k(x)^+\bigr)$ is a (non-complete)  affinoid field;
    \item the \emph{completed residue field} $\wdh{k(x)}$ is the topological completion of $k(x)$. This comes with a canonical valuation subring $\wdh{k(x)}^+\subset \wdh{k(x)}$ such that pair $\bigl(\wdh{k(x)}, \wdh{k(x)}^+\bigr)$ is a (complete) affinoid field.
\end{enumerate}
\end{definition}

\begin{remark} If $\varpi\in k(x)^+$ is a pseudo-uniformizer, then \cite[Lem.~1.6]{H0} ensures that $\wdh{k(x)}^+$ is the usual $\varpi$-adic completion of $k(x)^+$ and $\wdh{k(x)} = \wdh{k(x)}^+\bigl[\frac{1}{\varpi}\bigr]$.
\end{remark}

We first classify all connected adic spaces which are finite over the adic spectrum of a complete affinoid field: 

\begin{lemma}\label{lemma:finite-over-point} Let $Y=\Spa(K, K^+)$ be the adic spectrum of a complete affinoid field $(K, K^+)$. Let $y$ be the closed point of $Y$ corresponding to a valuation $v\colon K \to \Gamma_v\cup \{0\}$, $\m\subset K^+$ the maximal ideal.
Let $X$ be a reduced and connected adic space and let $f\colon X \to Y$ be a finite morphism. Then 
\begin{enumerate}[label=\upshape{(\roman*)}]
    \item\label{lemma:finite-over-point-1} $X=\Spa(A, A^+)$ is an affinoid space with $A=L$ being a finite field extension of $K$, and $A^+$ being the integral closure\footnote{We warn the reader that $A^+$ is not necessarily a valuation ring unless 
    $K^+$ is henselian along its maximal ideal.} of $K^+$ in $L$;
    \item\label{lemma:finite-over-point-2} the pre-image $f^{-1}(y)$ is equal to the set of valuations of $L$ that extend $v$;
    \item\label{lemma:finite-over-point-3} for each $x\in f^{-1}(y)$, we have $\rk x = \rk y$;
    \item\label{lemma:finite-over-point-4} there is an equality $A^+=\cap_{x\in f^{-1}(y)} L_x^+$, where $L_x^+$ is the valuation ring of (the valuation corresponding to) $x$, i.e., $L_x^+ = \{a\in L \suchthat v_x(a)\leq 1\}$. 
    \item\label{lemma:finite-over-point-5} for each $x\in f^{-1}(y)$, we have $\wdh{k(x)}^+=L_x^+$.
    \item\label{lemma:finite-over-point-6} $A^+$ is semi-local, and all maximal ideals are given by $\m_x\coloneqq \m_x^+\cap A^+$ for $x\in f^{-1}(y)$ and $\m_x^+\subset L_x^+$ the corresponding maximal ideal. Furthermore, the natural morphism $A^+_{\m_x} \to L_x^+$ is an isomorphism for each $x\in f^{-1}(y)$;
    \item\label{lemma:finite-over-point-7} we have $\rm{rad}(\m A^+)=\bigcap_{x\in f^{-1}(y)} \m_x$ .
\end{enumerate}
\end{lemma}
\begin{proof}
    \cref{lemma:finite-over-point-1}. First, we note that \cite[Satz 3.6.20 and Korollar 3.12.12]{Huber-thesis} imply that $X=\Spa(A, A^+)$ is an affinoid and $(K, K^+) \to (A, A^+)$ is a finite morphism of Huber pairs, i.e., $A$ is a finite $K$-algebra and $A^+$ is the integral closure of $K^+$ in $A$. Furthermore, the assumptions on $X$ imply that $A$ is a reduced finite $K$-algebra without idempotents. This implies that $A$ must be a field $L$ such that $K\subset L$ is a finite extension.  
    
    \cref{lemma:finite-over-point-2}. Now we note that, by definition of an adic space, we can identify $f^{-1}(y)$ with the set of valuation subrings $R_w\subset L$ such that 
    \begin{enumerate}[label=\textbf{(\Alph*)}]
        \item $K^+\subset R_w$ and the morphism $K^+ \to R_w$ is local;
        \item the corresponding valuation $w\colon L \to L^\times/R_w^\times \cup \{0\} = \Gamma_w \cup \{0\}$ is continuous;
        \item $w(A^+)\leq 1$.
    \end{enumerate}
    Therefore, for the purpose of proving \cref{lemma:finite-over-point-2}, it suffices to show that condition $\textbf{(A)}$ implies $\textbf{(B)}$ and $\textbf{(C)}$. In other words, we need to show that any valuation $w$ of $L$ that extends $v$ is automatically continuous and satisfies $w(A^+)\leq 1$.  
    
    Since $w|_{K} = v$, $v(K^+)\leq 1$, and $A^+$ is integral over $K^+$, we conclude that $w(A^+)\leq 1$ as well. Therefore, it suffices to show continuity of $w$. We choose some compatible rings of definition $K_0 \subset K$, $A_0 \subset A^+$, and a pseudo-uniformizer $\varpi \in K_0$. Therefore, \cite[L.~9, Cor.~9.3.3]{Seminar} ensures that it suffices to show that $w(\varpi)$ is cofinal in $\Gamma_w$ and $w(\varpi) < w(a)$ for any $a\in A_0$.   

    First, \cite[Ch.~VI, \S~8, n.~1, Prop.~1]{Bourbaki} gives that $\Gamma_w/\Gamma_v$ is a torsion group. Therefore $w(\varpi)=v(\varpi)$ is cofinal in $\Gamma_w$ since it is cofinal in $\Gamma_v$ due to \cite[L.~9, Cor.~9.3.3]{Seminar}. In particular, $w(\varpi) < 1$. Now we note that $w(A^+)\leq 1$, $\restr{v}{K^+}=\restr{w}{K^+}$, and thus 
    \[
    w(a\varpi)=w(a)w(\varpi)<w(a)\leq 1
    \]
    for any $a\in A_0 \subset A^+$. Therefore, we conclude that $w$ is continuous.  

    \cref{lemma:finite-over-point-3}. This follows directly from \cite[Ch.~VI, \S~8, n.~1, Cor.~1]{Bourbaki}.  
    
    \cref{lemma:finite-over-point-4}. This follows from \cref{lemma:finite-over-point-2} and \cite[Exercise 10.3]{Matsumura}.  
    
    \cref{lemma:finite-over-point-5}. We first note, for every $x\in f^{-1}(y)$, \cite[L.~14, ``Caveat on residue fields'' on pp.~2--3]{Seminar} implies that $\wdh{k(x)}=\wdh{L/\supp(x)}= \wdh{L}=L$ since $\supp(x)=(0)$ and $L$ is already a complete field. Therefore, $L_x^+$ and $\wdh{k(x)}^+$ are both equal to the valuation rings defined as $\{a\in L=\wdh{k(x)} \suchthat v_x(a)\leq 1\}$.  
    
    \cref{lemma:finite-over-point-6}. This follows directly from  \cref{lemma:finite-over-point-2} and \cite[Ch.~VI, \S~8, n.~6, Prop.~6]{Bourbaki}.  
    
    \cref{lemma:finite-over-point-7}. Since the ideal $\bigcap_{x\in f^{-1}(y)} \m_x$ is radical, the equality $\rm{rad}(\m A^+)=\bigcap_{x\in f^{-1}(y)} \m_x$ means that the fiber of $\Spec A^+ \to \Spec K^+$ over the closed point consists exactly of the closed points in $\Spec A^+$. This follows from \cite[Th.~9.3~(ii) and Th.~9.4~(i)]{Matsumura}
\end{proof}

Now we discuss the definition and basic properties of henselian affinoid fields.  

\begin{definition}\label{defn:henselian}
    An affinoid field $(K, K^+)$ is \emph{henselian} if $K^+$ is henselian with respect to its maximal ideal $\m\subset K^+$.
\end{definition}

\begin{remark}\label{rmk:equivalent-defns-henselian} We note that \cite[Th.~4.1.3]{ValuedFields} implies that $(K, K^+)$ is henselian in the sense of \cref{defn:henselian} if and only if it is henselian in the sense of \cite[p.~86]{ValuedFields}. In other words, $(K, K^+)$ is henselian if and only if its valuation $v$ uniquely extends to any finite field extension $K\subset L$.
\end{remark}

It turns out that we can always canonically make any affinoid field into a henselian one. 

\begin{definition}
\label{defn:henselize-field} 
    Let $(K, K^+)$ be an affinoid field. Its \emph{henselization} is an affinoid field $\bigl(K^{\h}, K^{+, \h}\bigr)$ where $K^{+, \h}$ is the henselization of $K^+$ with respect to its maximal ideal and $K^\h=K^{+, \h}\otimes_{K^+} K$.
\end{definition}

We note that $K^{+, \h}$ is a valuation ring by \cite[\href{https://stacks.math.columbia.edu/tag/0ASK}{Tag 0ASK}]{stacks-project}, and it is microbial since $K^+$ is. Thus $(K^\h, K^{+, \h})$ is indeed an affinoid field. 

\begin{warning} The notation $K^\h$ may be a bit misleading because this object depends on the valuation subring $K^+\subset K$ and not just on $K$ as a topological field.
\end{warning}

\begin{definition}\label{defn:residue-fields-2}
Let $x\in X$ be a point of an analytic adic space $(X, \O_X, \{v_x\})$. Then 
\begin{enumerate}[label=\upshape{(\roman*)}]
    \item the \emph{henselized residue field} $k(x)^{\rm{h}}$ is the henselization of $k(x)$. In particular, $\bigl(k(x)^\h, k(x)^{+, \h}\bigr)$ is a (henselian)  affinoid field;
    \item the \emph{henselized completed residue field} $\wdh{k(x)}^h$ is the henselization of $\wdh{k(x)}$ (see \cref{defn:residue-fields-1}). In particular, $\bigl(\wdh{k(x)}^\h, \wdh{k(x)}^{+, \h}\bigr)$ is a (henselian) affinoid field.
\end{enumerate}
\end{definition}

The following lemma is well-known in the rank-$1$ case. Even though it is probably also well-known in the higher rank case to the experts, we cannot find this explicitly stated in the literature and therefore include a proof here. 

\begin{lemma}\label{lemma:norms-compatible-with-valuations} Let $(K, K^+)$ be a henselian affinoid field with the valuation $v_K\colon K \to \Gamma_K \cup \{0\}$, and $i_{K/L}\colon K\hookrightarrow L$ a finite field extension with the (unique) compatible valuation $v_L\colon L \to \Gamma_L \cup \{0\}$. Then\footnote{Recall that we use the multiplicative notation for the group structure on any value group $\Gamma$.} 
\[
\left(v_L(-)\right)^{[L:K]} = v_L\left(i_{K/L}\rm{Nm}_{L/K}(-)\right),
\]
where $\rm{Nm}_{L/K}\colon L^\times \to K^\times$ is the norm map.
\end{lemma}
\begin{proof}
\begin{enumerate}[wide,label={\textit{Step~\arabic*}.},ref={Step~\arabic*}]
    \item\label{lemma:norms-compatible-with-valuations-normal} \textit{We assume that $L/K$ is normal.}
    We pick an element $f\in L^\times$. Then \cite[Ch.~5, \S~8, n.~3, Prop.~4]{Bourbaki-algebra-4-7} implies that we have
    \[
    \rm{Nm}_{L/K}(f) =\Bigl(\prod_{\sigma\in \rm{Aut}(L/K)} \sigma(f)\Bigr)^{[L:K]_i},
    \]
    where $[L:K]_i$ is the inseparable degree extension of $L/K$. Since $K$ is henselian, there is a unique extension of $v_K$ to $v_L$, so we conclude that $v_{L}(\sigma(f))=v_L(f)$ for any $\sigma\in \rm{Aut}(L/K)$. Therefore, we conclude that
    \begin{align*}
    v_L\left(i_{K/L}\rm{Nm}_{L/K}(-)\right) &= v_L\Biggl(\Bigl(\prod_{\sigma\in \rm{Aut}(L/K)} \sigma(f)\Bigr)^{[L:K]_i}\Biggr)\\
    & = \Biggl( v_L\Bigl(\prod_{\sigma\in \rm{Aut}(L/K)} \sigma(f)\Bigr)\Biggr)^{[L:K]_i} \\
    & = \left(v_L(f)\right)^{[L:K]_i[L:K]_s} = \left(v_L(f)\right)^{[L:K]}. 
    \end{align*}
    
    \item \textit{General case.}
    We consider a normal closure $L\subset M$ of $L$, so $M/K$ is normal. We denote by $i_{K/L}\colon K \to L$  the corresponding inclusion (and similarly for $i_{K/M}$ and $i_{L/M}$), and by $j_{K/L}\colon \Gamma_K \to \Gamma_L$ the induced morphism on the value groups (and similarly for $j_{K/M}$ and $j_{L/M}$).  

    Since $M/K$ is normal, we already know that
    \begin{align}
    \label{eqn:induction-nm} \tag{\epsdice{1}}
    \begin{split}
     \left(j_{L/M}\left(v_{L}\left(f\right) \right)\right)^{[L:K]}& = \left(v_M\left(i_{L/M}(f)\right)\right)^{[M:K] } \\
    & = v_M\left(i_{K/M}\rm{Nm}_{M/K}i_{L/M}(f)\right) \\
    & = v_M\left(i_{L/M}i_{K/L} \rm{Nm}_{L/K}\rm{Nm}_{M/L}i_{L/M}(f)\right) \\
    & = v_M\left(i_{L/M}i_{K/L} \rm{Nm}_{L/K} \left(f^{[M:L]}\right)\right) \\
    & = \left(j_{L/M} v_{L}\left(i_{K/L}\rm{Nm}_{L/K}(f)\right)\right)^{[M:L]}.
    \end{split}
    \end{align}
    Indeed, the first equality is formal. The second equality follows from \cref{lemma:norms-compatible-with-valuations-normal} applied to $M/K$ and $i_{L/M} f$. The third equality follows from the transitivity of Norm maps and the inclusion morphisms. The fourth equality follows from the formula $\rm{Nm}_{M/L}i_{L/M}f=f^{[M:L]}$. 
    The last equality is again formal.  

    Now we note that the morphism $j_{L/M}\colon \Gamma_L \to \Gamma_M$ is injective, and both $\Gamma_L$ and $\Gamma_M$ are torsion-free (since they are totally ordered abelian groups). Therefore, \cref{eqn:induction-nm} implies that $\left(v_{L}(f)\right)^{[L:K] }=v_L\left(i_{K/L}\rm{Nm}_{L/K}(f)\right)$. \qedhere
\end{enumerate}
\end{proof}

\subsection{Curve-like affinoid fields}

In this subsection, we define and study a particular class of curve-like affinoid fields. We will later show that ``boundary'' points on the universal compactification of a rigid-analytic curve are necessarilly curve-like (see \cref{lemma:extra-points-curve-like}). Throughout the subsection, we fix a nonarchimedean field $C$ with a rank-$1$ valuation $\abs{.}\colon C\to \Gamma_C \cup \{0\}$. We denote by $\O_C \subset C$ the corresponding valuation ring and by $\m_C \subset \O_C$ its maximal ideal.  

For the following definition, we fix a henselian affinoid field $(K, K^+)$ and a finite field extension $K\subset L$. \cref{rmk:equivalent-defns-henselian} and \cite[Exercise 10.3]{Matsumura} ensure that the integral closure $L^+$ of $K^+$ in $L$ is a henselian valuation ring. So we denote by $\m_K\subset K^+$ and $\m_L\subset L^+$ the unique maximal ideals of $K^+$ and $L^+$ respectively.

\begin{definition}[{\cite[Ch.~6, \S~8, n.~1]{Bourbaki}}] 
The \textit{ramification index} $e(L/K)$ is the cardinality $\abs{\Gamma_L/\Gamma_K}$. 
The \textit{residue class} $f(L/K)$ is the degree $[L^+/\m_{L}:K^+/\m_{K}]$.
\end{definition}

\begin{remark} Note that \cite[Ch.~6, \S~8, n.~1, Lem.~2]{Bourbaki} implies that we have an inequality $e(L/K)f(L/K) \leq [L:K]$. 
In particular, both $e(L/K)$ and $f(L/K)$ are finite numbers.
\end{remark}

\begin{definition} \cite{Swan}
    A henselian affinoid ring $(K, K^+)$ is \emph{defectless in every finite extension} if, for every finite field extension $K\subset L$, we have $e(L/K)f(L/K) = [L:K]$.
\end{definition}

Finally, we are essentially ready to define the notion of a curve-like affinoid field. 

\begin{definition}
    A \emph{$(C, \O_C)$-affinoid field} is an affinoid field $(K, K^+)$ with a continuous morphism $(C, \O_C) \to (K, K^+)$.
\end{definition}

For any $(C, \O_C)$-affinoid field, we have the natural induced morphism $j_K \colon \Gamma_C \to \Gamma_K$ of valuation groups. 

\begin{definition}\label{defn:curve-like}
    A $(C, \O_C)$-affinoid field $(K, K^+)$ is called \emph{curve-like} if 
\begin{enumerate}
    \item $(K, K^+)$ is henselian and defectless in every finite extension;
    \item the set\footnote{The element $1$ in the next formula means the neutral element of the group $\Gamma_K$.} $\{\gamma \in \Gamma_K \suchthat \gamma<1\}$ has a greatest element $\gamma_0$;
    \item $\Gamma_K$ is generated (as an abelian group) by $j_K(\Gamma_C)$ and the element $\gamma_0$. 
\end{enumerate}
\end{definition}

The following lemma (in conjunction with \cref{lemma:extra-points-curve-like}) will be at the heart of our construction of the analytic trace map (see \cref{defn:analytic-trace-curves}): 

\begin{lemma}\label{lemma:structure-curve-like-valuations} Let $(K, K^+)$ be a curve-like $(C, \O_C)$-affinoid field, and let $\Gamma_C\times \Z$ be a totally ordered abelian group with the lexicographical order. Then the natural morphism
\[
\alpha \colon \Gamma_C \times \Z \to \Gamma_K
\]
\[
\alpha(\gamma, n) = j_K(\gamma) \cdot \gamma_0^{-n}
\]
is an isomorphism of totally ordered abelian groups. 
\end{lemma}
\begin{proof}
    In this proof, we will freely use \cite[Obs.~3.6/10]{BGR} which guarantees that $\Gamma_C$ is divisible. We will also denote by $\langle \gamma_0\rangle \subset \Gamma_K$ the subgroup generated by $\gamma_0$. Since $\Gamma_K$ is torsion-free, we conclude that $\langle \gamma_0 \rangle$ is isomorphic to $\Z$ as abelian groups.  
    \begin{enumerate}[wide,label={\textit{Step~\arabic*}.},ref={Step~\arabic*}]
        \setcounter{enumi}{-1}
        \item\label{lemma:structure-curve-like-valuations-injective} \textit{The natural morphism $j_K\colon \Gamma_C \to \Gamma_K$ is injective.}
        Explicitly, we need to show that the natural morphism $C^\times/\O_C^\times \to K^\times/(K^+)^\times$ is injective. Since every element in $C^\times/\O_C^\times$ is either equal to the class of $\ov{\pi}$ or $\ov{\pi}^{-1}$ for some pseudo-uniformizer $\pi\in \O_C$, it suffices to show that the image of $\ov{\pi}$ is non-zero in $K^\times/(K^+)^\times$.
        Equivalently, we need to show that no pseudo-uniformizer $\pi\in \O_C$ becomes invertible in $K^+$.
        But since the morphism $C\to K$ is continuous, $\pi\in K^+$ is a topologically nilpotent.
        In particular, $\pi$ lies in the maximal ideal $\m_{K^+}$, so it is not invertible.  
    
        \item\label{lemma:structure-curve-like-valuations-intersection} \textit{We have $j_K(\Gamma_C) \cap \langle \gamma_0\rangle = \{1\}$}.
        Suppose that there is an element $1\neq \gamma \in j_K(\Gamma_C) \cap \langle \gamma_0\rangle$. Without loss of generality, we can assume that $\gamma=\gamma_0^n$ for some \emph{positive} integer $n$. 
        Since $\Gamma_C$ is divisible, there is $\gamma' \in j_K(\Gamma_C)\subset \Gamma_K$ such that $(\gamma')^{2n}=\gamma_0^n$.
        Since $\Gamma_K$ is torsion-free, we conclude that $\gamma_0=(\gamma')^2$. 
        In particular, $\gamma_0< \gamma'<1$.
        This contradicts the assumption that $\gamma_0$ is the greatest among the elements $<1$.  
    
        \item\label{lemma:structure-curve-like-valuations-abgrp} \textit{The map $\alpha$ is an isomorphism of abelian groups}.
        Our assumption on $K$ implies that $\alpha$ is surjective, while \cref{lemma:structure-curve-like-valuations-injective} and \cref{lemma:structure-curve-like-valuations-intersection} ensure that it is injective.   
    
        \item\label{lemma:structure-curve-like-valuations-convex} \textit{For any $\gamma\in j_K(\Gamma_C)$ such that $\gamma>1$, we have $\gamma>\gamma_0^N$ for any integer $N$.}
        We argue by contradiction. 
        Suppose there exists $\gamma\in j_K(\Gamma_{C, >1})$ and an integer $N$ such that $\gamma_0^N\geq \gamma$. Since $\gamma_0<1$ and $\gamma>1$, we see that $N<0$. Thus, we can write $N=-n$ for some positive integer $n$.
        Using that $\Gamma_C$ is divisible, we can then find $\gamma'\in j_K(\Gamma_C)$ such that $\gamma=(\gamma')^{-n}$.
        The inequality $\gamma_0^{-n} \geq \gamma = (\gamma')^{-n}>1$ implies that $\gamma_0 \leq \gamma' <1$. The choice of $\gamma_0$ implies that $\gamma_0=\gamma' \in j_K(\Gamma_C)$, but this is impossible due to \cref{lemma:structure-curve-like-valuations-intersection}.  
    
        \item \textit{The map $\alpha$ is an isomorphism of ordered abelian groups.}
        By \cref{lemma:structure-curve-like-valuations-abgrp}, it suffices to show that the subgroup $\langle \gamma_0\rangle\subset \Gamma_K$ is convex.
        This means that if $\gamma\in \Gamma_K$ satisfies $\gamma_0^n\leq \gamma<\gamma_0^m$ for some integers $n$ and $m$, then $\gamma\in \langle \gamma_0\rangle$. Since we already know that $\Gamma_K=\Gamma_C\times \langle \gamma_0\rangle$ as an abelian group, it suffices to show that the only element $\gamma \in j_K(\Gamma_C)$ satisfying
        \begin{equation}\label{eqn:convex-subgroup}
        \gamma_0^n \leq \gamma \leq \gamma_0^m
        \end{equation}
        for some integers $n$ and $m$ is the neutral object $1$. By passing to inverses, we can assume that $\gamma\geq 1$. But then it follows directly from \cref{lemma:structure-curve-like-valuations-convex}. \qedhere
    \end{enumerate}
\end{proof}

\begin{definition}\label{defn:sharp-map} For a curve-like $(C, \O_C)$-affinoid field $(K, K^+)$, we define the \emph{reduction morphism}
\[
\# \colon \Gamma_K \to \Z
\]
to be the unique homomorphism that sends $\gamma_0$ to $1$ and $\Gamma_C$ to $0$. 
\end{definition}

\begin{warning}\label{warning:sign-convention} Note that $\#$ coincides with the composition $-\rm{proj}_2\circ \alpha^{-1}$, where $\rm{proj}_2\colon \Gamma_C \times \Z \to \Z$ is the projection onto the second factor. In particular, $\#(\alpha(0, 1))=-1$.    
\end{warning}

\begin{lemma}\label{lemma:norm-commute} Let $C$ be an algebraically closed nonarchimedean field, and $(K, K^+) \subset (L, L^+)$ be a finite extension of curve-like $(C, \O_C)$-affinoid fields. If the residue field $K^+/\m_K$ is algebraically closed, then the diagram 
    \[
    \begin{tikzcd}[column sep = huge, row sep = huge]
    L^\times \arrow{d}{\rm{Nm}_{L/K}} \arrow{r}{\#\circ v_L} & \Z \\
    K^\times \arrow[ru, swap, "\#\circ v_K"]
    \end{tikzcd}
    \]
    commutes.
\end{lemma}
\begin{proof}
    Let us choose minimal elements $\gamma_{L, 0}\in \{\gamma\in \Gamma_L \suchthat \gamma<1\}$ and $\gamma_{K, 0} \in \{\gamma\in \Gamma_K \suchthat \gamma<1\}$ respectively. Then \cref{lemma:structure-curve-like-valuations} implies that
    \[
    \Gamma_L = \Gamma_C \times \langle \gamma_{L, 0}\rangle, \quad \Gamma_K = \Gamma_C \times \langle \gamma_{K, 0}\rangle
    \]
    with the lexicographic order. Thus, we have a commutative diagram
    \begin{equation}\label{eqn:inclusion-sharp}
    \begin{tikzcd}
    L^\times \arrow{r}{v_L} & \Gamma_L =\Gamma_C \times \langle \gamma_{L, 0}\rangle \arrow{r}{\sharp} & \Z \\
    K^\times \arrow{u}{i_{K/L}} \arrow{r}{v_K} & \Gamma_K= \Gamma_C \times \langle \gamma_{K, 0}\rangle\arrow{u}{j_{K/L}} \arrow{r}{\sharp} & \Z\arrow{u}{\cdot e_{L/K}}, 
    \end{tikzcd}
    \end{equation}
    where $j_{K/L}$ is the morphism of value groups induced by the inclusion $K\subset L$ and $e_{L/K}=\abs{\Gamma_{L}/\Gamma_K}$ is the ramification index of $L/K$. Since $\Z$ is torsion-free, it suffices to show that
    \[
    e_{L/K} \cdot \#\circ v_L(f) = e_{L/K}\cdot \#\circ v_K(\rm{Nm}_{L/K}(f))
    \]
    for any $f\in L^\times$. Therefore, \cref{eqn:inclusion-sharp} implies that it suffices to show that 
    \[
    e_{L/K} \cdot \#\circ v_L(f) = \#\circ v_L\left(i_{K/L}\rm{Nm}_{L/K}(f)\right).
    \]
    Now we show an even stronger\footnote{We recall again that we use the multiplicative notation for the group structure on any value group $\Gamma$.} claim that $v_L(f)^{e_{L/K} }=v_L\left(i_{K/L}\rm{Nm}_{L/K}(f)\right)$. Since $(K, K^+)$ is defectless in every finite extension and the residue field $K^+/\m_K$ is algebraically closed, we conclude that $f_{L/K}=1$ and $[L:K]=e_{L/K}$. Therefore, \cref{lemma:norms-compatible-with-valuations} implies that 
    \[
    v_L(f)^{e_{L/K}}=v_L\left(i_{K/L}\rm{Nm}_{L/K}(f)\right)
    \]
    for any $f\in L^\times$.
\end{proof}

\subsection{Finite morphisms and residue fields}

In this subsection, we record some results about the behaviour of various residue field (see \cref{defn:residue-fields-1} and \cref{defn:residue-fields-2}) with respect to finite morphisms. We expect that some of these results are probably well-known to the experts, but they do seem to appear in the existing literature.  

That being said, we first study the behavior of the completed residue fields with respect to finite morphisms.
We establish nice properties when $x$ is a point of rank-$1$.
To deal with higher rank points, we will need to pass to the henselized completed residue fields later in this subsection. 

\begin{lemma}\label{lemma:finite-same-rank} Let $f\colon X\to Y$ be a finite morphism of locally noetherian analytic adic spaces, and $x\in X$. Then $x$ and $y=f(x)$ have equal ranks. 
\end{lemma}
\begin{proof}
    Without loss of generality, we can assume that $Y=\Spa\bigl(\widehat{k(y)}, \widehat{k(y)}^+\bigr)$. Then we can replace $X$ by its reduction, and then pass to connected components to assume that $X$ is reduced and connected. In this case, the result follows from \cref{lemma:finite-over-point}\cref{lemma:finite-over-point-3}.
\end{proof}

\begin{lemma}\label{lemma:many-open-around-rk-1} Let $f\colon X \to Y$ be a finite morphism of locally noetherian analytic adic spaces, $y\in Y$ a rank-$1$ point, and $V\subset X$ an open subset of $X$ containing $f^{-1}(y)$. Then there is an open $U\subset Y$ containing $y$ such that $f^{-1}(U) \subset V$.
\end{lemma}
\begin{proof}
    The question is local on $Y$, so we can assume that $Y$ is an affinoid. Since $f$ is a finite morphism, we conclude that $X$ is also affinoid due to \cite[Korollar 3.12.12]{Huber-thesis}. In particular, both underlying topological spaces $\abs{X}$ and $\abs{Y}$ are spectral. Since $f^{-1}(y)$ is a finite set, we can refine $V$ to assume that it is  quasi-compact. In particular, the complement $Z\coloneqq X\smallsetminus V$ is a constructible subset of $X$.  
    
    Now we note that any rank-$1$ point in $Y$ is maximal due to \cite[Lem.~1.1.10~(ii)]{Huber-etale}, so $y = \cap_{i\in I} U_i$, where $\{U_i\}_{i\in I}$ is the filtered poset of quasi-compact opens containing $y$. Since $f^{-1}(y)= \bigcap_{i\in I} f^{-1}(U_i)$, we note that the condition that $f^{-1}(y) \subset V$ is equivalent to 
    \begin{equation}\label{eqn:empty-limit}
    Z \cap \bigcap_{i\in I} f^{-1}(U_i) = \bigcap_{i\in I} \bigl(Z\cap f^{-1}(U_i)\bigr) = \varnothing.  
    \end{equation}
    Now each $f^{-1}(U_i)$ is a quasi-compact open subset of $X$, and so $Z\cap f^{-1}(U_i)$ is a constructible subset of $X$ (in particular, it is closed in the constructible topology on $X$). Therefore, \cref{eqn:empty-limit} and \cite[\href{https://stacks.math.columbia.edu/tag/0A2W}{Tag 0A2W}]{stacks-project} guarantee that there is $i\in I$ such that $f^{-1}(U_i) \cap Z = \varnothing$. In other words, $f^{-1}(U_i) \subset V$. So $U=U_i$ does the job. 
\end{proof}

\begin{corollary}\label{cor:disjoint-opens} Let $f\colon X\to Y$ be a finite morphism of locally noetherian analytic adic spaces, $y\in Y$ a point of rank-$1$, and $\{x_i\}_{i=1}^n=f^{-1}(y)$. Then there is an open Tate affinoid $U\subset Y$ neighborhood of $y$ such that $f^{-1}(U) = \bigsqcup_{i=1}^n U_i$ and $x_i\in U_j$ if and only if $i=j$.
\end{corollary}
\begin{proof}
    \cref{lemma:finite-same-rank} ensures that all $x_i$ are points of rank-$1$. In particular, they are are maximal points of $X$. Therefore, there are no common generalizations among $x_i$'s. So \cite[\href{https://stacks.math.columbia.edu/tag/0904}{Tag 0904}]{stacks-project} implies that there are open neighborhoods $V_i \ni x_i$ such that $V_i\cap V_j = \varnothing$ if $i\neq j$.  
    
    Now we apply \cref{lemma:many-open-around-rk-1} to $V\coloneqq \bigcup V_i=\bigsqcup V_i \subset X$ to find $y\in U\subset Y$ such that $f^{-1}(U) = \bigsqcup_{i=1}^n U_i$ and $x_i\in U_j$ if and only if $i=j$. We can replace $U$ with any open Tate affinoid $y\in U'\subset U$ to find the desired \emph{affinoid} open subset.  
\end{proof}

We first show that stalks at rank-$1$ points behave nicely with respect to finite morphisms of affinoid rings. 

\begin{lemma}\label{lemma:decomposes-local-rings} Let $f\colon X=\Spa(B, B^+)\to Y=\Spa(A, A^+)$ be a finite morphism of stronly noetherian Tate affinoids, and let $y\in Y$ be a point of rank-$1$. Then the natural morphism
\[
\O_{Y, y} \otimes_A B \to \prod_{x_i\in f^{-1}(y)} \O_{X, x_i},
\]
is an isomorphism.
\end{lemma}
\begin{proof}
        We use \cref{cor:disjoint-opens} to replace $Y$ with $U$ to assume that  $X=\bigsqcup_i X_i$ and each $X_i$ contains \emph{exactly one} point over $y$. Therefore, we can replace $X$ with $X_i$ to assume that $f^{-1}(y)=\{x\}$. In this case, we need to show that the natural morphism
        \begin{equation}\label{eqn:stalks}
        \O_{Y, y}\otimes_A B \to \O_{X, x}
        \end{equation}
        is an isomorhism. This comes from the following sequence of isomorphisms:
        \begin{align*}
            \O_{Y, y}\otimes_A B & \simeq \bigl(\colim_{V\ni y} \O_Y\left(V\right)\bigr) \otimes_A B \\
                        & \simeq \colim_{V\ni y} \bigl(\O_Y\left(V\right)\otimes_A B\bigr) \\
                        & \simeq \colim_{V\ni y} \O_X\left(f^{-1}\left(V\right)\right) \\
                        & \simeq \colim_{W\ni x} \O_X\left(W\right) \\
                        & \simeq \O_{X, x},
        \end{align*}
        where the third isomorphism comes from \cite[Lem.~B.3.6]{Z-quotients}, and the fifth isomorphism comes from \cref{lemma:many-open-around-rk-1}. 
\end{proof}

Our next goal is to get an analogue of \cref{lemma:decomposes-local-rings} for completed residue fields at rank-$1$ points (under some further assumptions). We start with the following preliminary lemma: 

\begin{lemma}\label{lemma:preimage-of-generalization} Let $f\colon X \to Y$ be a finite morphism of locally noetherian analytic adic spaces, $y\in Y$ be a point with the unique rank-$1$ generalization $y_\gen$. Let $f^{-1}(y) = \{x_i\}_{i\in I}$, and $x_{i, \gen}$ the unique rank-$1$ generalization of $x_i$ for $i\in I$. Then $f^{-1}(y_\gen) = \{x_{i, \gen}\}_{i \in I}$ (some $x_{i, \gen}$ might coincide).
\end{lemma}
\begin{proof}
    First, \cite[Lem.~1.1.10~(iv)]{Huber-etale} implies that $\{x_{i, \gen}\}_{i\in I} \subset f^{-1}(y_\gen)$. Now \cite[Lem.~1.4.5~(ii)]{Huber-etale} implies that $f$ is closed. Therefore, \cite[\href{https://stacks.math.columbia.edu/tag/0066}{Tag 0066}]{stacks-project} implies that any $x\in f^{-1}(y_\gen)$ is a generalization of some $x_i$. Furthermore, \cref{lemma:finite-same-rank} ensures that $x$ is of rank-$1$, so it must be $x_{i, \gen}$.
\end{proof}

\begin{lemma}\label{lemma:reduced-implies-iso} Let $f\colon X=\Spa(B, B^+)\to Y=\Spa(A, A^+)$ be a finite morphism of strongly noetherian Tate affinoid adic spaces, and let $y\in Y$ be a rank-$1$ point. If $\wdh{k(y)} \otimes_A B$ is a reduced ring, then the natural morphism
\[
\wdh{k(y)}\otimes_A B \to \prod_{x_i\in f^{-1}(y)} \wdh{k(x_i)}
\]
is an isomorphism.
\end{lemma}
\begin{proof}
    Throughout this proof, we will freely use \cite[Lem.~B.3.6]{Z-quotients} that ensures that certain completed tensor products coincide with the usual tensor products.  

    That being said, we can use \cref{lemma:finite-same-rank} and \cref{cor:disjoint-opens} to reduce to the situation $f^{-1}(y)=x$ is a unique rank-$1$ point. Then we can replace $Y$ with $\Spa\bigl(\wdh{k(y)}, \wdh{k(y)}^\circ\bigr)$ to assume that $Y=\Spa(K, \O_K)$ for a nonarchimedean field $K$. Then our assumption on $X=\Spa(B, B^+)$ implies that it is a reduced adic space which is finite over $\Spa(K, \O_K)$, and $\abs{X}$ is a singleton (in particular, it is connected). 
    Therefore, \cref{lemma:finite-over-point}\cref{lemma:finite-over-point-1}, \cref{lemma:finite-over-point-4}, \cref{lemma:finite-over-point-5} imply that $B\simeq \wdh{k(x)}$. Thus the map
    \[
    \wdh{k(y)} \otimes_A B \simeq K\otimes_K \wdh{k(x)} \to \wdh{k(x)}
    \]
    is obviously an isomorphism. 
\end{proof}

Now we want to get an analogue of \cref{lemma:reduced-implies-iso} for higher rank points. For this, we will need to work with henselized completed residue fields.  

We start by proving the following general lemma in commutative algebra:

\begin{lemma}\label{lemma:CRT} Let $I_1, I_2, \dots, I_n \subset A$ be ideals in a ring $A$. Suppose that for each $1\leq i,j\leq n$, $I_i+I_j=A$. Then the natural morphism
\[
A^{\rm{h}}_{I_1\cap I_2\cap\dots\cap I_n} \to \prod_{i=1}^n A^{\rm{h}}_{I_i}
\]
is an isomorphism.
\end{lemma}
\begin{proof}
First, we notice that $(\cap_{j \not = i} I_j) + I_i = A$ for every $i=1, \dots, n$. So we can assume that $n=2$. Now we note $I_1A^\h_{I_1}$ is a radical ideal, so the assumption $I_1+I_2 = A$ implies $I_2A^\h_{I_1}= A^\h_{I_1}$. Thus, $(I_1\cap I_2)A^\h_{I_1} =I_1A^\h_{I_1}$ and, similarly, $(I_1\cap I_2)A^\h_{I_2} =I_2A^\h_{I_2}$. Therefore, $A^\h_{I_1} \times A^\h_{I_2}$ is henselian along $I_1\cap I_2$ and ind-\'etale over $A$. Thus, in order to check that the natural morphism
\[
A^\h_{I_1\cap I_2} \to A^\h_{I_1} \times A^\h_{I_2}
\]
is an isomorphism, it suffices to check it modulo $I_1 \cap I_2$. Using \cite[\href{https://stacks.math.columbia.edu/tag/0AGU}{Tag 0AGU}]{stacks-project} and the equalities $(I_1\cap I_2) A^\h_{I_i} = I_i A^\h_{I_i}$, we conclude that it suffices to show that $A/(I_1\cap I_2) \to A/I_1 \times A/I_2$ is an isomorphism. This follows from our assumption that $I_1+I_2=A$. 
\end{proof}

Finally, we are ready to prove the main result of this subsection: 

\begin{theorem}\label{thm:iso-henselized-completed-residue-fields} Let $f\colon X=\Spa(B, B^+)\to Y=\Spa(A, A^+)$ be a finite morphism of strongly noetherian Tate affinoids, and let $y\in Y$ be a point with the unique rank-$1$ generalization $y_{\rm{gen}}$. If $\wdh{k(y_\gen)}\otimes_A B$ is reduced, then the natural morphism
\[
\wdh{k(y)}^{\rm{h}}\otimes_A B \to \prod_{x_i\in f^{-1}(y)} \wdh{k(x_i)}^{\rm{h}}
\]
is an isomorphism.
\end{theorem}
\begin{proof}
    Let us denote by $x_{i, \rm{gen}}$ the unique rank-$1$ generalization of $x_i$ for each $x_i\in f^{-1}(y)$. Then \cref{lemma:preimage-of-generalization} implies that the sets $\{x_{i, \gen}\}_{i\in I}$ and $f^{-1}(y_\gen)$ coincide. Then \cref{cor:disjoint-opens} allows us to reduce to the situation when $x_{i, \rm{gen}}=x_{j, \rm{gen}}$ for any $i, j$. We denote this common rank-$1$ generalization by $x_{\gen}$.  
    
    Now \cref{lemma:reduced-implies-iso} implies that the natural morphism
    \begin{equation}\label{eqn:weak-iso}
    \wdh{k(y_\gen)} \otimes_A B \to \wdh{k(x_\gen)}
    \end{equation}
    is an isomorphism. In particular, $\wdh{k(x_{\rm{gen}})}$ is a finite extension of $\wdh{k(y_{\rm{gen}})}$. Let us denote by $R^+$ the integral closure of $\wdh{k(y)}^+$ in $\wdh{k(x_{\rm{gen}})}$. Then \cref{lemma:finite-over-point} implies that
    \[
    R^+ = \bigcap_{i=1}^n \wdh{k(x_i)}^+.
    \]
    Let us denote by $\m\subset \wdh{k(y)}^+$ the maximal ideal in $\wdh{k(y)}^+$, and by $\m'_i \subset \wdh{k(x_i)}^+$ the maximal ideal in $\wdh{k(x_i)}^+$. Now \cref{lemma:finite-over-point}\cref{lemma:finite-over-point-6} implies that $R^+$ is a semi-local ring with maximal ideals given by $\m_i\coloneqq R^+\cap \m'_i$, and that the natural morphism $R^+_{\m_i} \to \wdh{k(x_i)}^+$ is an isomorphism of local rings. Furthermore, \cref{lemma:finite-over-point}\cref{lemma:finite-over-point-7} implies that $\rm{rad}(\m R^+)=\cap_{i=1}^n \m_i$. Therefore, \cite[\href{https://stacks.math.columbia.edu/tag/0DYE}{Tag 0DYE}]{stacks-project} and \cref{lemma:CRT} (applied to the maximal ideals $\m_i$) imply that
    \[
    \wdh{k(y)}^{+, \rm{h}} \otimes_{\wdh{k(y)}^+} R^+ \simeq R^{+, \rm{h}}_{\cap_{i=1}^n \m_i} =
    \prod_{i=1}^n R^{+, \rm{h}}_{\m_i} = \prod_{i=1}^n \wdh{k(x_i)}^{+, \rm{h}}.
    \]
    Now we recall that the natural morphism $\wdh{k(y)} \to \wdh{k(y_\gen)}$ is an isomorphism due to \cite[Lem.~1.1.10~(iii)]{Huber-etale}. Thus, after inverting a pseudo-uniformizer, we get the isomorphism
    \begin{equation*}
    \wdh{k(y)}^{\rm{h}} \otimes_{\wdh{k(y_{\rm{gen}})}} \wdh{k(x_{\rm{gen}})} \simeq \prod_{i=1}^n \wdh{k(x_i)}^{\rm{h}}.
    \end{equation*}
    Combining it with \cref{eqn:weak-iso}, we conclude that the desired isomorphism
    \[
    \wdh{k(y)}^{\rm{h}} \otimes_A B \simeq \wdh{k(y)}^{\rm{h}} \otimes_{\wdh{k(y_{\rm{gen}})}} \wdh{k(y_{\rm{gen}})}
    \otimes_A B \simeq \wdh{k(y)}^{\rm{h}} \otimes_{\wdh{k(y_{\rm{gen}})}} \wdh{k(x_{\rm{gen}})} \simeq \prod_{i=1}^n \wdh{k(x_i)}^{\rm{h}}.\qedhere
    \]
\end{proof}

Even though the conclusion of \cref{thm:iso-henselized-completed-residue-fields} does not hold for an arbitrary finite morphism, it can still be used to study properties of arbitrary finite morphisms: 

\begin{corollary}\label{cor:finite-morphism-finite-on-henselized-completions} Let $f\colon X\to Y$ be a finite morphism of locally noetherian analytic adic spaces, and let $x\in X$ be an arbitrary point with $y=f(x)$. Then $\wdh{k(y)}^{\h} \to \wdh{k(x)}^\h$ is a finite field extension.
\end{corollary}
\begin{proof}
    Without loss of generality, we can assume that $Y=\Spa\bigl(\widehat{k(y)}, \widehat{k(y)}^+\bigr)$. Then we can replace $X$ by its reduction, and then pass to connected components to assume that $X$ is reduced and connected. In this case, the assumption of \cref{thm:iso-henselized-completed-residue-fields} is obviously satisfied since $\wdh{k(y)}\simeq \wdh{k(y_\gen)}$. Therefore the result follows directly from \cref{thm:iso-henselized-completed-residue-fields}.
\end{proof}

We finish the subsection by establishing two examples when the assumptions of \cref{thm:iso-henselized-completed-residue-fields} is automatic. 

\begin{example} The assumption of \cref{thm:iso-henselized-completed-residue-fields} is always satisfied if $f\colon X \to Y$ is a finite \emph{\'etale} morphism.
\end{example}

Now we give another, more elaborate example: 

\begin{lemma}\label{lemma:weakly-shilov-fiber-reduced} Let $K$ be a nonarchimedean field, let $f\colon X=\Spa(B, B^\circ) \to Y=\Spa(A, A^\circ)$ be a finite morphism of smooth rigid-analytic affinoid spaces over $K$, and let $y\in Y$ be a weakly Shilov point of $Y$ (in the sense of \cite[Def.~2.5]{BH}). Then $\wdh{k(y)} \otimes_A B$ is reduced. 
\end{lemma}
\begin{proof}
    First, there is an open affinoid $U\subset Y$ neighborhood of $y$ such that $y\in U$ is a Shilov point. Then \cite[Lem.~B.3.6]{Z-quotients} guarantees that \[
    \wdh{k(y)}\otimes_A B \simeq \wdh{k(y)} \otimes_{\O_Y(U)} \O_Y(U) \otimes_A B \simeq \wdh{k(y)} \otimes_{\O_Y(U)} \O_X(X_U).
    \]
    Therefore, we can replace $Y$ with $U$ and assume that $y\in Y$ is a Shilov point. Then \cite[Th.~2.25]{BH} implies that $\wdh{k(y)}\otimes_A B$ is a regular algebra. In particular, it is reduced.
\end{proof}

\subsection{Finite flat morphisms}

In this subsection, we collect some facts about flat and finite flat morphisms of locally noetherian analytic adic spaces. In particular, we show that the universal compactification of a finite flat morphism is always finite and flat.  

We start by recalling the following definition: 

\begin{definition}
    A morphism $f\colon X \to Y$ of locally noetherian analytic adic spaces is \emph{flat} if, for every point $x\in X$, the morphism $\O_{Y, f(x)} \to \O_{X, x}$ is a flat morphism of local rings. 
\end{definition}

\begin{remark} Flat morphisms are closed under compositions. But it is not clear (and probably false) whether they are closed under base change. 
\end{remark}

\begin{lemma}\label{lemma:flat-maximal-points} Let $f\colon X \to Y$ be a morphism of locally noetherian analytic adic spaces. Suppose that, for any rank-$1$ point $x\in X$, the morphism $\O_{Y, f(x)} \to \O_{X, x}$ is flat. Then $f$ is flat.
\end{lemma}
\begin{proof}
    Pick any $x_0\in X$, and let $x$ be its unique rank-$1$ generalization (it exists by \cite[Lem.~1.1.10]{Huber-etale}).
    Then \textit{loc.\ cit.} implies that $y\coloneqq f(x)$ is the unique rank-$1$ generalization of $y_0\coloneqq f(x_0)$. Therefore, we have a commutative diagram
    \[
    \begin{tikzcd}
    \O_{X, x_0} \arrow{r} & \O_{X, x}\\
    \O_{Y, y_0} \arrow{u} \arrow{r} & \O_{Y, y}\arrow{u}.
    \end{tikzcd}
    \]
    \textit{Loc.\ cit.} ensures that the horizontal maps are local and flat (in particular, they are faithfully flat), and the right vertical arrow is flat by the assumption. This implies that the left vertical is flat as well. Since $x$ was arbitrary, we conclude that $f$ is flat. 
\end{proof}

In general, it seems very difficult to test whether a morphism of locally noetherian adic spaces is flat. However, it turns out that the situation is much better in the case of finite morphisms: 

\begin{lemma}\label{lemma:flatness-finite-case} Let $f\colon X=\Spa(B, B^+) \to Y=\Spa(A, A^+)$ be a morphism of strongly noetherian Tate affinoids. Then $f$ is finite flat if and only if $(A, A^+) \to (B, B^+)$ is a finite morphism of Huber pairs and $A \to B$ is a flat morphism of rings.
\end{lemma}
\begin{proof}
    If $f$ is finite and flat, then $(A, A^+) \to (B, B^+)$ is a finite morphism of Huber pairs by \cite[Satz 3.6.20 and Korollar 3.12.12]{Huber-thesis} and $A\to B$ is flat by \cite[Lem.~B.4.3]{Z-quotients}.  
    
    Now suppose that $X=\Spa(B, B^+)$ is an affinoid, $(A, A^+) \to (B, B^+)$ is finite and $A\to B$ is flat. Then $\Spa(B, B^+) \to \Spa(A, A^+)$ is clearly finite. We only need to show that it is also flat. Let $x\in X$ is a rank-$1$ point, and $y=f(x)$. Then \cref{lemma:decomposes-local-rings} implies that the morphism
    \[
    \O_{Y, y} \to \prod_{x_i\in f^{-1}(y)} \O_{X, x_i} \simeq B\otimes_A \O_{Y, y}
    \]
    is flat. In particular, $\O_{Y,y} \to \O_{X, x}$ is flat. Therefore, \cref{lemma:flat-maximal-points} ensures that $X\to Y$ is flat. 
\end{proof}

\begin{remark} \Cref{lemma:flatness-finite-case} and \cite[Lem.~B.3.6]{Z-quotients} imply that finite flat morphisms are closed under arbitrary base change.
\end{remark}

Finally, we are ready to show that universal compactifications preserve finite flat morphisms. We refer to \cref{univ-comp} for a brief recollection on universal compactifications. 

\begin{lemma}\label{lemma:compactification-of-finite-morphism} Let $K$ be a nonarchimedean  field, $f\colon X=\Spa(B, B^+)\to Y=\Spa(A, A^+)$ is a finite (resp.~finite flat) morphism of rigid-analytic affinoid spaces over $K$, and $f^c\colon X^c=\Spa(B, B'^+)\to Y^c=\Spa(A, A'^+)$ the induced morphism on the universal compactifications. Then $f^c$ is a finite (resp.~finite flat) morphism.
\end{lemma}
\begin{proof}
    Since universal compactifications do not change rings of rational functions (see \cref{lemma:universal-compactification-affinoid}), \cref{lemma:flatness-finite-case} implies that it suffices to show that $(A, A'^+)\to (B, B'^+)$ is a finite morphism of Huber pairs. Cleary, $A \to B$ is finite, so it suffices to show that $A'^+\to B'^+$ is integral.  

    Now \cite[Lem.~3.12.10]{Huber-thesis} (see also \cite[Lem.~3.5]{adic-notes}) implies that the integral closures $B''^+$ of $A'^+$ in $B$ defines a Huber pair $(B, B''^+)$. By construction, it contains $\O_K$. \Cref{lemma:universal-compactification-affinoid} ensures that $B'^+$ is the minimal $+$-ring containing $\O_K$. Therefore, $B'^+\subset B''^+$, and thus $B'^+$ is integral over $A'^+$. 
\end{proof}

\subsection{Trace for finite flat morphisms}\label{section:finite-flat-trace}

The main goal of this subsection is to define a trace morphism for any finite flat morphism of locally noetherian analytic adic spaces. This construction will be an important tool in studying properties of the analytic trace map in \cref{section:analytic-trace}.

Before we start the construction, we mention that the algebraic counterpart of the finite flat trace map has been defined in \cite[Exp.~XVII, Th.~6.2.3]{SGA4} and \cite[\href{https://stacks.math.columbia.edu/tag/0GKI}{Tag 0GKI}]{stacks-project}. We carry over a similar strategy to the nonarchimedean situation.
We begin with some preliminary lemmas.

\begin{lemma}\label{lemma:stalks-finite-morphism} Let $f\colon X \to Y$ be a finite morphism of locally noetherian analytic adic spaces, $\ov{y} \to Y$ a geometric point, and $\F\in \cal{A}b(Y_\et)$ a sheaf of abelian groups. Then there is a natural isomorphism
\[
\bigoplus_{\ov{x}\in f^{-1}(\ov{y})} \F_{\ov{y}} \xrightarrow{\sim} \left(f_*f^*\F\right)_{\ov{y}}.
\]
\end{lemma}
\begin{proof}
    This follows from the sequence of isomorphisms
    \begin{align*}
        \bigoplus_{\ov{x}\in f^{-1}(\ov{y})} \F_{\ov{y}} & \simeq \bigoplus_{\ov{x}\in f^{-1}(\ov{y})} \left(f^*\F\right)_{\ov{x}}  \simeq \left(f_*f^*\F\right)_{\ov{y}},
    \end{align*}
    where the second isomorphism comes from \cite[Prop.~2.6.3]{Huber-etale}. 
\end{proof}

\begin{remark}\label{rmk:integral-projection-formula} 
\Cref{lemma:stalks-finite-morphism} implies that for a finite morphism $f\colon X \to Y$ and an abelian sheaf $\F\in \cal{A}b(Y_\et)$, the natural morphism
\[
f_*\ud{\Z} \otimes_{\ud{\Z}} \F \to f_*f^*\F 
\]
is an isomorphism.
\end{remark}

\begin{remark}\label{rmk:integral-base-change}
We also note that \cite[Prop.~2.6.3]{Huber-etale} guarantees that for a finite morphism $f\colon X \to Y$, the functor $f_*\colon \cal{A}b(X_\et) \to \cal{A}b(Y_\et)$ is exact and commute with arbitrary base change along $Y'\to Y$. 
\end{remark}

Recall that for every strongly noetherian Tate affinoid $S=\Spa(A, A^+)$, \cite[Cor.~1.7.3, (3.2.8)]{Huber-etale} constructs a functorial morphism of \'etale topoi $c_S\colon S_\et \to (\Spec A)_\et$.
In particular, for every morphism $f\colon T=\Spa(B, B^+)\to S=\Spa(A, A^+)$ with induced morphism $f^{\rm{alg}}\colon \Spec B\to \Spec A$, the following diagram commutes (up to canonical equivalence):
\[
\begin{tikzcd}
    T_\et \arrow{d}{f_\et} \arrow{r}{c_T}&  \left(\Spec B\right)_\et \arrow{d}{f^{\rm{alg}}_\et}\\
    S_\et \arrow{r}{c_S} & \left(\Spec A\right)_\et
\end{tikzcd}
\]

\begin{construction}[Relative analytification]\label{construction:relative-analytification}
Likewise, for every strongly noetherian Tate affinoid $S=\Spa(A, A^+)$ and a locally finite type $A$-scheme $g\colon X\to \Spec A$, \cite[Prop.~3.8]{Huber-2} defines the relative analytification $X^{\rm{an}/S}$ as an adic space which is locally of finite type over $S$.
By \cite[Cor.~1.7.3, (3.2.8)]{Huber-etale}, it comes equipped with a canonical morphism of \'etale topoi $c_{X/S}\colon X^{\rm{an}/S}_\et \to X_\et$ such that the diagram\footnote{
If $X=S$, we denote $c_{X/S}$ simply by $c_S$.}
\[
\begin{tikzcd}
X^{\rm{an}/S}_\et \arrow{d}{g^{\rm{an}/S}_\et} \arrow{r}{c_{X/S}}& X_\et \arrow{d}{g_\et}\\
S_\et \arrow{r}{c_S} & \Spec A_\et
\end{tikzcd}
\]
commutes (up to canonical equivalence).
\end{construction}

\begin{lemma}\label{lemma:comparison-finite-morphism} Let $f\colon X=\Spa(B, B^+)\to Y=\Spa(A, A^+)$ be a finite morphism of strongly noetherian Tate affinoids and $\F\in \cal{A}b(\Spec B_\et)$. Then the natural morphism
\[
\gamma\colon c_Y^*f^{\rm{alg}}_* \F \longrightarrow f_*c_X^*\F
\]
is an isomorphism.
\end{lemma}
\begin{proof}
    First, we note that $c_Y^*f^{\rm{alg}}_*\F$ and $c_X^*\F$ are both overconvergent (in the sense of \cite[Def.~8.2.1]{Huber-etale}) because they are analytifications of algebraic sheaves. Furthermore, $f_*c_X^*\F$ is also overconvergent due to \cite[Prop.~8.2.3]{Huber-etale}. Therefore, it suffices to show that $\gamma$ induces an isomorphism on stalks at geometric points of rank $1$.
    Since both $f_*$ and $f_*^{\rm{alg}}$ commute with arbitrary base change (see \cref{rmk:integral-base-change} and \cite[\href{https://stacks.math.columbia.edu/tag/0959}{Tag 0959}]{stacks-project}), we may thus assume that $Y=\Spa(C, \O_C)$ for an algebraically closed nonarchimedean field $C$.
    Then we can further replace $X$ by $X_\red$ to assume that $X$ is reduced.
    In this case, $X$ decomposes as a finite disjoint union $X=\bigsqcup_{i=1}^n \Spa(C, \O_C)$ and the result becomes trivial. 
\end{proof}
Finally, we are ready to prove the main result of this subsection: 
\begin{theorem}
\label{thm:flat-trace}
There is a unique way to assign
to any finite flat morphism
$f \colon X \to Y$ of locally noetherian analytic adic spaces and any abelian sheaf $\F \in \cal{A}b(Y_\et)$, a trace map $\ttr_{f, \F} \colon f_* f^* \F \to \F$ satisfying the following properties:
\begin{enumerate}[label=\upshape{(\arabic*)}]
\item\label{thm:flat-trace-1} (functoriality in $\F$) For any morphism $\varphi\colon \F \to \G$ in $\cal{A}b(Y_\et)$, the following diagram is commutative:
\[
\begin{tikzcd}[column sep = huge]
f_*f^*\F \arrow{r}{\ttr_{f, \F}}\arrow{d}{f_*f^*(\varphi)} & \F \arrow{d}{\varphi} \\
f_*f^*\G \arrow{r}{\ttr_{f, \G}} & \G
\end{tikzcd}
\]
\item\label{thm:flat-trace-2} (compatibility with compositions) For any two finite flat morphisms $f \colon X \to Y$ and $g \colon Y \to Z$ and any $\F\in \cal{A}b(Z_\et)$, the following diagram is commutative:
\[ 
    \begin{tikzcd}[column sep=huge]
        (g\circ f)_* (g\circ f)^* \F \arrow{r}{\ttr_{g\circ f, \F}} \arrow[d,sloped,"\sim"] & \F \\ 
        g_*\bigl(f_*f^*(g^*\F)\bigr) \arrow{r}{g_*(\ttr_{f, g^*\F})} & g_*g^*\F \arrow{u}{\ttr_{g, \F}}
    \end{tikzcd} 
\]
\item\label{thm:flat-trace-3} (compatibility with pullbacks) For any pullback diagram
\[ \begin{tikzcd}
    X' \arrow[r,"g'"] \arrow[d,"f'"] & X \arrow[d,"f"] \\
    Y' \arrow[r,"g"] & Y
\end{tikzcd} \]
in which $f$ and $f'$ are finite flat and any $\F\in \cal{A}b(Y_\et)$, the following diagram is commutative:
\[ \begin{tikzcd}
    g^*f_*f^*\F \arrow[d,sloped,"\sim"] \arrow{r}{g^*(\ttr_{f, \F})} & g^*\F \\
    f'_*g^{\prime,*}f^*\F \arrow{r}{\sim} & f'_*f^{\prime,*}g^*\F \arrow[u, swap, "\ttr_{f', g^*\F}"]
\end{tikzcd}
\]
\item\label{thm:flat-trace-4} (normalization) If $f$ is of constant rank $d$, the composition $\F\to f_*f^*\F \xrightarrow{\ttr_{f, \F}} \F$ is equal to the multiplication by $d$.
\end{enumerate}
Furthermore, this assignment satisfies the following properties: 
\begin{enumerate}[label=\upshape{(\arabic*)}, resume]   
    \item\label{flat-trace-etale} (compatibility with \'etale traces) If $f$ is finite \'etale, then $\ttr_{f, \F}$ is given by the counit \[ f_*f^* \F \simeq f_! f^* \F \to \F \] of the adjunction between $f_!$ and $f^*$.
    \item\label{thm:flat-trace-6} (compatibility with algebraic traces I) For any finite flat morphism $f\colon T=\Spa(B, B^+) \to S=\Spa(A, A^+)$ with associated morphism\footnote{We note that $f^{\rm{alg}}$ is finite flat due to \cref{lemma:flatness-finite-case}.} $f^{\rm{alg}} \colon \Spec B \to \Spec A$ and any sheaf $\F\in \cal{A}b\bigl((\Spec A)_\et\bigr)$, the diagram
    \begin{equation}\label{diagram:comp-alg-trace-1}
    \begin{tikzcd}[column sep = huge]
    c_S^* f_*^{\rm{alg}} f^{\rm{alg}, *} \F \arrow[d,sloped,"\sim"]  \arrow{r}{c_S^*(\ttr_{f^{\rm{alg}, \F}})} & c_S^* \F \\
    f_*c_T^* f^{\rm{alg}, *} \F \arrow{r}{\sim} & f_*f^* c_S^* \F \arrow[u, swap, "\ttr_{f, c_S^*\F}"]
    \end{tikzcd}
    \end{equation}
    commutes, where $c_A$ and $c_B$ are the morphisms of topoi defined just before \cref{construction:relative-analytification}.
    \item\label{thm:flat-trace-7} (compatibility with algebraic traces II) For any strongly noetherian Tate affinoid $S=\Spa(A, A^+)$, any finite flat morphism $g \colon X \to Y$ of locally finite type $A$-schemes with relative analytification $g^{\an/S} \colon X^{\an/S} \to X^{\an/S}$, and any sheaf $\F\in \cal{A}b(Y_\et)$, the diagram
    \begin{equation}\label{diagram:comp-alg-trace-2}
    \begin{tikzcd}
    c_{Y/S}^* g_*g^* \F \arrow[d,sloped,"\sim"] \arrow{r}{c_{Y/S}^*(\ttr_{g, \F})} & c_{Y/S}^* \F \\
    g^{\an/S}_*c_{X/S}^* g^* \F \arrow{r}{\sim} & g^{\an/S}_* g^{\an/S, *} c_{Y/S}^* \F \arrow[u, swap, "\ttr_{g^{\an/S}, c_{Y/S}^*\F}"]
    \end{tikzcd}
    \end{equation}
    commutes, where $c_{X/A}$ and $c_{Y/A}$ are the morphisms of topoi defined in \cref{construction:relative-analytification}.
\end{enumerate}
\end{theorem}
\begin{proof}
    \begin{enumerate}[wide,label={\textit{Step~\arabic*}.},ref={Step~\arabic*}]
        \item\label{thm:flat-trace-unique} \textit{Uniqueness.}
        First, we note that \cite[Prop.~2.5.5]{Huber-etale} and \cref{thm:flat-trace-3} ensure that $\ttr_{f, \F}$ is determined by the case when $Y=\Spa(C, C^+)$ for a complete algebraically closed nonarchimedean field $C$ and an open bounded valuation subring $C^+\subset C$.
        In that setting, we can write $X=\bigsqcup_{i=1}^n X_i$ with each $X_i$ connected (and finite flat over $Y$);
        we denote by $f_i \colon X_i \to Y$ the induced morphisms and set $\F_i \colonequals \restr{\F}{X_i}$.
        For the natural clopen immersions $j_i\colon X_i \to X$, properties \cref{thm:flat-trace-4} and \cref{thm:flat-trace-3} imply that the trace morphisms $\ttr_{j_i, \F} \colon j_{i, *}\F_i \to \F$ must be the adjunction morphisms $j_{i, *} \F_i = j_{i, !}\F_i \to \F$.
        Then $\ttr_{f, \F}=\sum_{i=1}^n \ttr_{f_i, \F_i}$ for any $\F \in \cal{A}b(Y_\et)$ due to \cref{thm:flat-trace-2}. 
        Therefore, it suffices to show uniqueness under the additional assumptions that $Y=\Spa(C, C^+)$ and $X$ is connected.
    
        In this case, we clearly have that $f\colon X \to Y$ is of constant rank. Since $C$ is algebraically closed, \cref{rmk:equivalent-defns-henselian} implies that $C^+$ is henselian.
        Therefore, \cref{lemma:finite-over-point} ensures that $X_{\rm{red}} \simeq Y$, so $X_\et \simeq Y_\et$. In particular, $f_*f^*\F = \F$ for any $\F\in \cal{A}b(Y_\et)$.
        Thus, \cref{thm:flat-trace-4} implies that $\ttr_{f, \F} \colon \F \to \F$ must be equal to the multiplication by $\dim_C \O_X(X)$ for any $\F \in \cal{A}b(Y_\et)$.
        This finishes the proof of uniqueness.  
    
        \item \textit{Construction of $\ttr_{f, \F}$.}
        We first construct trace maps $\ttr_{f, \ud{\Z}}$ that are compatible with base change.
        Thanks to the base change compatibility, it suffices to do so locally on $Y$.
        Hence, we may assume that both $X=\Spa(B, B^+)$ and $Y=\Spa(A, A^+)$ are affinoid.
        Then \cref{lemma:flatness-finite-case} implies that $A\to B$ is a finite flat ring map, so the induced morphism $f^{\rm{alg}}\colon \Spec B \to \Spec A$ is also finite flat. 
        Therefore, we use \cref{lemma:comparison-finite-morphism} to define
        \[
        \ttr_{f, \ud{\Z}}\coloneqq c_Y^*\left(\ttr_{f^{\rm{alg}}, \ud{\Z}}\right),
        \]
        where $\ttr_{f^{\rm{alg}}, \ud{\Z}}$ is the algebraic trace for finite flat morphisms constructed in \cite[Exp.~XVII, Th.~6.2.3]{SGA4}.
        We note that $\ttr_{f, \ud{\Z}}$ commutes with arbitrary base change since the same holds for the algebraic trace map.
    
        For a general sheaf $\F$, we define $\ttr_{f, \F}$ as the composition
        \begin{equation}\label{eqn:finite-flat-trace-in-general}
            f_*f^*\F \xrightarrow{\sim} f_*f^*\ud{\Z} \otimes \F \xrightarrow{\ttr_{f, \ud{\Z}} \otimes \rm{id}} \ud{\Z} \otimes \F \xrightarrow{\sim} \F,
        \end{equation}
        where the first map is the inverse of the projection formula isomorphism (see \cref{rmk:integral-projection-formula}). 
        This definition satisfies \cref{thm:flat-trace-1} by its very construction. 
        Next, we check that it also satisfies \cref{thm:flat-trace-6}.
        For this, we first note that Diagram~\cref{diagram:comp-alg-trace-1} commutes for $\F=\ud{\Z}$ by the very definition of $\ttr_{f, \ud{\Z}}$. Secondly, we note that the construction of  $\ttr_{f^{\rm{alg}},\F}$ is analogous to \cref{eqn:finite-flat-trace-in-general} except that one uses $\ttr_{f^{\rm{alg}}, \ud{\Z}}$ in place of $\ttr_{f, \ud{\Z}}$;
        see the proof of \cite[\href{https://stacks.math.columbia.edu/tag/0GKG}{Tag 0GKG}]{stacks-project} and \cite[\href{https://stacks.math.columbia.edu/tag/0GKI}{Tag 0GKI}]{stacks-project}.
        Since the algebraic projection formula analytifies to the analytic one, we conclude that Diagram~\cref{diagram:comp-alg-trace-1} commutes for any $\F$. 
        
        In order to check \cref{thm:flat-trace-2}, \cref{thm:flat-trace-3}, and \cref{thm:flat-trace-4}, we may assume that $X$, $Y$, $Z$, $Y'$ are affinoid, and $\F=\ud{\Z}$.
        Then the results follow from \cref{thm:flat-trace-6} and the analogous properties of algebraic finite flat trace (see \cite[Exp.~XVII, Th.~6.2.3]{SGA4}). Thus, we are only left to check \cref{flat-trace-etale} and \cref{thm:flat-trace-7}.
    
        We start with \cref{flat-trace-etale}. We first note that it suffices to treat the case $\F=\ud{\Z}$. Since $\ud{\Z}$ and $f_*\ud{\Z}$ are overconvergent (see \cite[Prop.~8.2.3~(ii)]{Huber-etale}), we can check equality over geometric points of rank $1$.
        Both $\ttr_{f, \ud{\Z}}$ and the $(f_!, f^*)$-adjunction commute with arbitrary base change, so we can assume that $Y=\Spa(C, \O_C)$ for an algebraically closed nonarchimedean field $C$. Then $X=\bigsqcup_{i=1}^n Y$ and the result follows from \cref{thm:flat-trace-4}. 
    
        Finally, we show \cref{thm:flat-trace-7}. First, we observe that both $c^*_{Y/S}g_*g^*\F$ and $c_{Y/S}^*\F$ are overconvergent as analytifications of algebraic sheaves. Therefore, we can check that Diagram~\cref{diagram:comp-alg-trace-2} commutes over geometric rank-$1$ points of $Y^{\an/S}$.
        Since both $\ttr_{g^{\an/S}}$ and $\ttr_{g}$ commute with arbitrary base change, we can assume that $A=C$ is an algebraically closed nonarchimedean field and $Y=\Spec C$.
        In this case, the result follows from \cref{thm:flat-trace-6}. \qedhere
    \end{enumerate}
\end{proof}
For later reference, we also discuss a version of \'etale traces for general \'etale morphisms.
Recall that \cite[Lem.~2.7.6]{Huber-etale} guarantees that in this case, $f^*$ admits an exact right adjoint functor $f_!$, which furthermore commutes with arbitrary base change.
\begin{definition}\label{etale-trace}
    Let $f \colon X \to Y$ be an \'etale morphism of locally noetherian analytic adic spaces and let $\F \in \cal{A}b(Y_\et)$ be any abelian sheaf on $Y$.
    Then the natural counit
    \[ \ttr^\et_{f, \F} \colon f_! f^* \F \to \F \]
    for the adjunction between $f_!$ and $f^*$ is called the \emph{\'etale trace map} for $f$.
\end{definition}
\begin{notation}\label{notation:etale-trace}
    We will often use the \'etale trace in the situation where we have fixed a ring of coefficients $\Lambda = \Z/n\Z$.
    In this case, we denote the map $\ttr^\et_{f, \Lambda}$ simply by $\ttr^{\et}_f$.
    Likewise, we denote $\ttr^\et_{f, \mu_n^{\otimes m}}$ by $\ttr^\et_{f}(m)$.     
\end{notation}

\begin{lemma}\label{etale-trace-compatibility}
    Let $f\colon X \to Y$ be an \'etale morphism of locally noetherian analytic adic spaces and let $\F \in \cal{A}b(Y_\et)$ be an abelian sheaf on $Y$.
    Then the \'etale trace map is compatible with compositions, pullbacks, and relative analytifications (in the sense of \cref{thm:flat-trace}\cref{thm:flat-trace-2},\cref{thm:flat-trace-3},\cref{thm:flat-trace-7}).
\end{lemma}
\begin{proof}
    The compatibility under compositions follows from the compatibility of counits of adjunctions with compositions (see e.g.\ \cite[\href{https://kerodon.net/tag/02DS}{Tag~02DS}]{kerodon}).
    Now we show the compatibility under pullbacks.
    Let 
    \[ \begin{tikzcd}
        X' \arrow[r,"g'"] \arrow[d,"f'"] & X \arrow[d,"f"] \\
        Y' \arrow[r,"g"] & Y
    \end{tikzcd} \]
    be a pullback diagram of locally noetherian adic spaces in which $f$ and $f'$ are \'etale. Thanks to \cite[Prop.~2.5.5]{Huber-etale} and the observation that the \'etale trace commutes with taking stalks, it suffices to check that $g^*\ttr_f$ agrees with $\ttr_{f'}$ after passing to stalks at each geometric point $\xi$ of $Y'$. Therefore, we can assume that the morphism $g$ is of the form $Y'=\Spa(C', C'^+) \to Y=\Spa(C, C^+)$ for some algebraically closed nonarchimedean field $C$ and $C'$ and a faithfully flat morphism $C^+ \to C'^+$ of open bounded valuation subrings. Furthermore, in this case, it suffices to show that $\Gamma(Y', g^* \ttr^\et_f) = \Gamma(Y', \ttr^\et_{f'})$.
    The claim is local on $X$, so we can assume that $X=\Spa(A, A^+)$ is affinoid. Then \cite[Lem.~2.2.8]{Huber-etale} implies that $X$ is an open subspace inside a finite \'etale $Y$-space $\ov{X}$. Since $C$ is algebraically closed, we conclude that $\ov{X}$ is a finite disjoint union of copies of $Y$, so we can assume that $f\colon X \to Y$ is an open immersion. If $f$ is an isomorphism, the claim is obvious. If $f$ is an open immersion that does not meet the closed point of $Y$, we see that $\Gamma(Y', g^*\ttr^\et_f) = 0 =\Gamma(Y', \ttr^\et_{f'})$.
\end{proof}

\section{Cycle classes}
\label{cycle classes section}

In this section, we develop a theory of cycle classes in analytic adic geometry. The cycle class considerations will be an important technical tool for verifying properties of the analytic trace in \cref{section:analytic-trace}. Furthermore, they will be absolutely crucial for our ``diagrammatic'' approach to Poincar\'e duality (see \cref{section:smooth-poincare-duality}).  

Our approach closely follows its algebraic counterpart developed in \cite{SGA41/2} in the case of divisors, and in \cite[\S~1]{Fujiwara-purity} in the case of lci immersion of higher codimension.

\subsection{Cycle classes of divisors}
\label{section:divisor-classes}

The goal of this subsection is to define the cycle class in the case of divisors. In later subsections, we will generalize this construction to higher codimensions as well.  

We refer the reader to \cite[\S~5]{adic-notes} for the definition and basic properties of effective Cartier divisors in analytic adic geometry.   

Throughout this subsection, we fix a locally noetherian analytic adic space $X$, an integer $n \in \O_X^\times$ (we do not assume that $n$ is invertible in $\O_X^+$),
and set $\Lambda \coloneqq \Z/n\Z$.

\subsubsection{Construction of cycle classes}

The goal of this subsection is to construct a class $c\ell_X(D) \in \rm{H}^2_D(X, \mu_n)$ for any effective Cartier divisor $D\subset X$. In order to start the construction, we will need the following explicit characterization of (local) \'etale cohomology of $\bf{G}_m$:

\begin{lemma}\label{lemma:local-cohomology} Let $X$ be a locally noetherian analytic adic space, and $D$ a Zariski-closed subspace with the open complement $U$. Then there are functorial identifications
\[
\rm{H}^1(X, \bf{G}_m) \simeq \rm{Pic}(X),
\]
\[
\rm{H}^1_D(X, \bf{G}_m)\simeq \bigl\{(\cal{L}, \varphi) \suchthat  \cal{L} \text{ a line bundle on } X, \varphi:\O_U\xrightarrow{\sim} \restr{\cal{L}}{U}\bigr\}/\sim.
\]
\end{lemma}
\begin{proof}
    The first claim follows from \cite[(2.2.7)]{Huber-etale}. Then the second statement follows from the argument identical to that of \cite[2.13]{Olsson} (and it is essentially formal).
\end{proof}

For the following discussion, we fix an effective Cartier divisor $i\colon D\hookrightarrow X$. Our current goal is to leverage \cref{lemma:local-cohomology} and the Kummer exact sequence to define the cycle class $c\ell_X(D) \in \rm{H}^2_D(X, \mu_n)$. We start with the following definition: 

\begin{definition}
    The line bundle \emph{associated to $D\subset X$} is $\O_X$-module $\O_X(D) \coloneqq \bigl(\rm{ker}(\O_X \to i_* \O_D)\bigr)^{\vee}$. We denote its dual by $\O_X(-D) \simeq \rm{ker}(\O_X \to i_* \O_D)$. 
\end{definition}

By definition, we have the following short exact sequence, 
\[
0 \to \O_X(-D) \to \O_X \to i_*\O_D \to 0.
\]
By passing to duals, we get a canonical morphism $\O_X \to \O_X(D)$ which is an isomorphism over $U\coloneqq X\smallsetminus D$. We denote the restriction of this morphism on $U$ by 
\begin{equation}\label{eqn:can-iso}
\varphi_D\colon \O_U \xrightarrow{\sim} \restr{\O_X(D)}{U}.
\end{equation}

\cref{lemma:local-cohomology} implies that the pair $(\O_X(D), \varphi_D)$ defines a class $[D]\in \rm{H}^1_D(X, \bf{G}_m)$. To get the (localized) cycle class of $D$, we combine the above discussion with the boundary map
\[
\delta_X\colon \rm{H}^1_D(X, \bf{G}_m) \to \rm{H}^2_D(X, \mu_{n})
\]
coming from the Kummer exact sequence
\[
0\to \mu_{n} \to \bf{G}_m \xrightarrow{f\mapsto f^{n}} \bf{G}_m\to 0.
\]

More precisely, we give the following definition:  

\begin{definition}
\label{defn:cycle-clas-divisors}
The (localized) \emph{cycle class} $c\ell_X(D)\in \rm{H}^2_D(X, \mu_{n})$ of an effective divisor $D \xhookrightarrow{i} X$ is defined to be $\delta_X([D]) \in \rm{H}^2_D(X, \mu_{n})$.
\end{definition}

\begin{variant}\label{variant:cycle-class-map-divisors} Sometimes, it will be more convenient to think of the cycle class as a map 
\[
\rm{cl}_X(D) \colon i_* \ud{\Lambda}_{D} \to \mu_{n, X}[2].
\]
\end{variant}

\begin{lemma}[Tranversal Base Change]
\label{lemma:base-change-divisor-classes} Let
\[
\begin{tikzcd}
D' \arrow{d}\arrow{r} & D \arrow{d}\\
X' \arrow{r}{f} & X
\end{tikzcd}
\]
be a cartesian diagram of locally noetherian analytic adic spaces such that the vertical arrows are effective Cartier divisors. Then  $f^*c\ell_X(D)=c\ell_{X'}(D')\in \rm{H}^2_{D'}(X', \mu_{n})$.
\end{lemma}
\begin{proof}
    The boundary map coming from the Kummer exact sequence commutes with an arbitrary base change. Therefore, it suffices to show that the pair $(\O_X(D), \varphi_{D})$ pullbacks to the pair $(\O_{X'}(D'), \varphi_{D'})$. This follows from \cite[Lem.~5.8]{adic-notes}.
\end{proof}

Now we show that our construction of cycle classes is compatible with the algebraic construction in \cite[Cycle, Def.~2.1.2]{SGA41/2}. 

\begin{lemma}
\label{lemma:comparison-cycles-of-divisors} 
Let $S=\Spa(A, A^+)$ be a strongly noetherian Tate affinoid, let $X$ be a locally finite type $A$-scheme with the relative analytification $X^{\an/S}$ (see \cref{construction:relative-analytification}), and let $D\subset X$ be an effective Cartier divisor. Then the natural comparison morphism\footnote{Here, we implicitly use \cite[Cor.~6.5]{adic-notes} (see also \cite[Prop.~5.5]{Guo-Li}) that guarantees that $D^{\an/S} \subset X^{\an/S}$ is an effective Cartier divisor}
\[
c_{X/S}^*\colon \rm{H}^2_D(X, \mu_n) \longrightarrow \rm{H}^2_{D^{\rm{an}/S}}(X^{\rm{an}/S}, \mu_n)
\]
sends $c\ell_X(D) \in \rm{H}^2_D(X, \mu_n)$ to $c\ell_{X^{\an/S}}(D^{\an/S}) \in \rm{H}^2_{D^{\an/S}}(X^{\an/S}, \mu_n)$.
\end{lemma}
\begin{proof}
    Let $[D]$ be the class in $\rm{H}^1_D(X, \bf{G}_m)$ corresponding to the pair $(\O_X(D), \varphi_D)$, where $\varphi_D \colon \O_U \xrightarrow{\sim} \restr{\O_X(D)}{U}$ is the algebraic counterpart of the isomorphism constructed in \cref{eqn:can-iso}. Then we note that $c\ell_X(D)$ is defined as $\delta_X([D])$, where $\delta_X$ is the boundary map in the algebraic Kummer sequence.  Since the relative analytification of the pair $(\O_X(D), \varphi_D)$ is isomorphism to the pair $(\O_{X^{\an/S}}(D^{\an/S}), \varphi_{D^{\an/S}})$, we conclude that the natural map $c_{X/S}\colon \rm{H}^1_D(X, \bf{G}_m) \to \rm{H}^1_{D^{\an/S}}(X^{\an/S}, \bf{G}_m)$ sends the class $[D]$ to $[D^{\an/S}]$. Using the compatibility between algebraic and analytic Kummer exact sequences, we conclude that
    \[
    c\ell_{X^{\an/S}}(D^{\an/S}) = \delta_{X^{\an/S}}([D^{\an/S}]) = \delta_{X^{\an/S}}(c_{X/S}^*[D]) = c_{X/S}^*\bigl(\delta_{X}([D])\bigr) = c_{X/S}^*(c\ell_X(D)),
    \]
    where we slightly abuse notation and denote by $c_{X/S}$ both natural comparison morphisms $\rm{H}^1_D(X, \bf{G}_m) \to \rm{H}^1_{D^{\an/S}}(X^{\an/S}, \bf{G}_m)$ and $\rm{H}^2_D(X, \mu_n) \to \rm{H}^2_{D^{\an/S}}(X^{\an/S}, \mu_n)$. 
\end{proof}

\subsubsection{First Chern classes}

In this subsection, we define the first Chern classes of line bundles.
Although our discussion is essentially just a variant of \cref{defn:cycle-clas-divisors}, it is convenient to treat this construction separately since it applies to a more general setup.  

More precisely, the goal is to define a class $c_1(\cal{L}) \in \rm{H}^2(X, \mu_n)$ for any line bundle $\cal{L}$ on $X$. For this, we recall that \cref{lemma:local-cohomology} ensures that $\rm{H}^1(X, \bf{G}_m)$ is canonically isomorphic to $\rm{Pic}(X)$, so the isomorphism class of a line bundle $\cal{L}$ defines a class $[\cal{L}]\in \rm{H}^1(X, \bf{G}_m)$. Now we combine it with the boundary map
\[
\delta_X\colon \rm{H}^1(X, \bf{G}_m) \to \rm{H}^2(X, \mu_{n})
\]
to get the first Chern class of $\cal{L}$:

\begin{variant}\label{variant:first-chern-class}
    The \emph{first Chern class $c_1(\cal{L})\in \rm{H}^2(X, \mu_n)$} of a line bundle $\cal{L}$ on $X$ is defined to be $\delta_X([\cal{L}])\in \rm{H}^2(X, \mu_n)$.
\end{variant}

\begin{remark}\label{rmk:cycle-class-chern-class} Let $D\subset X$ be an effective Cartier divisor, and $\iota \colon \rm{H}^2_D(X, \mu_n) \to \rm{H}^2(X, \mu_n)$ the natural morphism. Then $\iota(c\ell_X(D))=c_1(\O_X(D))$ as can be directly seen from the construction.  
\end{remark}

\begin{remark}\label{rmk:base-change-chern-classes}
The formation of the first Chern class commutes with base change along \emph{arbitrary} morphisms of locally noetherian analytic adic spaces $f\colon X' \to X$. The proof is analogous to that of \cref{lemma:base-change-divisor-classes} and boils down to the equality $f^*[\cal{L}]=[f^*\cal{L}]\in \rm{H}^2(X', \mu_n)$. 
\end{remark}

Now we show that the analytic first Chern classes are compatible with the algebraic first Chern classes:

\begin{lemma}
\label{lemma:comparison-first-chern-classes} 
Let $S=\Spa(A, A^+)$ be a strongly noetherian Tate affinoid, let $X$ be a locally finite type $A$-scheme with the relative analytification $X^{\an/S}$, and let $\cal{L}\in \rm{Pic}(X)$. Then the natural comparison morphism
\[
c_{X/S}^*\colon \rm{H}^2(X, \mu_n) \longrightarrow \rm{H}^2(X^{\rm{an}/S}, \mu_n)
\]
sends $c_1(\cal{L}) \in \rm{H}^2(X, \mu_n)$ to $c_1(c_{X/A}^*\cal{L}) \in \rm{H}^2(X^{\an/S}, \mu_n)$.
\end{lemma}
\begin{proof}
    The proof is essentially identical (and, in fact, easier) to that of \cref{lemma:comparison-cycles-of-divisors}. We leave details to the interested reader. 
\end{proof}

\subsection{Projective bundle and blow-up formulas}

In this subsection, we prove the projective bundle and blow-up formulas. This will be the crucial ingredient in the extension of (localized) cycle classes from the case of divisors to the case of general lci immersions $Y\subset X$.  

Throughout this subsection, we fix a locally noetherian analytic adic space $X$, an integer $n \in \O_X^\times$, and set $\Lambda \coloneqq \Z/n\Z$.

\subsubsection{Projective bundle formula}

The main goal of this subsection is to prove the projective bundle formula. For this, we fix a vector bundle $\cal{E}$ on $X$ of rank $d+1$ and consider the associated projective bundle
\[
f\colon P=\bf{P}_X(\cal{E}) \to X
\]
with the universal line bundle $\O_{P/X}(1)$;
see \cite[Def.~7.11]{adic-notes} for the precise definition of $\bf{P}_X(\cal{E})$ in the context of adic spaces. 

\begin{construction} The first Chern class $c_1\bigl(\O(1)\bigr)\in \rm{H}^2(P, \mu_n)$ defines a morphism $\ud{\Lambda}_{P} \to \mu_n[2]$. After twisting, this becomes a morphism
\[
c_1\colon \ud{\Lambda}_P(-1)[-2] \to \ud{\Lambda}_X.
\]
By the $(f^*, \rm{R}f_*)$-adjunction, this defines a morphism
\[
c_1\colon \ud{\Lambda}_X(-1)[-2] \to \rm{R}f_* \ud{\Lambda}_P
\]
Using the multiplicative structure on $\rm{R}f_* \ud{\Lambda}_P$ coming from the cup-product map, we get a morphism
\[
c_1^k \colon \ud{\Lambda}_X(-k)[-2k] \to \rm{R}f_* \ud{\Lambda}_P
\]
for any $k\geq 0$. 
\end{construction}

\begin{proposition}(Projective Bundle Formula) 
Let $\cal{E}$ be a vector bundle on $X$ of rank $d+1$.
Let $f\colon P=\bf{P}_X(\cal{E}) \to X$ be the associated projective bundle. Then the natural morphism
\[
\gamma = \bigoplus_{i=0}^d c_1^i \colon \bigoplus_{i=0}^d \ud{\Lambda}_X(-i)[-2i] \to \rm{R}f_* \ud{\Lambda}_P
\]
is an isomorphism.
\end{proposition}
\begin{proof}
    By \cite[Prop.~8.2.3~(ii)]{Huber-etale}, $\rm{R}^if_*\ud{\Lambda}_P$ is overconvergent for all $i\geq 0$. Therefore, it suffices to show that $\gamma$ is an isomorphism over geometric points of rank-$1$. Since first Chern classes commute with arbitrary base change (see \cref{rmk:base-change-chern-classes}) and the formation of $\rm{R}f_* \ud{\Lambda}_P$ commutes with taking stalks (see \cite[Prop.~2.6.1]{Huber-etale}), it suffices to prove the claim under the additional assumption that $X=\Spa(C, \O_C)$ for an algebraically closed nonarchimedean field $C$. Then $P$ algebraizes to a projective space $P^{\rm{alg}}=\bf{P}^d_C$.
    Therefore, the result follows from its algebraic counterpart established in \cite[Exp.~VII, Th.~2.2.1]{SGA5} as well as the comparison results \cite[Th.~3.7.2]{Huber-etale} and \cref{lemma:comparison-first-chern-classes}. 
\end{proof}

\subsubsection{Blow-up formula}

In this subsection, we discuss the blow-up formula in the context of adic spaces. We refer to \cite[Def.~5.3]{adic-notes} for the discussion of lci immersions in the context of adic spaces (see also \cite[Def.~5.4]{Guo-Li}) and to \cite[Def.~7.11]{adic-notes} for the definition of a blow-up.  

We fix a locally noetherian analytic adic space $X$ with an lci immersion $i\colon Z\hookrightarrow X$ of pure codimension $c$. Let $\cal{I}_Z$ be the coherent ideal sheaf of the Zariski-closed immersion $i$.
Define the conormal bundle to be $\cal{N}_{Z|X}\coloneqq \cal{I}_Z/\cal{I}_Z^2$;
this is a vector bundle\footnote{This uses the lci assumption.} over $Z$ of rank $c$. We consider the blow-up morphism
\[
\pi\colon \widetilde{X} \coloneqq \rm{Bl}_Z(X) \to X
\]
and define the exceptional divisor via the formula
\[
E\coloneqq \ud{\rm{Proj}}^{\an}_X \bigoplus_{n\geq 0} \cal{I}_Z^n/\cal{I}_Z^{n+1} \simeq \bf{P}_Z(\cal{N}_{Z|X}).
\]
Similarly as in algebraic geometry, one checks that $E$ is naturally an effective Cartier divisor in $\rm{Bl}_Z(X)$ and that $\O_{\widetilde{X}}(-E) \simeq \O_{\widetilde{X}/X}(1)$. 

In order to compute cohomology of the blow-up, we will need to use the construction of cup products in local cohomology. For this, we recall that given a Zariski-closed immersion $i\colon Z \hookrightarrow X$ and $\F, \G\in D(X_\et; \Lambda)$, there is a canonical morphism
\[
w_{\F, \G} \colon i^*\F \otimes^L \rR i^!\G \to \rR i^!(\F \otimes^L \G),
\]
which is adjoint to $i_*(i^*\F \otimes^L \rR i^!\G) \xr[\sim]{\PF^{-1}} \F \otimes^L i_* \rR i^!\G \xr{\id \otimes^L \epsilon_i} \F \otimes^L \G$. 

\begin{construction}[Cup product in local cohomology]\label{construction:local-cup-product} 
\begin{enumerate}[leftmargin=*,label=\upshape{(\roman*)}]
    \item Let $i\colon Z \hookrightarrow X$ be a Zariski-closed immersion and let $\F, \G \in D(X_\et; \Lambda)$.
    Then, for each integers $i$ and $j$, there is a functorial map
    \[
    \blank \cup \blank \colon \rm{H}^i(Z, i^*\F) \otimes \rm{H}^j_Z(X, \G) \to \rm{H}^{i+j}_Z(X, \F\otimes^L \G)
    \]
    defined as the composition
    \begin{multline*}
    \Hh^i(Z, i^*\F) \otimes \Hh^j_Z(X, \G) \simeq \Hom\bigl(\ud{\Lambda}_Z, i^*\F[i]\bigr) \otimes \Hom\bigl(\ud{\Lambda}_Z, \rR i^! \G[j]\bigr) \xrightarrow{(f,g) \mapsto f\otimes^L g} \\ 
    \Hom\bigl(\ud{\Lambda}_Z, i^*\F[i] \otimes^L \rR i^!\G[j]\bigr) 
     \xr{w_{\F, \G}\circ \blank} \Hom\bigl(\ud{\Lambda}_Z, \rR i^!(\F\otimes^L \G)[i+j]\bigr) \simeq \Hh^{i+j}_Z(X, \F\otimes^L \G).
    \end{multline*}
    \item Let $i_1\colon Z_1 \hookrightarrow X$, $i_2\colon Z_2 \hookrightarrow X$ be two Zariski-closed immersions, let $i \colon Z \coloneqq Z_1\times_X Z_2 \hookrightarrow X$ be their intersection, and let $\F, \G \in  D(X_\et; \Lambda)$. Then, for each integers $i$ and $j$, there is a functorial map
    \[
    - \cup -\colon \rm{H}^i_{Z_1}(X, \F) \otimes \rm{H}^j_{Z_2}(X, \G) \to \rm{H}^{i+j}_Z(X, \F\otimes^L \G)
    \]
    defined as the composition
    \begin{multline*}
    \Hh^i_{Z_1}(X, \F) \otimes \Hh^j_{Z_2}(X, \G) \simeq \Hom\bigl(i_{1, *}\ud{\Lambda}_{Z_1}, \F[i]\bigr) \otimes \Hom\bigl(i_{2, *}\ud{\Lambda}_{Z_2}, \G[j]\bigr) \xrightarrow{(f,g) \mapsto f\otimes^L g}  \\
    \Hom\bigl(i_{1, *}\ud{\Lambda}_{Z_1} \otimes^L i_{2, *}\ud{\Lambda}_{Z_2}, \F[i] \otimes^L \G[j]\bigr) \xr{\sim}
     \Hom(i_* \ud{\Lambda}_Z, \F \otimes^L \G [i+j]) \simeq \Hh^{i+j}_Z(X, \F\otimes^L \G),
    \end{multline*}
    where the third map is given by precomposing with the inverse of the K\"unneth formula isomorphism 
    \[i_{1, *} \ud{\Lambda}_{Z_1} \otimes^L i_{2, *} \ud{\Lambda}_{Z_2} \xr{\sim} i_* \ud{\Lambda}_Z.
    \]
\end{enumerate}
\end{construction}

\begin{convention}\label{convention:negative-divisors} We denote by $c\ell_{\widetilde{X}}(-E)$ the cohomology class $-c\ell_{\widetilde{X}}(E)\in \rm{H}^2_E\bigl(\widetilde{X}, \Lambda(1)\bigr)$.
\end{convention}
\begin{proposition}\label{prop:blow-up-formula}(Blow-up Formula) Let $Z \hookrightarrow X$ be an lci immersion of pure codimension $c$ and let $\pi \colon \widetilde{X} \to X$ be the blow-up of $Z$ in $X$. 
Then there is a canonical isomorphism
\[
\alpha\colon \bigoplus_{i=1}^{c-1} \rm{H}^{2(c-i)}\bigl(Z, \Lambda(c-i)\bigr) \oplus \rm{H}^{2c}_Z\bigl(X, \Lambda(c)\bigr) \to \rm{H}^{2c}_E\bigl(\rm{Bl}_Z(X), \Lambda(c)\bigr)
\]
given by the formula 
\[
\alpha\bigl((\gamma_i), \gamma\bigr) = \sum_{i=1}^{c-1} \gamma_i \cdot c\ell_{\widetilde{X}}(-E)^i + \pi^*\gamma,
\]
where $E\subset \rm{Bl}_Z(X)$ is the exceptional divisor of the blow-up and the product $\gamma_i \cdot c\ell_{\widetilde{X}}(-E)^i$ is from \cref{construction:local-cup-product}.
\end{proposition}
\begin{proof}
    The proof is a formal consequence of the projective bundle formula and the excision sequence. For example, the proof of \cite[Lem.~1.1.1]{Fujiwara-purity} applies verbatim in this context.
\end{proof}

\subsection{Higher-dimensional cycle classes}

In this subsection, we extend the theory of cycle classes to arbitrary lci immersions of pure codimension by following the strategy taken in \cite{Fujiwara-purity}. The case of effective Cartier divisors was already treated in \cref{section:divisor-classes};
the general case is reduced to this via certain blow-ups.
Throughout this subsection, we fix a locally noetherian analytic adic space $X$, an integer $n \in \O_X^\times$, and denote $\Lambda \coloneqq \Z/n\Z$.

Let $i\colon Z \hookrightarrow X$ be an lci immersion of pure codimension $c$; our current goal is to define the (localized) cycle class $c\ell_X(Z) \in \rm{H}^2_Z\bigl(X, \Lambda(c)\bigr)$. For this, we consider the blow-up 
\[
\pi\colon \widetilde{X}\coloneqq \rm{Bl}_Z(X) \to X
\]
with exceptional divisor $E\subset \widetilde{X}$.
Now \cref{defn:cycle-clas-divisors} provides us with a class 
\[
c\ell_{\widetilde{X}}(-E) \in \rm{H}^2_E\bigl(\widetilde{X}, \Lambda(1)\bigr),
\]
while the blow-up formula from \cref{prop:blow-up-formula} implies that there is a unique monic relation
of degree $c$
\begin{equation}\label{eqn:cycle-class-higher-dim}
c\ell_{\widetilde{X}}(-E)^c + \sum_{i=1}^c c_i \cdot c\ell_{\widetilde{X}}(-E)^{c-i} =0 \in \rm{H}^{2c}_E\bigl(\widetilde{X}, \Lambda(c)\bigr)
\end{equation}
with $c_i\in \rm{H}^{2i}\bigl(Z, \Lambda(i)\bigr)$ and $c_c\in \rm{H}^{2c}_Z\bigl(X, \Lambda(c)\bigr)$.

\begin{definition}\label{defn:higher-dimensional-cycle}
    The \emph{(localized)  cycle class} $c\ell_X(Z) \in \rm{H}^{2c}_Z\bigl(X, \Lambda(c)\bigr)$ is the class $c_c \in \rm{H}^{2c}_Z\bigl(X, \Lambda(c)\bigr)$ from \cref{eqn:cycle-class-higher-dim}.
\end{definition}

\begin{variant}
\label{variant:cycle-class-map} 
Similar to \cref{variant:cycle-class-map-divisors}, it will sometimes be more convenient to think of the cycle class as a map
\[
\rm{cl}_i = \rm{cl}_X(Z) \colon i_* \ud{\Lambda}_{Z} \to \ud{\Lambda}_{X}(c)[2c].
\]
\end{variant}

\begin{lemma}[Tranversal Base Change]
\label{lemma:base-change-lci-classes}
Let
\[
\begin{tikzcd}
Z' \arrow{d}\arrow{r} & Z \arrow{d}\\
X' \arrow{r}{f} & X
\end{tikzcd}
\]
be a cartesian diagram of locally noetherian analytic adic spaces (still with $n \in \cO^\times_X$) such that the vertical arrows are lci immersions of pure codimension $c$.
Then  $f^*c\ell_X(Z)=c\ell_{X'}(Z')\in \rm{H}^{2c}_{Z'}\bigl(X', {\Lambda}(c)\bigr)$.
\end{lemma}
\begin{proof}
    Let $\cal{I}_Z$ and $\cal{I}_{Z'}$ be the ideal sheaves of the Zariski-closed immersions $Z \hookrightarrow X$ and $Z' \hookrightarrow X'$, respectively.
    Then \cite[Lem.~5.8]{adic-notes} implies that $f^*\cal{I}_Z \simeq \cal{I}_{Z'}$. Then \cite[Rmk.~7.8]{adic-notes} ensures that there is a natural isomorphism
    \[
    \alpha\colon \rm{Bl}_{Z'}(X') \xlongrightarrow{\sim} \rm{Bl}_Z(X) \times_X X'.
    \]
    
    Denote by $E\subset \rm{Bl}_Z(X)$ and $E'\subset \rm{Bl}_{Z'}(X')$ the corresponding exceptional divisors.
    It is easy to see that $\alpha$ restricts to an isomorphism $\restr{\alpha}{E'}\colon E' \xrightarrow{\sim} E\times_{X} X'$. In particular, the cartesian diagram
    \[
    \begin{tikzcd}
    E' \arrow{d}\arrow{r} & E \arrow{d}\\
    \rm{Bl}_{Z'}(X') \arrow{r} & \rm{Bl}_Z(X)
    \end{tikzcd}
    \]
    is transversal.
    After unraveling \cref{defn:higher-dimensional-cycle}, the question then boils down to showing that the natural morphism $\rm{H}^2_{E}(\rm{Bl}_Z(X), \mu_n) \to \rm{H}^2_{E'}(\rm{Bl}_{Z'}(X
    ), \mu_n)$ sends $c\ell_{\rm{Bl}_Z(X)}(E)$ to $c\ell_{\rm{Bl}_{Z'}(X')}(E')$. This follows from \cref{lemma:base-change-divisor-classes}. 
\end{proof}

Now we study the behavior of cycle classes with respect to intersections. 

\begin{definition}\label{defn:normal-intersection}
    A collection of effective Cartier divisors $\{D_i\}_{i=1, \dots, c}$ on $X$ \emph{crosses normally} if, for every subset $I\subset \{1, \dotsc, c\}$, the (``scheme-theoretic'') intersection
    \[
    D_I \coloneqq \bigcap_{i\in I} D_i
    \]
    is an lci Zariski-closed subspace of $X$ of pure codimension $\abs{I}$.
\end{definition}

\begin{lemma}\label{lemma:composition-of-cycles} Let $\{D_i\}_{i=1, \dots, c}$ be a collection of effective Cartier divisors on $X$ which crosses normally.
Let $Z=\cap_{i\in I} D_i$ be their (``scheme-theoretic'') intersection. Then $c\ell_X(Z)$ is given by the cup product 
\[
c\ell_X(Z) = \bigcup_{i=1}^c c\ell_X(D_i) \in \rm{H}^{2c}_Z\bigl(X, \Lambda(c)\bigr).
\]
\end{lemma}
\begin{proof}
    We denote by $\widetilde{D}_i\coloneqq \rm{Bl}_Z(D_i) \subset \widetilde{X} = \rm{Bl}_Z(X)$ the strict transform of $D_i$ in the blow-up $\pi\colon \widetilde{X} \to X$. Since $Z\subset X$ is an lci Zariski-closed immersion, we see that $\widetilde{D}_i$ and $\pi^{-1}(D_i)$ are effective Cartier divisors and $\O_{\widetilde{X}}\bigl(\pi^{-1}(D_i)\bigr)\simeq \O_{\widetilde{X}}(\widetilde{D}_i) \otimes \O_{\widetilde{X}}(E)$. Therefore, we have an equality of cycle classes
    \[
    c\ell_{\widetilde{X}}(\pi^{-1}(D_i)) = c\ell_{\widetilde{X}}(\widetilde{D}_i) + c\ell_{\widetilde{X}}(E) \in \rm{H}^2_{\pi^{-1}(D_i)}\bigl(\widetilde{X}, \Lambda(1)\bigr).
    \]
    Now we consider the cup-product of all these classes to get
    \begin{equation}\label{eqn:cycle-classes-of-strict-transforms}
    \bigcup_{i=1}^c \Bigl(c\ell_{\widetilde{X}}\bigl(\pi^{-1}(D_i)\bigr) + c\ell_{\widetilde{X}}(-E) \Bigr) = \bigcup_{i=1}^c c\ell_{\widetilde{X}}(\widetilde{D}_i),
    \end{equation}
    where the equality takes place in $\rm{H}^{2c}_{\cap_{i=1}^c \pi^{-1}(D_i)}\bigl(\widetilde{X}, \Lambda(c)\bigr) = \rm{H}^{2c}_E\bigl(\widetilde{X}, \Lambda(c)\bigr)$. 
    Since the intersection of strict transforms $\cap_{i=1}^c \widetilde{D}_i=\varnothing$ is empty, the product $\bigcup_{i=1}^c c\ell_{\widetilde{X}}(\widetilde{D}_i)$
    factors through the natural morphism $\rm{H}^{2c}_{\varnothing}\bigl(\widetilde{X}, \Lambda(c)\bigr)=0 \to \rm{H}^{2c}_{\cap_{i=1}^c \pi^{-1}(D_i)}\bigl(\widetilde{X}, \Lambda(c)\bigr)$.
    In particular, we see that $\bigcup_{i=1}^c c\ell_{\widetilde{X}}(\widetilde{D}_i) = 0$.
    Hence, \cref{eqn:cycle-classes-of-strict-transforms} simplifies to the equation 
    \[
    \sum_{j=0}^c \sigma_j\biggl(c\ell_{\widetilde{X}}\left(\pi^{-1}\left(D_1\right)\right), \dots, c\ell_{\widetilde{X}}\left(\pi^{-1}\left(D_c\right)\right)\biggr) \cdot c\ell_{\widetilde{X}}(-E)^{c-j} =0 \in \rm{H}^{2c}_E(\widetilde{X}, \Lambda(c)), 
    \]
    where $\sigma_j$ denotes the $j$-th elementary symmetric polynomial and $\sigma_0=1$.
    Therefore, \cref{defn:higher-dimensional-cycle} directly implies that
    \[
    \pi^*c\ell_{X}(Z) = \sigma_c\Bigl(c\ell_{\widetilde{X}}\bigl(\pi^{-1}({D}_1)\bigr), \dots, c\ell_{\widetilde{X}}\bigl(\pi^{-1}({D}_c)\bigr)\Bigr) = \bigcup_{i=1}^c c\ell_{\widetilde{X}}\bigl(\pi^{-1}(D_i)\bigr). 
    \]
    Thus, \cref{lemma:base-change-divisor-classes} ensures that 
    \[
    \pi^* c\ell_X(Z)=\bigcup_{i=1}^c c\ell_{\widetilde{X}}(\pi^{-1}(D_i)) = \bigcup_{i=1}^c \pi^* c\ell_{X}(D_i) = \pi^*\left(\bigcup_{i=1}^c c\ell_{X}(D_i)\right).
    \]
    Finally, \cref{prop:blow-up-formula} implies that $c\ell_{X}(Z) = \bigcup_{i=1}^c c\ell_X(D_i)$. 
\end{proof}

\begin{corollary}\label{cor:gysin-composition} Let $\{D_1, \dots, D_c\}$ be a collection of normally crossing effective Cartier divisors on $X$, and let  $c_1$, $c_2$ be two nonnegative integers with $c_1+c_2\leq c$.
Set $Y \colonequals \cap_{i=1}^{c_1} D_i$ and $Z \colonequals \cap_{i=1}^{c_1+c_2} D_i$, giving rise to natural Zariski-closed immersions $Z\xhookrightarrow{i_1} Y \xhookrightarrow{i_Y} X$.
Then the following composition commutes:
\[
\begin{tikzcd}[column sep=large]
    i_{Y, *} \bigl(i_{1, *} \ud{\Lambda}_Z\bigr) \arrow[d,sloped,"\sim"] \arrow{r}{i_{Y, *}(\rm{cl}_{i_1})} & i_{Y, *} \ud{\Lambda}_Y(c_2)[2c_2] \arrow{d}{\rm{cl}_{i_Y}(c_2)[2c_2]} \\
    (i_Y \circ i_1)_* \ud{\Lambda}_Z\arrow{r}{\rm{cl}_{i_Y \circ i_1}} & \ud{\Lambda}_X(c_1+c_2)[2(c_1+c_2)]
\end{tikzcd}
\]
\end{corollary}
\begin{proof}
Let us put $W \colonequals \cap_{i = c_1 + 1}^{c_1 + c_2} D_i$, resulting in a Cartesian diagram
\[
\begin{tikzcd}
    Z \arrow{r}{i_1} \arrow{d}{i_2} \arrow{rd}{i_Z} & Y \arrow{d}{i_Y} \\
    W \arrow{r}{i_W} & X.
\end{tikzcd}
\]
Our assertion follows directly once we show that the following diagram commutes:
\[
\begin{tikzcd}
i_{Y, *} \bigl(i_{1, *} \ud{\Lambda}_Z\bigr) \arrow[rrr, bend left=18, "\sim"'] \arrow[rd, "i_{Y, *}\left(\rm{cl}_{i_1}\right)"'] \arrow[rrr,shift left=4ex,font=\scriptsize,eq=cor:gysin-composition-top] &[1em] 
i_{Y, *} \bigl(i_Y^* i_{W, *} \ud{\Lambda}_W\bigr) \arrow[l, "i_{Y, *}(\BC)"'] \arrow{d}{i_{Y, *}\left(i_Y^*\rm{cl}_{i_W}\right)} &
i_{Y, *} \ud{\Lambda}_Y \otimes i_{W, *} \ud{\Lambda}_W \arrow[l, "\PF_{i_Y}"'] \arrow{r}{\KM} \arrow{d}{\id \otimes \rm{cl}_{i_W}} \arrow{rd}{\rm{cl}_{i_Y} \otimes \rm{cl}_{i_W}} &
i_{Z, *} \ud{\Lambda}_Z \arrow{d}{\rm{cl}_{i_Z}} \\
&
i_{Y, *} \ud{\Lambda}_Y(c_2)[2c_2] \arrow[rr, bend right = 15, "{\rm{cl}_{i_Y}(c_2)[2c_2]}"] & 
i_{Y, *} \ud{\Lambda}_Y \otimes \ud{\Lambda}_X(c_2)[2c_2] \arrow[l, "\PF_{i_Y}"] \arrow{r}{\rm{cl}_{i_Y} \otimes \id} & 
\ud{\Lambda}_X(c_1+c_2)[2(c_1+c_2)]
\end{tikzcd}
\]
The commutativity of \cref{cor:gysin-composition-top} (as well as the meaning of $\KM$, $\PF$, and $\BC$) expressing the factorization of K\"unneth map
in terms of the projection formula and the base change map
is explained in \cref{claim factorizing Kunneth map}, whose proof is independent from the rest of our paper.
The left triangle commutes by virtue of \cref{lemma:base-change-lci-classes}, whereas the upper right triangle
commutes thanks to \cref{lemma:composition-of-cycles}.
The rest part is easily seen to be commutative.
\end{proof}

\begin{remark} One can adapt the proof of \cite[Exp.~XVI, Th.~2.3.3]{deGabber} to show that, more generally, for an lci immersion $Z\xhookrightarrow{i} X$ of pure codimension $c$ and an lci immersion $Y\xhookrightarrow{j} Z$ of pure codimension $c'$, the diagram
\[
\begin{tikzcd}
    i_*\bigl(j_* \ud{\Lambda}_Y\bigr) \arrow[d,sloped,"\sim"] \arrow{r}{i_*(\rm{cl}_j)} & i_*\bigl(\ud{\Lambda}_Z(c')[2c']\bigr) \arrow{d}{\rm{cl}_i(c')[2c']} \\
    (i\circ j)_*\ud{\Lambda}_Y \arrow{r}{\rm{cl}_{i\circ j}} & \ud{\Lambda}_X(c+c')[2(c+c')].
\end{tikzcd}
\]
commutes. We do not prove this claim as we never use it in this paper.
\end{remark}

\subsection{Comparison with algebraic cycle classes}

Throughout this subsection, we fix a strongly noetherian Tate affinoid $S=\Spa(A, A^+)$, a locally finite type $A$-scheme $X$, an integer $n\in A^\times$,
and denote $\Lambda \coloneqq \Z/n\Z$.

Recall that \cite[Cycle, Def.~2.1.2]{SGA41/2} defines a cycle class 
\[
c\ell_X(D) \in \rm{H}^2_D(X, \mu_n)
\]
for any effective Cartier divisor $D\subset X$.
This definition has been extended to general lci closed subschemes $Y \subset X$ of pure codimension $c$ by Gabber in \cite[Def.~1.1.2]{Fujiwara-purity}.
So \textit{loc.\ cit.} defines the cycle class
\[
c\ell_{X}(Y) \in \rm{H}^{2c}_Y\bigl(X, \Lambda(c)\bigr).
\] 
The main goal of this subsection is to show that this construction is compatible with \cref{defn:higher-dimensional-cycle} via the relative analytification functor from \cref{construction:relative-analytification}.

For this, we recall that \cite[Cor.~6.5]{adic-notes} (see also \cite[Prop.~5.5]{Guo-Li}) ensures that, for an lci immersion $Y\subset X$ of pure codimension $c$, the relative analytification $Y^{\an/S} \subset X^{\an/S}$ is also an lci immersion of pure codimension $c$. Therefore, \cref{defn:higher-dimensional-cycle} applies to this situation, providing us with the cycle class $c\ell_{X^{\an/S}}(Y^{\an/S}) \subset \rm{H}^{2c}_{Y^{\an/S}}\bigl(X^{\an/S}, \Lambda(c)\bigr)$.

\begin{lemma}\label{lemma:analytification-cycle-classes} 
Let $X$ and $Y$ be locally finite type $A$-schemes and let $Y\subset X$ be an lci immersion of pure codimension $c$.
Then the natural morphism
\[
c_{X/S}^*\colon \rm{H}^{2c}_Y\bigl(X, \Lambda(c)\bigr) \longrightarrow \rm{H}^{2c}_{Y^{\rm{an}/S}}\bigl(X^{\rm{an}/S}, \Lambda(c)\bigr)
\]
sends the algebraic cycle class $c\ell_X(Y)$ to the analytic cycle class $c\ell_{X^{\rm{an}/S}}(Y^{\rm{an}/S})$.
\end{lemma}
\begin{proof}
    We recall that \cite[Lem.~7.14]{adic-notes} provides an isomorphism between the analytification of the algebraic blow-up and the analytic blow-up 
    \[
    \alpha\colon \rm{Bl}_Y(X)^{\rm{an}/S} \simeq \rm{Bl}_{Y^\an/S}(X^\an/S)
    \]
    that restricts to an isomorphism of exceptional divisors.
    Thus, after unraveling the definitions and using the compatibility of the first Chern classes from \cref{lemma:comparison-first-chern-classes}, the question boils down to showing that the natural morphism
    \[
    c_{\rm{Bl}_Y(X)/S}\colon \rm{H}^2_E(\rm{Bl}_Y(X), \mu_n) \longrightarrow \rm{H}^2_{E^{\an/S}}(\rm{Bl}_Y(X)^{\rm{an}/S}, \mu_n)
    \]
    sends $c\ell_{\rm{Bl}_Y(X)}(E)$ to $c\ell_{\rm{Bl}_Y(X)^{\rm{an}/S}}(E^{\an/S})$, where $E$ is the exceptional divisor of $\rm{Bl}_Y(X)$. Therefore, we may and do assume that $Y=D\subset X$ is an effective Cartier divisor. Then the result follows from \cref{lemma:comparison-cycles-of-divisors}.
\end{proof}

\subsection{Cycle class of point}

In this subsection, we discuss a variant of \cref{defn:higher-dimensional-cycle} that takes values in compactly supported cohomology groups.
Throughout this subsection, we fix an algebraically closed nonarchimedean field $C$, an integer $n \in C^\times$,
and set $\Lambda \coloneqq \Z/n\Z$.

Recall that, for every taut, separated, finite type adic $C$-space $X$ and every complex $\F\in D^+(X_\et; \Lambda)$, \cite[Def.~5.4.4]{Huber-etale} defines the compactly supported \'etale cohomology complex $\rm{R}\Gamma_c(X, \F)$.
Now suppose that $X$ is a rigid-analytic space over $C$ of equidimension $d$.
Let $x\in X(C)$ be a classical point of $X$.
Then \cite[Cor.~5.11]{adic-notes} implies that $i\colon \{x\} \hookrightarrow X$ is an lci immersion of pure codimension $d$.
Thus, \cref{defn:higher-dimensional-cycle} provides us with the cycle class $c\ell_{X}(x)\in \rm{H}^{2d}_{x}\bigl(X, \Lambda(d)\bigr)$. 
Now since $x$ is proper over $\Spa(C, \O_C)$, the $\rm{R}\Gamma_c(X, -)$-functor applied to the counit morphism
\[
i_*\rm{R}i^!\ud{\Lambda}_X(c) \longrightarrow \ud{\Lambda}_X(c)
\]
yields a morphism
\[
\rm{R}\Gamma_x\bigl(X, \Lambda(c)\bigr) \longrightarrow \rm{R}\Gamma_c\bigl(X, \Lambda(c)\bigr).
\]
In particular, this induces a canonical morphism
\begin{equation}\label{eqn:canonical-map-comp-supported}
\rm{H}^{2c}_x\bigl(X,\Lambda(c)\bigr) \longrightarrow \rm{H}^{2c}_c\bigl(X, \Lambda(c)\bigr).
\end{equation}

\begin{definition}
\label{defn:cycle-clas-divisors-compact} 
In the situation above, the \emph{(compactly supported)  cycle class} $c\ell_X(x)\in \rm{H}^{2c}_c\bigl(X, \Lambda(c)\bigr)$ is the image of $c\ell_X(x)\in \rm{H}^{2c}_x\bigl(X, \Lambda(c)\bigr)$ under the morphism
\[
\rm{H}^{2c}_x\bigl(X, \Lambda(c)\bigr) \longrightarrow \rm{H}^{2c}_c\bigl(X, \Lambda(c)\bigr)
\]
from \cref{eqn:canonical-map-comp-supported}. 
\end{definition}

The main result of this subsection is that the compactly supported version of the cycle class of a point behaves well with respect to \'etale morphisms and Zariski-closed immersions.
We begin with the case of \'etale morphisms: 
\begin{lemma}
\label{lemma: cycle class of points compatible with etale morphism}
Let $f\colon X \to Y$ be an \'etale morphism between smooth separated taut rigid-analytic spaces over $C$ of equidimension $d$.
Let $x \in X(C)$ be a classical point.
Then the natural morphism
\[
\rm{H}^{2d}_c\bigl(X, \Lambda(d)\bigr) \longrightarrow \rm{H}^{2d}_c\bigl(Y, \Lambda(d)\bigr)
\]
sends $c\ell_X(x)$ to $c\ell_Y\bigl(f(x)\bigr)$.
\end{lemma}
\begin{proof}
    We first consider the case of an open immersion $f \colon X \hookrightarrow Y$. In this case, the claim follows from the following commutative diagram:
    \[
    \begin{tikzcd}
        \rm{H}^{2d}_x\bigl(X, \Lambda(d)\bigr)\arrow{d} &\arrow[l,"\sim"'] \rm{H}^{2d}_{f(x)}\bigl(Y,\Lambda(d)\bigr)\arrow{d} \\
        \rm{H}^{2d}_c\bigl(X, \Lambda(d)\bigr) \arrow{r} & \rm{H}^{2d}_c\bigl(Y, \Lambda(d)\bigr)
    \end{tikzcd}
    \]
    Now we treat the general case.
    The case of open immersions implies that, for the purposes of proving the statement, we may replace $X$ with any open neighborhood of $x$ and $Y$ with any open neighborhood of $f(x)$. Therefore, we can first reduce to the case when $X$ and $Y$ are affinoids. Then \cite[Cor.~1.7.4]{Huber-etale} implies that $f$ is quasi-finite.
    Thus, \cite[Lem.~1.4.7, Lem.~1.5.2~(f)]{Huber-etale} imply that we can further localize to the case of a finite \'etale morphism $f$.
    
    Choose a pseudo-uniformizer $\varpi\in \O_C$.
    We have
    \[
    \wdh{\O}^+_{Y, f(x)}\Bigl[\frac{1}{\varpi}\Bigr]\simeq \wdh{k\bigl(f(x)\bigr)^+}\Bigl[\frac{1}{\varpi}\Bigr] \simeq C,
    \]
    where $\wdh{(-)}$ stands for the $\varpi$-adic completion (see e.g.\ \cite[Prop.~2.25]{Scholze-perfectoid} or \cite[Prop.~7.5.5~5]{Bhatt-notes}).
    Since $\cO^+_{Y, f(x)}$ is $\varpi$-henselian (as a colimit of $\varpi$-complete, hence $\varpi$-henselian rings), \cite[Prop.~5.4.53]{GR} implies that 
    \[
    {\O_{Y, f(x)}}_{\fet} \simeq C_{\fet}.
    \]
    In particular, any finite \'etale $\O_{Y, f(x)}$-algebra is split. Therefore, a standard approximation argument implies that, after passing to an open neighborhood of $f(x)$, we can assume that $X=\bigsqcup_{i=1}^n Y$. Then, by passing to a neighborhood of $x$, we can assume that $X=Y$. 
    In this case, the claim is obvious. 
\end{proof}

\begin{lemma}\label{lemma:class-of-point-composes-well} Let $X$ and $Y$ be smooth separated taut rigid-analytic spaces over $C$ of equidimension $d_X$ and $d_Y$ respectively. Let $i\colon X \hookrightarrow Y$ be a Zariski-closed immersion, and let $x\in X(C)$ be a classical point. Then the morphism
\[
\Hh^{2d_X}_c\bigl(Y, \rm{cl}_i(d_X)\bigr)\colon \Hh^{2d_X}_c\bigl(X, \Lambda(d_X)\bigr) \to \Hh^{2d_Y}_c\bigl(Y, \Lambda(d_Y)\bigr)
\]
sends $c\ell_X(x)$ to $c\ell_Y(x)$.
\end{lemma}
\begin{proof}
    \cref{lemma: cycle class of points compatible with etale morphism} implies that we can replace $Y$ with any open subspace $x\in U\subset Y$ (and $X$ with $X\cap U$).
    Therefore, we can assume that there is a collection of normally crossing (in the sense of \cref{defn:normal-intersection}) effective Cartier divisors $\{D_1,\dotsc,D_{d_Y}\}$ on $Y$ such that $X = \cap_{i=1}^{d_Y - d_X} D_i$ and $\{x\} = \cap_{i=1}^{d_Y} D_i$.
    Then the result follows automatically from \cref{cor:gysin-composition}. 
\end{proof}

\section{Rigid-analytic curves}
\label{section:curves}

In this section, we study some properties of smooth rigid-analytic curves over an algebraically closed nonarchimedean field.
These results will be crucial in the construction of the trace map in \'etale cohomology.

Throughout this section, we fix an algebraically closed nonarchimedean field $C$. We denote its ring of integers by $\O_C$, its maximal ideal by $\m_C\subset \O_C$, and its residue field by $k_C \colonequals \cO_C/\fm_C$. We also choose a pseudo-uniformizer $\varpi \in \O_C$. 

\subsection{Geometry of curves}

In this subsection, we collect some results about the geometry of rigid-analytic curves. In particular, we recall that smooth (quasicompact and separated) rigid-analytic curves behave similarly to algebraic curves.
We also discuss that rigid-analytic curves often admit particularly nice formal models.

\begin{definition}
    A rigid-analytic $C$-space $X$ is a \emph{rigid-analytic $C$-curve} if $X$ is of pure dimension $1$ (in the sense of \cite[Def.~1.8.1]{Huber-etale}).
\end{definition}

We start with a technical lemma that will be handy in a number of situations later in this paper. 

\begin{lemma}
\label{lemma:finite-flat-curves} 
Let $f\colon X \to Y$ be a finite morphism of smooth rigid-analytic curves over $C$. Then $f$ is flat. 
\end{lemma}
\begin{proof}
    Without loss of generality, we can assume that both $X$ and $Y$ are affinoid.
    In this situation, $f\colon X=\Spa(B, B^\circ) \to Y=\Spa(A, A^\circ)$ is induced by a finite morphism $f^\#\colon A \to B$ of $K$-affinoid domains.  
    
    Then it suffices to show that, for every maximal ideal $\m\subset B$ with the pre-image $\fn\coloneqq f^{\#, -1}(\m)$ (this prime ideal is automatically maximal), the natural morphism $A_{\fn} \to B_{\m}$ is flat. Now \cite[Lem.~1.8.6~(ii)]{Huber-etale} and \cite[Th.~II.10.1.8]{FujKato} implies that $\dim A_{\fn} =\dim B_{\m} =1$. Furthermore, \cite[Th.~3.6.3]{FvdP04} ensures that $A_{\fn}$ are $B_{\m}$ are regular local rings. Therefore, flatness of $A_{\fn} \to B_{\m}$ follows directly from Miracle Flatness (see \cite[\href{https://stacks.math.columbia.edu/tag/00R4}{Tag 00R4}]{stacks-project}).
\end{proof}

Now we discuss formal models of rigid-analytic curves over $C$. We start with the following definition of semi-stable formal $\O_C$-curves: 

\begin{definition}
\label{defn:ss-formal} 
We say that an admissible formal $\O_C$-scheme $\cX$ is a \emph{semi-stable formal $\O_C$-curve} if the special fiber has pure dimension $1$ and
for any point $x\in \cX$, either the structure map
$f \colon \cX \to \Spf(\O_C)$ is smooth at $x$, or there is a pointed affine formal $\O_C$-scheme $(\Spf A, y)$ such that there exists a diagram of pointed formal schemes
\[
\begin{tikzcd}
& \left(\Spf A, y\right)  \arrow[swap]{ld}{g} \arrow{rd}{h}\\
\left(\cX, x\right) \arrow{dr}&  & \left(\Spf \frac{\O_C\langle S,T\rangle}{\left(ST-\alpha_x\right)}, \left\{u,0,0\right\}\right) \arrow{dl} \\
 & (\Spf \O_C, u), & 
\end{tikzcd}
\]
where $u$ is the only point of $\abs{\Spf \O_C}$, $g$ and $h$ are \'etale morphisms, and $\alpha_x \in \m_C$. 
\end{definition}

\begin{remark} A semi-stable formal $\O_C$-curve $\cX$ is rig-smooth if and only if all $\alpha_x$ in \cref{defn:ss-formal} are non-zero. 
\end{remark}

Our next goal is to discuss formal models of (smooth) rigid-analytic curves in more detail. We start with the following algebraization result that seems difficult to find in the existing literature: 

\begin{lemma}\label{lemma:algebraize-curves} Let $\cX$ be a finitely presented proper formal $\O_C$-scheme such that the special fiber $\cX_s$ is of pure dimension $1$. Then there is a finitely presented projective $\O_C$-scheme $\mathfrak{X}$ and an $\O_C$-isomorphism $\wdh{\mathfrak{X}} \simeq \cX$, where $\wdh{\mathfrak{X}}$ is the $\varpi$-adic completion of $\mathfrak{X}$. Furthermore, if $\cX$ is admissible (resp.~rig-smooth), then $\mathfrak{X}$ is $\O_C$-flat (resp.~$\mathfrak{X}_\eta$ is $C$-smooth). 
\end{lemma}
\begin{proof}
    First, we note that \cite[\href{https://stacks.math.columbia.edu/tag/09NZ}{Tag 09NZ}]{stacks-project} implies that the special fiber $\cX_s$ admits an ample line bundle $\cal{L}_s$. Then a standard approximation argument implies that we can find a pseudo-uniformizer $\pi\in \O_C$ such that $\cX_0\coloneqq \cX\times_{\Spec \O_C} \Spec \O_C/(\pi)$ admits a line bundle $\cal{L}_0$ such that its restriction $\restr{\cal{L}_0}{\cX_s}$ is isomorphic to $\cal{L}_s$. Furthermore, \cite[\href{https://stacks.math.columbia.edu/tag/0D2S}{Tag 0D2S}]{stacks-project} implies that $\cal{L}_0$ is automatically ample.  

    Now, for each integer $n\geq 0$, we denote by $\cX_n$ the $\O_C/\pi^{n+1}$-scheme $\cX\times_{\Spf \O_C} \Spec \O_C/(\pi^{n+1})$. Then \cite[\href{https://stacks.math.columbia.edu/tag/0C6R}{Tag 0C6R}]{stacks-project} and vanishing of $\rm{H}^2(\cX_n, \pi^n\O_{\cX_{n+1}})$ imply that $\rm{Pic}(\cX_n) \to \rm{Pic}(\cX_{n-1})$ is surjective for any $n\geq 1$. Therefore, we can find a compatible sequence of line bundles $\cal{L}_n$ on $\cX_n$ such that $\restr{\cal{L}_n}{\cX_{n-1}} \simeq \cal{L}_{n-1}$. By passing to the limit, we get a line bundle $\cal{L}$ on $\cX$ such that $\restr{\cal{L}}{\cX_0} \simeq \cal{L}_0$ is ample. Therefore, \cite[Prop.~I.10.3.2]{FujKato} implies that there is a finitely presented proper $\O_C$-scheme $\mathfrak{X}$ with an isomorphism $\wdh{\mathfrak{X}} \simeq \cX$.  

    Now suppose that $\cX$ is admissible. We wish to show that $\mathfrak{X}$ is then $\O_C$-flat. Choose an open subscheme $\Spec A\subset \mathfrak{X}$. Then \cite[Prop.~4.3.4 and Th.~7.3.2]{FGK} imply that the map $A\to \wdh{A}$ is flat. Therefore, the map $A\to \wdh{A} \times A\bigl[\tfrac{1}{\varpi}\bigr]$ is faithfully flat. In particular, it is injective. Our assumption that $\cX$ is admissible implies that $\wdh{A}$ is $\O_C$-flat (i.e.\ has no $\varpi$-torsion). Therefore, $A\hookrightarrow \wdh{A} \times A\bigl[\tfrac{1}{\varpi}\bigr]$ has no $\varpi$-torsion as well. Thus, it is $\O_C$-flat.  

    Finally, we assume that $\cX$ is rig-smooth. Then \cite[Th.~A.3.1]{Conrad99} implies that there is an isomorphism $\mathfrak{X}_\eta^{\rm{an}}\simeq \cX_\eta$. Our assumption implies that $\mathfrak{X}_\eta^{\rm{an}}$ is $C$-smooth. Therefore, $\mathfrak{X}_\eta$ is $C$-smooth due to \cite[Th.~A.2.1]{Conrad99}.
\end{proof}

Finally, we recall the following version of the semi-stable reduction for rigid-analytic curves. This result will be crucial in our proof of \cref{thm:analytic-trace-algebraic}: 

\begin{proposition}
\label{useful proposition on rigid curves}
\leavevmode
\begin{enumerate}[label=\upshape{(\roman*)}]
\item\label{useful proposition on rigid curves-1} \textup{(\cite[Th.~2]{FM86})}
Every irreducible quasi-compact separated rigid-analytic curve over $C$ is either affinoid or proper;
\item\label{useful proposition on rigid curves-2} \emph{(}\cite[Section 1.8]{Lut16}\emph{)} 
The category of smooth proper rigid-analytic curves over $C$ and the category of smooth proper algebraic curves over $C$ are equivalent;
\item\label{useful proposition on rigid curves-3} \emph{(}\cite[Th.~1.1]{vdPut80}\emph{)}
Every smooth affinoid rigid-analytic curve $X$ over $C$ is an open subdomain of a smooth proper rigid-analytic $C$-curve $\ov{X}$.
\item\label{useful proposition on rigid curves-4} Every pair $X \subset \ov{X}$ over $C$ in \cref{useful proposition on rigid curves-3} arises as the rigid generic fiber of open immersions
of admissible formal $\O_C$-schemes $\cX \subset \ov{\cX}$
where $\ov{\cX}$ is a semi-stable formal $\cO_C$-curve;
\item\label{useful proposition on rigid curves-5} In particular, every smooth quasi-compact separated rigid $C$-curve admits a semi-stable formal model over $\cO_C$;
\item\label{useful proposition on rigid curves-6} For a proper and smooth $C$-scheme $X$ of equidimension $1$ and an integer $n\in C^\times$, the natural morphism $\rm{R}\Gamma(X, \mu_n) \to \rm{R}\Gamma(X^{\an}, \mu_n)$ is an isomorphism. 
\end{enumerate}
\end{proposition}

Of course, \cite[Th.~3.7.2]{Huber-etale} guarantees that the conclusion of \cref{useful proposition on rigid curves-6} holds for any proper $C$-scheme $X$. However, we prefer to give a different elementary argument below. 

\begin{proof}
\cref{useful proposition on rigid curves-1}-\cref{useful proposition on rigid curves-3} are stated and proven in the said references. \cref{useful proposition on rigid curves-5} follows from combining \cref{useful proposition on rigid curves-1}-\cref{useful proposition on rigid curves-4}. Since we did not find the exact statement \cref{useful proposition on rigid curves-4} in literature, we give a proof below.  

By \cite[Lem.~2]{B}, we can find an admissible formal $\O_C$-scheme $\ov{\cX}$ and an open formal subscheme $\cX\subset \ov{\cX}$ such that the generic fiber of this immersion is equal to $X\subset \ov{X}$. Now \cite[Cor.~4.4 and 4.5]{Temkin2000} imply that $\ov{\cX}$ is automatically a proper admissible formal $\O_C$-scheme, and \cite[Cor.~B.4]{Z-thesis} implies that the special fiber $\ov{\cX}_s$ is of pure dimension $1$. Therefore, \cref{lemma:algebraize-curves} ensures that there is a projective, finitely presented, flat $\O_C$-scheme $\ov{\mathfrak{X}}$ such that $\wdh{\ov{\mathfrak{X}}} \simeq \ov{\cX}$. Furthermore, the generic fiber $\ov{\mathfrak{X}}_\eta$ is $C$-smooth.  

Now \cite[Th.~1.5]{Temkin-curve} implies that there is an $\eta$-modification\footnote{An $\eta$-modification is a proper morphism $f\colon \ov{\mathfrak{X}}' \to \ov{\mathfrak{X}}$ such that the generic fiber $\ov{\mathfrak{X}}'_\eta$ is schematically dense in $\ov{\mathfrak{X}}'$ and $f_\eta$ is an isomorphism. This automatically implies that $\ov{\mathfrak{X}}'$ is $\O_C$-flat and $f$ is finitely presented.} $f\colon \ov{\mathfrak{X}}' \to \ov{\mathfrak{X}}$ such that $\ov{\mathfrak{X}}'$ is a semi-stable $\O_C$-curve. We note that the completion $\wdh{\ov{\mathfrak{X}}}'$ is a semi-stable formal $\O_C$-curve in the sense of \cref{defn:ss-formal} (see, for example, \cite[Lem.~B.11]{Z-alterations}) and the morphism 
\[
\wdh{f}\colon \ov{\cX}'\coloneqq \wdh{\ov{\mathfrak{X}}}' \to \wdh{\ov{\mathfrak{X}}}\simeq \ov{\cX}
\]
is a rig-isomorphism (see, for example, \cite[Lem.~B.8]{Z-alterations}). Thus, the inclusion $X\subset \ov{X}$ arises as the rigid generic fiber of an open immersion $\wdh{f}^{-1}(\cX) \subset \ov{\cX}'$ where $\ov{\cX}'$ is a semi-stable formal $\O_C$-curve.  

Lastly we give a proof of \cref{useful proposition on rigid curves-6}. Using \cite[Lem.~6.1.4~(2)]{Z-revised} and \cite[\href{https://stacks.math.columbia.edu/tag/03RT}{Tag 03RT}]{stacks-project}, it suffices to show that the natural morphism $\rm{H}^i(X, \mu_n) \to \rm{H}^i(X^{\an}, \mu_n)$ is an isomorphism for $i\leq 2$. We use \cite[Lem.~6.1.4~(3)]{Z-revised}, \cite[\href{https://stacks.math.columbia.edu/tag/03RM}{Tag 03RM}]{stacks-project}, and the schematic and analytic Kummer exact sequence to reduce the question to proving that the natural morphisms
    \[
    \rm{H}^0(X, \bf{G}_m) = \O_X(X)^\times  \to \rm{H}^0(X^\an, \bf{G}_m) = \O_{X^\an}(X^\an)^\times \text{ and}
    \]
    \[
    \rm{H}^1(X, \bf{G}_m) = \rm{Pic}(X) \to \rm{H}^1(X^\an, \bf{G}_m) = \rm{Pic}(X^\an)
    \]
    are isomorphisms. Finally, we note that these maps are isomorphisms due to the rigid-analytic GAGA and properness of $X$ (see \cite[Th.~II.9.4.1 and Cor.~II.9.4.4]{FujKato})
\end{proof}

Now we end the subsection with a version of the Noether normalization result for semi-stable curves over $\O_C$. This result, in conjunction with \cref{lemma:finite-flat-curves} and \cref{useful proposition on rigid curves}, will be a very useful tool to reduce questions about general smooth rigid-analytic curves to the case of the closed unit disc. 

\begin{lemma}\label{lemma:good-noether} Let $\cX$ be a semi-stable proper formal $\O_C$-curve, and let $\{x_1, \dots, x_n\}$ be a finite set of closed points in $\abs{\cX} = \abs{\cX_s}$ that meets each irreducible component of $\abs{\cX}$. Then there is a finite morphism $f\colon \cX \to \wdh{\bf{P}}^1_{\O_C}$ such that $f^{-1}_s(\{\infty\})$ is set-theoretically equal to $\{x_1, \dots, x_n\}$.
\end{lemma}
\begin{proof}
    We first construct the map over the residue field $k_C$. 
    Since $\cX_s$ is nodal, we know that there is an effective Cartier divisor on $\cX_s$
    whose set-theoretic support is the set of nodes $\{x_1, \dots, x_n\}$. 
    Its associated line bundle $\cal{L}_s$ is ample
    on $\cX_s$ because $\{x_1, \dots, x_n\}$ hits every irreducible component of $\cX_s$ (see \cite[\href{https://stacks.math.columbia.edu/tag/0B5Y}{Tag 0B5Y}]{stacks-project}). Furthermore, $\cal{L}_s$ comes with a canonical section $\delta_s \in \cal{L}_s(\cX_s)$ such that the vanishing locus $\rm{V}(\delta_s)$ is set-theoretically equal to $\{x_1, \dots, x_n\}$. Therefore, \cite[Lem.~6]{Kedlaya-covers} ensures that there is an integer $d$ and a section $\alpha\in \cal{L}_s^{\otimes d}(\cX_s)$ such that the natural morphism $\O_{\cX_s}^{\oplus 2} \xrightarrow{\alpha_s +\delta^{\otimes d}_s} \cal{L}^{\otimes d}_s$ is surjective and defines a finite map $f_s\colon \cX_s \to \bf{P}^1_k$ such that the pre-image $f_s^{-1}(\{\infty\})$ is set-theoretically equal to $\{x_1, \dots, x_n\}$. By replacing $\cal{L}_s$ with $\cal{L}_s^{\otimes d}$ (and $\delta$ with $\delta^{\otimes d}$), we may and do assume that $d=1$. 

    Now a standard approximation argument implies that we can find a pseudo-uniformizer $\pi \in \O_C$ and a line bundle $\cal{L}_0$ on $\cX_0\coloneqq \cX\times_{\Spf \O_C} \Spec \O_C/(\pi)$ with two global sections $s_0$ and $\alpha_0$ such that $\restr{\cal{L}_0}{\cX_s} \simeq \cal{L}_s$, $\restr{\delta_0}{\cX_s} = \delta_s$, $\restr{\alpha_0}{\cX_s} = \alpha_s$, and the natural morphism $\O_{\cX_0}^{\oplus 2} \xrightarrow{\alpha_0 + \delta_0} \cal{L}_0$ is surjective and defines a finite morphism $f_0\colon \cX_0 \to \bf{P}^1_{\O_C/(\pi)}$. By a standard deformation theory argument (see the proof of \cref{lemma:algebraize-curves}), we can lift $\cal{L}_0$ to an ample line bundle $\cal{L}$ on $\cX$. Therefore, after replacing $\cal{L}$ with its high enough power (and replacing sections $\alpha_0$ and $\delta_0$ with their powers as well), we can assume that $\rm{H}^1(\cX, \cal{L})=0$. With this cohomology vanishing, we can lift the sections $\delta_0$ and $\alpha_0$ to some sections $\delta\in \cal{L}(\cX)$ and $\alpha\in \cal{L}(\cX)$ such that $\restr{\delta}{\cX_0} = \delta_0$ and $\restr{\alpha}{\cX_0} = \alpha_0$. Lastly, Nakayama's lemma ensures these sections define a surjection $\O_{\cX}^{\oplus 2} \xrightarrow{\alpha+ \delta} \cal{L}$ which, in turn, defines a finite morphism $f\colon \cX \to \wdh{\bf{P}}^1_{\O_C}$ such that $\restr{f}{s} = f_s$. This implies that $f^{-1}_s(\{\infty\})$ is set-theoretically equal to $\{x_1, \dots, x_n\}$. 
\end{proof}

\subsection{Universal compactifications of curves}

In this subsection, we study universal compactifications of curves;
cf.\ \cref{univ-comp}.
The description of universal compactifications obtained in this subsection will be an important input in our construction of analytic trace maps in \cref{section:analytic-trace}. 

\begin{lemma}\label{lemma:compactifications-compatible}
Let $f\colon X=\Spa(B, B^\circ) \to Y=\Spa(A, A^\circ)$ be a finite morphism of rigid-analytic affinoid $C$-curves,  
inducing a finite morphism $f^c\colon X^c\to Y^c$ of universal compactifications.
Then $f^{c, -1}(Y^c\smallsetminus Y) = X^c\smallsetminus X$. 
\end{lemma}
\begin{proof}
    Clearly, $f^{c, -1}(Y^c\smallsetminus Y) \subset X^c\smallsetminus X$. Therefore, it suffices to show that $X^c\smallsetminus X\subset f^{c, -1}(Y^c\smallsetminus Y)$. Equivalently, it suffices to show that the natural morphism
    \[
    j \colon X \to X'\coloneqq X^c\times_{Y^c} Y
    \]
    is an isomorphism. Since $X \to X^c$ is an open immersion, we conclude that $j$ is an open immersion as well. Since $\O_{Y^c}(Y^c) \simeq \O_Y(Y)$, $\O_{X^c}(X^c)\simeq \O_X(X)$, and $X$, $X^c$, $Y$, and $Y^c$ are all affinoids (see \cref{lemma:universal-compactification-affinoid}), we conclude that the natural morphism $\O_{X'}(X')\simeq \O_{Y}(Y) \wdh{\otimes}_{\O_{Y^c}(Y^c)} \O_{X^c}(X^c) \to \O_X(X)$ is a topological isomorphism.  
    
    Now \cref{lemma:compactification-of-finite-morphism} implies that $f^c$ is a finite morphism, and so $j$ is a morphism of finite adic $Y$-spaces. Therefore, $j$ is itself a finite morphism. Thus, \cite[Lem.~1.4.5~(ii)]{Huber-etale} implies that topologically $j$ is a closed morphism, and so \cite[Lem.~B.6.14]{Z-quotients} ensures that $j$ is a closed immersion. Therefore, \cite[Cor.~B.6.9]{Z-quotients} and the established above isomorphism $\O_{X'}(X')\xrightarrow{\sim} \O_X(X)$ implies that $j$ is an isomorphism. 
\end{proof}

\begin{lemma}\label{lemma:compacitification-of-the-disc} Let $X=\bf{D}^1$ be a one-dimensional closed unit disc. Then $X^c\smallsetminus X$ consists of a unique rank-$2$ point $x_+$.
Furthermore, under the isomorphism $X^c=\Spa\left(C\langle T\rangle, \O_C + T \m_C\langle T\rangle\right)$ of \cref{lemma:universal-compactification-affinoid}, the point $x_+$ comes from the valuation
\[
v_{x_+} \colon C\langle T\rangle \to \bigl(\Gamma_C \bigoplus \Z\bigr) \cup \{0\}
\]
\[
v_{x_+}\Bigl(\sum_n a_n T^n\Bigr) = \sup_n\bigl\{ (\abs{a_n}, n) \bigr\}.
\]
\end{lemma}
\begin{proof}
    First, \cref{lemma:universal-compactification-affinoid} implies that $X^c=\Spa\left(C\langle T\rangle, \O_C + T \m_C\langle T\rangle\right)$. Therefore, \cite[Prop.~3.9]{H0} implies that  
    \[
    \abs{\Spa\left(C\langle T\rangle, \O_C + T \m_C\langle T\rangle\right)} \smallsetminus  \abs{\Spa\left(C\langle T\rangle, \O_C\langle T\rangle\right)} = \abs{\Spa\left(C[T], \O_C + T\m_C[T]\right)}  \smallsetminus \abs{\Spa\left(C[T], \O_C[T]\right)},
    \]
    where we endow both $\O_C[T]$ and $\O_C + T\m_C[T]$ with the $\varpi$-adic topology. Therefore, the result follows directly from \cite[L.~11, Ex.~11.3.14]{Seminar} or \cite[Ex.~5.2]{Swan}.
\end{proof}

In order to get some intuition of how the valuation $v_{x_+}$ works, let us do the following easy computation: 

\begin{example}
By how it is defined, we see that $v_{x_+}\left(a_n T^n\right) \leq 1$ is equivalent to
either $\abs{a_n} < 1$ or $\abs{a_n} = 1$ and $n = 0$.
We also see that the subset of $C\langle T\rangle$ defined by the condition $v_{x_+} \leq 1$
is precisely $\O_C + T \m_C\langle T\rangle$.
\end{example}

\begin{lemma}[{\cite[Lem.~5.12]{Swan}}]
\label{lemma:extra-points}
    Let $X$ be a separated, quasi-compact rigid-analytic $C$-curve with universal compactification $X \hookrightarrow X^c$. Then 
    \begin{enumerate}[label={\upshape{(\roman*)}}]
        \item\label{lemma:extra-points-1} $\abs{X^c} \smallsetminus \abs{X}$ is finite and discrete;
        \item\label{lemma:extra-points-2} each point $x\in \abs{X^c} \smallsetminus \abs{X}$ is a rank-$2$ point;
        \item\label{lemma:extra-points-3} if $X$ is affinoid, then $\abs{X^c} \smallsetminus \abs{X}$ is non-empty.
    \end{enumerate}
\end{lemma}
\begin{proof}
    \cref{lemma:extra-points-1} follows directly from \cite[Lem.~5.12]{Swan}. For \cref{lemma:extra-points-2}, we first note that \cref{lemma:extra-points-higher-rank} implies that $\rk \Gamma_x >1$. Now \cite[Cor.~1.8.8, Cor.~5.1.14]{Huber-etale} imply that $\rm{tr.c}(\wdh{k(x)}/C)\leq 1$. Therefore, \cite[Ch.~VI.10.3, Cor.~1]{Bourbaki} ensures that $\rm{dim}_{\mathbf{Q}}(\Gamma_x/\Gamma_C \otimes_{\Z} \mathbf{Q}) \leq 1$. Therefore, \cite[Ch.\~VI.10.2, Prop.~3]{Bourbaki} implies that $\rk \Gamma_x \leq \rk \Gamma_C +1 \leq 2$. This implies that $\rk \Gamma_x =2$.

    To see \cref{lemma:extra-points-3}, we note that \cite[Prop.~1.4.6]{Huber-etale} implies that $X$ is not proper. Then \cite[Cor.~5.1.6]{Huber-etale} ensures that $X\neq X^c$. 
\end{proof}

Given $x \in X$ with residue field $k(x)$, we denote the associated valuation by $v_x \colon k(x) \to \Gamma_x \cup \{0\}$. We slightly abuse the notation and also denote by $v_x\colon \wdh{k(x)}^\h \to \Gamma_x \cup \{0\}$ the induced valuation of the henselized completed residue field $\wdh{k(x)}^\h$ (see \cite[\href{https://stacks.math.columbia.edu/tag/0ASK}{Tag 0ASK}]{stacks-project} for the fact that henselization does not change the value group).  

Our next goal is to study the henselized completed residue field $\wdh{k(x)}^\h$ for $x\in \abs{X^c}\smallsetminus \abs{X}$. It turns out that all these affinoid fields are curve-like in the sense of \cref{defn:curve-like}.

\begin{lemma}[{\cite[Prop.~5.1]{Swan}}]\label{lemma:extra-points-curve-like}
    Let $X$ be a separated rigid-analytic $C$-curve and $x$ a rank-$2$ point on $X^c$. Then $\wdh{k(x)}^\h$ is a curve-like affinoid field and the secondary residue field $\wdh{k(x)}^{+, \h}/\m_x\wdh{k(x)}^{+, \h}  \simeq k(x)^+/\m_x$ is isomorphic to $k_C$. 
\end{lemma}
\begin{proof}
    By construction, $\wdh{k(x)}^\h$ is henselian. Therefore, it suffices to show that $\wdh{k(x)}^\h$ is defectless in every finite extension, $(\Gamma_x)_{<1}$ has a greatest element $\gamma_x$, and $\Gamma_x$ is generated by $\Gamma_C$ and $\gamma_x$.  

    Now we note that the point $x$ is of Type III in the sense of \cite[\S~5, p.~184]{Swan}. Therefore, \cite[Lem.~5.3~(i, ii, iii)]{Swan} implies that $\wdh{k(x)}^\h$ is defectless in every finite extension, while \cite[Lem.~5.1~(iii)]{Swan} implies that $(\Gamma_x)_{< 1}$ admits a greatest element $\gamma_x$ and that it is generated by $\Gamma_C$ and $\gamma_x$. Furthermore, \textit{loc.\ cit.} implies that $\wdh{k(x)}^{+, \h}/\m_x\wdh{k(x)}^{+, \h} \simeq k(x)^+/\m_x$ is isomorphic to $k_C$. 
\end{proof}

\begin{lemma}\label{lemma:generalization-weakly-Shilov} Let $X$ and $x\in X$ be as in \cref{lemma:extra-points-curve-like}, and let $x_{\rm{gen}}$ be the unique rank-$1$ generalization of $x$. Then $x_{\rm{gen}}$ is weakly Shilov in the sense of \cite[Def.~2.5]{BH}.
\end{lemma}
\begin{proof}
Since $x_{\rm{gen}}$ admits a proper specialization, it is of type II in the sense of \cite[\S~5, p.~184]{Swan}. Therefore, \cite[Prop.~5.1~(ii)]{Swan} implies that the secondary residue field of $x_{\rm{gen}}$ has transcendence degree $1$ over $C$. Therefore, \cite[Prop.~2.9]{BH} implies that $x_{\rm{gen}}$ is weakly Shilov. %
\end{proof}

\begin{definition}\label{defn:reduction}
    Let $X$ be a rigid-analytic $C$-curve and $x$ a rank-$2$ point on $X^c$ with the corresponding valuation $v_x\colon \wdh{k(x)}^\h \to \Gamma_x \cup \{0\}$. We define the \emph{reduction morphism} $\#\colon \Gamma_x \to \Z$ to be the morphism from \cref{defn:sharp-map}. 
\end{definition}

\subsection{Relation to formal models}

The main goal of this subsection is to give a geometric interpretation of the reduction morphism from \cref{defn:reduction} in terms of formal models, which follows essentially from the discussion in \cite[pp.~199--200]{Swan};
see also \cite{Kobak} for a discussion in the setting of more general quasi-compact, separated rigid spaces.
We expand the argument here for the convenience of the reader.
This interpretation will play the key role in showing the compatibility of the analytic and algebraic trace maps (see \cref{thm:analytic-trace-algebraic}).  

Before we discuss this interpretation, we need to recall the construction of the specialization morhpism:  

\begin{construction}[{\cite[Prop.~1.9.1]{Huber-etale}}]\label{construction:specialization}
Let $\cX$ be an admissible formal $\O_C$-scheme with generic fiber $X=\cX_\eta$. Then we recall that $X$ is equipped with a \emph{specialization morphism} $\spec_{\cX} \colon (X,\cO^+_X) \to (\cX,\cO_{\cX})$ that is universal among such maps from rigid spaces.
This construction is affine local on $\cX$;
when $\cX = \Spf(A_0)$, then $\cX_\eta = \Spa\bigl(A_0\bigl[\frac{1}{\varpi}\bigr], A_0\bigl[\frac{1}{\varpi}\bigr]^\circ\bigr)$ and $\spec_{\cX}$ sends a valuation $v \colon A_0\bigl[\frac{1}{\varpi}\bigr] \to \Gamma_v$  to the prime ideal $v^{-1}(\Gamma_{v, <1}) \cap A_0 \subset A_0$, which is open due to continuity of $v$. 
\end{construction}
When there is no risk for confusion, we also often write $\spec$ instead of $\spec_\cX$. 

We now discuss the behaviour of the specialization map at some specific class of points of $X$: 

\begin{lemma}
\label{lemma:local-ring-at-generic-point-in-the-special-fiber}
Let $\cX$ be an admissible formal $\O_C$-scheme with reduced special fiber, and let $\zeta \in \abs{\cX_s}$ be a generic point in the special fiber. Then 
\begin{enumerate}[label=\upshape{(\roman*)}]
    \item\label{lemma:local-ring-at-generic-point-in-the-special-fiber-0} $\sp^{-1}(\zeta)$ consists of a unique point $z$;
    \item\label{lemma:local-ring-at-generic-point-in-the-special-fiber-1} the local ring $\O_{\cX, \zeta}$ is a rank-$1$ valuation ring which is $\varpi$-adically separated and $\varpi$-adically henselian;
    \item\label{lemma:local-ring-at-generic-point-in-the-special-fiber-2} the natural morphism $\O_{\cX, \zeta} \to k(z)^+$ is an isomorphism;
    \item\label{lemma:local-ring-at-generic-point-in-the-special-fiber-3} the ideal $\m_C\O_{\cX, \zeta}$ is the maximal ideal of $\O_{\cX, \zeta}$;
    \item\label{lemma:local-ring-at-generic-point-in-the-special-fiber-4} the natural morphism of value groups $\Gamma_C \to \Gamma_z$ is an isomorphism.
\end{enumerate}
\end{lemma}
\begin{proof}
    We first show \cref{lemma:local-ring-at-generic-point-in-the-special-fiber-0}. For this, we recall that the underlying topological space of $\cX_\eta$ is given by $\abs{\cX_\eta} \simeq \lim_{\cX' \to \cX} \abs{\cX'}$, where the limit is taken over all admissible blow-ups $\cX' \to \cX$ (see \cite[\S~2.2]{ALY} or \cite[Th.~II.A.4.7]{FujKato}). Furthermore, the specialization morphism is given simply by the projection $\sp\colon \abs{\cX_\eta} \simeq \lim_{\cX' \to \cX} \abs{\cX'} \to \abs{\cX}$. Therefore, it suffices to show that, for any admissible blow-up $f\colon \cX' \to \cX$, there is a dense open subset $\mathscr{U} \subset \cX$ such that $f$ is an isomorphism over $\mathscr{U}$. This follows directly from \cite[Cor.~B.14]{Z-thesis}.  

    Now we note that \cite[Lem.~A.2~(a)]{ALY} implies that $\cX$ is $\eta$-normal in the sense of \cite[Def.~A.1]{ALY}. Therefore, \cref{lemma:local-ring-at-generic-point-in-the-special-fiber-1}-\cref{lemma:local-ring-at-generic-point-in-the-special-fiber-2} follow from the combination of \cite[Def.~A.11, Lem.~A.12, and Prop.~A.15]{ALY}. We note that \cref{lemma:local-ring-at-generic-point-in-the-special-fiber-3} follows from the observation that $\O_{\cX, \zeta}/\m_C\O_{\cX, \zeta} \simeq \O_{\cX_s, \zeta}$ is a field because $\zeta$ is a generic point of $\cX_s$.   

    Finally, we show \cref{lemma:local-ring-at-generic-point-in-the-special-fiber-4}. The question is Zariski-local on $\cX$, so we can assume that $\cX=\Spf A_0$ is affine, smooth, and connected (thus, irreducible). Put $A\coloneqq A_0\bigl[\frac{1}{\varpi}\bigr]$, then \cite[Prop.~3.4.1]{Lut16} and the assumption that $C$ is algebraically closed imply that $A^\circ=A_0$ and $A^{\circ\circ} = \m_C A^{\circ} = \m_C A_0$. Therefore, $A^\circ/A^{\circ \circ} = A_0/\m_C A_0$ is an integral domain, and so \cite[Prop.~6.2/5]{BGR} implies that the supremum semi-norm $\abs{.}_{\rm{sup}} \colon A \to \Gamma_C\cup \{0\}$ is a valuation\footnote{Here, we implicitly use that $C$ is algebraically closed, \cite[Obs.~3.6/10]{BGR}, and \cite[Prop.~3.1/16]{B} to ensure that the value group of the supremum semi-norm is equal to $\Gamma_C$.} of $A$. The supremum norm is bounded on $A^\circ$ due to \cite[Th.~3.1/17]{B} and is continuous due to \cite[L.~9, Cor.~9.3.3~(2)]{Seminar}, thus it defines a point $z'\in \Spa(A, A^\circ) = \cX_\eta$. Now \cite[Cor.~3.1/18]{B} implies that $\abs{.}^{-1}_{\rm{sup}}(\Gamma_{C, <1}) = A^{\circ \circ} = \m_C A^\circ$. Therefore, we conclude that $\rm{sp}(z')=\zeta$. So \cref{lemma:local-ring-at-generic-point-in-the-special-fiber-0} ensures that $z=z'$. Then $\Gamma_{z}=\Gamma_{z'}=\Gamma_C$ by the very construction. 
\end{proof}

\begin{lemma}
\label{lemma:zig-zag of extra points are never closed}
Let $\cX$ be a quasi-compact admissible separated formal $\O_C$-scheme, then its rigid generic fiber
$X$ is separated and taut over $\Spa(C, \mathcal{O}_C)$.
Let $x \in \abs{X^c} \smallsetminus \abs{X}$ with its unique rank-$1$ generalization $x_{\rm{gen}} \in X$.
Then $\sp(x_{\rm{gen}}) \in \cX_s$ is \emph{not} a closed point.
\end{lemma}

\begin{proof}
By \cite[Prop.~4.7]{BL1}, we know that $X$ is separated.
Since $\cX$ is quasi-compact, we conclude that $X$ is quasi-compact. Since it is also separated, \cite[Lem.~5.1.3~(ii)]{Huber-etale} implies that it is taut.
All rank-$1$ points on $\abs{X^c}$ already lie on $\abs{X}$ thanks to \cref{lemma:extra-points-higher-rank}. 

Let us show the last statement, suppose to the contrary that $\sp(x_{\rm{gen}})$ is a closed point. Let $\cU$ be an affine open neighborhood of $\sp(x_{\rm{gen}})$, let $U=\cU_\eta$ be its generic fiber, and let $j^c \colon U^c \to X^c$ be the morphism induced by the natural open immersion $j\colon U \to X$. Then \cite[Cor.~1.3.9]{Huber-etale} implies that $\ov{\{x_{\rm{gen}}\}} \subset U^c$, where the closure $\ov{\{x_{\rm{gen}}\}}$ is taken inside $X^c$. This implies that $x$ lies in $U^c$, so we can replace $\cX$ with $\cU$ to assume that $\cX = \Spf A_0$ is affine.

In this situation, we put $A \coloneqq A_0[1/\varpi]$. We see that $X = \Spa(A, A^\circ)$ and \cref{lemma:universal-compactification-affinoid} implies that $X^c = \Spa(A, \mathcal{O}_C[A^{\circ \circ}]^+)$. Then the point $x$ (resp.~$x_{\rm{gen}}$) defines a continuous morphism $r_x \colon \mathcal{O}_C[A^{\circ \circ}]^+ \to \wdh{k(x)}^+$ (resp.~$r_{\rm{gen}} \colon A^\circ \to \wdh{k(x_{\rm{gen}})}^+$). We note that \cite[Th.~1.1.10]{Huber-etale} implies that $\wdh{k(x_{\rm{gen}})}^+$ is a rank-$1$ valuation ring and that $\rm{Frac}\Bigl(\wdh{k(x_{\rm{gen}})}^+\Bigr) = \rm{Frac}\Bigl(\wdh{k(x)}^+\Bigr)$. Thus, the specialization relation $x_{\rm{gen}}\rightsquigarrow x$ can be realized as the commutative diagram:
\[
\begin{tikzcd}
\mathcal{O}_C[A^{\circ \circ}]^+ \arrow{r}{r_x} \arrow[d, hook] & \wdh{k(x)}^+ \arrow[d, hook] \\
A^\circ \arrow{r}{r_{\rm{gen}}} & \wdh{k(x_{\rm{gen}})}^+.
\end{tikzcd}
\]
Since $r_{\rm{gen}}$ is continuous, we conclude that $r_{\rm{gen}}(\varpi)$ is a pseudo-uniformizer in $\wdh{k(x_{\rm{gen}})}^+$. Since $\wdh{k(x_{\rm{gen}})}^+$ is a rank-$1$ valuation ring, we conclude that $\m_C \wdh{k(x_{\rm{gen}})}^+ = \m_{\rm{gen}}$ is the maximal ideal of $\wdh{k(x_{\rm{gen}})}^+$. Therefore, we see that our assumption on $\rm{sp}(x)$ implies that the map of $k_C$-algebras
\[
\restr{r_{\rm{gen}}}{A_0} \, \rm{mod}\, \m_C A_0 \colon A_0/(\mathfrak{m}_C \cdot A_0) \to \wdh{k(x_{\rm{gen}})}^+/\mathfrak{m}_{\rm{gen}} = \wdh{k(x_{\rm{gen}})}^+/(\m_C \wdh{k(x_{\rm{gen}})}^+)
\]
factors through the natural morphism $k_C \to \wdh{k(x_{\rm{gen}})}^+/\mathfrak{m}_{\rm{gen}}$. Since  $\mathfrak{m}_{\rm{\gen}}$ and $\O_C$ lie in $\wdh{k(x)}^+$ (see \cite[Th.~10.1]{Matsumura} for the former claim), we conclude that the image of $A_0 \to \wdh{k(x_{\rm{gen}})}^+$ lies inside $\wdh{k(x)}^+$. Finally, we use \cite[Prop.~6.3.4/1]{BGR} and the fact that $\wdh{k(x)}^+$ is integrally closed in $\rm{Frac}\bigl(\wdh{k(x)}^+\bigr) = \rm{Frac}\bigl(\wdh{k(x_{\rm{gen}})}^+\bigr)$ to conclude that the morphism $r_{\rm{gen}}\colon A^\circ \to \wdh{k(x_{\rm{gen}})}^+$ factors through $\wdh{k(x)}^+$. This contradicts the assumption that $x \in \abs{X^c} \smallsetminus \abs{X}$.
\end{proof}

Finally, we are almost ready to discussed the promised above relation between \cref{defn:reduction} and formal models. But before we do this, we need to recall the following two lemmas: 

\begin{lemma}\label{lemma:completion-of-valuation} Let $k^\circ$ be a rank-$1$ valuation ring with fraction field $k$, and let $\wdh{k}$ be the completion of $k$ (with respect to the valuation topology). Let $\cal{V}(k, k^\circ)$ (resp.~$\cal{V}(\wdh{k}, \wdh{k}^\circ)$) be the set of valuation rings $A$ on $k$ (resp.~valuation rings $B$ on $\wdh{k}$) such that $A\subset k^\circ$ (resp.~$B\subset \wdh{k}^\circ$). Then the map
\[
\cal{V}(k, k^\circ) \to \cal{V}(\wdh{k}, \wdh{k}^\circ)
\]
\[
(A\subset k^\circ) \mapsto (\wdh{A} \subset \wdh{k}^\circ),
\]
is a bijection with the inverse given by
\[
(B\subset \wdh{k}^\circ) \mapsto (B\cap k \subset k^\circ).
\]
\end{lemma}

\begin{proof}
In this proof, we denote by $\m$ the maximal ideal of $k^\circ$ and choose a pseudo-uniformizer $\pi \in \m$. Then we note that \cite[Lem.~1(iii)]{Buzzard-Verberkmoes} implies that $\wdh{k}^\circ = \wdh{k^\circ}$. Since $\wdh{k}^\circ/(\pi) \simeq k^\circ/(\pi)$ (see \cite[\href{https://stacks.math.columbia.edu/tag/05GG}{Tag 05GG}]{stacks-project}), we conclude that $\wdh{\m}$ is equal to the maximal ideal of $\wdh{k}^\circ$. Now \cite[Th.~10.1]{Matsumura} implies that $\cal{V}(k, k^\circ)$ and $\cal{V}(\wdh{k}, \wdh{k}^\circ)$ are both in bijection with the set of valuation rings $R\subset \kappa\coloneqq k^\circ/\m \simeq \wdh{k}^\circ/\wdh{\m}$. Furthermore, both bijections are realized by taking the pre-image of $R$ along the reduction morphism $k^\circ \to \kappa$ or $\wdh{k}^\circ \to \kappa$.  

Therefore, we are only left to show that the composite bijection $\cal{V}(k, k^\circ) \to \cal{V}(\wdh{k}, \wdh{k}^\circ)$ is given by taking the $\pi$-adic completion, and the other composite bijection $\cal{V}(\wdh{k}, \wdh{k}^\circ) \to \cal{V}(k, k^\circ)$ is given by taking the intersection with $k$. The first claim follows from the fact that, for any $A\in \cal{V}(k, k^\circ)$, we have\footnote{We note that \cite[Th.~10.1]{Matsumura} implies that $\m\subset A$.} $A/\m \simeq \wdh{A}/\wdh{\m}$. The second claim can be easily seen by unravelling the definitions. %
\end{proof}

\begin{lemma}[{\cite[Lem.~1.3.6~i)]{Huber-etale}}]\label{lemma:valuation-generalization} 
    Let $X$ be an analytic adic space, $y\in X$ a point, $x\in X$ a generalization of $y$ inducing the morphism $\iota\colon k(y) \to k(x)$ of residue fields. Then there is a unique valuation ring $A_{y\rightarrow x}\subset k(x)^+$ such that $\iota^{-1}(A_{y\rightarrow x}) = k(y)^+$. Furthermore, the natural morphism
    \[
    \wdh{k(y)}^+\to \wdh{A}_{y\rightarrow x}
    \]
    is an isomorphism.
\end{lemma}
\begin{proof}
    This follows directly from \cref{lemma:completion-of-valuation} and the observation that $\wdh{k(x)} \simeq \wdh{k(y)}$ \cite[Lem.~1.1.10~iii)]{Huber-etale}.
\end{proof}
Now we get to the promised geometric interpretation of the reduction morphism from \cref{defn:reduction} in terms of formal models.

\begin{lemma}[{\cite{Swan}}]
\label{specialization}
Let $\cX$ be a quasi-compact admissible formal $\O_C$-scheme such that the special fiber $\cX_s$ is a reduced separated scheme of pure dimension $1$. Let $\cX^c_s$ be a schematically dense compactification of $\cX_s$, and let $\nu \colon \cX^{c,n}_s \to \cX^c_s$ be its normalization.
Let $X \colonequals \cX_\eta$ be the rigid generic fiber of $\cX$
with its universal compactification $X^c$.
Then:
\begin{enumerate}[label=\upshape{(\roman*)}]
    \item\label{specialization-points} For any generic point $\zeta \in \cX$ with $z = \spec^{-1}(\zeta)$, there is a bijection
    \[ \mu_\zeta \colon \overline{\{z\}} \xrightarrow{\sim} \abs{Y_\zeta} \]
    between points of the closure $\overline{\{z\}}$ of $z$ in $X^c$, and points of the corresponding connected component $Y_\zeta \subseteq \cX^{c,n}_s$;
    when $y \in \overline{\{z\}} \cap X$, then $\nu(\mu_\zeta(y)) = \spec(y)$.
    \item\label{vanishing-order-rings} Let $y \in \overline{\{z\}}$ be a specialization of $z$.
    Then the ring $A_{y\to z}$ from \cref{lemma:valuation-generalization} is the preimage of $\cO_{\cX^{c,n}_s,\mu_\zeta(y)}$ under
    \[ k(z)^+ \xleftarrow{\sim} \cO_{\cX,\zeta} \to \cO_{\cX_s,\zeta} \simeq \cO_{\cX^{c,n}_s,\zeta} \simeq \Frac\bigl(\cO_{\cX^{c,n}_s, \mu_\zeta(y)}\bigr) \]
    where the first isomorphism comes from \cref{lemma:local-ring-at-generic-point-in-the-special-fiber}\cref{lemma:local-ring-at-generic-point-in-the-special-fiber-2}.
    In particular, $k(y)^+=A_{y\to z} \cap k(y)$;
    \item\label{vanishing-order-valuation} Under the identification of \cref{vanishing-order-rings}, the map $\# \circ v_y$ is the composition
    \[ 
        k(y)^+ \to \cO_{\cX^{c,n}_s,\mu_\zeta(y)} \xrightarrow{\ord_{\mu_\zeta(y)}} \ZZ, 
    \]
    where $\ord_{\mu_\zeta(y)}$ is given by order of vanishing at $\mu_\zeta(y) \in \cX^{c,n}_s$;
    \item\label{specialization-comp} The bijections from \cref{specialization-points} induce a bijection
    \[ 
        \mu\colon \abs{X^c} \smallsetminus \abs{X} \xrightarrow{\sim} \abs{\cX^{c, n}_s} \smallsetminus \abs{\cX^n_s}. 
    \]
\end{enumerate}
\end{lemma}

\begin{proof}
We note that \cref{lemma:zig-zag of extra points are never closed} implies that $X$ is separated and taut, so the universal compactification $X^c$ exists due to \cref{compactifications exist}.  
 
\Cref{specialization-points}.
Fix a generic point $\zeta \in \cX_s$.
Let $Y_\zeta \subseteq \cX^{c,n}_s$ be the corresponding connected component of $\cX^{c,n}_s$ and let $z \colonequals \spec^{-1}_\cX(\zeta) \in X$ be the corresponding rank-$1$ point from \cref{lemma:local-ring-at-generic-point-in-the-special-fiber}\cref{lemma:local-ring-at-generic-point-in-the-special-fiber-0} with closure $\overline{\{z\}}$ in $X^c$.
By \cref{lemma:local-ring-at-generic-point-in-the-special-fiber}\cref{lemma:local-ring-at-generic-point-in-the-special-fiber-2}, the natural map $\cO_{\cX,\zeta} \to k(z)^+$ is an isomorphism.  

By sending $y \in \overline{\{z\}}$ to $A_{y\rightarrow z} \subset k(z)^+$ (see \cref{lemma:valuation-generalization}), the valuative criterion for properness \cite[Lem.~1.3.6, Cor.~1.3.9]{Huber-etale} gives a correspondence between the points of $\overline{\{z\}}$ and valuations rings $V \subseteq k(z)^+$ such that $V \cap C = \cO_C$. Now we note that \cite[Th.~10.1]{Matsumura} and \cref{lemma:local-ring-at-generic-point-in-the-special-fiber}\cref{lemma:local-ring-at-generic-point-in-the-special-fiber-2},\cref{lemma:local-ring-at-generic-point-in-the-special-fiber-3} imply that such valuation rings are in bijection\footnote{Explicitly, this bijection sends a valuation ring $V\subset k(z)^+$ to $\widetilde{V} \coloneqq V/\m_Ck(z)^+ \subset k(z)^+/\m_Ck(z)^+ \simeq k(Y_\eta)$. Its inverse is given by the map sending $\widetilde{V} \subset k(z)^+/\m_Ck(z)^+$ to $\pi^{-1}(\widetilde{V}) \subset k(z)^+$, where $\pi\colon k(z)^+ \to k(z)^+/\m_Ck(z)^+$ is the natural projection.} with valuation rings $\widetilde{V}$ on $k(z)^+/\m_Ck(z)^+ \simeq k(Y_\zeta)$ that contain $k_C = \O_C/\m_C$.
Now we apply \cite[Ch.~VI, \S~10.3, Cor.~2 and 3]{Bourbaki} to $K=k_C$ with the trivial valuation and $K'=k(Y_\zeta)$ to conclude that any such $\widetilde{V}$ is either trivial or a discrete valuation.
By the valuative criterion for properness for the smooth proper curve $Y_\zeta$, the natural map $\Spec k(Y_\zeta) \to Y_\zeta$ extends uniquely to $j_{\widetilde{V}} \colon \Spec \widetilde{V} \to Y_\zeta$;
the image of the closed point of $\Spec \widetilde{V}$ under $j_{\widetilde{V}}$ is a closed point $u_{\widetilde{V}} \in Y_\zeta(k_C)$.
Since the resulting map
\begin{equation}\label{curve-points}
\{ \text{discrete valuation rings } k_C \subset \widetilde{V} \subset k(Y_\zeta) \} \longrightarrow \{ u \in Y_\zeta(k_C) \text{ closed} \}, \quad \widetilde{V} \mapsto u_{\widetilde{V}}
\end{equation}
is a bijection (the inverse sends a closed point $u \in Y_\zeta(k_C)$ to $\cO_{Y_\zeta,u}$), the result follows directly.  

\Cref{vanishing-order-rings} and \cref{vanishing-order-valuation} follow from chasing the construction in the proof of \cref{specialization-points}.  

\Cref{specialization-comp}.
We fix a generic point $\zeta\in \cX_s$ and the corresponding rank-$1$ point $z=\spec^{-1}(\zeta)\in X$. We start the proof by showing the following claim: 

\begin{claim} A point $y \in \ov{\{z\}} \subset X^c$ lies in $X$ if and only if $\mu_\zeta(y) \in \cX^n_s \subseteq \cX^{c, n}_s$.
\end{claim}

\begin{proof}
If $y\in \overline{\{z\}} \cap X$, then \cref{specialization-points} implies that $\nu\bigl(\mu_\zeta(y)\bigl)=\spec(y)\in \cX_s$. Therefore, $\mu_\zeta(y)\in \cX_s^n$.  

Now we pick a point $y\in \overline{\{z\}}\subset X^c$ such that $\mu_\zeta(y)\in \cX_s$. We wish to show that $y$ lies in $X$.  

We first treat the case when $\cX=\Spf A$ is an affine admissible formal $\O_C$-scheme. In this situation, $X = \Spa\bigl(A\bigl[\tfrac{1}{\varpi}\bigr],A^+\bigr)$, where $A^+$ is the integral closure of $A$ in $A\bigl[\tfrac{1}{\varpi}\bigr]$. Then $y \in X$ if and only if $v_y(A^+) \le 1$, or equivalently $v_y(A) \le 1$. Now \cref{lemma:extra-points-curve-like} and \cref{lemma:structure-curve-like-valuations} imply that the valuation $v_y \colon A\bigl[\tfrac{1}{\varpi}\bigr] \to \Gamma_y$ has value group $\Gamma_y \simeq \Gamma_C \times \ZZ$. Now \cref{lemma:local-ring-at-generic-point-in-the-special-fiber}\cref{lemma:local-ring-at-generic-point-in-the-special-fiber-4} (applied to $z$) and the fact that $z\in X$ imply that, for every $a\in A$, the first coordinate $v_y(a)$ is less or equal to $1$. Thus, we only need to show that $\#\circ v_y(a) \leq 0$ for every $a\in A$. By \cref{vanishing-order-valuation}, this is equivalent to showing that 
\begin{equation}\label{eqn:lie-on-compactification}
\ord_{\mu_\zeta(y)}(A/\m_C A) \geq 0.
\end{equation}
Under the bijection \cref{curve-points}, \cref{eqn:lie-on-compactification} is equivalent to the condition that the image of $A/\fm_C A$ in $k_C(Y_\zeta)$ is contained in the valuation ring $\cO_{Y_\zeta,\mu_\zeta(y)}$, i.e., $\mu_\zeta(y) \in \abs{\cX^n_s}$.  

Now we explain how to reduce the case of a general $\cX$ to the case of an affine $\cX$. For this, we choose some open affine formal subscheme $\mathscr{U}\coloneqq \Spf A\subset \cX$ that contains the point $\nu\bigl(\mu_\zeta(y)\bigl)$. By construction, $\mathscr{U}$ also contains the point $\zeta \in \abs{\cX_s}=\abs{\cX}$, so the generic fiber $U\coloneqq \mathscr{U}_\eta=\spec^{-1}(\mathscr{U}_s)$ contains the point $z\in X$. To clarify the notation later on, we denote the point $z$ considered as a point of $U$ by $z_U$. Arguing as in the proof of \cref{lemma:zig-zag of extra points are never closed}, we conclude that $y\in \abs{U^c} \smallsetminus \abs{U}$. The construction of $\mu_\zeta$ in \cref{specialization-points} is compatible with the open immersion $\mathscr{U} \to \cX$, so it suffices to prove the claim for $\mathscr{U}$ that was treated above.
\end{proof}

As a consequence, $\mu_\zeta$ restricts to a natural bijection
\[ 
    \mu_\zeta \colon \overline{\{z\}} \cap (\abs{X^c} \smallsetminus \abs{X}) \xrightarrow{\sim} \abs{Y_\zeta} \cap (\abs{\cX^{c,n}_s} \smallsetminus \abs{\cX^n_s}). 
\]
By the last sentence of \cref{lemma:zig-zag of extra points are never closed},
we see that the disjoint union of the left hand side (where $z$ runs through all
preimages of generic points of $\cX_s$ under specialization map)
is exactly $\abs{X^c} \smallsetminus \abs{X}$.
Since $\abs{X^c} \smallsetminus \abs{X}$ is finite and discrete,
we can combine the various $\mu_\zeta$ for all generic points of $\cX_s$ to a bijection $\mu \colon \abs{X^c} \smallsetminus \abs{X} \xrightarrow{\sim} \abs{\cX^{c,n}_s} \smallsetminus \abs{\cX^n_s}$.
\end{proof}

\begin{remark}\label{rmk:extra-points-on-different-compactifications} We note that \cref{specialization}\cref{specialization-comp} implies that a smooth point $x\in \cX^c_s \smallsetminus \cX_s$ defines a unique rank-$2$ point $u_x \in X^c \smallsetminus X$ such that $\sp(u_x)=x$. Likewise, a nodal point $x\in \cX^c_s \smallsetminus \cX_s$ defines two rank-$2$ points $v_x, w_x\in \abs{X^c} \smallsetminus \abs{X}$ such that $\sp(v_x)=\sp(w_x)=x$.
\end{remark}

\subsection{Relation to formal models: nodes}

Given a quasi-compact rigid curve $X$ with a (quasi-compact) admissible formal model $\cX$, \cref{specialization} describes additional rank-2 points in $\abs{X^c} \smallsetminus \abs{X}$ in terms of ``points at infinity'' of the normalized special fiber $\cX^n_s$.
In this subsection, we explain the role played by the normalization, at least when $\cX = \Spf \cO_C \langle S,T \rangle / (ST - \pi)$ is a model rig-smooth semistable curve in the sense of \cref{defn:ss-formal}.

We recall that $\varpi \in \O_C$ is a fixed pseudo-uniformizer in $\O_C$. For the rest of this subsection, we choose another pseudo-uniformizer $\pi\in \mathfrak{m}_C \smallsetminus \{0\}$. 

\begin{notation}\label{formal-model-node-notation}
We set $R \coloneqq \Bigl(\bigl(\frac{\cO_C \langle S,T \rangle}{(ST - \pi)}\bigr)^\h_{(S, T)}\Bigr)^{\wedge}_{\varpi}$ to
be the $\varpi$-completion of the $(S, T)$-adic henselization of the standard rig-smooth semi-stable nodal curve.
In particular $(S,T,\fm_C \cdot R)$ is a maximal ideal in $R$.
We also set $\widetilde{R} \coloneqq R^{\wedge}_{(S,T,\varpi)}$ to be the $(S, T,\varpi)$-adic completion of $R$. 
\end{notation}

\begin{lemma}\label{lemma:equality-of-completions}
    The natural map $\frac{\cO_C\langle S,T \rangle}{(ST - \pi)} \to R$ induces an isomorphism $\frac{\cO_C \llbracket S,T \rrbracket}{(ST - \pi)} \xrightarrow{\sim} \widetilde{R}$.
\end{lemma}
\begin{proof}
We put $\widetilde{R}_n\coloneqq \frac{\bigl(\frac{\cO_C}{(\varpi^n)}\bigr)[S,T]}{(ST - \pi, S^n, T^n)} \simeq \frac{\O_C[S, T]}{(ST - \pi, S^n, T^n, \varpi^n)}$. Then $\frac{R}{(S^n, T^n, \varpi^n)} \simeq (\widetilde{R}_n\bigr)^\h_{(S, T)}$ due to \cite[\href{https://stacks.math.columbia.edu/tag/0DYE}{Tag 0DYE}]{stacks-project}. The ring $\widetilde{R}_n$ is already $(S, T)$-adically henselian due to \cite[\href{https://stacks.math.columbia.edu/tag/0F0L}{Tag 0F0L}]{stacks-project}, so we conclude that $\widetilde{R}_n \simeq \frac{R}{(S^n, T^n, \varpi^n)}$. Therefore, we conclude that 
\[
\frac{\cO_C \llbracket S,T \rrbracket}{(ST - \pi)} \simeq \Bigl(\frac{\O_C[S, T]}{(ST-\pi)} \Bigr)^{\wedge}_{(S, T, \varpi)} \simeq \lim_n \widetilde{R}_n \simeq \lim_n \frac{R}{(S^n, T^n, \varpi^n)} \simeq \widetilde{R}. \qedhere
\]
\end{proof}

\cref{lemma:equality-of-completions} implies that any element $f \in \widetilde{R}\bigl[\tfrac{1}{\varpi}\bigr]$ can be written uniquely 
as 
\[
f = a_0 + \sum_{i \geq 1} b_i S^i + \sum_{j \geq 1} c_j T^j
\]
where $a_0, b_i, c_j$ are elements in $C$ such that $\{a_0, b_i, c_i;\; i \in \ZZ_{\ge 1}\} \subset C$ is a bounded subset.
The two Gauss norms on the annulus $\bigl(\Spf(\frac{\cO_C \langle S,T \rangle}{(ST - \pi)})\bigr)_\eta$ extend to norms on $\widetilde{R}\bigl[\tfrac{1}{\varpi}\bigr]$ in the following fashion:
\begin{definition}
\label{regular-elements}
    Let $f\in \widetilde{R}\bigl[\tfrac{1}{\varpi}\bigr]$.
    \begin{enumerate}[label=\upshape{(\roman*)}]
        \item The \emph{$S$-Gauss norm} and \it{$T$-Gauss norm} of $f$ are given by the following valuations:
        \[ \abs{f}_S \colonequals \sup\{\abs{a_0}, \abs{b_i}, \abs{c_j \cdot \pi^j} \mid i \geq 1, j \geq 1\} \quad \text{and} \quad \abs{f}_T \colonequals \sup\{\abs{a_0}, \abs{b_i \cdot \pi^i}, \abs{c_j} \mid i \geq 1, j \geq 1\} \]
        \item\label{regular-elements-forreal} We say that $f$ is \textit{$S$-regular} (resp.\ \textit{$T$-regular}) if the supremum $\abs{f}_S$ (resp.\ $\abs{f}_T$) is attained by an element of the set.
        We say $f$ is \emph{regular} if it is both $S$-regular and $T$-regular.
        \item We denote the sets of $S$-regular (resp.\ $T$-regular, resp.\ regular) elements of $\widetilde{R}\bigl[\tfrac{1}{\varpi}\bigr]$ by $\widetilde{R}\bigl[\tfrac{1}{\varpi}\bigr]_{S\mhyphen\mathrm{reg}}$ (resp.\ $\widetilde{R}\bigl[\tfrac{1}{\varpi}\bigr]_{T\mhyphen\mathrm{reg}}$, resp.\ $\widetilde{R}\bigl[\tfrac{1}{\varpi}\bigr]_\mathrm{reg}$).
    \end{enumerate}
\end{definition}

To relate these norms to the classical Gauss norm, we need to introduce some further notation: 

\begin{notation}\label{notation:annulus-to-disk} We set $\widetilde{R}_S \coloneqq (\O_C[\![S]\!][\frac{1}{S}])^{\wedge}_{\varpi}$ and $\widetilde{R}_T \coloneqq (\O_C[\![T]\!][\frac{1}{T}])^{\wedge}_{\varpi}$. Then we see that $\widetilde{R}$ admits two ring-homomorphisms
\[
\beta_S \colon \w{R} \to \w{R}_S \quad \text{and} \quad \beta_T \colon \w{R} \to \w{R}_T
\]
defined by the rule 
\begin{gather*}
    \beta_S\Bigl( a_0 + \sum_{i \geq 1} b_i S^i + \sum_{j \geq 1} c_j T^j\Bigr) = a_0 + \sum_{i=1}^\infty b_i S^i + \sum_{i=1}^\infty c_{i}\cdot \pi^{i} S^{-i}, \\
    \beta_T\Bigl( a_0 + \sum_{i \geq 1} b_i S^i + \sum_{j \geq 1} c_j T^j\Bigr) = a_0 + \sum_{i=1}^\infty b_i\cdot \pi^i T^{-i} + \sum_{i=1}^\infty c_{i} T^{i}.
\end{gather*}
By abuse of notation, we denote by $\beta_S \colon \w{R}\bigl[\tfrac{1}{\varpi}\bigr] \to \w{R}_S\bigl[\tfrac{1}{\varpi}\bigr]$ and by $\beta_T \colon \w{R}\bigl[\tfrac{1}{\varpi}\bigr] \to \w{R}_T\bigl[\tfrac{1}{\varpi}\bigr]$ the natural morphisms induced by $\beta_S$ and $\beta_T$ from above. 
\end{notation}

\begin{remark}\label{regular-elements-revisited} Let $\abs{\blank}\colon \w{R}_S \to \Gamma_C \cup \{0\}$ be the classical Gauss norm $\abs{ \sum_{i \in \ZZ} a_i S^i} = \sup(\abs{a_i})$. Then we have the following equality 
\[
\abs{\blank}_S = \abs{\blank} \circ \beta_S \colon \w{R} \to \Gamma_C \cup \{0\}.
\]
We say that an element $f = \sum_{i\in \ZZ} a_i S^i \in \w{R}_S\bigl[\tfrac{1}{\varpi}\bigr]$ is \emph{regular} if $\abs{f} = \abs{a_i}$ for some $i\in \ZZ$. 
\end{remark}

To justify the name of the $S$- and $T$-Gauss norms, we make the following observations. First, both $\abs{\blank}_S$ and $\abs{\blank}_T$ are injective, submultiplicative, and satisfy the nonarchimedean triangle inequality. Therefore, they do define norms on $\w{R}$ (in the sense of \cite[Def.~1.2.1/1]{BGR}). One can also check that $S$- and $T$-Gauss norms are multiplicative by reducing the question to the Gauss norm on $\w{R}_S$ and then approximating this norm with the Gauss norms on the closed disks  of radius $r<1$. In this paper, we never use this multiplicativity, so we leave the details to the interested reader. 

Now we are ready to formulate one of the key results of this subsection: 

\begin{lemma}
\label{lemma:image under completion is regular}
    The image of the natural map $R\bigl[\tfrac{1}{\varpi}\bigr] \to \widetilde{R}\bigl[\tfrac{1}{\varpi}\bigr]$ is contained
    in $\widetilde{R}\bigl[\tfrac{1}{\varpi}\bigr]_{\mathrm{reg}}$.
\end{lemma}
In order to present the proof, we will first need to study the map $R \to \widetilde{R}$ in more detail.

\begin{lemma}
\label{injectivity of completing henselian rings}
Let $B$ be a ring. 
Then the natural maps $\bigl(B[S]\bigr)^\h_{S} \to B[\![S]\!]$ and $\bigl(\frac{B[S,T]}{(ST)}\bigr)^h_{(S, T)} \to \frac{B[\![S,T]\!]}{(ST)}$ are injective.
\end{lemma}

\begin{proof}
We show it for the map $\bigl(\frac{B[S,T]}{(ST)}\bigr)^\h_{(S, T)} \to \frac{B[\![S,T]\!]}{(ST)}$, a similar proof applies to $\bigl(B[S]\bigr)^\h_{S} \to B[\![S]\!]$. We start by writing $B \simeq \colim_{i\in I} B_i$ as a filtered colimit of its finitely generated $\ZZ$-subalgebras. Then the natural morphism $\colim_I \bigl(\frac{B_i \llbracket S, T \rrbracket}{(ST)}\bigr) \to \frac{B\llbracket S, T\rrbracket}{(ST)}$ is injective and the natural morphism $\colim_I \bigl(\frac{B_i [S, T]}{(ST)}\bigr)^\h_{(S, T)} \to \bigl(\frac{B[S, T]}{(ST)}\bigr)^\h_{(S, T)}$ is an isomorphism (see \cite[\href{https://stacks.math.columbia.edu/tag/0A04}{Tag 0A04}]{stacks-project}). Therefore, it suffices to show the lemma under the additional assumption that $B$ is a finitely generated $\ZZ$-algebra. In this case, the result follows directly from \cite[\href{https://stacks.math.columbia.edu/tag/0AGV}{Tag 0AGV}]{stacks-project}.
\end{proof}

Already we are getting some interesting statements concerning the map
$R \to \widetilde{R}$.

\begin{lemma}
\label{Telling integrality from its completion}
The maps $R/\pi \to \widetilde{R}/\pi$, $R \to \widetilde{R}$, and $R\bigl[\tfrac{1}{\varpi}\bigr] \to \widetilde{R}\bigl[\tfrac{1}{\varpi}\bigr]$ are injective.
Moreover, we have a pullback diagram of rings
\[ \begin{tikzcd}
R \arrow[d] \arrow[r] & \widetilde{R} \arrow[d] \\
R\bigl[\tfrac{1}{\varpi}\bigr] \arrow[r] & \widetilde{R}\bigl[\tfrac{1}{\varpi}\bigr].
\end{tikzcd} \]
\end{lemma}

\begin{proof}
By \cite[\href{https://stacks.math.columbia.edu/tag/0DYE}{Tag 0DYE}]{stacks-project}
and \cref{lemma:equality-of-completions}, the map $R/\pi \to \widetilde{R}/\pi$ is identified with the map 
\[
\Bigl(\frac{\frac{\O_C}{\pi}[S, T]}{(S, T)}\Bigr)^\h_{(S, T)} \to \Bigl(\frac{\frac{\O_C}{\pi}[S, T]}{(S, T)}\Bigr)^{\wedge}_{(S, T)}.
\]
Thus, \cref{injectivity of completing henselian rings} implies that $R/\pi \to \widetilde{R}/\pi$ is injective. Since $R$ is $\pi$-adically separated (it is even $\pi$-adically complete) and both $R$ and $\widetilde{R}$ are $\pi$-torsionfree, we conclude that $R\to \widetilde{R}$ is injective as well. This formally implies that $R\bigl[\tfrac{1}{\varpi}\bigr] \to \widetilde{R}\bigl[\tfrac{1}{\varpi}\bigr]$ is injective. For the last statement, the vertical and horizontal maps are injective, so an easy diagram chase implies that it suffices to show that the natural map 
\[
\frac{R\bigl[\tfrac{1}{\varpi}\bigr]}{R} \simeq \frac{R\bigl[\tfrac{1}{\pi}\bigr]}{R} \to \frac{\widetilde{R}\bigl[\tfrac{1}{\pi}\bigr]}{\widetilde{R}} \simeq \frac{\widetilde{R}\bigl[\tfrac{1}{\varpi}\bigr]}{\widetilde{R}}
\]
is injective. Since both $R$ and $\widetilde{R}$ are $\pi$-torsionfree, we can identify this map with the map $\colim_n R/\pi^n \to \colim_n \widetilde{R}/\pi^n$. This follows formally from injectivity of $R/\pi \to \widetilde{R}/\pi$, injectivity of $R \to \widetilde{R}$, and the observations that $R$ is $\pi$-adically separated and both $R$ and $\widetilde{R}$ are $\pi$-torsionfree. %
\end{proof}
With these technical preliminaries out of the way, we can prove \cref{lemma:image under completion is regular}.
\begin{proof}[Proof of \cref{lemma:image under completion is regular}]
We fix an element $f\in R\bigl[\tfrac{1}{\varpi}\bigr]$ and show that its image $\widetilde{f}\in \widetilde{R}\bigl[\tfrac{1}{\varpi}\bigr]$ is $S$-regular 
($T$-regularity follows as the role of $S$ and $T$ is symmetric). By inspection, we see that scaling by $C^{\times}$ or powers of $S$ does not change $S$-regularity. So we may and do assume that $f\in R$. Then we choose a non-negative integer $n$ such that 
\[
\abs{\pi^{n+1}} < \abs{\widetilde{f}}_S\leq \abs{\pi^n}.
\]
The definition of the ``$S$-Gauss norm'' and the assumption that $f\in R$ imply that the image $\frac{S^n\cdot \widetilde{f}}{\pi^n} \in \widetilde{R}$. Therefore, the last statement in \cref{Telling integrality from its completion} ensures that $\frac{S^n\cdot f}{\pi^n} \in R$. In particular, we can further replace $f$ with $\frac{S^n\cdot f}{\pi^n}$ to assume that $f\in R$ and $\abs{\pi}<\abs{\widetilde{f}}_S \leq \abs{1}$.

We put $\fn \coloneqq \{c \in \cO_C \mid \abs{c} < \abs{\widetilde{f}}_S\}$. Our assumptions on $f$  imply that $(\pi) \subset \fn$. Suppose that $\widetilde{f}$ is not $S$-regular, then $f$ lies in the kernel of the natural map $R \to \frac{R}{(\fn, T)} \simeq \bigl(\frac{\O_C}{\fn}[S]\bigr)^\h_{(S)} \hookrightarrow \frac{\O_C}{\fn}\llbracket S \rrbracket$. The last map is injective due to \cref{injectivity of completing henselian rings}, so $f$ lies in the kernel of the natural morphism $R \to \frac{R}{(\fn, T)}$. This means that $f = \pi' \cdot f_1 + T\cdot f_2$ for some $\pi'\in \fn$ and $f_1, f_2\in R$. This leads to the contradiction $\abs{\widetilde{f}}_S \leq \max\{\abs{\pi'}, \abs{\pi}\} < \abs{\widetilde{f}}_S$ and finishes the proof.
\end{proof}

\begin{definition}\label{valuation-regular-element} \begin{enumerate}[leftmargin=*,label=\upshape{(\roman*)}]
\item\label{valuation-regular-element-1} (Secondary Gauss valuation) Let $f\in \bigl(\w{R}_S\bigl[\tfrac{1}{\varpi}\bigr]\bigr)_{\mathrm{reg}} \subset \bigl(\w{R}_S\bigl[\tfrac{1}{\varpi}\bigr]\bigr) = \bigl(\O_C \llbracket S \rrbracket [S^{-1}]\bigr)^{\wedge}_{\varpi}\bigl[\tfrac{1}{\varpi}\bigr]$. We define the map $v \colon \bigl(\w{R}_S\bigl[\tfrac{1}{\varpi}\bigr]\bigr)_{\mathrm{reg}} \to \ZZ$ by the rule 
\[
v\Bigl( \sum_{i\in \ZZ} a_i S^i \Bigr) \coloneqq \min\{ r \in \ZZ \suchthat \abs{a_r} = \abs{f} \}.
\]
\item\label{valuation-regular-element-2} (Secondary $S$-Gauss valuation) Let $f\in \w{R}\bigl[\tfrac{1}{\varpi}\bigr]_{S\mhyphen\mathrm{reg}}$.
We define the map $v_S\colon \w{R}\bigl[\tfrac{1}{\varpi}\bigr]_{S\mhyphen\mathrm{reg}} \to \ZZ$ by the rule
\[
v_S(f) = v\bigl(\beta_S(f)\bigr),
\]
where $\beta_S \colon \w{R}\bigl[\tfrac{1}{\varpi}\bigr] \to \w{R}_S\bigl[\tfrac{1}{\varpi}\bigr]$ is the morphism from \cref{notation:annulus-to-disk}.  We define the map $v_T \colon \w{R}\bigl[\tfrac{1}{\varpi}\bigr]_{T\mhyphen\mathrm{reg}} \to \ZZ$ in a similar way. 
\end{enumerate}
\end{definition}

\begin{lemma}
\label{multiplicativity of the v_S and v_T}
    The set $\widetilde{R}\bigl[\tfrac{1}{\varpi}\bigr]_{S\mhyphen\mathrm{reg}}$ is stable under multiplication. Furthermore, we have $\abs{f\cdot g}_S = \abs{f}_S \cdot \abs{g}_S$ and $v_S(f\cdot g) = v_S(f)+ v_S(g)$ for any $f, g\in \widetilde{R}\bigl[\tfrac{1}{\varpi}\bigr]_{S\mhyphen\mathrm{reg}}$. The same results hold for $\widetilde{R}\bigl[\tfrac{1}{\varpi}\bigr]_{T\mhyphen\mathrm{reg}}$ and $v_T$. 
\end{lemma}
\begin{proof}
    By construction, it suffices to show that $\bigl(\w{R}_S\bigl[\tfrac{1}{\varpi}\bigr]\bigr)_{\rm{reg}} \subset \bigl(\w{R}_S\bigl[\tfrac{1}{\varpi}\bigr]\bigr)$ is closed under multiplication, and we have $\abs{f\cdot g} = \abs{f}\cdot \abs{g}$ and $v(f\cdot g) = v(f) + v(g)$ for any $f, g\in \bigl(\w{R}_S\bigl[\tfrac{1}{\varpi}\bigr]\bigr)_{\rm{reg}}$. This is a standard exercise on the usual Gauss norms; we leave details to the interested reader. 
\end{proof}

We come to the main statement of this subsection, which provides a concrete description of the valuations which correspond to nodes in the special fiber under the bijection $\mu_\zeta$ from \cref{specialization}\cref{specialization-points}. For the rest of the subsection, we fix an affine admissible formal $\O_C$-scheme $\cX = \Spf A$ with a nodal point $q\in \cX_s$ and a
$\varpi$-completely \'etale morphism $g\colon \bigl(\Spf A, q\bigr) \to \biggl( \Spf \Bigl(\frac{\O_C\langle S, R\rangle}{(ST - \pi)} \Bigr), \{0, 0\} \biggr)$ for $\pi\in \m_C \smallsetminus \{0\}$. 

Then the natural morphism 
\[
R = \biggl(\Bigl(\frac{\cO_C \langle S,T \rangle}{(ST - \pi)}\Bigr)^\h_{(S, T)}\biggr)^{\wedge}_{\varpi} \xr{g^*} \bigl(\O_{\cX, q}^\h\bigr)^{\wedge}_{(\varpi)}
\]
is an isomorphism, where $\O_{\cX, q}^\h$ denotes the localization of the local ring with respect to the maximal ideal. By slight abuse of notation, we define the map $\abs{\blank}_1 \colon A\bigl[\frac{1}{\varpi}\bigr] \to \Gamma_C \cup \{0\}$ as the composition
\[
    A\bigl[\tfrac{1}{\varpi}\bigr] \to \bigl(\O_{\cX, q}^\h\bigr)^{\wedge}_{(\varpi)}\bigl[\tfrac{1}{\varpi}\bigr] \xr[\sim]{(g^{*})^{-1}} R\bigl[\tfrac{1}{\varpi}\bigr] \to \w{R}\bigl[\tfrac{1}{\varpi}\bigr] \xr{\abs{\blank}_S} \Gamma_C\cup \{0\},
\]
similarly we define $\abs{\blank}_2 \colon A \to \Gamma_C \cup \{0\}$ using $\abs{\blank}_T$ in place of $\abs{\blank}_S$. Likewise, we define the map $v_1 \colon A\bigl[\frac{1}{\varpi}\bigr] \to \ZZ$ as the composition
\[
    A\bigl[\tfrac{1}{\varpi}\bigr] \to \bigl(\O_{\cX, q}^\h\bigr)^{\wedge}_{(\varpi)}\bigl[\tfrac{1}{\varpi}\bigr] \xr[\sim]{(g^{*})^{-1}} R\bigl[\tfrac{1}{\varpi}\bigr] \to \w{R}\bigl[\tfrac{1}{\varpi}\bigr] \xr{v_S} \ZZ,
\]
and similarly we define $v_2 \colon A\bigl[\frac{1}{\varpi}\bigr] \to \ZZ$ using $v_T$ in place of $v_S$. 
\begin{proposition}
\label{laurent-valuations}
    Let $g \colon \bigl(\cX, q\bigr) = \bigl(\Spf A, q\bigr) \to \biggl( \Spf \Bigl(\frac{\O_C\langle S, T\rangle}{(ST - \pi)} \Bigr), \{0, 0\} \biggr)$ be an \'etale map as above.
    Let $\zeta_1$ and $\zeta_2$ be the two generic points of $\cX_s$ corresponding to the two irreducible components containing $q$, whose
    images are the open prime ideals $(\m_C, T)$ and $(\m_C, S)$, respectively.
    For $i=1,2$, let $z_i = \spec^{-1}_{\cX}(\zeta_i) \in \cX_\eta$ (see \cref{lemma:local-ring-at-generic-point-in-the-special-fiber}\cref{lemma:local-ring-at-generic-point-in-the-special-fiber-0}), let $q_i\in \cX_s^n$ be the points of the normalization lying over the node $q$ that are on the component corresponding to $\zeta_i$, and let $y_i \in \ov{\{z_i\}} \subset \cX_\eta$ be the points corresponding to $q_i$ under the bijection from \cref{specialization}\cref{specialization-points}.
    Then the rank-$1$ points $z_1$ and $z_2$ correspond to the valuations $\abs{\blank}_1, \abs{\blank}_2\colon A\bigl[\frac{1}{\varpi}\bigr] \to \Gamma_C \cup \{0\}$ and the points $y_1$ and $y_2$ correspond to the valuations $\bigl(\abs{\blank}_1, -v_1(\blank)\bigr)$, $\bigl(\abs{\blank}_2, -v_2(\blank)\bigr) \colon A\bigl[\frac{1}{\varpi}\bigr] \to \Gamma_C \times \Z \cup \{0\}$, respectively.
\end{proposition}
\begin{proof}
    First, we note that $\abs{\blank}_i$ and $\bigl(\abs{\blank}_i, -v_i(\blank)\bigr)$ are multiplicative thanks to \cref{lemma:image under completion is regular} and \cref{multiplicativity of the v_S and v_T}.
    Thus, one easily deduces that they define valuations on $A$. Furthermore, it is evident from the definition that they are continuous and bounded above by $1$ on $A= \bigl(A\bigl[\frac{1}{\varpi}\bigr]\bigr)^\circ$.
    Therefore, all these valuations define elements of $\cX_\eta = \Spa\bigl(A\bigl[\frac{1}{\varpi}\bigr], A\bigr)$. 

    By the \'etaleness of $g$, we may replace $\cX$ by an open neighborhood of $q$ to ensure that $\zeta_1$
    is the only preimage of the open prime ideal $(\m_C, T) \subset \frac{\O_C\langle S, T\rangle}{(ST - \pi)}$,
    and that $q$ is the only preimage of the open maximal ideal $(\m_C, S, T) \subset \frac{\O_C\langle S, T\rangle}{(ST - \pi)}$.
    Then \cref{lemma:local-ring-at-generic-point-in-the-special-fiber}\cref{lemma:local-ring-at-generic-point-in-the-special-fiber-0} and \cref{construction:specialization} imply that, in order to show that the point $z_1$ corresponds to $\abs{\blank}_1$, it suffices to show that 
    \[
    \p_1\coloneqq \abs{\blank}_1^{-1}(\Gamma_{C, < 1}) \cap A = (\m_C, T) \cdot A.
    \]
    Since $g$ is \'etale and $g^{-1}(g(\zeta_1))=\{\zeta_1\}$, it suffices to show $g^{\#, -1}(\p_1) = (\m_C, T) \subset \frac{\O_C\langle S, T\rangle}{(ST - \pi)}$,
    which can be seen by direct inspection.

    Now we show that the point $y_1$ corresponds to the valuation $\bigl(\abs{\blank}_1, -v_1(\blank)\bigr)$.
    The previous step directly implies that $\bigl(\abs{\blank}_1, -v_1(\blank)\bigr)$ lies in $\ov{\{z_1\}}$. Therefore, \cref{specialization}\cref{specialization-points} and \cref{construction:specialization} imply that it suffices to show that
    \[
    \p'_1\coloneqq \Bigl(\bigl(\abs{\blank}_1, -v_1(\blank)\bigr)^{-1}\bigl((\Gamma_{C} \times \ZZ)_{< 1}\bigr)\Bigr) \cap A = (\m_C, S, T) \cdot A.
    \]
    The same argument as above reduces the question to the case $A=\frac{\O_C\langle S, T\rangle}{(ST - \pi)}$, where it can be seen directly from the definition. 
\end{proof}

\subsection{Line bundles on semi-stable formal curves}

In this subsection, we collect some results about line bundles on rig-smooth semi-stable formal $\O_C$-curves. Some of these results might be well-known to the experts. However, since these results play an important role in \cref{section:analytic-trace} and do not seem to appear in the literature, we decide to provide full proofs. We recall that $\varpi \in \O_C$ is a fixed pseudo-uniformizer. 

We start with the case of more general smooth formal $\O_C$-schemes. In this case, we show that any line bundle on generic fiber can be extended to a line bundle integrally. 

\begin{lemma}\label{lemma:perfect-complex} Let $\cX$ be a smooth formal $\cO_C$-scheme (resp.~smooth $\cO_C$-scheme), and let $\F$ be an adically quasi-coherent (resp.~quasi-coherent) $\pi$-torsionfree $\O_{\cX}$-module of finite type. Then $\F$ is a perfect complex.
\end{lemma}
\begin{proof}
    The claim is local on $\cX$. So we can assume that $\cX =\Spf R$ (resp.~$\cX =\Spec R$) and $\F = M^{\Updelta}$ (resp.~$\F = \widetilde{M}$) for some finitely generated $\varpi$-torsionfree $R$-module $M$. We wish to show that $M$ is a perfect $R$-complex. We prove this by verifying all conditions of  \cite[\href{https://stacks.math.columbia.edu/tag/068X}{Tag 068X}]{stacks-project}, with $A \to B$ in \textit{loc.\ cit.}~being $\cO_C \to R$. Conditions (1) and (2) follows from the assumption on $R$, and condition (4)
    follows from the fact that $M$ is $\varpi$-torsionfree and thus $\O_C$-flat (see \cite[\href{https://stacks.math.columbia.edu/tag/0539}{Tag 0539}]{stacks-project}).
    In order to check condition (3), we first note that $R$ is coherent,
    see \cite[Prop.~1.3]{BL1} (resp.~\cite[Cor.~0.9.2.8]{FujKato}).
    Hence, it suffices to show $M$ is finitely presented over $R$.
    Since $M$ is finitely generated over $R$ and flat over $\cO_C$,
    the result follows from \cite[Th.~7.3/4]{B} (resp.~\cite[\href{https://stacks.math.columbia.edu/tag/053E}{Tag 053E}]{stacks-project}). 
\end{proof}

\begin{lemma}
\label{formal version:line bundle has integral model lemma}
Let $\cX$ be a quasi-compact smooth formal $\cO_C$-scheme (resp.~a quasi-compact smooth $\cO_C$-scheme).
Then the natural map
$\Pic(\cX) \to \Pic(\cX_\eta)$ is surjective.
\end{lemma}
\begin{proof}
Let $\cal{L}$ be a line bundle on $\cX_\eta$. Pick any adically quasi-coherent (resp.~quasi-coherent) $\O_{\cX}$-module of finite type $\cal{L}_0$ such that $\cal{L}_{0, \eta} \simeq \cal{L}$ (such $\cal{L}_0$ exists by virtue of \cite[Cor.~II.5.3.3]{FujKato} and its algebraic counterpart). Then we replace $\cal{L}_0$ with its $\varpi$-free torsionfree quotient $\cal{L}_{0}/\cal{L}_0[\varpi^\infty]$ to assume that $\cal{L}_0$ is $\pi$-torsionfree adically quasi-coherent (resp.~quasi-coherent) $\O_{\cX}$-module of finite type. Now \cref{lemma:perfect-complex} ensures that $\cal{L}_0$ is a perfect complex on $\cX$. We use \cite[Th.~2]{KM76} to construct $\det(\cal{L}_0)$ which is the desired line bundle over $\cX$ whose generic fibre is $\det(\cal{L}) \simeq \cal{L}$ as the determinant construction commutes with base change.
\end{proof}

Now we discuss line bundles on certain rig-smooth semi-stable formal curves over $\O_C$. The exact analogue of \cref{formal version:line bundle has integral model lemma} is probably false in this case. Instead, we prove a weaker substitute showing that we can always trivialize a line bundle $\cal{L}\in \rm{Pic}(\cX_\eta)$ \'etale localy on $\cX$; this result is good enough for all applications in this paper. We start with the following basic lemma: %
\begin{lemma}
\label{compatibility of specialization and etale push/pull} 
Let $f\colon \cX \to \cY$ be an \'etale morphism of admissible formal $\O_C$-schemes, let $f_\eta\colon X \to Y$ be its generic fiber, and let $y\in Y(C)$ be a classical point. Then, for any two points $x_1, x_2\in f_\eta^{-1}(y) \subset X(C)$, we have $\sp_{\cX}(x_1)\neq \sp_{\cX}(x_2)$.
\end{lemma}

\begin{proof}
The point $y\in Y(C)$ uniquely extends to a morphism $i_y\colon \Spf \O_C \to \cY$. Since the formation of the specialization map is functorial with respect to morphisms of formal models, we can freely replace $\cY$ with $\Spec \O_C$ and $\cX$ with $\cX \times_{\cY} \Spf \O_C$. In this case, $\cX$ is of the form $\cX =\bigsqcup_{i\in I} \cY$, so the claim becomes trivial. 
\end{proof}

\begin{lemma}\label{lemma:local-trivialization-semistable-curves}
Let $(\cX, q)$ be a pointed rig-smooth semi-stable admissible formal $\O_C$-scheme of pure relative dimension $1$, and let $L$ be a line bundle on $X\coloneqq \cX_\eta$. Then there is an \'etale morphism $f\colon (\cU, u) \to (\cX, q)$ of pointed admissible formal $\O_C$-schemes such that $f_\eta^* L \simeq \O_{\cU_\eta}$.
\end{lemma}
\begin{proof}
    If $q\in \cX$ is a smooth point, the result immediately follows from \cref{formal version:line bundle has integral model lemma} (in fact, one can choose $f$ to be an open immersion). Therefore, we can assume that $q\in \cX$ is a nodal point. In this case, there is a pointed \'etale morphism $(\cU, u) \to (\cX, q)$ together with a pointed \'etale morphism 
    \[
    g\colon (\cU, u) \to (\cY_\pi, 0) \coloneqq \Bigl(\Spf \frac{\O_C\langle S, T\rangle}{(ST - \pi)} ,(0, 0)\Bigr)
    \]
    for some pseudo-uniformizer\footnote{The assumptions that $q\in \cX$ is a nodal point and that $\cX$ is rig-smooth imply that $\pi$ must be a pseudo-uniformizer.} $\pi\in \m_C\smallsetminus \{0\}$. Therefore, we can replace $(\cX, q)$ with $(\cU, u)$ to assume that there is an \'etale morphism $g\colon (\cX, q) \to (\cY_\pi, 0)$. Also, we can shrink $\cX$ around $q$ to assume that $g^{-1}(\{0\}) = \{q\}$. 

    Furthermore, we can assume that $\cX= \Spf A_0$ is affine and connected, so $X = \Spa(A, A^\circ)$ with $A=A_0\bigl[\tfrac{1}{\varpi}\bigr]$. Our assumptions on $\cX$ imply that $X$ is a smooth affinoid rigid-analytic curve. Furthermore, \cite[Lem.~B.12]{Z-thesis} ensures that $X$ is connected. In particular, $A = \O_X(X)$ is a Dedekind domain (see \cite[Th.~3.6.3]{FvdP04}). Thus, there is a Cartier divisor $D$ on $X$ such that $L = \O_X(D)$. 

    We put $Y_\pi \coloneqq \cY_{\pi, \eta}$. Then we consider the Cartier divisor $D'\coloneqq f_\eta(D)$ on $Y_\pi$ and the associated line bundle $L' \coloneqq \O_{Y_\pi}(D')$. Then \cite[Th.~2.1]{vdPut80} or \cite[Th.~2.2.9~(3)]{FvdP04} imply that $\rm{Pic}(Y_\pi)=0$. So $L'\simeq \O_{Y_\pi}$ and there is a meromorphic function $h\in \rm{Frac}\bigl(\O_{Y_\pi}(Y_\pi)\bigr)$ such that $D' = \rm{V}_{Y_\pi}(h)$. Therefore, we conclude that
    \[
    L \simeq L \otimes f^* (L')^\vee \simeq \O_X(D - f^{-1}(D')).
    \]
    Now \cref{compatibility of specialization and etale push/pull} and our assumption that $\{q\} = g^{-1}(\{0\})$ imply that $D - f^{-1}(D') = \sum_{x\in X(C)} a_x[x]$ such that this sum is finite and $\sp_{\cX}(x)\neq q$ if $a_x\neq 0$.  
    In other words, we can choose $D$ such that $q \notin \sp_{\cX}\bigl(\rm{Supp}(D)\bigr)$ and $\O_X(D)\simeq L$. Now we can choose an open subspace $q\subset \cU \subset \cX$ such that $\cU \cap  \sp_{\cX}\bigl(\rm{Supp}(D)\bigr) =\varnothing$, so $\restr{L}{\cU_\eta} \simeq \O_{\cU_\eta}$. This finishes the proof.
\end{proof}

\section{The trace map for smooth affinoid curves}\label{section:analytic-trace}

In this section, we build on the material developed in \cref{section:curves} and construct a trace morphism $\rm{H}^2_c(X, \mu_n) \to \Z/n\Z$ for any smooth affinoid curve.  

Throughout this section, we fix an algebraically closed nonarchimedean field $C$. We denote its ring of integers by $\O_C$, its maximal ideal by $\m_C\subset \O_C$, and its residue field by $k_C \colonequals \cO_C/\fm_C$. We also choose a pseudo-uniformizer $\varpi \in \O_C$ and an integer $n\in C^\times$. 

\subsection{Construction of the analytic trace map}
\label{Construction of the analytic trace map}

The main goal of this subsection is to construct the analytic trace morphism for any smooth affinoid curve over $C$.   

\begin{setup}
\label{setup:constructing analytic trace} We work with a smooth affinoid curve $X=\Spa(A, A^\circ)$ over $C$. We recall that \cref{lemma:universal-compactification-affinoid}, \cref{lemma:extra-points}, \cref{lemma:extra-points-curve-like}, and \cref{lemma:structure-curve-like-valuations} imply that $X$ admits a universal compactification $j \colon X \hookrightarrow X^c \colonequals \Spa(A,\cO_C[A^{\circ\circ}]^+)$ such that the complement $\abs{X^c} \smallsetminus \abs{X}$ consists of finitely many points $x_1,\dotsc,x_r$ corresponding to rank-$2$ valuations $v_{x_i} \colon k(x_i) \to \Gamma_C \times \Z \cup \{0\}$ and meets every connected component of $\abs{X^c}$.  By slight abuse of notation, we denote by the same letter the induced valuation on the henselized completed residue field
$v_{x_i}\colon \wdh{k(x_i)}^\h \to \bigl(\Gamma_C \times \Z\bigr) \cup \{0\}$.   
\end{setup}

We start by studying \'etale cohomology of $X$ and $\{x_i\}$ (considered as a pseudo-adic space) with coefficients in $\mu_n$. 

\begin{proposition}\label{cohomology-affine-curve} We have
\[ 
\Hh^i_c(X^c, \mu_n) \simeq \Hh^i(X^c,\mu_n) \simeq \Hh^i(X,\mu_n) \simeq \begin{cases}
  \mu_n(C)^{\# \pi_0(\abs{X})} \cong \left(\Z/n\Z\right)^{\# \pi_0(\abs{X})}  & i = 0 \\
  0 & i \geq 2
\end{cases} \]
and a short exact sequence
\begin{equation}\label{H1-interpretation}
0 \to A^\times/A^{\times,n} \to \Hh^1(X,\mu_n) \to \Pic(X)[n] \to 0.
\end{equation}
\end{proposition}
\begin{proof}
The first isomorphism follows from the fact that $X^c$ is proper over $\Spa(C, \O_C)$. The second isomorphism follows from \cite[Cor.~2.6.7~(ii)]{Huber-etale}. Since $C$ is algebraically closed, $\mu_n$ is non-canonically isomophic to the constant sheaf $\ud{\Lambda}$. Therefore, we conclude that $\Hh^0(X,\mu_n) \simeq \left(\mu_n(C)\right)^{\# \pi_0(\abs{X})}\cong \left(\Z/n\Z\right)^{\# \pi_0(\abs{X})}$.  

Now we show the vanishing of $\Hh^i(X,\mu_n)$ for $i \ge 2$.
Since $X$ is smooth of pure dimension $1$, Elkik's approximation theorem \cite[Th.~7, Rmk.~2, p.~587]{Elkik} lets us pick an $\cO_C$-algebra $B$ of finite type such that the $\varpi$-adic completion of $B$ is isomorphic to $A^\circ$ and $B\bigl[\tfrac{1}{\varpi}\bigr]$ is smooth over $C$ of pure dimension $1$ (see \cite[Lem.~B.5]{Z-thesis} for the dimension claim). 
Applying \cite[Th.~3.2.1, Ex.~3.1.13~iii)]{Huber-etale} to the decompleted Huber pair $(B\bigl[\tfrac{1}{\varpi}\bigr],B^+)$, where $B^+$ denotes the integral closure of $B$ in $B\bigl[\tfrac{1}{\varpi}\bigr]$,
we get
\[ \Hh^i(X,\mu_n) \simeq \Hh^i\bigl(\Spec B^\h_{(\varpi)}\bigl[\tfrac{1}{\varpi}\bigr],\mu_n\bigr); \]
the $B^\h_{(\varpi)}$ stands for the henselization of $B$ along the principal ideal $(\varpi)$.
In particular, $\Spec B^\h_{(\varpi)}\bigl[\tfrac{1}{\varpi}\bigr]$ is ind-\'etale over $\Spec B\bigl[\frac{1}{\varpi}\bigr]$. Thus, a standard approximation argument (see \cite[\href{https://stacks.math.columbia.edu/tag/09YQ}{Tag 09YQ}]{stacks-project}) and the Artin--Grothendieck vanishing theorem (see \cite[\href{https://stacks.math.columbia.edu/tag/0F0V}{Tag 0F0V}]{stacks-project}) imply that $\Hh^i\bigl(\Spec B^\h_{(\varpi)}\bigl[\tfrac{1}{\varpi}\bigr],\mu_n\bigr) = 0$ for $i\geq 2$.  

To get the short exact sequence describing $\Hh^1(X,\mu_n)$, we use the Kummer exact sequence
\[ 0 \to \mu_n \to \GG_m \xrightarrow{\cdot n} \GG_m \to 0 \]
on $X$ and the fact that $\Pic(X)$ can be identified with $\Hh^1(X,\GG_m)$ (see \cite[(2.2.7)]{Huber-etale}).
\end{proof}
\begin{lemma}
\label{cohomology-rk-2-pt}
Let $x_i \in \abs{X^c} \smallsetminus \abs{X}$, $i=1,\dotsc,r$, and let $\{x_i\}$ be the pseudo-adic space $(X^c, x_i)$ for each $i=1,\dots,r$.
Then 
\[ \Hh^i(\{x_i\},\mu_n) \simeq \begin{cases}
  \mu_n(C) \cong \Z/n\Z  & i = 0 \\
  \wdh{k(x_i)}^{\h,\times}/\Bigl(\wdh{k(x_i)}^{\h,\times}\Bigr)^n
  & i = 1 \\
  0 & i \geq 2
\end{cases}. \]
\end{lemma}
\begin{proof}
    \cref{thm:topos-of-a-point} ensures that $\Hh^i(\{x_i\}, \mu_n)\simeq \Hh^i(\Spec \widehat{k(x_i)}^\h, \mu_n)$.
    To show the vanishing part $i \ge 2$ of the lemma, it suffices to prove that the $p$-cohomological dimension of $\widehat{k(x_i)}^\h$ is $\leq 1$ for every prime $p\mid n$. We note that \cite[Lem.~2.8.4]{Huber-etale} and \cite[Lem.~1.8.6]{Huber-etale} imply that
    \[
    \rm{cd}_p \bigl(\widehat{k(x_i)}^\h\bigr) \leq \rm{tr.c}\bigl(\widehat{k(x_i)}/C\bigr) \le 1.
    \]
    The computation of $\Hh^i(\{x_i\},\mu_n)$ for $i = 0,1$ is similar to the analogous computation in the proof of \cref{cohomology-affine-curve}, noting that $\Hh^1(\Spec \wdh{k(x_i)}^\h, \GG_m)=0$ as $\wdh{k(x_i)}^\h$ is a field. 
\end{proof}

\begin{proposition}
\label{comp-supp-cohomology-affine-curve}
Keep notation as above, 
let us further denote $s \colonequals \#\bigl(\pi_0(X)\bigr)$.
Then we have
\[ \Hh^0_c(X,\mu_n) = \Hh^{\geq 3}_c(X,\mu_n) = 0 \]
and a natural exact sequence
\[ 0 \to \mu_n(C)^{\oplus (r-s)} \to \Hh^1_c(X, \mu_n) \to \Hh^1(X^c,\mu_n) \to \bigoplus^r_{i=1} \wdh{k(x_i)}^{\h,\times}/\Bigl(\wdh{k(x_i)}^{\h,\times}\Bigr)^n \xrightarrow{\partial_X} \Hh^2_c(X,\mu_n) \to 0. \]
\end{proposition}
\begin{proof} 
Everything except for the vanishing of $\rm{H}^0_c(X, \mu_n)$ follows from the long exact sequence in cohomology for the exact triangle
\begin{equation}\label{eqn:excision-sequence-proof} 
\rR\Gamma_c(X,\mu_n) \to \rR\Gamma(X^c,\mu_n) \to \bigoplus^r_{i=1}\rR\Gamma(\{x_i\},\mu_n) 
\end{equation}
coming out of the decomposition of pseudo-adic spaces $X \hookrightarrow X^c \hookleftarrow \{X, \coprod^r_{i = 1}  x_i \}$ (see \cref{rmk:excision}) together with \cref{lemma:pre-adic-dijoint-union}, \cref{cohomology-affine-curve} and \cref{cohomology-rk-2-pt}. Now we address vanishing of $\rm{H}^0_c(X, \mu_n)$. For this, we can assume that $X$ is connected, then \cref{lemma:extra-points}\cref{lemma:extra-points-3} ensures $X^c\smallsetminus X$ is non-empty. We pick some point $x\in X^c\smallsetminus X$. 
Then \cref{eqn:excision-sequence-proof} implies that the vanishing of $\rm{H}^0_c(X, \mu_n)$ follows from the injectivity of the map $\Z/n\Z \cong \rm{H}^0(X^c, \mu_n) \to \rm{H}^0(\{x\}, \mu_n) \cong \Z/n\Z$.
\end{proof}

As a nice application of \cref{comp-supp-cohomology-affine-curve}, we prove the following result:

\begin{corollary}\label{cor:first-cohomology-affinoids} Let $X$ be a smooth affinoid curve over $C$. Then $\rm{H}^1_c(X, \mu_n)$ is finite. Furthermore, $\rm{H}^1_c(\bf{D}^1, \mu_n)=0$.
\end{corollary}
\begin{proof}
    \cref{comp-supp-cohomology-affine-curve} shows that it suffices to show that the image of the map $\alpha_X \colon \rm{H}^1_c(X, \mu_n) \to \rm{H}^1(X, \mu_n)=\rm{H}^1(X^c, \mu_n)$ is finite. For this, we use \cref{useful proposition on rigid curves}\cref{useful proposition on rigid curves-3} to get an algebraic proper compactifcation $j\colon X\hookrightarrow \overline{X}$. Then the natural morphism $j_!\mu_{n, X} \to \rm{R}j_*\mu_{n, X}$ canonically factors through $\mu_{n, \overline{X}}$. Therefore, $\alpha_X$ factors as the composition 
    \[
    \rm{H}^1_c(X, \mu_n) \to \rm{H}^1(\ov{X}, \mu_n) \to \rm{H}^1(X, \mu_n).
    \]
    Thus, it suffices to show that $\rm{H}^1(\ov{X}, \mu_n)$ is finite. This follows from \cref{useful proposition on rigid curves}\cref{useful proposition on rigid curves-6} and finiteness of algebraic \'etale cohomology. To see that $\rm{H}^1_c(\bf{D}^1, \mu_n)=0$, we use the embedding $j\colon \bf{D}^1 \to \bf{P}^{1, \an}$ and a similar argument to reduce to showing that $\rm{H}^1(\bf{P}^{1, \an}, \mu_n)$. This again follows from \cref{useful proposition on rigid curves}\cref{useful proposition on rigid curves-6} and the analogous claim in algebraic geometry.
\end{proof}

Now we are ready to start constructing the analytic trace map:

\begin{definition}\label{analytic-pre-trace} For a smooth affinoid $X$, we define an \textit{analytic pre-trace}
\[ 
\widetilde{t}_X \colonequals \sum_{i=1}^r \# \circ v_{x_i} \colon \bigoplus^r_{i=1} \wdh{k(x_i)}^{\h,\times}/\Bigl(\wdh{k(x_i)}^{\h,\times}\Bigr)^n \to \ZZ/n\ZZ, 
\]
where $\#$ is defined as in \cref{defn:reduction}.
\end{definition}

In order to justify the name ``pre-trace'' morphism, we recall the exact sequence 
\[
\rm{H}^1(X^c, \mu_n) \xrightarrow{\rm{res}} \bigoplus^r_{i=1} \wdh{k(x_i)}^{\h,\times}/\Bigl(\wdh{k(x_i)}^{\h,\times}\Bigr)^n \xrightarrow{\partial_X} \rm{H}^2_c(X, \mu_n) \to 0
\]
from \cref{comp-supp-cohomology-affine-curve}. This ensures that the analytic pre-trace morphism $\widetilde{t}_X$ descends to a morphism $t_X\colon \rm{H}^2_c(X, \mu_n) \to \Z/n\Z$ if and only if $\widetilde{t}_X$ vanishes on the image of $\rm{res}$. 
The following vanishing is one of the main results of this subsection: 

\begin{theorem}
\label{thm:trace-well-defined} 
In \cref{setup:constructing analytic trace}, the composition
\[
\rm{H}^1(X^c, \mu_n) \to \bigoplus^r_{i=1} \wdh{k(x_i)}^{\h,\times}/\Bigl(\wdh{k(x_i)}^{\h,\times}\Bigr)^n \xrightarrow{\widetilde{t}_X} \Z/n\Z
\]
is zero.
\end{theorem}

In \cref{first proof} and \cref{second proof} below, we give two different proofs of \cref{thm:trace-well-defined}. But before we start discussing the proofs, we give the official definition of the analytic trace map assuming validity of \cref{thm:trace-well-defined}:

\begin{definition}\label{defn:analytic-trace-curves}
    The \emph{analytic trace morphism} $t_X\colon \rm{H}^2_c(X, \mu_n) \to \Z/n\Z$ is the unique group homomorphism such that the composition
    \[
    \bigoplus^r_{i=1} \wdh{k(x_i)}^{\h,\times}/\Bigl(\wdh{k(x_i)}^{\h,\times}\Bigr)^n \xrightarrow{\partial_X} \rm{H}^2_c(X, \mu_n) \xrightarrow{t_X} \Z/n\Z
    \]
    is equal to $\widetilde{t}_X$.
\end{definition}
\begin{remark}\label{rmk:trace-surjective} We note that each morphism $\#\circ v_{x_i} \colon \wdh{k(x_i)}^{\h,\times}/\Bigl(\wdh{k(x_i)}^{\h,\times}\Bigr)^n \xrightarrow{\widetilde{t}_X} \Z/n\Z$ is surjective. Then \cref{lemma:extra-points}\cref{lemma:extra-points-3} formally implies that $t_X \colon \rm{H}^2_c(X, \mu_n) \to \Z/n\Z$ is surjective for any smooth affinoid curve $X$.
\end{remark}

Before embarking on the proof of \cref{thm:trace-well-defined}, let us explicate its statement and the resulting \cref{defn:analytic-trace-curves} in the simple case of the $1$-dimensional closed unit disk. This will require the following sequence of lemmas: 

\begin{lemma}\label{lemma:trivial-picard-disk} We have $\rm{Pic}(\bf{D}^1)=0$.
\end{lemma}
\begin{proof}
    First, \cite[Th.~1.4.2]{KedAr} ensures that $\rm{Pic}(\bf{D}^1)\simeq \rm{Pic}(C\langle T\rangle)$. Now \cite[Cor.~2.2/10]{B} implies that $C\langle T\rangle$ is a UFD, thus $\rm{Pic}(C\langle T\rangle)=0$ by \cite[\href{https://stacks.math.columbia.edu/tag/0BCH}{Tag 0BCH}]{stacks-project}.
\end{proof}

\begin{lemma}
\label{invertible functions on closed disk}
As an abelian group, we have a decomposition
\[
C\langle T \rangle^{\times} = C^{\times} \times (1 + \mathfrak{m}_C T \langle T \rangle, \times).
\]
\end{lemma}

Here $1 + \mathfrak{m}_C T \langle T \rangle \coloneqq \{f = \sum_{i} a_i T^i \in C\langle T \rangle \mid a_0 = 1, a_{\geq 1} \in \mathfrak{m}_C\}$,
viewed as a group via multiplication.

\begin{proof}
We first note that if $f=\sum a_i T^i$ is a unit, then $a_0$ is a unit in $C$.
Moreover, by \cite[Cor.~2.2/4]{B}, $f \in C\langle T \rangle$ is a unit if and only if $\abs{a_0}>\abs{a_i}$ for $i>1$. Then one simply has $f = f(0) \cdot \frac{f}{f(0)}$ proving the lemma.
\end{proof}

\begin{remark}\label{rmk:infinite-cohomology} Let $C$ be an algebraically closed nonarchimedean field of mixed characteristic $(0, p)$. Then \cref{invertible functions on closed disk} implies that there is a surjection 
\[
\Hh^1(\DD^1, \mu_p) = (1 + \mathfrak{m}_C T \langle T \rangle, \times)/p \to \m_C/p\m_C
\]
defined by the rule $1 +\sum_{i\geq 1} a_i T^i \mapsto a_1$. In particular, $\Hh^1(\DD^1, \mu_p)$ is infinte and its cardinality is at least cardinality of $\m_C/p\m_C$. 
\end{remark}

\begin{example}\label{an-tr-disk}
We explain \cref{thm:trace-well-defined} in case $X = \DD^1$.
By \cref{lemma:compacitification-of-the-disc}, the universal compactification of the $1$-dimensional closed unit disk $\DD^1 \subset \DD^{1,c}$ consists of one additional point $x_+$ of rank-$2$ ``pointing toward $\infty$'' which corresponds to the valuation $v_{x_+}\colon k(x_+) \to \bigl(\Gamma_C\times \Z\bigr) \cup \{0\}$.
The explicit description of the corresponding valuation $v_{x_+}$ in \cref{lemma:compacitification-of-the-disc} shows that $\#\circ v_{x_+}\colon \wdh{k(x_+)}^{\h, \times} \to \Z$ vanishes on the image of the morphism $C\langle T\rangle^\times \to \wdh{k(x_+)}^{\h, \times}$:
by \cref{invertible functions on closed disk}, this boils down to the fact that $\#\circ v_{x_+}$ is zero on the scalars $c\in C^\times$ and functions of the form $f=1+\sum_{i\geq 1} a_i T^i$ with $a_i\in \fm$.
Thanks to \cref{cohomology-affine-curve} and \cref{lemma:trivial-picard-disk}, this yields the vanishing of $\#\circ v_{x_+}$ on $\Hh^1(\DD^{1,c},\mu_n)$.
As a consequence, we obtain the analytic trace morphism $t_{\DD^1} \colon \Hh^2_c(\DD^1,\mu_n) \to \ZZ/n\ZZ$.
\end{example}

\subsection{Construction of the analytic trace: first proof}\label{first proof}

In this subsection, we give the first proof of \cref{thm:trace-well-defined}. 
The idea is to reduce the case of a general smooth affinoid curve $X$ via Noether normalization to the case of the closed unit disk which was already treated in \cref{an-tr-disk}.  

In order to implement this strategy, we will need to verify certain technical lemmas about the trace morphisms (see \cref{thm:flat-trace}) for finite flat morphisms of smooth rigid-analytic curves over $C$. This will occupy the most part of this subsection. As an application of these methods, we also show that the analytic trace is compatible with the finite flat trace defined in \cref{thm:flat-trace}. 

\begin{setup}
\label{setup:finite-flat-morphism} 
We fix a finite flat morphism $f\colon X=\Spa(B, B^\circ) \to Y=\Spa(A, A^\circ)$ of smooth affinoid curves over $C$ with induced morphism $f^c\colon X^c \to Y^c$ between the universal compactifications. We also denote by $Z_Y=\{y_i\}_{i\in I}$ the finite complement $\abs{Y^c}\smallsetminus \abs{Y}$, and by $Z_i=f^{c, -1}(\{y_i\})=\{x_{i, j_i}\}_{j_i\in J_i}$ the pre-image of $y_i$ in $X^c$. 
\end{setup}

In \cref{setup:finite-flat-morphism}, \cref{lemma:compactifications-compatible}, \cref{lemma:extra-points}, and \cref{lemma:extra-points-curve-like}  ensure that $Z_Y$ and $Z_i$ are finite discrete sets consisting of rank-$2$ curve-like points and that $Z_X\coloneqq \sqcup_{i\in I} Z_i = X^c\smallsetminus X$. For each $i\in I$, we denote by 
\[
f_i^c\colon (X^c, Z_i)\to (Y^c, \{y_i\})
\]
the morphism of pseudo-adic spaces induced by $f^c$. Similarly, for $i=1, \dots, r$, we denote by 
\[
g_i\colon \bigsqcup_{j_i\in J_i} \Spec \wdh{k(x_{i, j_i})}^{\rm{h}}\to \Spec \wdh{k(y_i)}^{\rm{h}}
\]
the induced morphism of the henselized completed residue fields. These morphisms fit into the following commutative diagram of topoi:
\begin{equation}\label{eqn:maps-above-and-below}
\begin{tikzcd}
    X_\et \arrow{d}{f_\et} \arrow{r}{j_{X,\et}} &  X^c_\et \arrow{d}{f^c_\et} & (X^c, Z_X)_\et\simeq \prod_{i\in I}(X^c, Z_{i})_\et \arrow{r}{\gamma_{Z_X}} \arrow[r, swap, "\sim"] \arrow[l, swap, "i_{X, \et}"] \arrow{d}{f'_\et\coloneqq \prod f_{i,\et}^c} & \prod_{i\in I, j_i\in J_i} \Spec \wdh{k(x_{i, j_i})}^\h_\et \simeq \biggl(\bigsqcup_{i\in I, j_i\in J_i} \Spec \wdh{k(x_{j_i})}^{\rm{h}}\biggr)_\et \arrow{d}{g_\et=(\bigsqcup_{i\in I} g_i)_\et}   \\
    Y_\et \arrow{r}{j_{Y,\et}} & Y^c_\et & \arrow[l, swap, "i_{Y, \et}"] (Y^c, Z_Y)_\et \simeq \prod_{i\in I}(Y^c, \{y_i\})_\et \arrow{r}{\gamma_{Z_Y}} \arrow[r, swap, "\sim"] & \prod_{i\in I} \Spec \wdh{k(y_i)}^\h_\et \simeq \biggl(\bigsqcup_{i\in I} \Spec \wdh{k(y_i)}^{\rm{h}} \biggr)_\et,
\end{tikzcd}
\end{equation}
where the horizontal equivalences come from \cref{lemma:pre-adic-dijoint-union} and \cref{thm:topos-of-a-point}. 
Now we note that \cref{lemma:compactification-of-finite-morphism} ensures that $f^c$ is finite flat, so \cref{thm:flat-trace} provides us with trace morphisms $\ttr_{f, \mu_n}\colon f_*\,\mu_{n, X} \to \mu_{n, Y}$ and $\ttr_{f^c, \mu_n} \colon f^c_*\,\mu_{n, X^c} \to \mu_{n, Y^c}$. Furthermore, \cref{cor:finite-morphism-finite-on-henselized-completions} implies that 
\[
g\colon \bigsqcup_{\substack{i\in I,\\ j_i\in J_i}} \Spec \wdh{k(x_{j_i})}^{\rm{h}} \to \bigsqcup_{i\in I} \Spec \wdh{k(y_i)}^{\rm{h}}
\]
is a finite flat morphism (of schemes). Therefore, using the horizontal equivalences in Diagram~\cref{eqn:maps-above-and-below}, we can define the trace morphism
\[
\ttr_{f', \mu_n}\colon f'_{*}\, \mu_{n, Z_X} \to \mu_{n, Z_Y}
\]
as $\ttr_{f', \mu_n}=\gamma_{Z_Y}^*(\ttr_{g, \mu_n})$, where $\ttr_{g, \mu_n}$ is the algebraic finite flat trace map constructed in \cite[Exp.~XVII, Th.~6.2.3]{SGA4} and \cite[\href{https://stacks.math.columbia.edu/tag/0GKI}{Tag 0GKI}]{stacks-project}.
Next lemma will be the key to our (first) proof of \cref{thm:trace-well-defined}: 

\begin{lemma}\label{lemma:compatibility-of-traces} 
Let $f\colon X\to Y$ and $f^c\colon X^c \to Y^c$ be as in \cref{setup:finite-flat-morphism}. Then the following diagram commutes:
\[ \begin{tikzcd}[column sep = huge, row sep = large]
    0 \arrow{r} & j_{Y, !} f_{*}\,\mu_{n, X} \arrow{r} \arrow{d}{j_{Y,!}\left(\ttr_{f, \mu_n}\right)} & f^c_{*}\, \mu_{n, X^c}\arrow{d}{\ttr_{f^c, \mu_n}} \arrow{r} & i_{Y, *} f'_* \,\mu_{n, Z_X} \arrow{d}{i_{Y, *}\left(\ttr_{f', \mu_n}\right)} \arrow{r} &  0 \\
    0 \arrow{r} & j_{Y, !} \,\mu_{n, Y}\arrow{r} & \mu_{n, Y^c} \arrow{r} & i_{Y,*} \,\mu_{n, Z_Y} \arrow{r} & 0 
\end{tikzcd} \]
\end{lemma}
\begin{proof}
    It suffices to show that each square commutes separately. In order to check that the left square commutes, it suffices to show that $j^*_Y\left(\ttr_{f^c, \mu_n}\right)=\ttr_{f, \mu_n}$. This follows from the fact that $\ttr_{f^c, \mu_n}$ commutes with arbitrary base change (see \cref{thm:flat-trace}).  

    In order to show that the right square commutes, it suffices to show that $i^*_Y\left(\ttr_{f^c, \mu_n}\right)= \ttr_{f', \mu_n}$. It will be more convenient to check this equality after applying the equivalence $\gamma_{Z_Y}$ to both sides. For this, we recall that the discussion before \cref{construction:relative-analytification} ensures that we have a commutative diagram of topoi
    \[
        \begin{tikzcd}
            X^c_\et \arrow{d}{f^c_\et} \arrow{r}{c_X}&  \left(\Spec B\right)_\et \arrow{d}{f^{c, \rm{alg}}_\et}\\
            Y^c_\et \arrow{r}{c_Y} & \left(\Spec A\right)_\et,
        \end{tikzcd}
    \]
    where $c_A$ and $c_B$ are the analytification morphisms. Furthermore, this fits into a bigger commutative diagram:
    \begin{equation}\label{eqn:big-diagram}
        \begin{tikzcd}[column sep=small]
        (X^c, Z_X)_\et \arrow{rd}{i_{X, \et}} \arrow{rr}{\gamma_{Z_X}} \arrow[rr, swap, "\sim"]\arrow{dd}{f'_\et} & & \Biggl(\bigsqcup_{i\in I,\, j_i\in J_i} \Spec \widehat{k(x_{i, j_i})}^\h\Biggr)_\et \arrow{dd}[pos = 0.3]{g_\et} \arrow{rd}{i_{X, \et}^{\rm{alg}}} & \\
        & (X^c)_\et \arrow[rr,crossing over,pos=0.3,"c_X"] & & \left(\Spec B\right)_\et \arrow{dd}{f^{c, \rm{alg}}_\et} \\
        (Y^c, Z_Y)_\et \arrow{rd}{i_{Y, \et}} \arrow[rr,pos=0.6,"\gamma_{Z_Y}","\sim"'] & &  \Biggl(\bigsqcup_{i\in I}  \Spec \widehat{k(y_i)}^\h\Biggr)_\et \arrow{rd}{i_{Y, \et}^{\rm{alg}}}& \\
        & (Y^c)_\et \arrow[from=uu, crossing over, swap, pos = 0.3, "f^c_\et"] \arrow{rr}{c_Y} & & \left(\Spec A\right)_\et.
    \end{tikzcd}
    \end{equation}
    We recall that $\ttr_{f', \mu_n}$ is defined as $\gamma_{Z_Y}^*(\ttr_{g, \mu_n})$, and $\ttr_{f^c, \mu_n} = c_Y^*(\ttr_{f^{c, \rm{alg}}, \mu_n})$ (see \cref{thm:flat-trace}\cref{thm:flat-trace-6}). Therefore, it suffices to show that $i_{Y}^{\rm{alg}, *}(\ttr_{f^{c, \rm{alg}}, \mu_n})=\ttr_{g, \mu_n}$. Now we note that \cite[\href{https://stacks.math.columbia.edu/tag/0GKI}{Tag 0GKI}]{stacks-project} guarantees that the algebraic finite flat trace map commutes with arbitrary base change. Therefore, it suffices to show that the natural morphism
    \[
    \widehat{k(y_i)}^\h \otimes_A B \to \prod_{j_i\in J_i} \widehat{k(x_{i, j_i})}^\h
    \]
    is an isomorphism for any $i\in I$. This follows directly from the combination of \cref{lemma:generalization-weakly-Shilov}, \cref{lemma:weakly-shilov-fiber-reduced}, and \cref{thm:iso-henselized-completed-residue-fields}.
\end{proof}

For each $i\in I$ and $j_i\in J_i$, let $\rm{Nm}_{x_{i, j_i}/y_i} \colon \wdh{k(x_{i, j_i})}^{\h,\times}/\Bigl(\wdh{k(x_{i, j_i})}^{\h,\times}\Bigr)^n \to \wdh{k(y_{i})}^{\h,\times}/\Bigl(\wdh{k(y_i)}^{\h,\times}\Bigr)^n$ be the morphism induced by the norm morphism $\rm{Nm}\colon \wdh{k(x_{i, j_i})}^{\h,\times} \to \wdh{k(y_{i})}^{\h,\times}$;
note that the norm map is well-defined due to \cref{cor:finite-morphism-finite-on-henselized-completions}). 

\begin{corollary}\label{cor:norm-maps-are-traces} Let $f\colon X\to Y$ and $f^c\colon X^c \to Y^c$ be as in \cref{setup:finite-flat-morphism}. Then there is the following diagram of exact sequences:
\begin{equation*}\label{eqn:noether-analytic-trace}
    \begin{tikzcd}
        \rm{H}^1(X^c, \mu_n) \arrow{r} \arrow{d}{\rm{H}^1(\ttr_{f^c, \mu_n})}& \bigoplus_{\substack{i\in I,\\j\in J_i}} \wdh{k(x_{i, j_i})}^{\h,\times}/\Bigl(\wdh{k(x_{i, j_i})}^{\h,\times}\Bigr)^n \arrow[d,shorten <= -1.5em,pos=.1,"\bigoplus_{i\in I} \bigl(\sum_{j_i\in J_i} \rm{Nm}_{x_{i, j_i}/y_i}\bigr)"]  \arrow{r}{\partial_X} & \rm{H}^2_c(X, \mu_n) \arrow{d}{\rm{H}^2_c(\ttr_{f, \mu_n})}\arrow{r} & 0\\
        \rm{H}^1(Y^c, \mu_n) \arrow{r} & \bigoplus_{i\in I} \wdh{k(y_{i})}^{\h,\times}/\Bigl(\wdh{k(y_{i})}^{\h,\times}\Bigr)^n \arrow{r}{\partial_Y} & \rm{H}^2_c(Y, \mu_n) \arrow{r} & 0.
    \end{tikzcd}
\end{equation*}
\end{corollary}
\begin{proof}
    Exactness of horizontal sequences follows directly from \cref{comp-supp-cohomology-affine-curve}. Now let $g_{i, j_i}$ be the natural morphism $\Spec \wdh{k(x_{i, j_i})}^\h \to \Spec \wdh{k(y_{i})}^\h$. Then \cite[Exp.~XVII, Diagram~(6.3.18.2) on p.~198]{SGA4} implies that, under the identification $\rm{H}^1\Bigl(\widehat{k(x_{i, j_i})}^\h, \mu_n\Bigr) \simeq \widehat{k(x_{i, j_i})}^{\h, \times}/\Bigl(\widehat{k(x_{i, j_i})}^{\h, \times}\Bigr)^n$ and under the similar identification for $y_i$, the trace map $\rm{H}^1(\ttr_{g_{i, j_i}, \mu_n})$ becomes equal to $\rm{Nm}_{x_{i, j_i}/y_i}$. Therefore, the result follows directly from \cref{lemma:compatibility-of-traces}. 
\end{proof}

Now we are finally ready to give the first proof of \cref{thm:trace-well-defined}: 

\begin{proof}[First Proof of \cref{thm:trace-well-defined}]
In this proof, we use the notation from \cref{setup:constructing analytic trace}; we warn readers that the notation is slightly different from \cref{setup:finite-flat-morphism}. We start the proof by noting that \cite[Prop.~3.1/2]{B}, \cite[\href{https://stacks.math.columbia.edu/tag/00OK}{Tag 00OK}]{stacks-project}, and \cref{lemma:finite-flat-curves} allow us to find a finite flat morphism $f\colon X \to \bf{D}^1$. We denote by $f^c\colon X^c \to \bf{D}^{1, c}$ the induced morphism of universal compactifications.  

Now we consider the open immersion $j\colon X \hookrightarrow X^c$ with $\abs{X^c}\smallsetminus \abs{X}$ consisting of finitely many rank-$2$ points $\{x_1, \dots, x_r\}$. Then \cref{cor:norm-maps-are-traces}, \cref{lemma:norm-commute}, and \cref{lemma:extra-points-curve-like} imply that the following diagram
\begin{equation}\label{eqn:noether-analytic-trace-2}
    \begin{tikzcd}[column sep = huge]
        \rm{H}^1(X^c, \mu_n) \arrow{r} \arrow{d}{\rm{H}^1(\ttr_{f^c, \mu_n})}& \bigoplus_{i=1}^r \wdh{k(x_{i})}^{\h,\times}/\Bigl(\wdh{k(x_{i})}^{\h,\times}\Bigr)^n \arrow[d, shorten <= -.7em, swap, "\sum_{i=1}^r \rm{Nm}_{x_i/x_+}"] \arrow[r,  "\sum_{i=1}^r \#\circ v_{x_i}"] &[1em] \Z/n\Z \\
        \rm{H}^1(\bf{D}^{1, c}, \mu_n) \arrow{r} & \wdh{k(x_{+})}^{\h,\times}/\Bigl(\wdh{k(x_{+})}^{\h,\times}\Bigr)^n \arrow[ru, swap, "\#\circ v_{x_+}"] & 
    \end{tikzcd}
\end{equation}    
commutes. Therefore it suffices to show that the composition
    \[
    \rm{H}^1(X^c, \mu_n) \xrightarrow{\rm{H}^1(\ttr_{f^c, \mu_n})} \rm{H}^1(\bf{D}^{1, c}, \mu_n) \to \widehat{k(x_+)}^{\times, \h}/\bigl(\widehat{k(x_+)}^{\h, \times}\bigr)^n \xrightarrow{\#\circ v_{x_+}} \Z/n\Z
    \]
    is zero. To that end, we just note that \cref{an-tr-disk} implies that the composition of the last two maps is already zero.
\end{proof}

As an application of our methods, we also show that the analytic trace morphism is compatible with the finite flat trace morphisms: 

\begin{theorem}\label{thm:analytic-trace-compatible-finite-flat-trace} Let $f\colon X\to Y$ be a finite flat morphism of smooth affinoid rigid-analytic $C$-curves. Then the diagram
\[
\begin{tikzcd}
\rm{H}^2_c(X, \mu_n) \arrow[d, swap, "\rm{H}^2_c(\ttr_{f, \mu_n})"] \arrow{r}{t_X} & \Z/n\Z \\
\rm{H}^2_c(Y, \mu_n) \arrow[ru, swap, "t_Y"] &
\end{tikzcd}
\]
commutes, where $\ttr_f$ is the finite flat trace morphism from \cref{thm:flat-trace} and $t_X$, $t_Y$ are analytic traces from \cref{defn:analytic-trace-curves}.
\end{theorem}
\begin{proof}
    Keeping the notation of \cref{setup:finite-flat-morphism}, \cref{cor:norm-maps-are-traces} ensures that the diagram
    \[
    \begin{tikzcd}
    \bigoplus_{i\in I,\, j\in J_i} \wdh{k(x_{i, j_i})}^{\h,\times}/\Bigl(\wdh{k(x_{i, j_i})}^{\h,\times}\Bigr)^n \arrow[d,shorten <= -1em,pos=.4,"\bigoplus_{i\in I} \bigl(\sum_{j_i\in J_i} \rm{Nm}_{x_{i, j_i}/y_i}\bigr)"'] \arrow{r}{\partial_X} & \rm{H}^2_c(X, \mu_n) \arrow{d}{\rm{H}^2_c(\ttr_{f, \mu_n})} \\
    \bigoplus_{i\in I} \wdh{k(y_{i})}^{\h,\times}/\Bigl(\wdh{k(y_{i})}^{\h,\times}\Bigr)^n \arrow{r}{\partial_Y} & \rm{H}^2_c(Y, \mu_n)
    \end{tikzcd}
    \]
    commutes. Now we use that both $\partial_X$ and $\partial_Y$ are surjective and the definition of the analytic trace map (see \cref{defn:analytic-trace-curves}) to conclude that it suffices to show that the diagram 
    \[
    \begin{tikzcd}[column sep=huge]
    \bigoplus_{i\in I,\, j\in J_i} \wdh{k(x_{i, j_i})}^{\h,\times}/\Bigl(\wdh{k(x_{i, j_i})}^{\h,\times}\Bigr)^n \arrow[d, shorten <= -1em, pos=.4, "\bigoplus_{i\in I} \bigl(\sum_{j_i\in J_i} \rm{Nm}_{x_{i, j_i}/y_i}\bigr)"'] \arrow{r}{\sum \#\circ v_{x_{j_i}} }&  \Z/n\Z \\
    \bigoplus_{i\in I} \wdh{k(y_{i})}^{\h,\times}/\Bigl(\wdh{k(y_{i})}^{\h,\times}\Bigr)^n \arrow[ru, swap, "\sum \#\circ v_{y_i}"] & 
    \end{tikzcd}
    \]
    commutes. This now follows directly from \cref{lemma:norm-commute} and \cref{lemma:extra-points-curve-like}. 
\end{proof}

\subsection{Construction of the analytic trace: second proof}
\label{second proof}

Now we give another proof of \cref{thm:trace-well-defined} which does not resort to Noether normalization, finite flat traces or the results of \cref{first proof}. 
Instead, we use the interpretation of $\Hh^1(X^c,\mu_n)$ as isomorphism classes of $\mu_n$-torsors to generalize \cref{an-tr-disk} to the setting of smooth affinoid curves.
This strategy necessitates a more detailed analysis of $\Hh^1(X^c,\mu_n) \simeq \Hh^1(X,\mu_n)$.
\begin{construction}
\label{mun-torsor-class}
Recall that \cref{H1-interpretation} induces a canonical identification
\[ \Hh^1(X^c,\mu_n) \simeq \{ (L,s) \suchthat L \in \Pic(X^c),\, s \colon \cO \xrightarrow{\sim} L^{\otimes n} \}/\sim \]
with isomorphism classes of line bundles together with a trivialization of their $n$-th power.
With this interpretation, we can attach to any point $x \in \abs{X^c}$ a natural homomorphism
\[ \rho_x \colon \Hh^1(X^c,\mu_n) \to k(x)^\times/\bigl(k(x)^\times\bigr)^n \]
as follows:
Given $(L,s) \in \Hh^1(X^c,\mu_n)$, choose an open affinoid neighborhood $U \subseteq X^c$ of $x$ on which the restricted line bundle $\restr{L}{U}$ admits a trivialization $a \colon \cO_U \xrightarrow{\sim} \restr{L}{U}$.
Then the image of $(L,s)$ under $\Hh^1(X^c,\mu_n) \to \Hh^1(U,\mu_n)$ lies in the image of the boundary map $\cO(U)^\times/\bigl(\cO(U)^\times)^n \to \Hh^1(U,\mu_n)$ of the Kummer sequence;
concretely, it is the well-defined (independent of the choice of $a$) element of $\cO(U)^\times/\bigl(\cO(U)^\times)^n$ determined by the isomorphism
\[ a^{-n} \circ s_U \colon \cO_U \xrightarrow{\sim} \bigl(\restr{L}{U}\bigr)^{\otimes n} \xrightarrow{\sim} \cO^{\otimes n}_U \simeq \cO_U. \] 
We define $\rho_x(L,s)$ to be the image of this element under the natural map
\[ \cO(U)^\times/\bigl(\cO(U)^\times)^n \to \cO_{U,x}^\times/\bigl(\cO_{U,x}^\times)^n \to k(x)^\times/\bigl(k(x)^\times\bigr)^n. \]
By passing to common open affinoids trivializing several line bundles, one checks that this defines a group homomorphism.
\end{construction}
\begin{variant}\label{mun-torsor-class-h}
By \cref{thm:topos-of-a-point} and Hilbert's Theorem 90, we have the identification $\Hh^1(\{x\},\mu_n) \simeq \Hh^1\Bigl(\Spec \wdh{k(x)}^\h,\mu_n\Bigr) \simeq \wdh{k(x)}^{\h,\times}/\Bigl(\wdh{k(x)}^{\h,\times}\Bigr)^n$.
The functoriality of the Kummer sequence then shows that under this isomorphism, the composition
\[ \Hh^1(X^c,\mu_n) \xlongrightarrow{\rho_x} k(x)^\times/\bigl(k(x)^\times\bigr)^n \longrightarrow \wdh{k(x)}^{\h,\times}/\Bigl(\wdh{k(x)}^{\h,\times}\Bigr)^n \]
is given by the natural map $\Hh^1(X^c,\mu_n) \to \Hh^1(\{x\},\mu_n)$.
Concretely, it can again be described as in \cref{mun-torsor-class} using trivializations of the pullback of $L$ along the map of ringed \'etale topoi $\bigl(\Spec\bigl(k(x)^\h\bigr)_\et,\cO\bigr) \to \bigl(X^c_\et,\cO\bigr)$ induced by the map of \'etale topoi $\Spec\bigl(k(x)^\h\bigr)_\et \xleftarrow[\gamma]{\sim} \bigl(\Spa\bigl(k(x)^\h,k(x)^{+,\h}\bigr),\{x\}\bigr)_\et \to X^c_\et$ from \cref{thm:topos-of-a-point}.
We still denote this map by $\rho_x$ when there is no risk for confusion.
\end{variant}
If $x \in \abs{X}$, then $\rho_x$ factors through $\Hh^1(X,\mu_n)$ and we can also work with line bundles $L$ on $X$ instead of $X^c$ in \cref{mun-torsor-class} and \cref{mun-torsor-class-h}.
While the ultimate construction of the analytic trace in \cref{defn:analytic-trace-curves} only uses rank-$2$ valuations for points in $\abs{X^c} \smallsetminus \abs{X}$, the second proof of \cref{thm:trace-well-defined} also uses \cref{mun-torsor-class} for points in $\abs{X}$.
We then have the following compatibility with the inclusion of residue fields from \cite[Lem.~1.1.10~iii)]{Huber-etale}:
\begin{lemma}
\label{mun-torsor-class-specialization}
    Let $z \in \abs{X}$ be point of rank $1$ and $y \in \overline{\{z\}} \subset{X^c}$ a specialization of $z$.
    Then the composition
    \[ \Hh^1(X^c,\mu_n) \xlongrightarrow{\rho_y} k(y)^\times/\bigl(k(y)^\times\bigr)^n \longrightarrow k(z)^\times/\bigl(k(z)^\times\bigr)^n \]
    with the morphism induced by the inclusion of residue fields $k(y) \to k(z)$ is equal to $\rho_z$.
\end{lemma}
\begin{proof}
    Unwinding \cref{mun-torsor-class}, this follows from commutativity of the diagram of natural maps
    \[ \begin{tikzcd}[row sep=.1em]
        & \cO_{U,y} \arrow[r] \arrow[dd] & k(y) \arrow[dd] \\
        \cO(U) \arrow[ru] \arrow[rd] && \\
        & \cO_{U,z} \arrow[r] & k(z)
    \end{tikzcd} \]
    for any open affinoid neighborhood $U \subseteq X^c$ of $y$ (and thus also $z$).
\end{proof}

\begin{lemma}
\label{mun-torsor-further-specialization}
Let $\cX$ be an admissible formal $\O_C$-model of $X$ whose special fiber $\cX_s$
is a reduced separated scheme of pure dimension $1$.
Let $z \in \abs{X}$ be point of rank $1$ whose specialization $\zeta = \spec(z) \in \abs{\cX_s}$
is a generic point.
Then there is a pushout diagram of groups
\[ 
\begin{tikzcd}
(\cO_C)^{\times} \arrow[r] \arrow[d] & (k(z)^+)^{\times} \arrow[d] \\
C^{\times} \arrow[r] & k(z)^{\times}.
\end{tikzcd} 
\]
In particular, using \cref{lemma:local-ring-at-generic-point-in-the-special-fiber}\cref{lemma:local-ring-at-generic-point-in-the-special-fiber-2} we get the well-defined ``\emph{further specialization}'' map 
\[ \begin{tikzcd}%
    k(z)^{\times}/n \arrow[rrr, bend left=16, "\spec_z"] & \arrow[l, swap, "\sim"] (k(z)^+)^{\times}/n \arrow[r, "\sim"] & \O_{\cX, \zeta}^\times/n \arrow[r,two heads] & k(\zeta)^{\times}/n.
\end{tikzcd} \]
\end{lemma}
\begin{proof}
The vertical maps are injective with cokernels the respective value groups.
Thus, for the first statement we only need to show that the induced map of value groups is a bijection,
which follows from \cref{lemma:local-ring-at-generic-point-in-the-special-fiber}\cref{lemma:local-ring-at-generic-point-in-the-special-fiber-4}.
The second statement then follows directly from the Snake Lemma and the fact that $\Gamma_C$ is divisible (see \cite[Obs.~3.6/10]{BGR}).   
\end{proof}

\begin{notation}\label{notation:specialization-of-torsor}
In the situation of \cref{mun-torsor-further-specialization}, let
$(L, s) \in \Hh^1(X^c,\mu_n)$.
Then we denote the further specialization of $\rho_z(L, s)$
under the map $\rm{sp}_z\colon k(z)^{\times}/n \to k(\zeta)^{\times}/n$
by $\sp_z(L, s)$.
\end{notation}

We recall that, for a smooth irreducible $k$-curve $Y$, a point $y\in Y(k)$, and a rational function $f\in k(Y)$, we denote by $\rm{ord}_{y}(f) \in \ZZ$ the order of vanishing of $f$ at the point $y$. 

\begin{lemma}
\label{valuation-order-comparison}
In the situation of \cref{mun-torsor-further-specialization},
let $y \in \overline{\{z\}} \subset \abs{X^c}$ be a specialization of $z$.
Then for any $\ov{f}\in k(y)^\times/n$, we have
\begin{equation}\label{equation:specializtion-reduction}
    \bigl(\# \circ v_y\bigr)\bigl(\ov{f}\bigr) \equiv \ord_{\mu_\zeta(y)}\bigl(\sp_z(\ov{f})\bigr) \mod n
\end{equation}
under the identification $\O_{\cX, \zeta} \xrightarrow{\sim} k(z)^+$ from \cref{lemma:local-ring-at-generic-point-in-the-special-fiber}\cref{lemma:local-ring-at-generic-point-in-the-special-fiber-2} and the correspondence $\mu_\zeta$ from \cref{specialization}. In particular, for any $(L,s) \in \Hh^1(X^c,\mu_n)$, we have \[ \bigl(\# \circ v_y\bigr)\bigl(\rho_y(L,s)\bigr) \equiv \ord_{\mu_\zeta(y)}\bigl(\sp_z(L,s)\bigr) \mod n \] 
\end{lemma}
\begin{proof}
We choose a lift $f\in k(y)^\times$ of $\ov{f}$. Throughout this proof, we will freely identify $f\in k(y)$ with its image under the (injective) map $k(y) \to k(z)$. \cref{mun-torsor-further-specialization} ensures that there is an element $c \in C^{\times}$ such that $c \cdot f \in (k(z)^+)^{\times}$. Since $C^\times$ is $n$-divisible, we conclude that
\[
\sp_z\bigl(\ov{c\cdot f}\bigr) = \sp_z\bigl(\ov{f}\bigr) \in k(\zeta)^\times/n.
\]
We note that also $(\# \circ v_y)(c) = 0$ for any $c\in C^\times$, so we may and do replace $f$ with $c \cdot f$ to assume that $f\in k(y)^\times \cap (k(z)^+)^{\times}$. In this case, \cref{specialization}\cref{vanishing-order-valuation} implies that $(\# \circ v_y)\bigl(\ov{f}\bigr) = \ord_{\mu_\zeta(y)}\bigl(\rm{sp}_z(\ov{f})\bigr)$.  

The ``in particular'' part now follows directly from \cref{equation:specializtion-reduction}, \cref{mun-torsor-class-specialization}, and \cref{notation:specialization-of-torsor}.
\end{proof}

Next, we show that $(\# \circ v_x)(\rho_x(L,s))$ vanishes for rank-$2$ points $x \in \abs{X}$ such that $\sp_\cX(x)$ is a smooth point of some admissible formal $\O_C$-model $\cX$ of $X$. 

\begin{lemma}\label{section-integral-model}
Let $\cX=\Spf R$ be a smooth formal $\cO_C$-scheme with irreducible special fiber.
Let $(L,s)$ be a $\mu_n$-torsor on $\cX_\eta = \Spa(R\bigl[\tfrac{1}{\varpi}\bigr], R)$ and let $\cL \in \Pic(R)$ such that $\cL\bigl[\tfrac{1}{\varpi}\bigr] \simeq L$.
Then there exists an isomorphism $\sigma \colon \O_{\cX} \xrightarrow{\sim} \cL^{\otimes n}$ such that $\sigma\bigl[\tfrac{1}{\varpi}\bigr] = c\cdot s$ for some $c \in C^\times$.
\end{lemma}
\begin{proof}
Let $v$ be the supremum semi-norm on $R\bigl[\tfrac{1}{\varpi}\bigr]$ \cite[\S~3.8]{BGR}.
Since $\Spf(R)$ is smooth and connected (thus irreducible), $v$ is a valuation on $R\bigl[\tfrac{1}{\varpi}\bigr]$ due to \cite[Prop.~6.2.3/5]{BGR};
this is exactly the unique rank-$1$ valuation corresponding to the generic point $\eta$
of $\cX_s$ under 
\cref{lemma:local-ring-at-generic-point-in-the-special-fiber}.
\cref{lemma:local-ring-at-generic-point-in-the-special-fiber-0} (see the proof of \cref{lemma:local-ring-at-generic-point-in-the-special-fiber}\cref{lemma:local-ring-at-generic-point-in-the-special-fiber-4} for the justification). Then \cref{lemma:local-ring-at-generic-point-in-the-special-fiber}\cref{lemma:local-ring-at-generic-point-in-the-special-fiber-4} implies that there is a scalar $c\in C^\times$ such that $v(\rho_v(L, c\cdot s)) = 1$. Now, we claim that $c\cdot s$ can be extended to an isomorphism $\sigma$.

We note that if an extension $\sigma$ exists, it is unique.
We may therefore localize on $\Spf R$ and assume that $\cL \simeq \O_{\cX}$.
In this case, $s$ just corresponds to an element of $R\bigl[\tfrac{1}{\varpi}\bigr]^\times$ and we wish to show that $c\cdot s$ lies in $R^\times$. By our assumption on $s$ and $c$, we have $v(c\cdot s)=1$. Now using \cite[Prop.~3.4.1]{Lut16} and \cite[Prop.~6.2.3/1]{BGR}, we conclude that $R = \{ r \in R\bigl[\tfrac{1}{\varpi}\bigr] \suchthat v(r) \le 1 \}$. Therefore, using multiplicativity of $v$, we conclude that $v(c\cdot s)=1$ implies that $c\cdot s\in R^\times$ finishing the proof.
\end{proof}

We are ready to prove the promised above vanishing:

\begin{proposition}
\label{smooth-value}
Let $\cX$ be an admissible formal $\O_C$-scheme such that the special fiber $\cX_s$ is a reduced separated scheme of pure dimension $1$. 
Let $\zeta$ be a generic point of $\cX_s$ corresponding to $z \in X$ via \cref{specialization}.
Let $y \in \ov{\{z\}}$ be a point given by a rank-$2$ valuation $v_y$.
Assume that  $\sp_{\cX}(y)\in \cX_s$ is a smooth point. 
Then for any $(L, s) \in \Hh^1(X^c,\mu_n)$, we have $\bigl(\# \circ v_y\bigr)\bigl(\rho_y(L,s)\bigr) = 0$.
\end{proposition}
\begin{proof}
    The question is Zariski-local on $\cX$, so we may and do assume that $\cX = \Spf R$
    is smooth affine formal $\mathcal{O}_C$-scheme with irreducible special fiber.
    By \cref{formal version:line bundle has integral model lemma}, 
    the line bundle $L$ can then be extended to a line bundle $\cL$ on $\cX$.
    \Cref{section-integral-model} guarantees that $s$ can be extended to an isomorphism $\sigma \colon \O_{\cX} \xrightarrow{\sim} \cL^{\otimes n}$ after scaling by some $c \in C^\times$. Thus, for the purpose of showing that $\bigl(\# \circ v_y\bigr)\bigl(\rho_y(L,s)\bigr)=0$, we can replace $s$ with $c\cdot s$ to assume that $s$ extends to an isomorphism $\sigma\colon \O_\cX \xr{\sim} \cal{L}^{\otimes n}$.
    
    Over the local ring $\cO_{\cX,\spec(y)}$, we choose a trivialization $a \colon \cO_{\cX,\spec(y)} \xrightarrow{\sim} \restr{\cL}{\cO_{\cX,\spec(y)}}$ and consider $a^{-n} \circ \sigma_x$ as an element of $\O_{\cX, \spec(y)}^\times$.
    Using \cref{valuation-order-comparison}, we can identify $(\# \circ v_x)(\rho_x(L,s))$ with $\ord_{\mu_\zeta(y)}\bigl(a^{-n} \circ \sigma_x\bigr)$.
    Since $a^{-n} \circ \sigma_x \in \O_{\cX, \spec(y)}^\times$, we conclude that this valuation is zero.
\end{proof}
Next, we prove a vanishing statement for the nodes in the special fiber.
\begin{proposition}\label{node-value}
    Let $\cX$ be a separated semistable formal $\cO_C$-curve; see
    \cref{defn:ss-formal}.
    Let $q \in \cX_s$ be a node and $q_1,q_2 \in \cX^n_s$ be the two points in the normalization lying above $q$.
    Let $y_1$ and $y_2$ be the points of $\abs{X}$ with rank-$2$ valuation $v_{y_1}$ and $v_{y_2}$ which correspond to $q_1$ and $q_2$ under \cref{specialization}, respectively.
    Then $(\# \circ v_{y_1})(\rho_{y_1}(L,s)) + (\# \circ v_{y_2})(\rho_{y_2}(L,s)) = 0$.
\end{proposition}
The proof relies on the following statement:
\begin{proposition}\label{vanishing-regular-element}
    Fix an element $\pi \in \mathfrak{m}_C \smallsetminus \{0\}$ and set $\widetilde{R} \coloneqq \frac{\O_C[\![S, T]\!]}{(ST-\pi)}$. Let $f \in \widetilde{R}\bigl[\tfrac{1}{\varpi}\bigr]^\times$ such that both $f$ and $f^{-1}$ are regular in the sense of \cref{regular-elements}\cref{regular-elements-forreal}. Then 
    \[
    v_S(f) + v_T(f) = 0,
    \]
    where $v_S$ and $v_T$ are defined in \cref{valuation-regular-element}\cref{valuation-regular-element-2}. 
\end{proposition}
\begin{proof}
    The multiplicativity of both $v_S$ and $v_T$ (see \cref{multiplicativity of the v_S and v_T}) implies that it suffices to show that $v_S(f) +v_T(f)\geq 0$ for any non-zero regular $f\in \widetilde{R}\bigl[\tfrac{1}{\varpi}\bigr]$. 

    If $f= S^n$, we see that $v_S(S^r) + v_T(S^r) = v_S(S^r) + v_T(\tfrac{\pi^r}{T^r}) = r - r =0$. Therefore, we may replace $f$ with $f \cdot S^r$ for any integer $r$. Since $v_S(S)=1$, we can use the above observation to reduce to the case when $v_S(f)=0$. Thus, we conclude that the $S$-adic expansion $f = \sum_{i\in \ZZ} a_i S^i$ inside $\w{R}\bigl[\tfrac{1}{\varpi}\bigr]$ satisfies the property that $\abs{a_i} < \abs{a_0}$ if $i<0$ and $\abs{a_i} \leq \abs{a_0}$ if $i>0$.

    Now if we look at the $T$-adic expansion of $f$, we have $f = \sum_{i \in \ZZ} a_{-i} \pi^{-i} T^i$. Thus, the coefficients of negative powers of $T$ are of the form $a_{> 0} \cdot (\text{positive powers of } \pi)$. In particular, their norm is strictly less than $\abs{a_0}$. Hense, the very definition of $v_T(f)$ implies that $v_T(f)\geq 0$. 
\end{proof}

\begin{proof}[Proof of \cref{node-value}]
\cref{lemma:local-trivialization-semistable-curves} and the definition of semi-stable formal $\O_C$-curves imply that we can find a diagram of pointed admissible formal $\O_C$-schemes
\[
\begin{tikzcd}
& (\cU, u) \arrow[ld, swap, "h"]\arrow[rd, "g"] & \\
(\cX, q) && (\cY_\pi, 0) \coloneqq \Bigl(\Spf\Bigl( \frac{\mathcal{O}_C\langle S, T \rangle)}{(ST - \pi)}\Bigr), (0,0)\Bigr)
\end{tikzcd}
\]
such that $g$ and $h$ are \'etale, $\pi\in \m_C \smallsetminus \{0\}$, and there is a trivialization $a \colon \mathcal{O}_X \xrightarrow{\sim} L$. 

Let $U\coloneqq \cU_\eta$ be the rigid generic fiber of $\cU$, let $u_i$ be the unique (due to the \'{e}taleness) points of $\cU_s^n$ lying above $q_i$, and let $w_i \in \abs{U}$ be the rank-$2$ points corresponding to $u_i$ under \cref{specialization}, respectively. Using \cref{valuation-order-comparison} and \cref{specialization}\cref{vanishing-order-valuation}, we see that
\[
(\# \circ v_{y_i})(\rho_{y_i}(L,s)) = (\# \circ v_{w_i})\bigl(\rho_{w_i}\bigl(\restr{(L,s)}{U}\bigr)\bigr).
\]
Hence we may replace $(\cX, q)$ by $(\cU, u)$
and assume that our pointed semistable formal $\cO_C$-curve
admits an \'{e}tale map $g \colon (\cX, q) \to (\cY_\pi, 0)$ and there is a chosen trivialization $a\colon \mathcal{O}_X \xrightarrow{\sim} L$. By shrinking $(\cX, q)$ we may assume that $\cX = \Spf A_0$ is affine. We set $f \coloneqq a^{-n} \circ s \in (A_0\bigl[\tfrac{1}{\varpi}\bigr])^{\times}$. %
Then it suffices to show that 
\[
(\# \circ v_{y_1})(f) + (\# \circ v_{y_2})(f) = 0.
\]
Now we note that the map $g$ induces an isomorphism 
\[
R \coloneqq \Bigl(\bigl(\frac{\cO_C \langle S,T \rangle}{(ST - \pi)}\bigr)^\h_{(S, T)}\Bigr)^{\wedge}_{\varpi} \xr{\sim} \bigl(\O_{\cX, q}^\h\bigr)^{\wedge}_{\varpi}.
\]
Following \cref{formal-model-node-notation}, we have the natural morphism $A_0\bigl[\tfrac{1}{\varpi}\bigr]=\O_{\cX}(\cX)\bigl[\tfrac{1}{\varpi}\bigr] \to \bigl(\O_{\cX, q}^\h\bigr)^{\wedge}_{\varpi}\bigl[\tfrac{1}{\varpi}\bigr] \xr{\sim} R\bigl[\tfrac{1}{\varpi}\bigr] \to \w{R}\bigl[\tfrac{1}{\varpi}\bigr]  \coloneqq R^{\wedge}_{(S,T,\varpi)}\bigl[\tfrac{1}{\varpi}\bigr]$. We denote by $\widetilde{f}$ the image of $f$ in $\w{R}\bigl[\tfrac{1}{\varpi}\bigr]$. Then \cref{laurent-valuations} ensures that $\#\circ v_{y_1}(f)= v_S(\widetilde{f})$ and $\#\circ v_{y_2}= v_T(\widetilde{f})$. Thus, we reduced the question to showing that 
\[
v_S(\widetilde{f}) + v_T(\widetilde{f}) =0.
\]
This follows immediately from \cref{lemma:image under completion is regular} and \cref{vanishing-regular-element}. 
\end{proof}

\begin{proof}[{Second Proof of \cref{thm:trace-well-defined}}]
    Let $(L,s) \in \Hh^1(X,\mu_n)$.
    By \cref{useful proposition on rigid curves}\cref{useful proposition on rigid curves-5},
    we may choose a semistable $\cO_C$-formal model $\cX$ of $X$.
    Let $\cX^c_s$ be the compactification of $\cX_s$ such that $\cX_s \subset \cX^c_s$ is schematically dense and contains all the singular points of $\cX^c_s$, and let $\nu \colon \cX^{c,n}_s \to \cX^c_s$ be its normalization.
    Now \cref{specialization}\cref{specialization-comp} and the assumption that all singular points of $\cX^c$ are contained in $\cX_s$ imply that the specialization map induces a bijection $\abs{X^c} \smallsetminus \abs{X} \xrightarrow{\sim} \abs{\cX^{c, n}_s} \smallsetminus \abs{\cX_s^n}\xrightarrow{\sim} \abs{\cX^c_s} \smallsetminus \abs{\cX_s}$.
    For any $q \in \abs{\cX^c_s}$, let $\zeta_q$ be the generic point of the irreducible component containing $q$.
    By \cref{lemma:local-ring-at-generic-point-in-the-special-fiber}, the set $\spec_{\cX}^{-1}(\zeta_q) = \{z_q\}$ for some rank-$1$ point $z_q$.
    Then we have
    \begin{align*}
        \tilde{t}_X(L,s) \overset{\text{Lem.~\ref{specialization},}}{\underset{\text{Def.~\ref{analytic-pre-trace}}}{\equiv}} & \sum_{x_i \in \abs{X^c} \smallsetminus \abs{X}} (\# \circ v_i)\bigl(\rho_{x_i}(L,s)\bigr) \\
        \overset{\text{Lem.~\ref{valuation-order-comparison}, Prop.~\ref{smooth-value},}}{\underset{\text{Prop.~\ref{node-value}}}{\equiv}} & \sum_{q \in \abs{\cX^c_s} \smallsetminus \abs{\cX_s}} \ord_q\bigl(\spec_{z_q}(L,s)\bigr) + \sum_{q \in \abs{\cX_s} \text{ smooth}} \ord_q\bigl(\spec_{z_q}(L,s)\bigr) \\
        &+ \sum_{\substack{q \in \abs{\cX_s} \text{ node}\\ \text{with } \nu^{-1}(q) = \{q_1,q_2\}}} \ord_{q_1}\bigl(\spec_{z_q}(L,s)\bigr) + \ord_{q_2}\bigl(\spec_{z_q}(L,s)\bigr) \\
        \overset{\text{combining}}{\underset{\text{terms}}{\equiv}} & \sum_{\substack{Y \subseteq \cX^{c,n}_s \\ \text{connected component}}} \sum_{q \in \abs{Y} \text{ closed}} \ord_q\bigl(\spec_{z_q}(L,s)\bigr) \mod n.
    \end{align*}
    
    Now fix a connected component $Y_\zeta \subseteq \cX^{c,n}_s$ with generic point $\zeta$.
    This is an algebraic curve over $k$ on which $\spec_{z_q}(L,s) \in k(\zeta)^\times/\bigl(k(\zeta)^\times)^n$ defines a principal divisor (mod $n$).
    Thus,
    \[ \sum_{q \in \abs{Y} \text{ closed}} \ord_q\bigl(\spec_{z_q}(L,s)\bigr) \equiv \deg \bigl(\rm{Div}(\spec_{z_q}(L,s))\bigr) \equiv 0 \mod n \]
    and we win.
\end{proof}

\subsection{Compatibility with the algebraic trace map}\label{algebraic-analytic-trace-comparison}

The main goal of this subsection is to formulate the statement that the analytic trace map constructed in \cref{defn:analytic-trace-curves} is compatible with the algebraic one.
Its proof is the content of the next two subsections.
We begin by recalling the algebraic trace map.
\begin{definition}\label{defn:algebraic-trace-map}
Let $\overline{X}$ be a smooth proper rigid-analytic curve over $C$ and let $\overline{X}^{\rm{alg}}$ be its unique algebraization (see \cref{useful proposition on rigid curves}\cref{useful proposition on rigid curves-2}).
The \textit{algebraic trace map} on $\overline{X}$ is the homomorphism 
\[
t^{\rm{alg}}_{\overline{X}} \colon \rm{H}^2\bigl(\overline{X}, \mu_n\bigr) \simeq \rm{H}^2\bigl(\overline{X}^{\rm{alg}}, \mu_n\bigr) \xrightarrow{t_{\overline{X}^{\rm{alg}}}} \Z/n\Z
\]
which is obtained as the composition of the (inverse of the) isomorphism from \cref{useful proposition on rigid curves}\cref{useful proposition on rigid curves-6} and the schematic trace map (see \cite[Exp.~XVIII, Th.~2.9]{SGA4}). 
\end{definition}
Then the main statement about the compatibility of analytic and algebraic traces is the following: 
\begin{theorem}\label{thm:analytic-trace-algebraic}
Let $\overline{X}$ be a smooth proper rigid-analytic curve over $C$ and let $j \colon X \hookrightarrow \overline{X}$ be a quasi-compact open affinoid subspace.
Then the diagram
\[
\begin{tikzcd}
\rm{H}^2_c(X, \mu_n) \arrow{rr}{\Hh^2_c(\ttr^\et_j(1))} \ar{rd}{t_X} & & \rm{H}^2(\overline{X}, \mu_n) \arrow{ld}[swap]{t_{\overline{X}}^{\rm{alg}}} \\
& \ZZ/n\ZZ &
\end{tikzcd}
\]
commutes, where $t_X$ denotes the analytic trace from \cref{defn:analytic-trace-curves} and $\ttr^\et_j(1)$ is the \'etale trace from \cref{etale-trace} and \cref{notation:etale-trace}.
\end{theorem}
Before we start the proof of \cref{thm:analytic-trace-algebraic}, we record an application of the statement.
\begin{corollary}\label{curve-trace-etale-compatibility}
    Let $f \colon X \to Y$ be an \'etale morphism of smooth affinoid curves over $C$.
    Then the diagram
    \[ \begin{tikzcd}
        \Hh^2_c(X,\mu_n) \arrow[r,"t_X"] \arrow[d,"\Hh^2_c(\ttr^\et_{f}(1))"'] & \ZZ/n \\
        \Hh^2_c(Y,\mu_n) \arrow[ru,"t_Y"'] &
    \end{tikzcd} \]
    commutes, where $t_X$ and $t_Y$ denote the analytic traces from \cref{defn:analytic-trace-curves} and $\ttr^\et_f(1)$ is the \'etale trace from \cref{etale-trace} and \cref{notation:etale-trace}.
\end{corollary}
\begin{proof}
    First, when $f$ is finite \'etale, the statement follows from \cref{thm:analytic-trace-compatible-finite-flat-trace} and property \cref{flat-trace-etale} in \cref{thm:flat-trace}.
    Next, we contemplate the case where $f$ is an open immersion.
    By \cref{useful proposition on rigid curves}\cref{useful proposition on rigid curves-3}, we may choose an open immersion $g \colon Y \hookrightarrow Z$ into a smooth proper curve $Z$ over $C$.
    Consider the following diagram:
    \[ \begin{tikzcd}[column sep = 5em]
        \Hh^2_c(X,\mu_n) \arrow[r,"\Hh^2_c(\ttr^\et_f(1))"] \arrow[rrrd,bend right=10,"t_X"] & \Hh^2_c(Y,\mu_n) \arrow[r,"\Hh^2_c(\ttr^\et_g(1))"] \arrow[rrd,bend right=5,"t_Y"] & \Hh^2_c(Z,\mu_n) \arrow[rd,"t^\alg_Z"] & \\
        &&& \ZZ/n
    \end{tikzcd} \]
    Since both $X$ and $Y$ are quasicompact open affinoid subspaces of $Z$ via $g \circ f$ and $g$, respectively, and $\Hh^2_c\bigl(\ttr^\et_g(1)\bigr) \circ \Hh^2_c\bigl(\ttr^\et_f(1)\bigr) = \Hh^2_c\bigl(\ttr^\et_{g \circ f}(1)\bigr)$, \cref{thm:analytic-trace-algebraic} gives
    \[ t_Y \circ \Hh^2_c(\ttr^\et_f) = t^\alg_Z \circ \Hh^2_c(\ttr^\et_g(1)) \circ \Hh^2_c(\ttr^\et_f(1)) = t_X. \]

    Now we can treat the general case.
    By \cite[Lem.~2.2.8]{Huber-etale} (cf.\ also \cite[Prop.~3.1.4]{dJ-vdP}), we can choose a finite open affinoid cover $Y = \bigcup_{j \in J} V_j$ and factorizations
    \[ \begin{tikzcd}[column sep=tiny]
        f^{-1}(V_j) \arrow[rd,"f"'] \arrow[rr,"h_j"] && X_j \arrow[ld,"g_j"] \\
        & V_j &
    \end{tikzcd} \]
    such that the $h_j$ are open immersions and the $g_j$ are finite \'etale.
    They give rise to the diagram
    \[ \begin{tikzcd}
        \bigoplus_{j \in J} \Hh^2_c(f^{-1}(V_j),\mu_n) \arrow[r] \arrow[rd,start anchor=-10] & \Hh^2_c(X,\mu_n) \arrow[r,"t_X"] \arrow[d] & \ZZ/n \\
        & \Hh^2_c(Y,\mu_n) \arrow[ru,"t_Y"] &
    \end{tikzcd} \]
    in which all unlabeled arrows are the \'etale traces given by the respective counits of adjunction.
    The compatiblity of adjunction counits under composition then guarantees that the left triangle commutes.
    Moreover, the top horizontal composition is given by $\bigoplus_j t_{f^{-1}(V_j)}$ thanks to the already established case of open immersions.
    
    The natural map $\bigoplus_{j \in J} i_{j,!} \mu_n \to \mu_n$ induced by the open cover $\{ i_j \colon f^{-1}(V_j) \hookrightarrow X \}_{j \in J}$ is an epimorphism.
    Since $\Hh^2_c(f^{-1}(V_j),\blank) \simeq \Hh^2_c\bigl(X,i_{j,!}(\blank)\bigr)$ and $\Hh^2_c(X,\blank)$ is right exact \cite[Prop.~5.5.6, Prop.~5.5.8]{Huber-etale}, the top left horizontal arrow in the diagram above is then also an epimorphism.
    To prove that the right triangle commutes, it therefore suffices to show that the outer triangle commutes, which can be checked on each factor $\Hh^2_c(f^{-1}(V_j),\mu_n)$ separately.
    Since $f^{-1}(V_j) \to Y$ factors into a composition of open immersions and finite \'etale morphisms, this follows from the first paragraph.
\end{proof}

\subsection{Compatibility with the algebraic trace map: closed unit disk}\label{section:disc-p1-compatibility}

The main goal of this subsection is to prove \cref{thm:analytic-trace-algebraic} in the case of the standard open immersion $\bf{D}^1 \hookrightarrow \bf{P}^{1, \an}$. We recall that \cref{comp-supp-cohomology-affine-curve} gives some control over $\rm{H}^i_c(\bf{D}^1, \mu_n)$;
however, the main drawback of this description is that it seems difficult to relate $\rm{H}^2_c(\bf{D}^1, \mu_n)$ to the cohomology of $\bf{P}^{1, \an}$. 

For this reason, we take a different approach in this subsection and relate the compactly supported cohomology of $\bf{D}^1$ and the cohomology of $\bf{P}^{1,\an}$ directly.
An explicit understanding of their difference will also be the key input in our proof of \cref{thm:analytic-trace-algebraic}.

\subsubsection{Preliminaries. The ring $A(Z)$.}

To start the proof, we denote by $j\colon \bf{D}^1 \hookrightarrow \bf{P}^{1, \an}$ the usual open immersion and by $Z\coloneqq \abs{\bf{P}^{1, \an}}\smallsetminus \abs{\bf{D}^1}$ the closed complement of $\bf{D}^1$ inside $\bf{P}^{1, \an}$.
This space does \emph{not} admit any structure of an analytic adic space; 
instead, we consider $Z$ as a pseudo-adic space $(\bf{P}^{1, \an}, Z)$ (see \cref{section:pseudo-adic}) which we will, by abuse of notation, simply call $Z$.
\Cref{rmk:excision} implies that we have an exact triangle
\begin{equation}\label{eqn:disc-vs-p1}
\rR\Gamma_c(\bf{D}^1, \mu_n) \to \rR\Gamma(\bf{P}^{1, \an}, \mu_n) \to \rR\Gamma(Z, \mu_n).
\end{equation}
To put it plainly, the difference between $\rR\Gamma_c(\bf{D}^1, \mu_n)$ and $\rR\Gamma(\bf{P}^{1, \an}, \mu_n)$ is exactly controlled by the \'etale cohomology of $Z$.

To study these cohomology groups, we need to study the geometry of $j$ in more detail. For this, we view $\bf{P}^{1, \an}$ as the glueing of two closed unit disks\footnote{In the formulas below, we endow $\O_C[T]$ and $\O_C[S]$ with the $\varpi$-adic topology.}
\[
\bf{D}^1(0) = \Spa(C[T], \cO_C[T]) \text{ and } \bf{D}^1(\infty) = \Spa(C[S], \cO_C[S])
\]
along the torus 
\[
\Spa(C[T^{\pm 1}], \cO_C[T^{\pm 1}]) \simeq \Spa(C[S^{\pm 1}], \cO_C[S^{\pm 1}])
\]
via $T = S^{-1}$. Inclusion $j$ just becomes the inclusion $\bf{D}^1(0)\hookrightarrow \bf{P}^{1, \an}$. The complement of this inclusion is a special closed subset (see \cref{example:main-examples}) 
\[
Z=\bf{D}^1(\infty)\bigl(\abs{S}<1\bigr) \subset \bf{D}^1(\infty).
\]

Now, when we realize $Z$ as a special closed subset inside an affinoid $\bf{D}^1(\infty)$, \cref{thm:cohomology-prospecial-subsets} (see also \cref{defn:henselize-prospecial}) and \cite[Cor.~2.3.8]{Huber-etale} ensure that there are canonical isomorphisms 
\[
\rR\Gamma(Z, \mu_n) \simeq \rR\Gamma((\bf{D}^1(\infty), Z), \mu_n) \simeq \rR\Gamma(\Spec A(Z), \mu_n),
\]
where $A(Z)=\O_C[S]_{(\varpi, S)}^\h\bigl[\frac{1}{\varpi}\bigr]$. For the notational convenience, we introduce the following notation:

\begin{notation}\label{notation:ring-AZ} We put $A(Z)^+ \coloneqq \O_C[S]_{(\varpi, S)}^\h$ and $A(Z)\coloneqq A(Z)^+\bigl[\frac{1}{\varpi}\bigr]$.
\end{notation}

So we reduce the question of studying cohomology of $Z$ to the question of understanding alegbraic cohomology $\rm{R}\Gamma(\Spec A(Z), \mu_n)$. For this, we start we establishing certain algebraic properties of the ring $Z$. We stary by studying the Picard group of $A(Z)$: %

\begin{proposition}
\label{A(Z) has no Picard}
Let $A(Z)$ be as above. Then $\rm{Pic}(A(Z))=0$.
\end{proposition}

\begin{proof}
We recall that the ring $A(Z)^+=\cO_C[S]^{\h}_{(\varpi, S)}$ is a filtered colimit of rings which are \'{e}tale over $\cO_C[S]$.
Since any rank-$1$ projective module $L$ on $A(Z)=A(Z)^+\bigl[\tfrac{1}{\varpi}\bigr]$ must come from a rank-$1$ projective module on the $\bigl[\tfrac{1}{\varpi}\bigr]$-fibre of one of these algebras in the filtered colimit,
we summon \cref{formal version:line bundle has integral model lemma}
to see that $L$ must be the base change of a rank-$1$ projective module
on $A(Z)^+$. So it suffices to show that $\rm{Pic}\left(A(Z)^+\right)=0$. Now \cite[\href{https://stacks.math.columbia.edu/tag/0F0L}{Tag 0F0L}]{stacks-project} implies that $A(Z)^+=\cO_C[S]^{\h}_{(\varpi, S)} \simeq \cO_C[S]^{\h}_{(\fm_C, S)}$. But the latter algebra is local, hence its Picard group vanishes.
\end{proof}

Our next goal is to understand the unit group $A(Z)^\times$. For this, we will need some preliminary lemmas: 

\begin{lemma}\label{lemma:A(z)-separated} The ring $A(Z)^+$ is $\varpi$-adically separated, i.e. $\bigcap_{n\geq 1} \varpi^n A(Z)^+= \{0\}$.
In particular, every nonzero element $f \in A(Z)$ can be scaled by a power of
$\varpi$ so that it lies in $A(Z)^+ \smallsetminus \varpi \cdot A(Z)^+$.
\end{lemma}
\begin{proof}
    We note that \cite[\href{https://stacks.math.columbia.edu/tag/0F0L}{Tag 0F0L}]{stacks-project} ensures that $A(Z)^+$ is isomorphic to the $(\fm_C, S)$-adic henselization of $\O_C[S]$. In particular, $A(Z)^+$ is a local ring and can be written as a filtered colimit $A(Z)^+=\colim_{i\in I} B_i$ such that each $B_i$ is the localization of an \'etale $\O_C[S]_{(\fm_C, S)}$-algebra at a maximal ideal lying over the maximal ideal of $\O_C[S]_{(\fm_C, S)}$. 
    
    Suppose $0\neq f\in \cap_{n\geq 1} \varpi^n A(Z)^+$, then $f$ comes from an element $0\neq f_i\in B_i$ for some $i\in I$. Since $B_i \to A(Z)^+$ is faithfully flat, we conclude that $\varpi^n A(Z)^+\cap B_i = \varpi^n B_i$. Therefore, $0\neq f_i \in \cap_{n\geq 1} \varpi^n B_i$. So it suffices to show that each $B_i$ is $\varpi$-adically separated. We pick one and rename it as $B$. 
    
    Now $B$ is a localization of an \'etale $\O_C[S]$-algebra, so $B\bigl[\frac{1}{\varpi}\bigr]$ is noetherian. 
    Therefore, \cite[Cor.~0.9.2.7 and Prop.~0.8.5.10]{FujKato} imply that $B$ is $\varpi$-adically adhesive (see \cite[Def.~0.8.5.1]{FujKato}). We note that $J\coloneqq \cap_{n\geq 1} \varpi^n B$ is a saturated ideal of $B$ (because $\varpi$ is a non-zero divisor in $B$). Thus \cite[Prop.~0.8.5.3(c)]{FujKato} implies that $J$ is finitely generated. Since also $J = \varpi \cdot J$ and $\varpi$ lies inside the maximal ideal of $B_i$, Nakayama's lemma \cite[\href{https://stacks.math.columbia.edu/tag/00DV}{Tag 00DV}]{stacks-project} ensures that $J=0$. 

    For the last sentence, first let us scale $f$ by multiples of $\varpi$ so that
    it lies in $A(Z)^+$. Then $\rm{max} \bigl\{ n \suchthat f \in \varpi^n \cdot A(Z)^+ \bigr\}$ exists because $A(Z)^+$ is $\varpi$-adically separated. Denote this number by $n_0$, then $\varpi^{-n_0} f$ does the job.
\end{proof}

The following lemma is certainly well-known to the experts, however, it seems difficult to find a reference in the existing literature. For this reason, we spell out the proof below: 

\begin{lemma}
\label{k_C to O_C/pi}
Let $\pi \in \cO_C$ be a pseudo-uniformizer such that
$\cO_C/(\pi)$ shares the same characteristic\footnote{This condition is only relevant in the
mixed characteristic situation, in which case we are just saying $(p) \subset (\pi)$.} 
as the residue field $k_C$.
Then the natural surjection $\rho \colon \cO_C/(\pi) \to k_C$ admits a section.
\end{lemma}

\begin{proof}
Let $\mathbf{F} \subset k_C$ be the prime field, \cite[\href{https://stacks.math.columbia.edu/tag/030F}{Tag 030F}]{stacks-project} implies that we can choose a set of transcendental
basis $\{x_i\}_{i \in I}$, so $\mathbf{F}(\underline{x}) \subset k_C$ is an algebraic extension.
Let $A$ be the perfection\footnote{If $\charac k_C=p>0$, then $A=\bf{F}[\ud{x}^{1/p^\infty}]$. If $\charac k_c = 0$, then we put $A=\bf{F}[\ud{x}]$.} of $\mathbf{F}[\underline{x}]$ which can be realized as a $\mathbf{F}[\underline{x}]$-subalgebra inside $k_C$. This induces a further inclusion $\Frac(A) \subset k_C$ of $A$-algebras. 

We choose some lifts $\widetilde{x}_i\in \O_C/(\pi)$ of $x_i \in k_C$. This defines a morphism $\alpha \colon \bf{F}[\ud{x}] \to \O_C/(\pi)$ such that $\rho \circ \alpha$ is equal to the natural inclusion $\bf{F}[\ud{x}]\hookrightarrow k_C$. Now we wish to construct morphisms $\beta$, $\gamma$, and $\delta$ such that the  diagram
\[ \begin{tikzcd} 
    &&& \cO_C/(\pi) \arrow[d,two heads,"\rho"'] \\
    \FF[\underline{x}] \arrow[rrru,start anchor=north east,end anchor=west,"\alpha"] \arrow[r,hook] & A \arrow[rru,dotted,start anchor=north east,"\beta" description] \arrow[r,hook] & \Frac(A) \arrow[ru,dotted,start anchor=north east,"\gamma" description] \arrow[r,hook] & k_C \arrow[u,bend right,dotted,"\delta"']
\end{tikzcd} \]
commutes and $\rho \circ \delta = \rm{id}$. We do this step by step.

First, we extend $\alpha$ to a morphism $\beta$. This is only an issue when $k_C$ has characteristic $p$, in which case $A=\bf{F}[\ud{x}^{1/p^\infty}]$ and $\O_C/(\pi)$ is a semi-perfect algebra. Therefore, we can choose compatible $p$-power roots $\widetilde{x}_{i,r}$ of $\widetilde{x}_i$ and let $\beta$ to be the unique ring homomorphism that sends $x_i^{1/p^r}$ to $\widetilde{x}_{i,r}$. Then, in order to extend $\beta$ to $\gamma$, we need to check that $\beta(a) \in \bigl(\O_C/(\pi)\bigr)^\times$ for each nonzero $a \in A$.
This follows from the observation that the kernel of $\rho$ is locally nilpotent and $\rho \circ \beta$ is the natural inclusion
$A \hookrightarrow k_C$. Thus, $\beta$ admits a unique extension $\gamma$.

Finally, we construct $\delta$. For this, we notice that $\Frac(A)$ is a perfect field, so the algebraic extension
$\Frac(A) \subset k_C$ is ind-\'{e}tale.
Now the maximal ideal in $\cO_C/(\pi)$ is locally nilpotent,
hence by \cite[\href{https://stacks.math.columbia.edu/tag/0ALI}{Tag 0ALI}]{stacks-project}
we know it is a henselian local ring.
Finally applying \cite[\href{https://stacks.math.columbia.edu/tag/08HR}{Tag 08HR}]{stacks-project}
with the $(R \to S, R \to A)$ there being our $(\Frac(A) \to \cO_C/(\pi), \Frac(A) \subset k_C)$ here,
we see that the section $\delta$ exists and uniquely depends on $\gamma$.
\end{proof}

Finally, we are ready to get the desired control over the units in $A(Z)$, in analogy with \cref{invertible functions on closed disk}:

\begin{lemma}\label{lemma:A(z)-units} We have an equality
\[
A(Z)^\times = C^\times \cdot A(Z)^{+, \times}.
\]
\end{lemma}

\begin{proof}
    We first choose a pseudo-uniformizer $\pi\in \O_C$ such that $\pi \mid p$ if $\O_C$ is of mixed characteristic $(0, p)$.  

    Now we claim that $\rm{min}_{c \in C} \bigl\{ \abs{c} \suchthat f/c\in A(Z)^+ \bigr\}$ exists for any nonzero $f \in A(Z)$. First, \cref{lemma:A(z)-separated} implies that we may replace $f$ by $f/\pi^N$ for some $N$ to assume that $f \in A(Z)^+ \smallsetminus \pi A(Z)^+$. 
    Therefore in order to show that the desired minimum exists, it suffices to show that $A(Z)^+/(\pi)$ is a free $\O_C/(\pi)$-modules.  
    
    For this, we choose a section $k_C \to \cO_C/(\pi)$ of the natural projection $\O_C/(\pi) \to k_C$,
    which exists due to \cref{k_C to O_C/pi}. 
    Since the maximal ideal of $\O_C/(\pi)$ is locally nilpotent, 
    we conclude that the section $k_C \to \O_C/(\pi)$ is integral. Therefore, \cite[\href{https://stacks.math.columbia.edu/tag/0DYE}{Tag 0DYE}]{stacks-project} implies that 
    \[
    A(Z)^+/(\pi) = \O_C[S]^{\h}_{(\pi, S)}/(\pi) \simeq \bigl(\cO_C/(\pi)[S]\bigr)^{\h}_{(S)} 
    \simeq \cO_C/(\pi) \otimes_{k_C} k_C[S]^{\h}_{(S)}.
    \]
    Since any $k_C$-module is free, we conclude that $A(Z)^+/(\pi)$ is a free $\O_C/(\pi)$-module as well.   
    
    Now let us define a function $\abs{\blank}_\eta\colon A(Z) \to \Gamma_C\cup \{0\}$ by the rule
    \[
    \abs{f}_\eta=\rm{min}_{c\in C} \bigl\{ \abs{c} \suchthat f/c\in A(Z)^+ \bigr\}.
    \]
    A standard argument using that $A(Z)^{+}/\fm_C A(Z)^{+} \simeq k_C[S]_{(S)}^\h$ is a domain shows that $\abs{.}_\eta$ is multiplicative (see \cite[p.~13]{B}
    for a version of this argument). Now let $f\in A(Z)^\times$, we choose some $c\in C$ such that $\abs{c}=\abs{f}_\eta$. It suffices to show that $f'=f/c$ is a unit in $A(Z)^+$. By construction, $\abs{f'}_\eta=1$ and $f'$ is invertible in $A(Z)$. Thus, multiplicativity of $\abs{.}_\eta$ implies that $\abs{f'^{-1}}_\eta=1$ so $(f')^{-1} \in A(Z)^+$ finishing the proof.
\end{proof}

\begin{corollary}
\label{F_p-cohomology of 1 dimensional open disk}
We have
\[
 \rm{H}^i(Z, \mu_n) \simeq \Hh^i(\Spec A(Z), \mu_n)
\simeq \begin{cases}
  \mu_n(C) \cong \Z/n\Z  & i = 0 \\
  A(Z)^{\times}/(A(Z)^{\times})^n = A(Z)^{+, \times}/(A(Z)^{+, \times})^n
  & i = 1 \\
  0 & i \geq 2
\end{cases}.
\]
\end{corollary}
\begin{proof}
    The proof is essentially the same as that of \cref{cohomology-affine-curve}. One uses that $A(Z)$ is ind-\'etale over $C[S]$ and the Artin-Grothendieck vanishing theorem (see \cite[\href{https://stacks.math.columbia.edu/tag/0F0V}{Tag 0F0V}]{stacks-project}) to get vanishing in higher degrees. Then one uses the Kummer exact sequence and \cref{A(Z) has no Picard} to get the calculation in lower degrees. 
\end{proof}

\begin{corollary}
\label{second compactification proposition}
We have $\rm{H}^{i}_c(\bf{D}^1, \mu_n)= 0$ for $i\neq 2$ and a natural exact sequence
\[
0 \to
A(Z)^{\times}/(A(Z)^{\times})^n 
\to \rm{H}^{2}_{c}(\mathbf{D}^1, \mu_n)
\to \rm{H}^{2}(\mathbf{P}^{1, \an}, \mu_n)
\to 0
\]
\end{corollary}
\begin{proof}
    The first claim follows directly from \cref{comp-supp-cohomology-affine-curve} and \cref{cor:first-cohomology-affinoids}. The second claim follows directly from \cref{eqn:disc-vs-p1}, \cref{useful proposition on rigid curves}\cref{useful proposition on rigid curves-6}, and \cref{F_p-cohomology of 1 dimensional open disk}. 
\end{proof}

\subsubsection{Beginning of the proof}

In this subsubsection, we show that \cref{thm:analytic-trace-algebraic} holds for the open immersion $j\colon \bf{D}^1 \hookrightarrow \bf{P}^{1, \an}$ up to an invertible constant $\lambda\in \Z/n\Z^\times$. In the next subsubsection, we will show that this constant must be $1$ due to some cycle class considerations. 

We start the proof by relating the short exact sequence in \cref{second compactification proposition} to the one in \cref{comp-supp-cohomology-affine-curve}. For this, we recall that $Z$ can be realized as the closed subset of $\bf{D}^1(\infty)\bigl(\abs{S}<1\bigr) \subset \bf{D}^1(\infty)$ and we denote by $x_+$ the unique (rank-$2$) point of $\abs{\bf{D}^{1, c}(0)} \smallsetminus \abs{\bf{D}^1(0)}$ (see \cref{lemma:compacitification-of-the-disc}) . We now consider the following commutative diagram\footnote{Here, we implicitly use \cite[Lem.~4.2.5 and Prop.~4.2.11]{rigid-motives} that ensures that $\abs{\bf{D}^{1, c}}$ coincides with the topological closure of $\bf{D}^1$ inside $\bf{P}^{1, \an}$.} of pseudo-adic spaces:
\begin{equation}\label{eqn:different-compactifications}
\begin{tikzcd}
\Bigl(\Spa\bigl(\wdh{k(x_+)}, \wdh{k(x_+)}^+\bigr), x_+\Bigr) \arrow{r}\arrow{d}{\alpha} & \Bigl(\bf{D}^{1, c}(0), x_+\Bigr) \arrow{d}{\beta} \\
\Bigl(\bf{D}^1(\infty), Z\Bigr) \arrow{r}& \Bigl(\bf{P}^{1, \an}, Z\Bigr).
\end{tikzcd}
\end{equation}
By \cite[Prop.~2.3.7]{Huber-etale}, the horizontal morphisms are equivalences on the associated \'etale topoi, so they do not change cohomology.

Now the morphism $\alpha$ (due to \cref{example:different-henselizations}) induces the natural morphism $A(Z) \to \wdh{k(x_+)}^\h$ such that the image of $A(Z)^+$ lands inside $\wdh{k(x_+)}^{+, \h}$. After inverting $\varpi$, we denote the induced morphism by
\[
\rm{Res}\colon A(Z) \to \wdh{k(x_+)}^\h.
\]

\begin{proposition}
\label{Thm: cpt supp coh of disk}
There is a commutative diagram
between two natural exact sequences
\[ \begin{tikzcd}
& 0 \arrow[r] \arrow[d] & A(Z)^{\times}/(A(Z)^{\times})^n \arrow[r] \arrow[d,"\mathrm{Res}"] & \Hh^2_c(\mathbf{D}^1, \mu_n) \arrow[r] \arrow[d,equals] & \Hh^2(\mathbf{P}^{1, \an}, \mu_n) \arrow[r] \arrow[d] & 0 \\
0 \arrow[r] & C\langle T \rangle^{\times}/(C\langle T \rangle^{\times})^n  \arrow[r] & \wdh{k(x_+)}^{\h, \times}/\Bigl(\wdh{k(x_+)}^{\h, \times}\Bigr)^n \arrow[r] & \Hh^{2}_{c}(\mathbf{D}^1, \mu_n) \arrow[r] & 0. & 
\end{tikzcd} \]
\end{proposition}
\begin{proof}
The upper short exact sequence comes from \cref{second compactification proposition}. The lower short exact sequence comes from \cref{cohomology-affine-curve}, \cref{lemma:trivial-picard-disk}, \cref{comp-supp-cohomology-affine-curve}, and \cref{cor:first-cohomology-affinoids}. Now the diagram of \'etale topoi below 
\[
\begin{tikzcd}
\mathbf{D}^1_{ \et} \arrow{r} \arrow[d,equals] & \bf{D}^{1, c}_{\et} \ar{d}{\mathrm{incl}} &
\biggl[ \Bigl(\bf{D}^{1, c}, x_+\Bigr)_\et \arrow{l} \arrow{d}{\alpha}& \arrow[l, swap, "\sim"] \Bigl(\Spa\bigl(\wdh{k(x_+)}, \wdh{k(x_+)}^+\bigr), x_+\Bigr)_\et \arrow{d}{\beta} \biggr]\\
\mathbf{D}^1_{\et} \arrow{r} & \mathbf{P}^{1, \an}_{\et} & \biggl[\Bigl(\bf{P}^{1, \an}, Z\Bigr)_\et \arrow{l} & \Bigl(\bf{D}^1_2, Z\Bigr)_\et \arrow[l, swap, "\sim"]\biggr].
\end{tikzcd}
\]
and \cref{thm:cohomology-prospecial-subsets} gives rise to the following commutative diagram of exact triangles
\[ \begin{tikzcd}
\rR\Gamma_c(\mathbf{D}^1, \mu_n) \arrow[r] \arrow[d,equals] & \rR\Gamma(\mathbf{P}^{1, \an}, \mu_n) \arrow[r] \arrow[d,"\res"] & \rR\Gamma(\Spec A(Z), \mu_n)  \arrow[d,"\res"] \\
\rR\Gamma_c(\mathbf{D}^1, \mu_n) \arrow[r] & \rR\Gamma(\mathbf{D}^{1, c}, \mu_n) \arrow[r] & \rR\Gamma\Bigl(\Spec \wdh{k(x_+)}^\h, \mu_n\Bigr).
\end{tikzcd} \]
This yields the desired commutative diagram upon passing to cohomology and observing that $\rR\Gamma(\bf{D}^{1, c}, \mu_n) \simeq \rR\Gamma(\bf{D}^1, \mu_n)$ (see \cref{cohomology-affine-curve}).
\end{proof}

Now we are finally ready to start the proof of \cref{thm:analytic-trace-algebraic} for the open immersion $\bf{D}^1 \hookrightarrow \bf{P}^{1, \an}$. 

\begin{lemma}
\label{A(Z) has zero an tr}
The composition $A(Z)^\times \xrightarrow{\rm{Res}} \wdh{k(x_+)}^{\h, \times} \xrightarrow{\#\circ v_{x_+}} \Z$ is zero.
\end{lemma}
\begin{proof}
    By \cref{lemma:A(z)-units}, it suffices to show that the composition is zero on $C^\times$ and $A(Z)^{+, \times}$. The composition is zero on $C^\times$ by construction. To deal with the elements of $A(Z)^{+, \times}$, we observe that $\rm{Res}$ maps them to the elements in $\Bigl(\wdh{k(x_+)}^{+, \h}\Bigr)^\times$ (see the discussion before \cref{Thm: cpt supp coh of disk}). Therefore, the whole valuation $v_{x_+}$ vanishes on these elements. 
\end{proof}
 
\begin{corollary}
\label{traces on disk differ by a constant} Let $j\colon \bf{D}^1 \hookrightarrow \bf{P}^{1, \an}$ be the standard immersion. Then there is an invertible constant $\lambda \in (\Z/n\Z)^\times$ such that the diagram
\[
\begin{tikzcd}
\rm{H}^2_c(\bf{D}^1, \mu_n) \arrow{rr}{\Hh^2_c(\ttr^\et_j(1))} \ar{rd}{\lambda \cdot t_{\bf{D}^1}} & & \rm{H}^2(\bf{P}^{1, \an}, \mu_n) \arrow{ld}[swap]{t_{\bf{P}^{1,\an}}^{\rm{alg}}} \\
& \ZZ/n\ZZ &
\end{tikzcd}
\]
commutes. 
\end{corollary}
\begin{proof}
    First, we note that \cref{rmk:trace-surjective} ensures that $t_{\bf{D}^1}$ is surjective.
    Furthermore, \cref{Thm: cpt supp coh of disk} ensures that the \'etale trace $\Hh^2_c\bigl(\ttr^\et_j(1)\bigr)$ is surjective, while classical algebraic theory ensures that $t^{\rm{alg}}_{\bf{P}^{1, \an}}$ is an isomorphism.
    This implies that the composition $t^{\rm{alg}}_{\bf{P}^{1, \an}} \circ \Hh^2_c\bigl(\ttr^\et_j(1)\bigr)$ is also surjective.  

    Now \cref{Thm: cpt supp coh of disk} and \cref{A(Z) has zero an tr} imply that $t^{\rm{an}}_{\bf{D}^1}$ vanishes on the image of $A(Z)^\times$ in $\rm{H}^2_c(\bf{D}^1, \mu_n)$. Therefore, \cref{Thm: cpt supp coh of disk} ensures that the analytic trace map factors through the surjection
    \[
    \rm{H}^2_c(\bf{D}^1, \mu_n) \xrightarrow{\Hh^2_c(\ttr^\et_j(1))} \rm{H}^2(\bf{P}^{1, \an}, \mu_n) \simeq \Z/n\Z.
    \]
    Since both $t_{\bf{D}^1}$ and $t^{\rm{alg}}_{\bf{P}^{1, \an}} \circ \Hh^2_c\bigl(\ttr^\et_j(1)\bigr)$ are surjective and factor through $\rm{can}$, we formally conclude that they must differ by an element $\rm{Aut}(\Z/n\Z) = (\Z/n\Z)^\times$. This finishes the proof.
\end{proof}

\subsubsection{End of the proof}

In this subsubsection, we finally finish the proof of \cref{thm:analytic-trace-algebraic} in the case of the open immersion $j\colon \bf{D}^1 \hookrightarrow \bf{P}^{1, \an}$.  

We note that \cref{traces on disk differ by a constant} implies that the only thing we are left to do is to pin down the constant $\lambda$. 
This will be done via cycle class considerations. For this, we recall that, for each classical point $a\in \O_C = \bf{D}^1(C)$, we can attach the localized cycle class in $\rm{H}^2_{a}(\bf{D}^1, \mu_{n})$ and the compactly supported cycle class in $\rm{H}^{2}_c(\bf{D}^1, \mu_n)$ (see \cref{defn:cycle-clas-divisors} and \cref{defn:cycle-clas-divisors-compact} respectively). To clarify the exposition in this subsubsection, we denote the localized cycle class by $c\ell_{\bf{D}^1}^{\rm{loc}}(a)\in \rm{H}^2_{a}(\bf{D}^1, \mu_{n})$ and the compactly supported cycle class by $c\ell_{\bf{D}^1}(a) \in \rm{H}^{2}_c(\bf{D}^1, \mu_n)$; they are related via the natural map $\rm{H}^2_{a}(\bf{D}^1, \mu_{n}) \to \rm{H}^{2}_c(\bf{D}^1, \mu_n)$. 

In order to verify $\lambda=1$, we will show that $t_{\bf{D}^1}$ and $t_{\bf{P}^{1, \an}}^{\rm{alg}} \circ \Hh^2_c\bigl(\ttr^\et_j(1)\bigr)$ are both equal to $1$ when evaluated on the cycle class of any point $a\in \bf{D}^1(C)=\O_C$. 

\begin{proposition}
\label{alg tr disk is 1}
    Following the notation of \cref{traces on disk differ by a constant}, we have $t^{\rm{alg}}_{\bf{P}^{1, \an}}\Bigl(\Hh^2_c\bigl(\ttr^\et_j(1)\bigr)\bigl(c\ell_{\bf{D}^1}(a)\bigr)\Bigr) = 1$ for any $a\in \O_C=\bf{D}^1(C)$.
\end{proposition}
\begin{proof}
    \Cref{lemma:analytification-cycle-classes} and \cref{lemma: cycle class of points compatible with etale morphism} imply that it suffices to show that $t_{\bf{P}^1}\left(c\ell_{\bf{P}^1}(a)\right) = 1$, where $t_{\bf{P}^1}$ is the schematic trace map and $c\ell_{\bf{P}^1}$ is the schematic cycle class. This is classical and follows from the equality
    \[
    t_{\bf{P}^1}\left(c\ell_{\bf{P}^1}(a)\right) = \deg \O_{\bf{P}^1}(a) = \deg \O_{\bf{P}^1}(1) = 1. \qedhere
    \]
\end{proof}

Now we compute the analytic trace map applied to $c\ell_{\bf{D}^1}(a)$. We start with the following preliminary lemma: 

\begin{lemma}
\label{cycle class in terms of first compactification}
The cycle class $c\ell_{\bf{D}^1}(a)$
is the image of $(T - a)^{-1}$ under the map
$\wdh{k(x_+)}^{\rm{h}, \times}/\Bigl(\wdh{k(x_+)}^{\rm{h},\times}\Bigr)^{n} 
\xrightarrow{\partial_{\bf{D}^1}} \Hh^2_c(\mathbf{D}^1, \mu_n)$ from
\cref{Thm: cpt supp coh of disk}.
\end{lemma}
\begin{proof}
We consider the open immersion $j\colon \bf{D}^1 \smallsetminus \{a \} \hookrightarrow \bf{D}^1$ and the closed complement $i\colon \{a\} \hookrightarrow \bf{D}^1$. Then we apply \cite[\href{https://stacks.math.columbia.edu/tag/05R0}{Tag 05R0}]{stacks-project} to the following commutative diagram (with distinguished rows and columns)
\[
\begin{tikzcd}
    i_*\rm{R}i^! \mu_n \arrow{r} \arrow{d} & \mu_n  \arrow{d} \arrow{r}& \rm{R}j_* \mu_n \arrow{d} \\
    i_*\rm{R}i^! \bf{G}_m \arrow{r} \arrow{d} & \bf{G}_m \arrow{r} \arrow{d} & \rm{R}j_* \bf{G}_m \arrow{d} \\
    i_*\rm{R}i^! \bf{G}_m \arrow{r} & \bf{G}_m \arrow{r} & \rm{R}j_* \bf{G}_m 
\end{tikzcd}
\]
to conclude that the diagram 
\[ \begin{tikzcd}
\Hh^0(\mathbf{D}^1 \smallsetminus \{a\}, \mathbf{G}_m) \arrow[r] \arrow[d] & \Hh^1_a(\mathbf{D}^1, \mathbf{G}_m) \arrow[d] \\
\Hh^1(\mathbf{D}^1 \smallsetminus \{a\}, \mu_n) \arrow[r] & \Hh^2_a(\mathbf{D}^1, \mu_n)
\end{tikzcd} \]
is \emph{anti}-commutative. Following \cref{defn:cycle-clas-divisors}, we see that the localized cycle class
$c\ell^{\rm{loc}}_{\bf{D}^1}(a) \in \Hh^2_a(\mathbf{D}^1, \mu_n)$
is the image of $(T-a)$ going through the right top corner.
Hence if we instead go through the left bottom corner,
it then becomes the image of $(T-a)^{-1}$.  

We denote by $j^c\colon \bf{D}^1 \hookrightarrow \bf{D}^{1, c}$ the open immersion of $\bf{D}^1$ into its universal compactification, we denote its closed complement by $i^c \colon \{x_+\} \hookrightarrow \bf{D}^{1, c}$. We also denote by $\ov{j}\colon \bf{D}^{1, c} \smallsetminus \{a\} \hookrightarrow \bf{D}^{1, c}$ the natural open immersion, and its closed complement by $\ov{i}\colon \{a\} \hookrightarrow \bf{D}^{1, c}$. Then we have the following commutative diagram of pseudo-adic spaces:
\[
\begin{tikzcd}
\{a\} \arrow[d, equals] \arrow[r, hook, "i"] & \bf{D}^1 \arrow[d, hook, "j^c"] & \arrow[l,hook',"j"'] \bf{D}^1 \smallsetminus \{a\} \arrow[d, hook] \\
\{a\} \arrow[d, hook, "i"] \arrow[hook]{r}{\ov{i}} & \arrow[d, equals] \bf{D}^{1, c} & \bf{D}^{1, c} \smallsetminus \{a\}\arrow[l,hook',"\ov{j}"'] \\ 
\bf{D}^1 \arrow[hook]{r}{j} & \bf{D}^{1, c} & \{x_+\} \arrow[l,hook',"i^c"'] \arrow[u, hook']. 
\end{tikzcd}
\]
This induces the following commutative diagram of distinguished triangles in $D(\bf{D}^{1, c}_\et; \Z/n\Z)$:
\begin{equation}\label{eqn:local-vs-compact}
\begin{tikzcd}
\rm{R}j^{c}_* \, i_*\, \rm{R}i^!\, j^{c, *}\mu_{n} \arrow{r} &\rm{R}j^{c}_*\,j^{c, *}\mu_{n} \arrow{r} & \rm{R}j^{c, *}\, \rm{R}j_*\,j^*\,j^{c, *} \mu_{n} \\
\overline{i}_* \rm{R}\overline{i}^! \mu_n \arrow{r} \arrow{d} \arrow[u,sloped,"\sim"] & \mu_n \arrow{u} \arrow[d,equals] \arrow{r} & \rm{R}\ov{j}_* \ov{j}^* \mu_n \arrow{d} \arrow{u}\\  
 j^c_!j^{c, *} \mu_{n} \ar[r] & \mu_{n} \arrow{r} & i^c_*i^{c, *} \mu_{n},
\end{tikzcd}
\end{equation}
where the top left vertical map is an isomorphism due to the observation that $\rm{R}\ov{i}^! \simeq \rm{R}i^! j^{c, *}$ and $\rm{R}j^c_*\, i_* \simeq \ov{i}_*$. Now we apply the derived global sections to \cref{eqn:local-vs-compact} to get the following commutative diagram of distinguished triangles:
\begin{equation}\label{eqn:local-vs-compact-2}
\begin{tikzcd}
\rR\Gamma_a(\mathbf{D}^1, \mu_n) \arrow[r] & \rR\Gamma(\mathbf{D}^1, \mu_n) \arrow[r] & \rR\Gamma(\mathbf{D}^1 \smallsetminus \{a\}, \mu_n) \\
\rR\Gamma_a(\mathbf{D}^{1, c}, \mu_n) \arrow[r] \arrow[d] \arrow[u,"\res","\sim"'{sloped}] & \rR\Gamma(\mathbf{D}^{1, c}, \mu_n) \arrow[r] \arrow[d,equals] \arrow[u,"\res","\sim"'{sloped}] & \rR\Gamma(\mathbf{D}^{1, c} \smallsetminus \{a\}, \mu_n) \arrow[d,"\res"] \arrow[u,"\res","\sim"'{sloped}] \\
\rR\Gamma_c(\mathbf{D}^1, \mu_n) \arrow[r] & \rR\Gamma(\mathbf{D}^{1, c}, \mu_n) \arrow[r] & \rR\Gamma(\{x_+\}, \mu_n).
\end{tikzcd}
\end{equation}
\cref{eqn:local-vs-compact} implies that the left top vertical map in \cref{eqn:local-vs-compact-2} is an isomorphism. Furthermore, \cref{cohomology-affine-curve} ensures that the middle top vertical map is an isomorphism. Therefore, the same holds for the right top vertical map as well.  

One checks easily that the composition of the left column agrees with the
map $\rR\Gamma_a(\mathbf{D}^1, \mu_n) \to
\rR\Gamma_c(\mathbf{D}^1, \mu_n)$ appeared before
\cref{defn:cycle-clas-divisors-compact} (with the $c$ being $1$).
Using the Kummer exact sequence and \cref{eqn:local-vs-compact-2}, we get the following commutative diagram:
\begin{equation}\label{eqn:local-vs-compact-3}
\begin{tikzcd}
\Hh^0(\mathbf{D}^1 \smallsetminus \{a\}, \mathbf{G}_m) \arrow[r] & \Hh^1_{\et}(\mathbf{D}^1 \smallsetminus \{a\}, \mu_n) \arrow[r] & \Hh^2_a(\mathbf{D}^1, \mu_n) \\
\Hh^0(\mathbf{D}^{1, c} \smallsetminus \{a\}, \mathbf{G}_m) \arrow[r] \arrow[d,"\res"] \arrow[u,"\res","\sim"'{sloped}] & \Hh^1(\mathbf{D}^{1, c} \smallsetminus \{a\}, \mu_n) \arrow[r] \arrow[d,"\res"] \arrow[u,"\res","\sim"'{sloped}] & \Hh^2_a(\mathbf{D}^{1, c}, \mu_n) \arrow[d] \arrow[u,"\res","\sim"'{sloped}] \\
\Hh^0(\Spec \widehat{k(x_+)}^{\rm{h}}, \mathbf{G}_m) \arrow[r] & \Hh^1(\{x_+\}, \mu_n) \arrow[r] & \Hh^2_c(\mathbf{D}^1, \mu_n).
\end{tikzcd}
\end{equation}
Here the left squares are the boundary maps coming from the corresponding Kummer exact sequences, and the right squares are the corresponding boundary maps of the distinguished triangles in \cref{eqn:local-vs-compact-2}. In particular, using the inverse of the top right vertical map, the right column sends localized cycle class $c\ell^{\rm{loc}}_{\bf{D}^1}(a)\in \rm{H}^2_a(\bf{D}^1, \mu_n)$ to the compactly supported cycle class $c\ell_{\bf{D}^1}(a) \in \rm{H}^2_c(\bf{D}^1, \mu_n)$.  

Using first paragraph, we see that the compactly supported
cycle class $c\ell_{\bf{D}_C}(\{a\})$
is the image of $(T-a)^{-1}$ under the composite
of maps where we start at the top left corner of \cref{eqn:local-vs-compact-3} and go right-right-down-down to the bottom right corner. Since the invertible function $(T-a)^{-1}$ on $\mathbf{D}^1 \smallsetminus \{a\}$ extends to the invertible function $(T-a)^{-1}$ on $\mathbf{D}^{1, c} \smallsetminus \{a\}$, \cref{eqn:local-vs-compact-3} shows that $c\ell_{\bf{D}^1}(\{a\})$ can be obtained from $(T-a)^{-1} \in 
\Hh^0(\Spec \widehat{k(x_+)}^{\rm{h}}, \mathbf{G}_m) = 
\wdh{k(x_+)}^{\rm{h}, \times}$ by composing the two bottom horizontal arrows. This finishes the proof. %
\end{proof}

\begin{corollary}
\label{an tr disk is 1}
    For any $a\in \O_C=\bf{D}^1(C)$, we have $t_{\bf{D}^1}\bigl(c\ell_{\bf{D}^1}(a)\bigr) = 1$.
\end{corollary}
\begin{proof}
    Using \cref{cycle class in terms of first compactification} and the definition of $t_{\bf{D}^1}$ (see \cref{defn:analytic-trace-curves}), we conclude that 
    \[
    t_{\bf{D}^1}\bigl(c\ell_{\bf{D}^1}(a)\bigr) = \# \circ v_{x_+}((T-a)^{-1}).
    \]
    Using the explicit formula for $v_{x_+}$ from \cref{lemma:compacitification-of-the-disc} and the implicit negative sign in the definition of $\#$ (see \cref{warning:sign-convention}), we easily conclude that $t_{\bf{D}^1}\bigl(c\ell_{\bf{D}^1}(a)\bigr) = 1$ for any $a\in \O_C$.
\end{proof}

We have finally arrived at the following statement.
\begin{theorem}
\label{an trace equals alg trace for disk}
Following the notation of \cref{traces on disk differ by a constant}, we have $\lambda=1$. 
\end{theorem}
\begin{proof}
    \cref{traces on disk differ by a constant} implies that it suffices to show that there exists an element $x\in \rm{H}^2_c(\bf{D}^1, \mu_n)$ such that 
    \begin{equation}\label{eqn:pin-down-constant}
    t_{\bf{D}^1}(x) = 1 = t_{\bf{P}^{1, \an}}\bigl(\Hh^2_c\bigl(\ttr^\et_j(1)\bigr)(x)\bigr).
    \end{equation}
    Now we note that the combination of \cref{alg tr disk is 1} and \cref{an tr disk is 1} ensures that \cref{eqn:pin-down-constant} holds for $x=c\ell_{\bf{D}^1}(a)$ for any $a\in \O_C$.
\end{proof}

We recall that, for an algebraic smooth connected curve $X$ over $C$, the cycle class of a point $c\ell_X(a) \in \rm{H}^2_c(X, \mu_n)$ generates $\rm{H}^2_c(X, \mu_n)$ and is independent of the point $a\in X(C)$. It is natural to wonder if the same thing happens for a smooth affinoid connected curves over $C$. We show that this hope fails already for the closed unit disk:

\begin{lemma}
\label{distance and cycle class}
Let $C$ be an algebraically closed nonarchimedean field of mixed characteristic $(0, p)$.
Let $a,b \in \bf{D}^1(C)=\O_C$ be two classical point with corresponding cycle classes $c\ell_{\bf{D}^1}(a),c\ell_{\bf{D}^1}(b) \in \rm{H}^2_c(\bf{D}^1, \mu_{p^r})$ for some integer $r$.
Then
\begin{enumerate}[label=\upshape{(\roman*)}]
    \item\label{distance and cycle class-1} If $\abs{b - a} = 1$, then $c\ell_{\bf{D}^1}(a) \not= c\ell_{\bf{D}^1}(b)$.
    \item\label{distance and cycle class-2} If $\abs{b - a} < \abs{p^r \cdot (\zeta_p - 1)}$, 
    then $c\ell_{\bf{D}^1}(a) = c\ell_{\bf{D}^1}(b)$.
\end{enumerate}
In particular, $\Hh^2_c(\DD^1, \mu_p)$ is infinite and its cardinality is at least cardinality of the residue field $k_C \coloneqq \O_C/\m_C$.
\end{lemma}
\begin{proof}
    Since the statements are unchanged under automorphisms of $\bf{D}^1$, we can assume that $b=0$.
    Then \cref{cycle class in terms of first compactification} ensures that $c\ell_{\bf{D}^1}(b)-c\ell_{\bf{D}^1}(a)$ is given by the image of $\frac{T - a}{T} = 1 - \frac{a}{T}$ under the map $\wdh{k(x_+)}^{\rm{h}, \times}/\Bigl(\wdh{k(x_+)}^{\rm{h},\times}\Bigr)^{p^r}  \to \rm{H}^{2}_{c}(\bf{D}^1, \mu_{p^r})$.
    \Cref{comp-supp-cohomology-affine-curve} and \cref{invertible functions on closed disk} imply that we have the exact sequence
    \[ \bigl(1 + \mathfrak{m}_C T \langle T \rangle\bigr) \times \Bigl(\wdh{k(x_+)}^{\rm{h},\times}\Bigr)^{p^r} \longrightarrow \wdh{k(x_+)}^{\rm{h}, \times} \longrightarrow \Hh^2_c(\mathbf{D}_{C}, \mu_{p^r}) \longrightarrow 0. \]
    We are now reduced the question when $1 - \frac{a}{T}$ comes from an element in
    $\bigl(1 + \mathfrak{m}_C T \langle T \rangle\bigr) \times \Bigl(\wdh{k(x_+)}^{\rm{h},\times}\Bigr)^{p^r}$, depending on the norm of $a$.
    
    The explicit description of the valuation $v_{x_+}$ for the point $x_+$ (see \cref{lemma:compacitification-of-the-disc} and \cref{specialization}\cref{vanishing-order-rings}) implies that $1 - \frac{a}{T} \in \Bigl(\wdh{k(x_+)}^{+, \rm{h}}\Bigr)^{\times}$ and $1 + \mathfrak{m}_C T \langle T \rangle \subset \Bigl(\wdh{k(x_+)}^{+, \rm{h}}\Bigr)^{\times}$.
    Combining these two observations, we see that the question is now further reduced to when the element
    $1 - \frac{a}{T}$ comes from an element in
    $\bigl(1 + \mathfrak{m}_C T \langle T \rangle\bigr) \times 
    \Bigl(\wdh{k(x_+)}^{+, \rm{h}}\Bigr)^{\times, p^r}$. 

    For statement \cref{distance and cycle class-1}, note that \cref{specialization}\cref{vanishing-order-rings} and \cite[\href{https://stacks.math.columbia.edu/tag/0DYE}{Tag 0DYE}]{stacks-project} imply that 
    \[
    \wdh{k(x_+)}^{+, \rm{h}}/\mathfrak{m}_C\wdh{k(x_+)}^{+, \rm{h}} \simeq \O_{\bf{P}^1_{k_C}, \infty}^\h \simeq k_C[T^{-1}]^\h_{(T^{-1})}.
    \]
    Under this quotient, the set $\bigl(1 + \mathfrak{m}_C T \langle T \rangle\bigr)$ is projected to $1$.
    Thus, it suffices to know that $1 - \frac{\overline{a}}{T}$ is not a $p$-th power in $k_C[T^{-1}]^{\rm{h}}_{(T^{-1})}$
    whenever $\overline{a} \neq 0$ in $k_C$, which can be seen once one further completes with respect to $T^{-1}$.
    
    To prove statement \cref{distance and cycle class-2}, we use that $\frac{1}{T} \in \wdh{k(x_+)}^{+}$ by virtue of \cref{specialization}\cref{vanishing-order-rings}.
    Therefore, it suffices to show that the power series 
    \[
    \Bigl(1 - \frac{a}{T}\Bigr)^{p^{-r}} \coloneqq \sum_{i \geq 0} \binom{p^{-r}}{i} \cdot \Bigl(- \frac{a}{T}\Bigr)^i = \sum_{i\geq 0} \biggl( \binom{p^{-r}}{i} (-a)^i \biggr) \cdot \Bigl(\frac{1}{T}\Bigr)^i
    \]
    converges $p$-adically to an element in $\wdh{k(x_+)}^{+, \rm{h}}$.
    In other words, we need to show that the additive $p$-adic valuation $\ord_p\bigl( \binom{p^{-r}}{i} (-a)^i \bigr) \to \infty$ as $i \to \infty$. 
    For this, we note that this $p$-adic valuation equals
    \[
    \ord_p(a)i -ri - \ord_p(i!) = \ord_p(a)i -ri - \sum_{m \geq 1} \lfloor i/p^m \rfloor > \Bigl(\ord_p(a) - r - \frac{1}{p-1}\Bigr) i.
    \]
    Our assumption on $a$ implies that $\ord_p(a) > \ord_p\bigl(p^r\cdot (\zeta_p-1)\bigr) = r + \frac{1}{p-1}$. Therefore, $\bigl(\ord_p(a) - r - \frac{1}{p-1}\bigr) i \to \infty$ as $i\to \infty$. This finishes the proof. %
\end{proof}
\begin{remark}
It would be interesting to know if the distance inequality in
\cref{distance and cycle class}\cref{distance and cycle class-2} is sharp.
\end{remark}

Following a suggestion of Scholze,
we also show that $\Hh^2_c(\DD^1_C, \mu_p)$ is \emph{not} generated by cycle classes when $C$ is
an algebraically closed nonarchimedean field
of mixed characterstic $(0, p)$. For this, we first need the following lemma:

\begin{lemma} 
\label{bounded ring of disk henselized at infty}
Let $C$ be an algebraically closed nonarchimedean field of mixed characterstic $(0, p)$ with maximal ideal $\fm\subset \O_C$ and residue field $k\coloneqq \O_C/\fm$. Let $\DD^1_C = \Spa (C\langle T \rangle, \O_C\langle T\rangle)$ be the closed unit disk with coordinate $T$ and let $x_+$ be the unique (rank-$2$) point of $\DD_C^{1, c}\smallsetminus \DD^1_C$. Then there is a canonical isomorphism
\begin{equation}\label{eqn:mod-m-circ-henselized}
\frac{\bigl(k(x_+)^\h\bigr)^\circ}{\fm\cdot \bigl(k(x_+)^\h\bigr)^\circ}  \simeq \bigl(\O_{\PP^1_k, \infty}\bigr)^\h_{(T^{-1})}\Bigl[\frac{1}{T^{-1}}\Bigr].
\end{equation}
\end{lemma}
\begin{proof}
    \cref{lemma:valuation-generalization} and \cref{specialization}\cref{vanishing-order-rings} show that $\frac{k(x_+)^+}{(\fm, T^{-1})} \simeq \frac{\O_{\PP^1_k, \infty}}{(T^{-1})} \simeq k$. In particular, $(\fm, T^{-1})$ is the maximal ideal in $k(x_+)^+$.
    By the construction of henselizations, $(\fm, T^{-1})$ remains the maximal ideal in $k(x_+)^{+, \h}$.
    Since $T\in k(x_+)^\circ$, we conclude that $T^{-1}$ is invertible in $\bigl(k(x_+)^\h\bigr)^\circ$. 
    Therefore, there is a natural morphism $\bigl(k(x_+)^{+, \h}\bigr)\bigl[\frac{1}{T^{-1}}\bigr] \to \bigl(k(x_+)^\h\bigr)^\circ$. We claim that this is an isomorphism. A combination of \cite[Prop.~0.6.2.9 and Prop.~0.6.3.1]{FujKato} implies that $\bigl(k(x_+)^\h\bigr)^\circ$ is the unique rank-$1$ valuation subring of $k(x_+)^\h$ containing $k(x_+)^{+, \h}$, so we only need to show that $\bigl(k(x_+)^{+, \h}\bigr)\bigl[\frac{1}{T^{-1}}\bigr]$ is a rank-$1$ valuation subring. First, we note that $\bigl(k(x_+)^{+, \h}\bigr)\bigl[\frac{1}{T^{-1}}\bigr]$ is a valuation ring of rank $\leq 1$ because it is a non-trivial localization of a valuation ring of rank $2$. Therefore, it suffices to show that $\bigl(k(x_+)^{+, \h}\bigr)\bigl[\frac{1}{T^{-1}}\bigr]$ is not a field. 
    Now \cite[\href{https://stacks.math.columbia.edu/tag/0DYE}{Tag 0DYE}]{stacks-project} and \cref{specialization}\cref{vanishing-order-rings} imply that 
    \begin{equation}\label{eqn:mod-m-henselized}
    \frac{\bigl(k(x_+)^{+, \h}\bigr)\bigl[\frac{1}{T^{-1}}\bigr]}{\fm \cdot \bigl(k(x_+)^{+, \h}\bigr)\bigl[\frac{1}{T^{-1}}\bigr]} \simeq \bigl(\O_{\PP^1_k, \infty}\bigr)^\h_{(T^{-1})}\Bigl[\frac{1}{T^{-1}}\Bigr] \neq 0. 
    \end{equation}
    Hence, we conclude that $\bigl(k(x_+)^{+, \h}\bigr)\bigl[\frac{1}{T^{-1}}\bigr]$ is not a field, and so $\bigl(k(x_+)^{+, \h}\bigr)\bigl[\frac{1}{T^{-1}}\bigr] \simeq \bigl(k(x_+)^\h\bigr)^\circ$. 
    Finally, \cref{eqn:mod-m-circ-henselized} follows formally from this isomorphism and \cref{eqn:mod-m-henselized}. 
\end{proof}

Before we show that $\Hh^2_c(\DD^1, \mu_p)$ is not generated by cycle classes, we need another technical lemma from algebraic geometry.
In what follows, given a smooth connected $k$-curve $C$, we denote by $k(C)$ the field of rational functions on $C$. Likewise, we denote by $k(T^{-1})\coloneqq \bigl(\O_{\PP^1_k, \infty}\bigr)\bigl[\frac{1}{T^{-1}}\bigr] \simeq \Frac\bigl(\O_{\PP^1_k, \infty}\bigr)$ the field of rational functions on $\PP^1_k$ and by $k(T^{-1})^\h \coloneqq \bigl(\O_{\PP^1_k, \infty}\bigr)^\h_{(T^{-1})}\bigl[\frac{1}{T^{-1}}\bigr]$ its ``henselization''. 
\begin{lemma}
\label{explicit elements in the residue}
There is an element $f \in k(T^{-1})^\h$ satisfying $f^p - f = T^{-1}$
such that there is \emph{no} expression of the form
\[
f = a^p \cdot b
\]
with $a \in k(T^{-1})^\h$
and $b \in k(T^{-1})$.
\end{lemma}

\begin{proof}
We consider the Artin--Schreier map $\alpha \colon C_2 \coloneqq \PP^1_k \to C_1 \coloneqq \PP^1_k$ 
given by $[X:Y] \mapsto [Y^p: X^p - XY^{p-1}]$ in homogeneous coordinates. We denote by $x$ the ``coordinate'' on $C_2$, then it satisfies the equation $x^p-x=T^{-1}$ in $k(C_2)$. The map $\alpha$ is \'etale at $\infty \in C_1(k)$, the vanishing locus of the rational function $x$ is exactly $0\in C_2(k)$, and $\alpha(0)=\infty$. Therefore, we conclude that there is a natural map of $\O_{\PP^1_k, \infty}\simeq \O_{C_1, \infty}$-algebras $\O_{C_2, 0} \to \big(\O_{C_1, \infty}\big)^\h_{(T^{-1})}$. Consequently, we get the field extensions
\[
k(T^{-1}) \simeq k(C_1) \subset L \coloneqq k(C_2) \subset k(T^{-1})^\h.
\]

We claim that $f \colonequals x\in L \subset k(T^{-1})^\h$ does the job.
Indeed, suppose there are $a\in k(T^{-1})^\h$ and $b\in k(T^{-1})$ such that $x = a^p\cdot b$.
We first show that $a\in L$. For this, we observe that $a^p = \frac{x}{b} \in L$, so the extension $L\subset L(a)$ is purely inseparable. On the other hand, $L\subset k(T^{-1})^\h$ is separable algebraic by construction, so $L\subset L(a)$ must also be separable.
Combining these two observations, we obtain $L=L(a)$. In other words, $a\in L$.

Now suppose $x = a^p \cdot b$ with $a\in L$ and $b\in k(T^{-1})$. Since $\alpha$ is invariant under the action of $\ZZ/p\ZZ$ on $C_2$ by usual translations, we conclude that $\rm{ord}_{0}(b) = \rm{ord}_{1}(b)$ where we consider $b$ as a rational function on $C_2$ (via the natural inclusion $k(C_1) \hookrightarrow k(C_2)$). Therefore, we deduce that
\[
1 = \rm{ord}_0(x) - \rm{ord}_1(x) = p\bigl( \rm{ord}_0(a) -\rm{ord}_1(a)\bigr)
\]
is divisible by $p$. This is clearly absurd, so no such $a$ and $b$ exist. 
\end{proof}
We are finally ready to show the promised result:
\begin{lemma}
\label{cycle class of points do not generate}
Let $C$ be an algebraically closed nonarchimedean field of mixed characteristic $(0, p)$.
Then $\Hh^2_c(\DD^1_C, \mu_p)$ is \emph{not} generated by the cycle classes of points. 
\end{lemma}
\begin{proof}
First, we note that the explicit description of $x_+$ from \cref{lemma:compacitification-of-the-disc} implies that its unique rank-$1$ generalization is the Gauss point $\eta\in \DD^1_C$. 
Consider the rank-$1$ valuation $v\colon k(x_+)^\h \to \Gamma_v \cup \{0\}$ that corresponds to the valuation ring $\bigl(k(x_+)^\h\bigr)^\circ$. The observation above implies that the composition $\Frac(C\langle T\rangle) \to k(x_+)^\h \xr{v} \Gamma_v \cup \{0\}$ is equivalent to the usual Gauss norm on $\Frac(C\langle T\rangle)$. 

\Cref{comp-supp-cohomology-affine-curve} and \cref{cycle class in terms of first compactification} (and \cref{thm:topos-of-a-point} to get rid of completions) imply that it suffices to show that the natural morphism
\[
\bigl(\Frac C\langle T \rangle\bigr)^{\times} \to k(x_+)^{\h, \times}/p
\]
is not surjective.
By virtue of \cref{bounded ring of disk henselized at infty}, we can lift the element $f$ from \cref{explicit elements in the residue} to an element $F \in \bigl(k(x_+)^\h\bigr)^{\circ, \times} \subset k(x_+)^{\h, \times}$.
We claim that its image $\overline{F} \in \bigl(k(x_+)^\h\bigr)^{\times}/p$ does not come from
$\bigl(\Frac C\langle T \rangle\bigr)^{\times}$.

Suppose otherwise. Then we have an expression
\begin{equation}\label{eqn:in-rank-1-of-rank-2}
F = A^p \cdot B
\end{equation}
with $A \in k(x_+)^\h$ and $B \in \Frac(C\langle T \rangle)$. Since $F \in \bigl(k(x_+)^\h\bigr)^{\circ, \times}$, we conclude that $v(F)=1$. Furthermore, since the restriction of $v$ to $\Frac(C\langle T\rangle)$ is equivalent to the Gauss norm, we can find a scalar $c\in C^\times$ such that $\abs{c} = v(B)$. Therefore, we can replace $A$ and $B$ by $c^{1/p} \cdot A$ and $B/c$ to assume that $v(B)=1$ (and then $v(A)=1$ as well).
Consequently, we may assume that \cref{eqn:in-rank-1-of-rank-2} is an identity inside $\bigl(k(x_+)^\h\bigr)^{\circ, \times}$.

Now consider the reduction morphism $r\colon \bigl(k(x_+)^\h\bigr)^{\circ} \to k(T^{-1})^\h \coloneqq \bigl(\O_{\PP^1_k, \infty}\bigr)^\h_{(T^{-1})}\bigl[\frac{1}{T^{-1}}\bigr]$ from  \cref{bounded ring of disk henselized at infty}.
We claim that $r\bigl(\Frac(C\langle T\rangle)_{v \leq 1}\bigr) \subset k(T^{-1})$.
Indeed, it suffices to prove the claim for $\O_C\langle T\rangle$ where it is clear because the restriction of $v$ to $\Frac(C\langle T\rangle)$ coincides with the Gauss norm.
Therefore, after applying $r$ to \cref{eqn:in-rank-1-of-rank-2}, we get an expression
\[
f = a^p \cdot b
\]
with $a\in k(T^{-1})^\h$ and $b\in k(T^{-1})$, which is impossible due to \cref{explicit elements in the residue}.
\end{proof}

\subsection{Compatibility with the algebraic trace map: general case}

The main goal of this subsection is to prove \cref{thm:analytic-trace-algebraic} in full generality. We recall that \cref{an trace equals alg trace for disk} already proves the result for the standard open immersion $\bf{D}^1 \hookrightarrow \bf{P}^{1,\an}$.
In the general case, our strategy is to use the refined version of Noether normalization from \cref{lemma:good-noether} to reduce the general case to the case of the disk.  
In order to run these reductions, it will be convenient to introduce some definitions: 

\begin{definition}\label{defn:pointed-semi-stable} A \textit{pointed semi-stable formal $\O_C$-curve} $(\cX, \{x_1, \dots, x_n\})$ is a pair consisting of a rig-smooth connected semi-stable proper formal $\O_C$-curve $\cX$ and a finite non-empty set of closed point $\{x_1, \dots, x_n\} \subset \abs{\cX} = \abs{\cX_s}$. We denote by $j_{\ud{x}} \colon \cU_{\ud{x}} \hookrightarrow \cX$ the unique open formal $\O_C$-subscheme with special fiber $\cU_{\ud{x}, s} = \cX_s \smallsetminus \{x_1, \dots, x_n\}$. 
\end{definition}

\begin{remark} We note that \cite[Lem.~B.12]{Z-thesis} ensures that $\cX_\eta$ is connected for any pointed semi-stable formal $\O_C$-curve $(\cX, \{x_1, \dots, x_n\})$. Furthermore, \cref{useful proposition on rigid curves}\cref{useful proposition on rigid curves-1} then implies that $\cU_{\ud{x}, \eta}$ is always affinoid.
\end{remark}

\begin{definition}
A pointed semi-stable formal $\O_C$-curve $(\cX, \{x_1, \dots, x_n\})$ is \textit{trace-friendly} if the diagram
\[
\begin{tikzcd}
    \rm{H}^2_c(\cU_{\ud{x}, \eta}, \mu_n) \arrow{rr}{\Hh^2_c\bigl(\ttr^\et_{j_{\ud{x},\eta}}(1)\bigr)} \arrow{rd}{t_{\cU_{\ud{x},\eta}}} & & \rm{H}^2(\cX_\eta, \mu_n) \arrow[ld, swap, "t_{\cX_\eta}^{\rm{alg}}"] \\
    & \Z/n\Z &
\end{tikzcd}
\]
commutes. 
\end{definition}

\begin{example}[\cref{an trace equals alg trace for disk}] The pointed semi-stable curve $(\bf{\wdh{P}}^1_{\O_C}, \{\infty\})$ is trace-friendly.
\end{example}

Our first goal is to show that every pointed semi-stable curve is trace-friendly (\cref{thm:everything-is-trace-friendly}). This will be the hardest part in our proof of \cref{thm:analytic-trace-algebraic}. Before we start showing this claim, we need to introduce some further notation: 

\begin{notation}\label{notation:Z_i} Let $(\cX, \{x_1, \dots, x_n\})$ be a pointed semi-stable formal $\O_C$-curve.
Then, for each smooth (resp.\ nodal) point $x_i$, let $Z_{x_i}\subset \cU_{\ud{x}, \eta}^c$ be the pseudo-adic space consisting of the unique rank-$2$ point $\{u_i\}$ in $\cU_{\ud{x}, \eta}^c$ (resp.~the two rank-$2$ points $\{v_i\} \sqcup \{w_i\}$ in $\cU_{\ud{x}, \eta}^c$) from \cref{rmk:extra-points-on-different-compactifications}.
\end{notation}

We note that \cref{specialization}\cref{specialization-comp} together with \cref{lemma:pre-adic-dijoint-union} and \cref{thm:topos-of-a-point} imply that $\abs{\cU_{\ud{x}, \eta}^c} \smallsetminus \abs{\cU_{\ud{x}, \eta}} = \bigsqcup_{i=1}^n Z_{x_i}$ and that
\begin{multline*}
    \rm{R}\Gamma(\cU_{\ud{x}, \eta}^c  \smallsetminus \cU_{\ud{x}, \eta}, \mu_n)  \simeq \rm{R}\Gamma\Bigl(\bigsqcup_{i=1}^n Z_{x_i}, \mu_n\Bigr) \simeq \bigoplus_{i=1}^n \rm{R}\Gamma(Z_{x_i}, \mu_n) \\
    \simeq \biggl(\bigoplus_{i \suchthat x_i \in \cX_s^{\rm{sm}}} \rm{R}\Gamma\Bigl(\Spec \wdh{k(u_i)}^\h, \mu_n\Bigr) \biggr) \oplus \biggl(\bigoplus_{i \suchthat x_i\in \cX_s^{\rm{sing}}} \rm{R}\Gamma\Bigl(\Spec \wdh{k(v_i)}^\h, \mu_n\Bigr) \oplus \rm{R}\Gamma\Bigl(\Spec \wdh{k(w_i)}^\h, \mu_n\Bigr)\biggr),
\end{multline*}
where we treat $\cU_{\ud{x}, \eta}^c  \smallsetminus \cU_{\ud{x}}$ and $Z_{x_i}$ as pseudo-adic spaces inside $\cU_{\ud{x}, \eta}^c$. With this notation, the boundary morphism $\partial_{\cU_{\ud{x}, \eta}}$ from \cref{comp-supp-cohomology-affine-curve} has the form $\partial_{\cU_{\ud{x}, \eta}} \colon \bigoplus_{i=1}^n \rm{H}^1(Z_{x_i}, \mu_n) \to \Hh^2_c(\cU_{\ud{x}, \eta}, \mu_n)$. 

The next two results form the core of our proof that every pointed semi-stable formal $\O_C$-curve is trace-friendly. 

\begin{lemma}
\label{further reduction thm:analytic-trace-algebraic} Let $(\cX, \ud{x}) = (\cX, \{x_1, \dots, x_n\})$ be a pointed semi-stable $\O_C$-curve. For each $1\leq m\leq n$, let $(\cX, \ud{x}')=(\cX, \{x_1, \dots, x_m\})$ be a pointed semi-stable $\O_C$-curve obtained by forgetting the last $n-m$ marked points. Then the diagram
\[
\begin{tikzcd}[column sep=5.5em]
    \bigoplus_{i=1}^m \rm{H}^1(Z_{x_i}, \mu_n) \arrow[r,hook,"\rm{incl}"] \arrow{d}{\partial_{\cU_{\ud{x}', \eta}}} & \bigoplus_{i=1}^n \rm{H}^1(Z_{x_i}, \mu_n)  \arrow{d}{\partial_{\cU_{\ud{x}, \eta}}} \\
    \rm{H}^2_c(\cU_{\ud{x}', \eta}, \mu_n) &  \arrow[l,"\Hh^2_c\bigl(\ttr^\et_{j'_{\ud{x},\eta}}(1)\bigr)"'] \rm{H}^2_c(\cU_{\ud{x}, \eta}, \mu_n),
\end{tikzcd}
\]
commutes, where $\rm{incl}$ is the evident inclusion morphism and $H^2_c\bigl(\ttr^\et_{j'_{\ud{x},\eta}}(1)\bigr)$ is the \'etale trace coming from the open immersion $j'_{\ud{x},\eta} \colon \cU_{\ud{x}, \eta} \hookrightarrow \cU_{\ud{x}', \eta}$. 
\end{lemma}
\begin{proof}
In this proof, all spaces are regarded as pseudo-adic spaces;
see \cref{section:pseudo-adic} for the necessary definitions and results.
By the discussion before this lemma, we have canonical isomorphisms
\[
\bigoplus_{i=1}^n \rm{H}^1(Z_{x_i}, \mu_n) \simeq \rm{H}^1(\cU_{\ud{x}, \eta}^c \smallsetminus \cU_{\ud{x}, \eta}, \mu_n), \text{ and }  \bigoplus_{i=1}^m \rm{H}^1(Z_{x_i}, \mu_n) \simeq \rm{H}^1\bigl(\cU_{\ud{x}', \eta}^c \smallsetminus \cU_{\ud{x}', \eta}, \mu_n\bigr). 
\]
\begin{enumerate}[wide,label={\textit{Step~\arabic*}.},ref={Step~\arabic*}]
\item \textit{We recall the definitions of $\partial_{\cU_{\ud{x}', \eta}}$, $\partial_{\cU_{\ud{x}, \eta}}$, and $\rm{can}$.}
Consider the following diagram of inclusions:
\[ \begin{tikzcd}
\cU_{\ud{x}, \eta} \arrow[r,hook] \arrow[d,hook] & \cU_{\ud{x}', \eta} \arrow[d,hook] \\
\cU_{\ud{x}, \eta}^c \arrow[r,hook] & \cU_{\ud{x}', \eta}^c.
\end{tikzcd} \]
Using the triangles for both compositions from the top left to the bottom right, we obtain the following commutative diagram:
\begin{equation}
\label{eqn:proof-1-extra-point}
\begin{tikzcd}
\rR\Gamma_c(\cU_{\ud{x}, \eta}, \mu_n) \arrow{r} \arrow[d,equals] & \rR\Gamma(\cU_{\ud{x}, \eta}^c, \mu_n) \arrow{r} & 
\rR\Gamma(\cU_{\ud{x}, \eta}^c \smallsetminus \cU_{\ud{x}, \eta}, \mu_n) \arrow{r}{\delta_1} & \rm{R}\Gamma_c(\cU_{\ud{x}, \eta}, \mu_n)[1] \arrow[d,equals] \\
\rR\Gamma_c(\cU_{\ud{x}, \eta}, \mu_n) \arrow{r} \arrow{d}[swap]{\rR\Gamma_c\bigl(\ttr^\et_{j'_{\ud{x},\eta}}(1)\bigr)} & 
\rR\Gamma(\cU_{\ud{x}', \eta}^c, \mu_n) \ar[r] \ar[u] \arrow[d,equals] & 
\rR\Gamma(\cU_{\ud{x}', \eta}^c \smallsetminus \cU_{\ud{x}, \eta}, \mu_n) \arrow{u}{\alpha} \arrow{d}[swap]{\beta} \arrow{r}{\delta_2} & \rm{R}\Gamma_c(\cU_{\ud{x}, \eta}, \mu_n)[1] \arrow{d}[swap]{\rR\Gamma_c\bigl(\ttr^\et_{j'_{\ud{x},\eta}}(1)\bigr)[1]} \\
\rR\Gamma_c(\cU_{\ud{x}', \eta}, \mu_n) \arrow{r} & \rR\Gamma(\cU_{\ud{x}', \eta}^c, \mu_n) \arrow{r} & 
\rR\Gamma(\cU_{\ud{x}', \eta}^c \smallsetminus \cU_{\ud{x}', \eta}, \mu_n) \arrow{r}{\delta_3} & \rm{R}\Gamma_c(\cU_{\ud{x}', \eta}, \mu_n)[1]
\end{tikzcd}
\end{equation}
By the very definition, we have
\[ \partial_{\cU_{\ud{x}, \eta}}=\rm{H}^1\bigl(\delta_1\bigr), \quad \partial_{\cU_{\ud{x}', \eta}}=\rm{H}^1\bigl(\delta_3\bigr), \quad \text{and} \quad \Hh^2_c\bigl(\ttr^\et_{j'_{\ud{x},\eta}}(1)\bigr) = \Hh^1\bigl(\rR\Gamma_c\bigl(\ttr^\et_{j_{\ud{x},\ud{x}',\eta}}(1)\bigr)[1]\bigr). \]

\item \textit{We express $\rm{incl}$ in more geometric terms.} Next, in order to get our hands on the morphism $\rm{incl}$, we need to understand the space $\cU_{\ud{x}', \eta}^c\smallsetminus \cU_{\ud{x}, \eta}$ better.  

First, we note that $\cU_{\ud{x}', \eta}^c \smallsetminus \cU_{\ud{x}, \eta}$ decomposes as a set into the disjount union $\bigl(\cU_{\ud{x}', \eta}^c \smallsetminus \cU_{\ud{x}', \eta}\bigr)\sqcup \bigl(\cU_{\ud{x}', \eta} \smallsetminus \cU_{\ud{x}, \eta}\bigr)$. Clearly, $\cU_{\ud{x}', \eta}^c \smallsetminus \cU_{\ud{x}', \eta}$ is closed in $\cU_{\ud{x}', \eta}^c \smallsetminus \cU_{\ud{x}, \eta}$, but we claim that it is also open.
Indeed, \cite[\href{https://stacks.math.columbia.edu/tag/0903}{Tag 0903}]{stacks-project} ensures that it suffices to show that $\cU_{\ud{x}', \eta}^c \smallsetminus \cU_{\ud{x}', \eta}$ is stable under generalizations in $\cU_{\ud{x}', \eta}^c \smallsetminus \cU_{\ud{x}, \eta}$.
Equivalently, we need to prove that for each $i=1, \dotsc, m$, the (unique) rank-$1$ generalization of points in $Z_{x_i}$ lies in $\cU_{\ud{x}, \eta}$. This follows from the fact that these rank-$1$ generalizations correspond to generic points of the special fiber (see \cref{specialization}).%
Therefore, we conclude that $\cU_{\ud{x}', \eta}^c \smallsetminus \cU_{\ud{x}, \eta}$ has a clopen decomposition $(\cU_{\ud{x}', \eta}^c \smallsetminus \cU_{\ud{x}', \eta})\sqcup \left(\cU_{\ud{x}', \eta} \smallsetminus \cU_{\ud{x}, \eta}\right)$.

Thanks to this decomposition, \cref{lemma:pre-adic-dijoint-union} ensures that we have a canonical isomorphism
\[
\rR\Gamma(\cU_{\ud{x}', \eta}^c \smallsetminus \cU_{\ud{x}, \eta}, \mu_n) \simeq \rm{R}\Gamma(\cU_{\ud{x}', \eta}^c \smallsetminus \cU_{\ud{x}', \eta}, \mu_n) \oplus \rm{R}\Gamma(\cU_{\ud{x}', \eta} \smallsetminus \cU_{\ud{x}, \eta}, \mu_n). 
\]
This implies that the map $\beta$ from \cref{eqn:proof-1-extra-point}
admits a canonical section 
\[
\iota \colon \rm{R}\Gamma(\cU_{\ud{x}', \eta}^c \smallsetminus \cU_{\ud{x}', \eta}, \mu_n) \to \rm{R}\Gamma(\cU_{\ud{x}', \eta}^c \smallsetminus \cU_{\ud{x}', \eta}, \mu_n) \oplus \rm{R}\Gamma(\cU_{\ud{x}', \eta} \smallsetminus \cU_{\ud{x}, \eta}, \mu_n) \simeq  \rm{R}\Gamma(\cU_{\ud{x}', \eta}^c \smallsetminus \cU_{\ud{x}, \eta}, \mu_n)
\]
such that the composition $\alpha \circ \iota$ coincides with the map 
\[ \begin{tikzcd}[row sep=tiny]
 \bigoplus_{i=1}^m \rm{R}\Gamma(Z_{x_i}, \mu_n) \arrow[r,hook,"\widetilde{\rm{incl}}"] \arrow[d,phantom,sloped,"\simeq"] & \bigoplus_{i=1}^n \rm{R}\Gamma(Z_{x_i}, \mu_n) \arrow[d,phantom,sloped,"\simeq"] \\
 \rm{R}\Gamma(\cU_{\ud{x}', \eta}^c\smallsetminus \cU_{\ud{x}', \eta}, \mu_n) \arrow[r,hook, "\alpha \circ \iota"] & \rm{R}\Gamma(\cU_{\ud{x}, \eta}^c\smallsetminus \cU_{\ud{x}, \eta}, \mu_n),
\end{tikzcd} \]
where $\widetilde{\rm{incl}}$ is induced by the inclusion $\bigsqcup_{i=1}^m Z_{x_i} \to \bigsqcup_{i=1}^n Z_{x_i}$.
In particular, we conclude that the composition 
\[
\rm{H}^1(\alpha) \circ \rm{H}^1(\iota) =\rm{incl},
\]
where $\rm{incl}$ is from \cref{eqn:proof-1-extra-point}. 

\item \textit{End of proof.}
Now the result follows by applying $\rm{H}^1$ to the following sequence of equalities:
\[
\delta_3 =\delta_3 \circ \beta \circ \iota  = \rR\Gamma_c\bigl(\ttr^\et_{j'_{\ud{x},\eta}}(1)\bigr)[1] \circ \delta_2 \circ \iota = \rR\Gamma_c\bigl(\ttr^\et_{j'_{\ud{x},\eta}}(1)\bigr)[1] \circ \delta_1 \circ \alpha \circ \iota. \qedhere
\]
\end{enumerate}
\end{proof}

\begin{corollary}
\label{cor:add-points}
Let $\cX$ be a rig-smooth connected semi-stable proper formal $\O_C$-curve,
and let $\{x_1, \dots, x_n\}$ and $\{y_1, \dots, y_m\}$ be finite non-empty sets of closed points in $\cX_s$. Then the pair $(\cX, \{x_1, \dots, x_n, y_1, \dots, y_m\})$ is trace-friendly if and only if both $(\cX, \{x_1, \dots, x_n\})$ and $(\cX, \{y_1, \dots, y_m\})$ are so.
\end{corollary}
\begin{proof}
Denote by $\cU_{\ud{x}}$ (resp.\ $\cU_{\ud{y}}$, resp.\ $\cU_{\ud{x, y}}$) the open ``complement'' of $\{x_1, \dots, x_m\}$ 
(resp.~$\{y_1, \dots, y_m\}$, resp.~$\{x_1, \dots, x_m, y_1, \dots, y_n\}$) in $\cX$.
We also denote by $j'_{\ud{x,y}} \colon \cU_{\ud{x,y}} \hookrightarrow \cU_{\ud{x}}$ (resp.\ $j''_{\ud{x,y}} \colon \cU_{\ud{x,y}} \hookrightarrow \cU_{\ud{y}}$) the canonical open immersions and by $Z_{x_i}$ (resp.\ $Z_{y_i}$) the pseudo-adic spaces from \cref{notation:Z_i} applied to $(\cX, \{x_1, \dots, x_n\})$
(resp.\ $(\cX, \{y_1, \dots, y_m\})$).  
Then \cref{further reduction thm:analytic-trace-algebraic} implies that the diagram
\begin{equation}\label{eqn:add-points}
\begin{tikzcd}[column sep=5.5em]
    \bigoplus_{i=1}^n \rm{H}^1(Z_{x_i}, \mu_n) \arrow{r}{\rm{incl}'} \arrow[d,shorten <= -1em,"\partial_{\cU_{\ud{x}, \eta}}"]  & \bigoplus_{i=1}^n \rm{H}^1(Z_{x_i}, \mu_n) \oplus  \bigoplus_{j=1}^m \rm{H}^1(Z_{y_j}, \mu_n) \arrow[d,shorten <= -1em,"\partial_{\cU_{\ud{x, y}, \eta}}"] & \arrow[l, swap, "\rm{incl}''"] \arrow[d,shorten <=-1em,"\partial_{\cU_{\ud{y}, \eta}}"] \bigoplus_{j=1}^m \rm{H}^1(Z_{y_j}, \mu_n) \\
    \rm{H}^2_c(\cU_{\ud{x}, \eta}, \mu_n) \arrow[dr, swap, "\Hh^2_c\bigl(\ttr^\et_{j_{\ud{x},\eta}}(1)\bigr)"] & \arrow[l, swap, "\Hh^2_c\bigl(\ttr^\et_{j'_{\ud{x,y},\eta}}(1)\bigr)"] \rm{H}^2_c(\cU_{\ud{x, y}, \eta}, \mu_n) \arrow[d,pos=.4,"\Hh^2_c\bigl(\ttr^\et_{j_{\ud{x,y},\eta}}(1)\bigr)"] \arrow[r,"\Hh^2_c\bigl(\ttr^\et_{j''_{\ud{x,y},\eta}}(1)\bigr)"] &  \rm{H}^2_c(\cU_{\ud{y}, \eta}, \mu_n) \arrow{dl}{\Hh^2_c\bigl(\ttr^\et_{j_{\ud{y},\eta}}(1)\bigr)} \\[.5em]
    & \rm{H}^2(\cX_\eta, \mu_n) &
\end{tikzcd}
\end{equation}
commutes. Since the top vertical arrows in \cref{eqn:add-points} are surjective by \cref{comp-supp-cohomology-affine-curve}, and the images of $\rm{incl}'$ and $\rm{incl}''$ generate the group $\bigoplus_{i=1}^n \rm{H}^1(Z_{x_i}, \mu_n) \oplus  \bigoplus_{j=1}^m \rm{H}^1(Z_{y_j}, \mu_n)$, we conclude that $(\cX, \{x_1, \dots, x_n, y_1, \dots, y_m\})$ is trace-friendly if and only if 
\begin{equation}\label{eqn:trace-friendly-proof}
\rm{t}_{\cX_\eta}^{\rm{alg}} \circ \Hh^2_c\bigl(\ttr^\et_{j_{\ud{x,y},\eta}}(1)\bigr)(f) = t_{\cU_{\ud{x, y}, \eta}}(f) \quad \text{for}
\end{equation}
\begin{equation}\label{eqn:case-1-add-points}
f \in \rm{Im}(\partial_{\cU_{\ud{x, y}, \eta}} \circ \rm{incl}') \quad \text{and} 
\end{equation}
\begin{equation}\label{eqn:case-2-add-points}
f \in \rm{Im}(\partial_{\cU_{\ud{x, y}, \eta}} \circ \rm{incl}'').
\end{equation}
Using \cref{eqn:add-points} and the definition of the analytic trace map (see \cref{defn:analytic-trace-curves}), we conclude that Equation~\cref{eqn:trace-friendly-proof} for $f$ as in \cref{eqn:case-1-add-points} is equivalent to $(\cX, \{x_1, \dots, x_n\})$ being trace-friendly, while Equation~\cref{eqn:trace-friendly-proof} for $f$ as in \cref{eqn:case-2-add-points} is equivalent to $(\cX, \{y_1, \dots, y_m\})$ being trace-friendly. Combining these results, we get that $(\cX, \{x_1, \dots, x_n, y_1, \dots, y_m\})$ is trace-friendly if and only if both $(\cX, \{x_1, \dots, x_n\})$ and $(\cX, \{y_1, \dots, y_m\})$ are so.
\end{proof}

\begin{theorem}\label{thm:everything-is-trace-friendly} Let $(\cX, \{x_1, \dots, x_n\})$ be a pointed semi-stable formal $\O_C$-curve. Then it is trace-friendly. 
\end{theorem}
\begin{proof}
By \cref{cor:add-points}, it suffices to prove the claim after replacing $(\cX, \{x_1, \dots, x_n\})$ with \\
$(\cX, \{x_1, \dots, x_n, x_{n+1}, \dots, x_m\})$ for any set of closed points $x_{n+1}, \dots, x_m \in \cX_s$. Therefore, we may and do assume that each irreducible component of $\cX_s$ contains at least one point from the set $\{x_1, \dots, x_n\}$. In this situation, \cref{lemma:good-noether} and \cref{lemma:finite-flat-curves} imply that we can find a finite flat morphism $h\colon \cX_\eta \to \bf{P}^{1, \an}$ such that $h^{-1}(\bf{D}^1) = \cU_{\ud{x}, \eta}$. We denote by $h'\colon \cU_{\ud{x}, \eta} \to \bf{D}^1$ the restriction of $h$ and by $j \colon \bf{D}^1 \hookrightarrow \bf{P}^{1,\an}$ the natural open immersion.
Consider the diagram
\begin{equation}\label{eqn:trace-trick}
\begin{tikzcd}[column sep =5em,row sep=large]
\rm{H}^2_c(\cU_{\ud{x}, \eta}, \mu_n) \arrow{d}{\rm{H}^2_c(\ttr_{h'})} \arrow{r}{\Hh^2_c\bigl(\ttr^\et_{j_{\ud{x},\eta}}(1)\bigr)} & \rm{H}^2(\cX_\eta, \mu_n) \arrow{r}{t_{\cX_\eta}^{\rm{alg}}} \arrow{d}{\rm{H}^2(\ttr_h)} & \Z/n\Z. \\
\rm{H}^2_c(\bf{D}^1, \mu_n) \arrow{r}{\Hh^2_c(\ttr^\et_j(1))} & \rm{H}^2(\bf{P}^{1, \an}, \mu_n) \arrow[ru, swap, "t_{\bf{P}^{1, \an}}^{\rm{alg}}"], & 
\end{tikzcd}
\end{equation}
where $\ttr_h\coloneqq \ttr_{h,\mu_n}$ and $\ttr_{h'}\coloneqq \ttr_{h', \mu_n}$ are the finite flat trace maps from \cref{thm:flat-trace}. Then \cref{thm:flat-trace}\cref{thm:flat-trace-3} implies that the left square in \cref{eqn:trace-trick} commutes, while 
\cref{thm:flat-trace}\cref{thm:flat-trace-7} and  \cite[Exp.~ XVIII, Th.~2.9 and Prop.~2.10]{SGA4} imply that the right triangle commutes. Therefore, we conclude that \cref{an trace equals alg trace for disk}, \cref{thm:analytic-trace-compatible-finite-flat-trace}, and \cref{eqn:trace-trick} imply that 
\[
t^{\rm{alg}}_{\cX_\eta} \circ \Hh^2_c\bigl(\ttr^\et_{j_{\ud{x},\eta}}(1)\bigr) = t^{\rm{alg}}_{\bf{P}^{1, \an}} \circ \Hh^2_c\bigl(\ttr^\et_j(1)\bigr) \circ \rm{H}^2_c(\ttr_{h'}) = t_{\bf{D}^1} \circ \rm{H}^2_c(\ttr_{h'}) = t_{\cU{\ud{x}, \eta}}.
\]
In other words, $(\cX, \{x_1, \dots, x_n\})$ is trace-friendly. 
\end{proof}

Finally, we are ready to give the full proof of \cref{thm:analytic-trace-algebraic}: 

\begin{proof}[Proof of \cref{thm:analytic-trace-algebraic}]
First, we can clearly assume that $\ov{X}$ is connected. Then \cref{useful proposition on rigid curves}\cref{useful proposition on rigid curves-4} implies that $X \subset \overline{X}$ is the adic generic fiber of
$\cX \subset \cX^c$, where
$\cX^c$ is a semistable connected proper $\O_C$-curve and $\cX$ is an open formal subscheme of $\cX^c$.  

Now we consider the open subscheme $\cY_s\coloneqq \cX^c_s \smallsetminus \overline{\cX_s}$, and denote the corresponding open formal $\O_C$-subscheme of $\cX^c$ by $\cY$. We also denote its rigid generic fiber $Y$. By construction, $\cY$ and $\cX$ are disjoint in $\cX^c$ (so $X$ and $Y$ are disjoint as well), and $(\cY \sqcup \cX)_s\subset \cX^c_s$ is dense.  

Now we note that $X \sqcup Y = (\cX \sqcup \cY)_\eta$ is affinoid due to \cref{useful proposition on rigid curves}. Therefore, validity of \cref{thm:analytic-trace-algebraic} for the inclusion $X \sqcup Y \subset \overline{X}$ implies validity of \cref{thm:analytic-trace-algebraic} for both $X\subset \overline{X}$ and $Y \subset \overline{X}$. Therefore, we can replace $X$ with $X \sqcup Y$ to assume that $\cX_s \subset \cX_s^c$ is dense.  

Let $(x_1, \dots, x_n)$ be the finite non-empty set of points of $\abs{\cX_s^c} \smallsetminus \abs{\cX_s}$. Then $(\cX^c, \{x_1, \dots, x_n\})$ is a pointed semi-stable $\O_C$-curve and $\cU_{\ud{x}, \eta} = X$ (see \cref{defn:pointed-semi-stable}). Therefore, \cref{thm:analytic-trace-algebraic} for the inclusion $X\subset \ov{X}$ follows directly from \cref{thm:everything-is-trace-friendly}.
\end{proof}

\section{The trace map for smooth morphisms}\label{smooth-traces}

In this section, we discuss the trace map for separated taut smooth morphisms between locally noetherian analytic adic spaces for constant coefficients and use it to revisit the behavior of lisse complexes under smooth proper pushforwards.
In \cref{proper-trace}, we will then deal with traces for dualizing complexes along maps of rigid-analytic spaces.
Throughout, we fix a positive integer $n$ and set $\Lambda \colonequals \ZZ/n$.

\subsection{Construction}\label{constructing smooth traces}
We begin with the construction of smooth traces, loosely following the strategy in \cite[Exp.~XVII]{SGA4} and \cite[\S~7.2]{Berkovich} by reducing to the case of curves.
Even though our eventual Poincar\'e duality statement in \cref{Poincare dualizability theorem} uses smooth proper morphisms, for our construction of the smooth trace it will be vital to allow nonproper morphisms as well.
\begin{theorem}
\label{smooth-trace-constant}
There is a unique way to assign to any separated taut smooth of equidimension $d$ morphism $f \colon X \to Y$ of locally noetherian analytic adic spaces with $n \in \cO^\times_Y$ a trace map
$\ttr_f \colon \rR f_!\,\LLambda_X(d)[2d] \to \LLambda_Y$
of complexes on $Y_\et$, satisfying the following properties:
\begin{enumerate}[label=\upshape{(\arabic*)}]
\item\label{smooth-trace-constant-composition} (compatibility with compositions) For any two morphisms $f \colon X \to Y$ and $g \colon Y \to Z$ as above of equidimension $d$ and $e$, respectively, the following diagram is commutative:
\[ \begin{tikzcd}[column sep=huge]
   \rR(g \circ f)_!\, \LLambda_X(d+e)[2(d+e)] \arrow[r,"\ttr_{g \circ f}"] \arrow[d,sloped,"\sim"] & \LLambda_Z \\ 
   \rR g_! \Bigl((\rR f_!\, \LLambda_X(d)[2d]) \otimes \LLambda_Y(e)[2e] \Bigr) \arrow[r,"{\rR g_!\, (\ttr_f(e)[2e])}"] & \rR g_!\,\LLambda_Y(e)[2e] \arrow[u,"\ttr_g"]
\end{tikzcd} \]
\item\label{smooth-trace-constant-pullback} (compatibility with pullbacks) For any pullback diagram
\[ \begin{tikzcd}
    X' \arrow[r,"g'"] \arrow[d,"f'"] & X \arrow[d,"f"] \\
    Y' \arrow[r,"g"] & Y
\end{tikzcd} \]
in which $f$ and $f'$ are separated taut smooth of equidimension $d$ as above, the following diagram is commutative (with the top row induced by the base change map from \cite[Th.~5.4.6]{Huber-etale}\footnote{
We warn the reader that the base change map is not always an isomorphism unless $n$ is invertible in $\cO^+$.}):
\[ \begin{tikzcd}
    g^*\rR f_!\,\LLambda_X(d)[2d] \arrow[r] \arrow[d,"g^*\ttr_f"] & \rR f'_!\, \LLambda_{X'}(d)[2d] \arrow[d,"\ttr_{f'}"] \\
    g^* \LLambda_Y \arrow[r,"\sim"] & \LLambda_{Y'}
\end{tikzcd}\]
\item\label{smooth-trace-constant-etale} (compatibility with the \'etale traces from \cref{etale-trace}) If $f$ is \'etale, then $\ttr_f$ is given by the counit
\[ \rR f_!\,\LLambda_X \simeq f_! f^* \LLambda_Y \to \LLambda_Y \]
of the adjunction between $f_!$ and $f^*$.
\item\label{smooth-trace-constant-P1} (compatibility with algebraic traces) If $f$ is the structure morphism $\PP^{1, \an}_C \to \Spa(C,\cO_C)$ for some complete, algebraically closed nonarchimedean field $C$, then $\ttr_f$ is identified with the algebraic trace from \cref{defn:algebraic-trace-map}.
\end{enumerate}
\end{theorem}

For the general construction of trace maps, the following lemmas will turn out to be useful.
\begin{lemma}
\label{smooth trace lives in discrete space}
Let $f \colon X \to Y$ be a separated taut smooth of dimension $d$ morphism between locally noetherian analytic adic space.
Then $\rR f_!\, \LLambda_X(d)[2d]$ lies in $D^{\leq 0}(Y_\et; \Lambda)$. In particular, $\rR\cHom(\rR f_!\, \LLambda_X(d)[2d], \LLambda_Y)$ lies in $D^{\geq 0}(Y_\et; \Lambda)$, and every morphism $\rR f_!\, \LLambda_X(d)[2d] \to \LLambda_Y$ uniquely factors as the composition
\[
\rR f_!\, \LLambda_X(d)[2d] \xrightarrow{\tau^{\ge 0}} \rR^{2d} f_!\, \LLambda_X(d) \to \LLambda_Y.
\]
\end{lemma}
\begin{proof}
Note that by \cite[Prop.~1.8.7~ii)]{Huber-etale} $\dim.\ttr(f) = d$.
The lemma now simply follows from the fact that $\rR f_!$ has cohomological dimension $\leq 2d$;
see \cite[Prop.~5.5.8]{Huber-etale}.
\end{proof}

\begin{lemma}[Gluing trace maps locally on the source]
\label{trace local source}
Let $f \colon X \to Y$ be a separated taut smooth of dimension $d$ morphism between locally noetherian analytic adic spaces.
Let $X = \bigcup_{i \in I} U_i$ be an open cover for which the corresponding open immersions $U_i \hookrightarrow X$ are taut.
For all $i,i' \in I$, let $j_i \colon U_i \hookrightarrow X$ and $j_{i'i} \colon U_{i, i'}\coloneqq U_i \cap U_{i'} \hookrightarrow U_i$ be the natural open immersions.
Let $\ttr^\et_{j_{i'i}}$ be the corresponding \'etale traces in the sense of \cref{etale-trace} and let $f_i \colonequals \restr{f}{U_i}$ be the restriction of $f$ to $U_i$.
Assume there exist maps $\tau_i \colon \rR f_{i,!}\, \LLambda_{U_i}(d)[2d] \to \LLambda_Y$ such that for all $i,i' \in I$, the diagram
\begin{equation}\label{trace local source compatibility} 
\begin{tikzcd}[row sep=tiny,column sep=7em]
(\rR f_{i,!} \circ j_{i'i,!}) \LLambda_{U_{i, i'}}(d)[2d] \arrow[r,"{\rR f_{i,!}(\ttr^\et_{j_{i'i}}(d)[2d])}"] \arrow[dd,equals] & \rR f_{i,!}\, \LLambda_{U_i}(d)[2d] \arrow[rd,"\tau_i"] & \\
&& \LLambda_Y \\
(\rR f_{i',!} \circ j_{ii',!}) \LLambda_{U_{i', i}}(d)[2d] \arrow[r, "{\rR f_{i',!}(\ttr^\et_{j_{ii'}}(d)[2d])}"'] & \rR f_{i',!}\, \LLambda_{U_{i'}}(d)[2d] \arrow[ru, swap, "\tau_{i'}"] & \\
\end{tikzcd}
\end{equation}
commutes.
Then there exists a unique map $\tau \colon \rR f_!\, \LLambda_X(d)[2d] \to \LLambda_Y$ such that for all $i \in I$ the following diagram is commutative:
\[ \begin{tikzcd}
\rR f_{i,!}\, \LLambda_{U_i}(d)[2d] \arrow[rr,"{\rR f_!\,(\ttr^{\et}_{j_i}(d)[2d])}"] \arrow[rd,"\tau_i"] && \rR f_!\,\LLambda_X(d)[2d] \arrow[ld,swap, "\tau"] \\
& \LLambda_Y &
\end{tikzcd} \]
\end{lemma}
\begin{proof}
For any finite subset $J \subseteq I$, let $U_J \colonequals \bigcap_{i \in J} U_i$ and $f_J \colonequals \restr{f}{U_J}$ be the restriction of $f$ to $U_J$.
Note that giving a map $\rR f_{J,!}\, \LLambda_{U_J}(d)[2d] \to \LLambda_Y$ is equivalent to giving a map $\rR^{2d}f_{J,!}\, \LLambda_{U_J}(d) \to \LLambda_Y$ due to \cref{smooth trace lives in discrete space}.
By \cite[Rmk.~5.5.12~iii)]{Huber-etale}, we have a spectral sequence
\[ \mathrm{E}^{pq}_1 \colonequals \bigoplus_{\substack{J \subseteq I \\ \abs{J} = -p+1}} \rR^q f_{J,!}\, \LLambda_{U_J} \Longrightarrow \rR^{p+q}f_!\,\LLambda_Y. \]
Since the $\rR f_{J,!}$ have cohomological dimension $\leq 2d$ \cite[Prop.~5.5.8]{Huber-etale}, the associated abutment filtration for the antidiagonal $p + q = 2d$ reduces to an isomorphism
\[ \rR^{2d} f_! \,\LLambda_X(d) \simeq \coker\Bigl(\bigoplus_{\{i,i'\} \subseteq I} \rR^{2d}f_{ii',!} \,\LLambda_{U_{i, i'}}(d) \to \bigoplus_{i \in I} \rR^{2d}f_{i,!}\, \LLambda_{U_i}(d)\Bigr) \]
and the $\tau_i$ assemble to a map $\tau \colon \bigoplus_{i \in I} \rR^{2d}f_{i,!}\, \LLambda_{U_i}(d) \to \LLambda_Y$.
On the other hand, the assumption on the commutativity of \cref{trace local source compatibility} guarantees that $\tau$ factors uniquely through the cokernel, so we win.
\end{proof}
\begin{lemma}\label{overconvergent-target}
    Let $Y$ be a locally noetherian analytic adic space and let $a,b \colon \F \to \G$ be two morphisms of \'etale sheaves on $Y$.
    Assume that $\cG$ is overconvergent (in the sense of \cite[Def.~8.2.1]{Huber-etale}) and that $a_{\overline{\eta}} = b_{\overline{\eta}}$ for every geometric point $\overline{\eta} \colon \Spa(C_{\overline{\eta}},\cO_{C_{\overline{\eta}}}) \to Y$ of rank $1$.
    Then $a = b$. 
\end{lemma}
\begin{proof}
    Any $y \in \abs{Y}$ has an associated geometric point (\cite[(2.5.2)]{Huber-etale})
    \[ \overline{y} \colon \Bigl(\Spa\bigl(\wdh{\overline{k(y)}},\wdh{\overline{k(y)}}^+\bigr),\{y\}\Bigr) \to Y, \]
    where the source is the (strongly) pseudo-adic space whose underlying topological space is the closed point $\{y\}$ of $\Spa\Bigl(\wdh{\overline{k(y)}},\wdh{\overline{k(y)}}^+\Bigr)$.
    By \cite[Prop.~2.5.5]{Huber-etale}, it suffices to check that the induced maps on stalks $a_{\overline{y}},b_{\overline{y}} \colon F_{\overline{y}} \to G_{\overline{y}}$ coincide.
    The rank-$1$ generalization of $\overline{y}$ is given by the geometric point 
    \[ \overline{\eta} \colon \Spa\Bigl(\wdh{\overline{k(y)}},\wdh{\overline{k(y)}}^\circ\Bigr) \to Y. \]
    The resulting specialization maps for $\F$ and $\G$ from \cite[(2.5.16)]{Huber-etale} fit into the commutative diagram
    \[ \begin{tikzcd}
        \F_{\overline{y}} \arrow[r,shift left,"a_{\overline{y}}"] \arrow[r,shift right,"b_{\overline{y}}"'] \arrow[d,"\spec_\F"] & \G_{\overline{y}} \arrow[d,"\spec_\G"]  \\
        \F_{\overline{\eta}} \arrow[r,shift left,"a_{\overline{\eta}}"] \arrow[r,shift right,"b_{\overline{\eta}}"'] & \G_{\overline{\eta}}
    \end{tikzcd}\]
    and the overconvergence assumption on $G$ guarantees that $\spec_G$ is an isomorphism.
    Thus, we obtain the desired
    \[ a_{\overline{y}} = \spec^{-1}_\G \circ a_{\overline{\eta}} \circ \spec_\F = \spec^{-1}_\G \circ b_{\overline{\eta}} \circ \spec_\F = b_{\overline{y}}. \qedhere \]
\end{proof}

This finishes the sequence of preliminary lemmas.
We are ready to show the uniqueness part of \cref{smooth-trace-constant}.
For the proof, recall that for any locally noetherian analytic adic space $Y$, the $d$-dimensional unit disk over $Y$ is defined as $\DD^d_Y \colonequals \Spa(\ZZ[T_1,\dotsc,T_d],\ZZ[T_1,\dotsc,T_d]) \times_{\Spa(\ZZ,\ZZ)} Y$.%
\begin{proof}[Proof of \cref{smooth-trace-constant}, uniqueness]
    Suppose there are two ways to assign to any separated taut smooth of equidimension $d$ morphism $f \colon X \to Y$ of locally noetherian analytic adic spaces with $n \in \cO^\times_Y$ trace morphisms $\ttr_f$ and $\ttr'_f$ satisfying the properties of the statement.
    We need to show that $\ttr_f = \ttr'_f$ for all $f$.

    First, for any such $f$, we may pick an open affinoid cover $Y = \bigcup_{j \in J} V_j$ and an affinoid open cover $f^{-1}(V_j) = \bigcup_{i \in I_j} U_{ji}$ such that $f_{ji} \colonequals \restr{f}{U_{ji}} \colon U_{ji} \to V_j$ factors as
    \[ \begin{tikzcd}
        U_{ji} \arrow[r,"g_{ji}"] \arrow[rd,"f_{ji}"'] & \DD^d_{V_j} \arrow[d,"\pi_{V_j}"] \\
        & V_j &
    \end{tikzcd} \]
    with $g_{ji}$ \'etale \cite[Cor.~1.6.10]{Huber-etale}.
    Further, since the $U_{ji}$ are qcqs, the open immersions $U_{ji} \hookrightarrow X$ are taut due to \cite[Lem.~5.1.3]{Huber-etale}. 
    By the uniqueness assertions in \cref{trace local source} and \cref{smooth trace lives in discrete space}, it suffices to show that $\ttr_{f_{ji}} = \ttr'_{f_{ji}}$ for all $f_{ji}$.
    
    Moreover, the compatibility with \'etale traces gives $\ttr_{g_{ji}} = \ttr^\et_{g_{ji}} = \ttr'_{g_{ji}}$.
    Thanks to the compatibility with compositions, it then suffices to show that $\ttr_{\pi_{V_j}} = \ttr'_{\pi_{V_j}}$.
    In fact, by writing $\pi_{V_j}$ as a composition of projections $\pi_n \colon \DD^n_{V_j} \simeq \DD^1_{\DD^{n-1}_{V_j}} \to \DD^{n-1}_{V_j}$ away from the last coordinate for $n = 1,\dotsc,d$, it is enough to check that $\ttr_{\pi_n} = \ttr'_{\pi_n}$ for all $n$.
    By \cref{smooth trace lives in discrete space} and \cref{overconvergent-target}, this can be checked on stalks at every geometric point of $\DD^{n-1}_{V_j}$ of rank $1$.
    The compatibility with pullbacks and (weak) proper base change \cite[Cor.~5.4.8]{Huber-etale} guarantee that the stalks of the trace maps at these points are just the trace maps of the fibers.
    In conclusion, we are therefore reduced to the verification that $\ttr$ and $\ttr'$ agree on the closed unit disk $\DD^1_C$ over a complete, algebraically closed nonarchimedean field $C$.
    This is a consequence of property \cref{smooth-trace-constant-P1}, the compatibility with compositions, and the compatibility with the \'etale trace for the open immersion $j \colon \DD^1_C \hookrightarrow \PP^{1,\an}_C$:
    \[ \ttr_{\DD^1_C} = \ttr_{\PP^1_C} \circ \ttr_j(1)[2] = \ttr_{\PP^1_C} \circ \ttr^\et_j(1)[2] = \ttr'_{\PP^1_C} \circ \ttr^\et_j(1)[2] = \ttr'_{\PP^1_C} \circ \ttr'_j(1)[2] = \ttr'_{\DD^1_C}. \qedhere \]
\end{proof}

The uniqueness proof already suggests that we should define the trace of a smooth morphism by locally factoring it into an \'etale morphism and a relative disk.
The main work is to show that this is well-defined, that is, independent of the factorization.
For technical reasons (cf.\ \cref{affine-space-better}), it will be advantageous to work with affine spaces instead of disks; 
recall that the $d$-dimensional affine space over any locally noetherian analytic adic space $Y$ is defined as $\AA^{d,\an}_Y \colonequals \Spa(\ZZ[T_1,\dotsc,T_d],\ZZ) \times_{\Spa(\ZZ,\ZZ)} Y$.
We begin by constructing the trace map for such families of affine spaces, following again the strategy in the uniqueness part of the proof of \cref{smooth-trace-constant}.
\begin{lemma}\label{trace-relative-disk}
    For any a locally noetherian analytic adic space $Y$ with $n \in \cO^\times_Y$ and structure map $\pi_Y \colon \AA^{d,\an}_Y \to Y$, there exist maps $\ttr_{\pi_Y} \colon \rR\pi_{Y,!}\LLambda(d)[2d] \to \LLambda$ with the following properties:
    \begin{enumerate}[label=\upshape{(\arabic*)}]
        \item\label{trace-relative-disk-pullback} $\ttr_{\pi_Y}$ is compatible with pullbacks in $Y$ in the sense of \cref{smooth-trace-constant}.
        \item\label{trace-relative-disk-permutation-invariance} $\ttr_{\pi_Y}$ is invariant under permutations;
        that is, if
        \[ \sigma_Y \colon \AA^{d,\an}_Y \xrightarrow{\sim} \AA^{d,\an}_Y, \quad (y_1,\dotsc,y_d) \mapsto (y_{\sigma(1)},\dotsc,y_{\sigma(d)}) \]
        is the isomorphism permuting the coordinates according to some $\sigma \in \fS_d$, then the map
        \[ \rR\pi_{Y,!} \LLambda(d)[2d] \simeq \rR (\pi_Y \circ \sigma_Y)_! \LLambda(d)[2d] \simeq \rR \pi_{Y,!}(\sigma_{Y,!}\LLambda)(d)[2d] \xrightarrow[\sim]{\rR \pi_{Y,!}(\ttr_{\sigma_Y}(d)[2d])} \rR \pi_{Y,!}\LLambda(d)[2d] \xrightarrow{\ttr_{\pi_Y}} \LLambda \]
        agrees with $\ttr_{\pi_Y}$ in $\Hom\bigl(\rR \pi_{Y,!}\, \LLambda(d)[2d], \LLambda\bigr)$.
        \item\label{trace-relative-disk-dim1} Assume that $d=1$ and that $Y = \Spa(C,\cO_C)$ for some complete, algebraically closed nonarchimedean field $C$.
        Denote by $\tilde{j} \colon \DD^1_Y \hookrightarrow \AA^{1,\an}_Y$ the canonical open immersion.
        Then the map $\Hh^0\bigl(\ttr_{\pi_Y} \circ R\pi_{Y,!}(\ttr^\et_{\tilde{j}}(1)[2]) \bigr)$ is the analytic trace map from \cref{an-tr-disk}.
    \end{enumerate}
\end{lemma}
\begin{proof}
    \begin{enumerate}[wide,label={\textit{Step~\arabic*}.},ref={Step~\arabic*}]
    \item\label{trace-relative-disk-d1} \textit{Proof for $d=1$.}
    Let $\ov{\pi}_Y \colon \PP^{1,\an}_Y \to Y$ be the structure map of the relative (analytic) projective line;
    see e.g.\ \cite[\S~7]{adic-notes} for an account of the latter in the locally noetherian adic context.
    Below, we will describe a trace map $\ttr_{\ov{\pi}_Y} \colon \rR \ov{\pi}_{Y, !}\LLambda(1)[2] \to \LLambda$.
    Granted the existence of $\ttr_{\ov{\pi}_Y}$, we can use the commutative diagram
    \[ \begin{tikzcd}
        \AA^{1,\an}_Y \arrow[r,hook,"\overline{j}"] \arrow[d,"\pi_Y"] & \PP^{1,\an}_Y \arrow[ld,"\ov{\pi}_Y"] \\
        Y &
    \end{tikzcd}\]
    and the trace map $\ttr^{\et}_{\overline{j}} \colon \overline{j}_!\LLambda \to \LLambda$ for the (\'etale) open immersion $\overline{j}$ from \cref{etale-trace} to define $\ttr_{\pi_Y}$ as the composition
    \[ \ttr_{\pi_Y} \colon \rR \pi_{Y,!}\,\LLambda(1)[2] \simeq \rR {\ov{\pi}_Y}_!\bigl(\overline{j}_!\,\LLambda(1)[2]\bigr) \xrightarrow{\rR {\ov{\pi}_Y}_!(\ttr^\et_{\overline{j}}(1)[2])} \rR {\ov{\pi}_Y}_!\,\LLambda(1)[2] \xrightarrow{\ttr_{\ov{\pi}_Y}} \LLambda. \]
    
    One way to define $\ttr_{\ov{\pi}_Y}$ comes from algebraic geometry:
    By \cref{smooth trace lives in discrete space}, it suffices to do it locally on $Y$. So we may and do assume that $Y = \Spa(A,A^+)$ for a strongly noetherian Tate--Huber pair $(A, A^+)$ with $n \in A^\times$.
    Then $\ov{\pi}_Y$ is the relative analytification of the morphism of schemes $\PP^1_A \to \Spec(A)$ along the map of locally ringed spaces $\Spa(A,A^+) \to \Spec(A)$ as in \cref{construction:relative-analytification}. Thus we can construct $\ttr_{\ov{\pi}_Y}$ following \cref{algebraic-analytic-trace-comparison} from the algebraic trace map via \cite[Th.~3.7.2]{Huber-etale}.
    However, in order to avoid any reliance on the algebraic trace map, we can alternatively proceed as follows:
    By \cite[Prop.~6.1.6]{Z-revised}, the \'etale first Chern class of the universal line bundle $\cO_{\PP^{1,\an}_Y}(1)$ (defined in the analytic context in \cite[Def.~6.1.2]{Z-revised}) induces an isomorphism
    \[ c^{\et}_1\bigl(\cO_{\PP^{1,\an}_Y}(1)\bigr) \colon \LLambda \xrightarrow{\sim} \rR^2\ov{\pi}_{Y, *}\LLambda(1) = \rR^2\ov{\pi}_{Y, !}\LLambda(1). \]
    Using \cite[Prop.~5.5.8]{Huber-etale}, one then sets
    \[ \ttr_{\ov{\pi}_Y} \colon \rR \ov{\pi}_{Y, !}\LLambda(1)[2] \to \tau^{\ge 0}\rR \ov{\pi}_{Y, !}\LLambda(1)[2] \simeq \rR^2 \ov{\pi}_{Y, !}\LLambda(1) \xrightarrow[\sim]{c^{\et}_1\bigl(\cO_{\PP^{1,\an}_Y}(1)\bigr)^{-1}} \LLambda. \]

    Since the formation of the adjunction map $\overline{j}_!\LLambda \to \LLambda$ and of the first Chern classes for $\cO_{\PP^{1,\an}_Y}(1)$ commutes with arbitrary base change in $Y$ (\cref{etale-trace-compatibility}, \cref{rmk:base-change-chern-classes}), we conclude \cref{trace-relative-disk-pullback} for $d=1$.
    Further, \cref{trace-relative-disk-permutation-invariance} is an empty statement for $d=1$, so we are only left to show \cref{trace-relative-disk-dim1}. For this, let $j \colon \DD^1_Y \hookrightarrow \PP^{1,\an}_Y$ be the canonical open immersion.
    Since $j = \overline{j} \circ \tilde{j}$, \cref{etale-trace-compatibility} guarantees that the composition 
    \[ \ttr_{\pi_Y} \circ \rR\pi_{Y,!}\bigl(\ttr^\et_{\tilde{j}}(1)[2]\bigr) \colon \rR(\ov{\pi}_Y \circ j)_!\LLambda(1)[2] \simeq \rR(\pi_Y \circ \tilde{j})_!\LLambda(1)[2] \to \LLambda \]
    is given by $\ttr_{\ov{\pi}_Y} \circ \rR \ov{\pi}_{Y, !}\bigl(\ttr^\et_j(1)[2]\bigr)$.
    By \cref{section:disc-p1-compatibility}, especially \cref{an trace equals alg trace for disk}, and \cref{rmk:cycle-class-chern-class}, the latter map agrees with the one from \cref{an-tr-disk} when $Y = \Spa(C,\cO_C)$, yielding \cref{trace-relative-disk-dim1}. 

    \item \textit{Proof for general $d \in \ZZ_{\ge 1}$.}
    By successive projections away from the last coordinate, one can factor $\pi_Y$ as
    \[ \AA^{d,\an}_Y \simeq \AA^{d-1,\an}_Y \times_Y \AA^{1,\an}_Y \simeq \AA^{1,\an}_{\AA^{d-1,\an}_Y} \xrightarrow{\pi_d} \AA^{d-1,\an}_Y \simeq \AA^{d-2,\an}_Y \times_Y \AA^{1,\an}_Y \simeq \AA^{1,\an}_{\AA^{d-2,\an}_Y} \xrightarrow{\pi_{d-1}} \dotsb \xrightarrow{\pi_2} \AA^{1,\an}_Y \xrightarrow{\pi_1} Y. \]
    We set 
    \begin{equation}\label{trace-relative-disk-higher-dim}
        \ttr_{\pi_Y} \colonequals \ttr_{\pi_1} \circ \rR \pi_{1,!}\bigl(\ttr_{\pi_2}(1)[2]\bigr) \circ \dotsb \circ \rR (\pi_1 \circ \dotsb \pi_{d-1})_!\bigl(\ttr_{\pi_d}(d-1)[2d-2]\bigr) \colon \rR \pi_{Y,!}\LLambda(d)[2d] \to \LLambda.
    \end{equation}
    Then $\ttr_{\pi_Y}$ satisfies again \cref{trace-relative-disk-pullback} because the $\ttr_{\pi_i}$ do so separately by \cref{trace-relative-disk-d1}.
    Thus, it remains to verify property \cref{trace-relative-disk-permutation-invariance} that the definition of $\ttr_{\pi_Y}$ is independent of the coordinates under permutation.
    
    Since $\fS_d$ is generated by adjacent transpositions, we may assume for this that $d=2$ and $\sigma \in \fS_2$ is the nontrivial element.
    Thanks to \cref{smooth trace lives in discrete space}, it suffices to verify that
    \[ \rR^4 \pi_{Y,!}(\sigma_{Y,!}\LLambda)(2) \xrightarrow[\sim]{\rR^4\pi_{Y,!}(\ttr_{\sigma_Y}(2))} \rR^4 \pi_{Y,!}\LLambda(2) \xrightarrow{\Hh^0(\ttr_{\pi_Y})} \LLambda \]
    agrees with $\Hh^0(\ttr_{\pi_Y})$.
    The sheaf $\LLambda$ is overconvergent, so \cref{overconvergent-target} allows us to check this on stalks at geometric points of rank $1$.
    Since the formation of derived proper pushforwards commutes with talking stalks \cite[Th.~5.4.6]{Huber-etale} and $\ttr_{\pi_Y}$ is compatible with pullbacks by \cref{trace-relative-disk-pullback}, we may therefore further assume that $Y = \Spa(C,\cO_C)$ for some algebraically closed complete nonarchimedean field $C$.
 
    In this case, we note that it suffices to show a stronger claim that $\fS_2$ acts trivially on $\Hh^4_c\bigl(\AA^{2,\an}_C,\Lambda(2)\bigr)$. In order to justify this, we prove an even stronger claim that $\GL_2(C)$ acts trivially on $\Hh^4_c\bigl(\AA^{2,\an}_C,\Lambda(2)\bigr)$. For this, we note that $\Hh^4_c\bigl(\AA^{2,\an}_C,\Lambda(2)\bigr) \simeq \Lambda$:
    this follows from the analogous statement for the algebraic compactly supported cohomology $\Hh^4_c\bigl(\AA^{2, \an}_C,\Lambda(2)\bigr)$ by Huber's comparison theorem \cite[Th.~5.7.2]{Huber-etale}.
    Alternatively (if one wants to avoid using the nontrivial \cite[Th.~3.2.10]{Huber-etale}), one can adapt the proof of \cite[Th.~7.1.1]{Berkovich} to the adic context. Now we observe that the action of $\GL_2(C)$ on $\Lambda^*$ has to factor through the maximal abelian quotient
    \[ \GL_2(C)/\bigl[ \GL_2(C),\GL_2(C)\bigr] \simeq \GL_2(C) / \SL_2(C) \xrightarrow[\sim]{\det} C^*, \]
    which is divisible and therefore cannot admit any nontrivial maps to the torsion group $\Lambda^*$. \qedhere %
    \end{enumerate}
\end{proof}
\begin{remark}\label{affine-space-better}
As we explain below, the action of $\fS_2$ on $\rm{H}^4_c(\bf{D}^2_C, \mu_p^{\otimes 2})$ is nontrivial when the ground field $C$ is of mixed characteristic $(0, p)$. 
Therefore, in the \emph{proof} of \cref{trace-relative-disk}, it is crucial to use $\bf{A}^{d, \an}_Y$ instead of $\bf{D}^d_Y$. 

To see that the action of $\fS_2$ is nontrival, we set $x_1 \colonequals (0, 1)\in \DD^2$ and $x_2 \colonequals (1, 0) \in \DD^2$.
Then \cref{lemma:base-change-lci-classes} ensures that $\sigma^*c\ell_{\DD^2}(x_1) = c\ell_{\DD^2}(\sigma(x_1)) = c\ell_{\DD^2}(x_2)$, so it suffices to show that 
\[
c\ell_{\DD^2}(x_1) \neq c\ell_{\DD^2}(x_2) \in \Hh^4_c(\bf{D}^2_C, \mu_p^{\otimes 2}).
\]
For this, we consider the hyperplane sections $\{x_i\} \xhookrightarrow{\iota_{x_i}} \{x_i\} \times \DD^1 \xhookrightarrow{\iota'_{x_i}} \DD^2$ for $i=1, 2$ and the projection onto the second factor $\rm{pr}_2\colon \DD^2 \to \DD^1$.
\Cref{cor:gysin-composition} and \cref{tr-vs-cl} below (whose proof does not use this remark) imply that $\ttr_{\pr_2}\bigl(c\ell_{\DD^2}(x_1)\bigr) = c\ell_{\DD^1}\bigl(\{0\}\bigr)$ and $\ttr_{\rm{pr}_2}\bigl(c\ell_{\DD^2}(x_2)\bigr) = c\ell_{\DD^1}\bigl(\{1\}\bigr)$.
On the other hand,
\[
c\ell_{\DD^1}(\{0\}) \neq c\ell_{\DD^1}(\{1\}) \in \Hh^2_c(\DD^1, \mu_p).
\]
thanks to \cref{distance and cycle class}\cref{distance and cycle class-1}, yielding the claim.
\end{remark}

The trace map for affine spaces from \cref{trace-relative-disk} is related to the trace maps for smooth affinoid curves from \cref{section:analytic-trace}:
\begin{lemma}\label{compatibility-smooth-curve-trace}
    Let $X$ be a smooth affinoid curve over an algebraically closed complete nonarchimedean field $C$ and $n \in C^\times$.
    Assume that the structure morphism $f \colon X \to \Spa(C,\cO_C)$ factors as 
    \[ \begin{tikzcd}
        X \arrow[r,"g"] \arrow[rd,"f"'] & \AA^{1,\an}_C \arrow[d,"\pi_C"] \\
        & \Spa(C,\cO_C)
    \end{tikzcd} \]
    with $g$ \'etale.
    Then the map $\Hh^0\bigl(\ttr_{\pi_C} \circ \rR\pi_{C,!}(\ttr^\et_g(1)[2]) \bigr)$ is the analytic trace $t_X \colon \Hh^2_c(X,\mu_n) \to \ZZ/n$ from \cref{defn:analytic-trace-curves}.
\end{lemma}
\begin{proof}
    Since $X$ is quasicompact, $g$ factors through a closed unit disk $\DD^1(r)$ of some radius $r$.
    By the same argument as in the last paragraph of the proof of \cref{trace-relative-disk}, the natural action of $C^* = \GL_1(C)$ on $\Hh^2_c(\AA^1_C,\mu_n)$ is trivial.
    After a renormalization, we therefore end up in the situation
        \[ \begin{tikzcd}[column sep=large]
        U \arrow[rr, bend left,"g"] \arrow[r,"\tilde{g}"] \arrow[rrd,bend right=15,"f"'] & \DD^1_C \arrow[r,hook,"\tilde{j}"] \arrow[rd,bend right=10] & \AA^{1, \an}_C  \arrow[d,"\pi_C"] \\
        && \Spa(C,\cO_C),
    \end{tikzcd} \]
    where $\tilde{j} \colon \DD^1_C \hookrightarrow \AA^{1, \an}_C$ denotes the canonical open immersion.
    In that case, \cref{trace-relative-disk}\cref{trace-relative-disk-dim1} guarantees that $\Hh^0\bigl(\ttr_{\pi_C} \circ \rR\pi_{C,!}(\ttr^\et_{\tilde{j}}(1)[2]) \bigr)$ is the analytic trace morphism $t_{\DD^1_C}$ from \cref{an-tr-disk}.
    Moreover, the analytic trace for smooth affinoid curves in \cref{defn:analytic-trace-curves} is compatible with \'etale morphisms (\cref{curve-trace-etale-compatibility}) and the \'etale trace is compatible with compositions (\cref{etale-trace-compatibility}), so we conclude that
    \[ \Hh^0\bigl(\ttr_{\pi_C} \circ \rR \pi_{C,!}(\ttr^\et_g(1)[2])\bigr) = \Hh^0\bigl(\ttr_{\pi_C} \circ \rR\pi_{C,!}(\ttr^\et_{\tilde{j}}(1)[2]) \circ (\rR\pi_{C,!} \circ \tilde{j}_!)(\ttr^\et_{\tilde{g}}(1)[2])\bigr) = t_{\DD^1_C} \circ \Hh^2_c\bigl(\ttr^\et_{\tilde{g}}(1)\bigr) = t_X. \qedhere \]
\end{proof}
\Cref{compatibility-smooth-curve-trace} leads to the following uniqueness statement.
\begin{lemma}\label{factorization independence}
    Let $f \colon X \to Y$ be a separated taut smooth of equidimension $d$ morphism of locally noetherian analytic adic spaces with $n \in \cO^\times_Y$.
    Let
    \[ \begin{tikzcd}
        X \arrow[r,"g_1"] \arrow[rd,"f"] \arrow[d, swap, "g_2"] & \bf{A}^{d, \an}_Y \arrow[d, "\pi_Y"] \\
        \bf{A}^{d, \an}_Y \arrow[r, "\pi_Y"] & Y
    \end{tikzcd}\]
    be a commutative diagram of factorizations of $f$ such that $g_i$ is \'etale for $i = 1,2$. Then the induced diagram
    \[ \begin{tikzcd}[row sep=tiny]
        & (\rR\pi_{Y,!} \circ g_{1,!})\LLambda(d)[2d] \arrow[r,"{\rR\pi_{Y,!} (\ttr^{\et}_{g_1}(d)[2d])}"] &[6em] \rR\pi_{Y,!} \LLambda(d)[2d] \arrow[rd,near start,"\ttr_{\pi_Y}"] &[2em] \\
        \rR f_!\LLambda(d)[2d] \arrow[ru,equals] \arrow[rd,equals] &&& \LLambda \\ 
        & (\rR\pi_{Y,!}  \circ g_{2,!})\LLambda(d)[2d] \arrow[r, swap, "{\rR\pi_{Y,!} (\ttr^{\et}_{g_1}(d)[2d])}"] & \rR\pi_{Y,!}\LLambda(d)[2d] \arrow[ru,swap,near start,"{\ttr_{\pi_Y}}"] &
    \end{tikzcd}\]
    involving the maps $\ttr^{\et}_{g_i}$ from \cref{etale-trace} and the map $\ttr_{\pi_Y}$ from \cref{trace-relative-disk} commutes as well.
\end{lemma}
\begin{proof}
    Since the sheaf $\LLambda$ is overconvergent, \cref{smooth trace lives in discrete space} and \cref{overconvergent-target} allow us to check the commutativity of the diagram on stalks at geometric points of rank $1$.
    Again, the formation of derived proper pushforwards commutes with taking such stalks \cite[Th.~5.4.6]{Huber-etale} and the trace maps $\ttr^{\et}_{g_i}$ and $\ttr_{\pi_Y}$ are compatible with pullbacks by \cref{etale-trace-compatibility} and \cref{trace-relative-disk}\cref{trace-relative-disk-pullback}.
    Thus, we may assume that $Y = \Spa(C,\cO_C)$ for some algebraically closed complete nonarchimedean field $C$ and that $X$ is a separated taut smooth rigid space of equidimension $d$ over $C$.
    
    The two functions $g_i \colon X \to \bf{A}^{d, \an}_C$ are given by tuples $(g^{(1)}_i,\dotsc,g^{(d)}_i)$ with $g^{(n)}_i \in \cO(X)$.
    Let $x \in X$.
    Since the $g_i$ are \'etale at $x$, the differentials $dg^{(1)}_i,\dotsc,dg^{(d)}_i$ reduce to a basis of the fiber $\Omega^1_{X/C} \otimes k(x)$ for $i=1,2$ \cite[Prop.~1.6.9~iii)]{Huber-etale}.
    By the Steinitz exchange lemma, we can find $\sigma \in \fS_d$ such that for each $n = 1,\dotsc,d$, the differentials $dg^{(\sigma(1))}_1,\dotsc,dg^{(\sigma(n-1))}_1,dg^{(n)}_2,\dotsc,dg^{(d)}_2$ reduce to a basis of $\Omega^1_{X/C} \otimes k(x)$.
    Since the zero locus of $\bigl(dg^{(\sigma(1))}_1 \wedge \dotsb \wedge dg^{(\sigma(n-1))}_1 \wedge dg^{(n)}_2 \wedge \dotsb \wedge dg^{(d)}_2\bigr) \in \Omega^d_X(X)$ is Zariski-closed as the section of a line bundle, its complement then gives a Zariski-open neighborhood $U$ of $x$ over which each
    \[ \Bigl(g^{(\sigma(1))}_1,\dotsc,g^{(\sigma(n-1))}_1,g^{(n)}_2,\dotsc,g^{(d)}_2\Bigr) \colon U \to \bf{A}^{d, \an}_C \]
    is \'etale (again thanks to \cite[Prop.~1.6.9~iii)]{Huber-etale}).
    As the Zariski-open neighborhoods for various points $x$ cover the adic space $X$ and the corresponding Zariski-open embeddings are taut \cite[Lem.~5.1.4!i), Lem.~5.1.3~iii)]{Huber-etale}, it suffices to show that the diagram commutes over each $U$ by the uniqueness assertion in \cref{trace local source}.
    Combined with \cref{trace-relative-disk}\cref{trace-relative-disk-permutation-invariance}, we may therefore assume that $(g^{(1)}_1,\dotsc,g^{(n-1)}_1,g^{(n)}_2,\dotsc,g^{(d)}_2) \colon U \to \bf{A}^{d, \an}_C$ is \'etale for all $n = 1,\dotsc,d$. %

    Arguing one coordinate at a time, it now suffices to show the statement when $\pi_n \circ g_1 = \pi_n \circ g_2 \equalscolon f'$, where $\pi_n \colon \AA^{d, \an}_C \to \AA^{d-1, \an}_C$ is the projection away from the $n$-th coordinate for some $1 \le n \le d$.
    In that case, we are in the situation of the following commutative diagram:
    \[ \begin{tikzcd}
        U \arrow[r,"g_1"] \arrow[rd,"f'"] \arrow[d,"g_2"] & \AA^{d, \an}_C \arrow[d, "\pi_n"]  \arrow[rdd,bend left=15,"\pi_Y"] &\\
        \AA^{d, \an}_C \arrow[r,"\pi_n"]  \arrow[rrd,bend right=15,"\pi_Y"]& \AA^{d-1, \an}_C \arrow[rd, "\pi'_Y"] & \\
        & & \Spa(C, \O_C)
    \end{tikzcd} \]
    By the definition of the trace morphisms in \cref{trace-relative-disk} (and the invariance under permutation of coordinates), we have $\ttr_{\pi_Y} = \ttr_{\pi'_Y} \circ \rR\pi'_{Y,!}(\ttr_{\pi_n}(d-1)[2d-2])$.
    Therefore, we only need to prove that 
    \[ \ttr_{\pi_n} \circ \rR\pi_{n,!}\bigl(\ttr^\et_{g_1}(1)[2]\bigr) = \ttr_{\pi_n} \circ  \rR\pi_{n,!} \bigl(\ttr^\et_{g_2}(1)[2]\bigr). \]
    As in the first paragraph of the proof, we may check this statement on stalks at geometric rank-$1$ points and use (weak) proper base change \cite[Th.~5.4.6]{Huber-etale} to reduce to the case where $d=n=1$.

    Now we note that for each affinoid open $j\colon V\hookrightarrow U$, the morphism is taut due to \cite[Lem.~5.1.3~(i),(iii)]{Huber-etale}. 
    Therefore, we can shrink $U$ even further to assume that $U$ is a smooth affinoid curve over $C$.
    In that case, we have
    \[ \Hh^0\bigl(\ttr_{\pi_1} \circ \rR \pi_{1, !}(\ttr_{g_1}^\et(1)[2])\bigr) = t_U = \Hh^0\bigl(\ttr_{\pi_1} \circ \rR \pi_{1, !}(\ttr_{g_2}^\et(1)[2])\bigr) \]
    by \cref{compatibility-smooth-curve-trace}.
    The desired statement follows from \cref{smooth trace lives in discrete space}.
\end{proof}

We are finally ready to discuss the trace for smooth morphisms in general.
The following construction, which was forced upon us by the uniqueness part of the proof, is essentially independent of the construction of the analytic trace map in \cref{section:analytic-trace}.
However, in order to see that it does not depend on any of the choices made in the process, the existence of an \textit{a priori} well-defined analytic trace map, which was used in the proof of \cref{factorization independence}, is indispensable.
\begin{proof}[{Proof of \cref{smooth-trace-constant}, existence}]
    As in the uniqueness part of the proof, we may pick open affinoid (and thus also taut) covers $Y = \bigcup_{j \in J} V_j$ and $f^{-1}(V_j) = \bigcup_{i \in I_j} U_{ji}$ such that $f_{ji} \colonequals \restr{f}{U_{ji}} \colon U_{ji} \to V_j$ factors as
    \begin{equation}\label{smooth trace factorization} \begin{tikzcd}
        U_{ji} \arrow[r,"g_{ji}"] \arrow[rd,"f_{ji}"'] & \AA^{d,\an}_{V_j} \arrow[d,"\pi_{V_j}"] \\
        & V_j &
    \end{tikzcd} \end{equation}
    with $g_{ji}$ \'etale.
    Set $\ttr_{f_{ji}} \colonequals \ttr_{\pi_{V_j}} \circ \rR\pi_{V_j,!}(\ttr^\et_{g_{ji}}(d)[2d])$.
    This is independent of the factorization by \cref{factorization independence};
    in particular, all the various traces agree on intersections and we may glue them to a trace
    \[ \ttr_f \colon \rR f_!\,\LLambda_X(d)[2d] \to \LLambda_Y \]
    thanks to \cref{trace local source} and \cref{smooth trace lives in discrete space}.

    It remains to verify that $\ttr_f$ satisfies the desired properties.
    In the situation of \cref{smooth-trace-constant-composition}, pick open affinoid covers $Z = \bigcup_{k \in K} W_k$, $g^{-1}(W_k) = \bigcup_{j \in J_k} V_{kj}$, and $f^{-1}(V_{kj}) = \bigcup_{i \in I_{kj}} U_{kji}$ which fit into the commutative diagram
    \begin{equation}\label{smooth-trace-constant-compostion-compatibility} \begin{tikzcd}
        X \arrow[r,phantom,"\supseteq"] \arrow[d,"f"'] &[-2em] U_{kji} \arrow[r,"h^{\prime\prime}_{kji}"] \arrow[d,"f_{kji}"'] & \AA^{d, \an}_{V_{kj}} \simeq \AA^{d, \an}_{W_k} \times_{W_k} V_{kj} \arrow[r,"h'_{kj}"] \arrow[ld,"\pi_{V_{kj}}"'] & \AA^{d+e, \an}_{W_k} \arrow[ld,"\pi_{\AA^{e, \an}_{W_k}}"] \\
        Y \arrow[r,phantom,"\supseteq"] \arrow[d,"g"'] & V_{kj} \arrow[r,"h_{kj}"] \arrow[d,"g_{kj}"'] & \AA^{e, \an}_{W_k} \arrow[ld, swap, "\pi_{W_k}"'] & \\
        Z \arrow[r,phantom,"\supseteq"] & W_k &&
    \end{tikzcd} \end{equation}
    with $h_{kj}$ and $h^{\prime\prime}_{kji}$ \'etale.
    The construction of the traces in the first paragraph makes it clear that
    \[ \ttr_{\pi_{V_{kj}}} \circ \rR\pi_{V_{kj},!}(\ttr^\et_{h^{\prime\prime}_{kji}}(d)[2d]) = \ttr_{f_{kji}} \quad \text{and} \quad \ttr_{\pi_{W_k}} \circ \rR\pi_{W_k,!}(\ttr^\et_{h_{kj}}(e)[2e]) = \ttr_{g_{kj}}. \]
    Therefore, the compatibility of traces under composition boils down to the verification that the diagram
    \[ \begin{tikzcd}[row sep=tiny,column sep=huge]
        (\rR\pi_{\AA^{e, \an}_{W_k},!} \circ \rR h'_{kj,!})\LLambda(e)[2e] \arrow[r,"{\rR\pi_{\AA^{e, \an}_{W_k},!}(\ttr^\et_{h'_{kj}}(e)[2e])}"] \arrow[dd,equals] &[5em] \rR\pi_{\AA^{e, \an}_{W_k},!}\LLambda(e)[2e] \arrow[rd,"\ttr_{\pi_{\AA^{e, \an}_{W_k}}}"] & \\
        && \LLambda \\ 
        (\rR h_{kj,!} \circ \rR\pi_{V_{kj},!})\LLambda(e)[2e] \arrow[r,"\rR h_{kj,!}(\ttr_{\pi_{V_{kj}}})"] & \rR h_{kj,!}\LLambda \arrow[ru,"\ttr^\et_{h_{kj}}"] &
    \end{tikzcd}\]
    of traces in the parallelogram commutes.
    As before, \cite[Prop.~2.5.5]{Huber-etale}, \cref{smooth trace lives in discrete space} and \cref{overconvergent-target} allow us to check this on stalks at geometric points of rank $1$ of $\AA^{e, \an}_{W_k}$.
    Moreover, the base change isomorphism of derived pushfowards with compact support from \cite[Th.~5.4.6]{Huber-etale} is compatible with the formation of $\ttr_{\pi_{V_{kj}}}$ and $\ttr_{\pi_{\AA^{e, \an}_{W_k}}}$ (\cref{trace-relative-disk}\cref{trace-relative-disk-pullback}) as well as $\ttr^\et_{h_{kj}}$ and $\ttr^\et_{h'_{kj}}$ (\cref{etale-trace-compatibility}).
    Therefore, we may check the commutativity of traces in the parallelogram after pulling back along a geometric point of rank $1$ of $\DD^e_{W_k}$, where the statement is clear because the pullback of $V_{kj}$ decomposes as finite disjoint union of geometric points of rank $1$.
   
    Next, the $\ttr_{\pi_{V_j}}$ are compatible with pullbacks by \cref{trace-relative-disk}\cref{trace-relative-disk-pullback} and the $\ttr^\et_{g_{ji}}$ are compatible with pullbacks by \cref{etale-trace-compatibility}, so $\ttr_f$ satisfies property \cref{smooth-trace-constant-pullback}.
    Property \cref{smooth-trace-constant-etale} holds by definition.  

    Lastly, we show \cref{smooth-trace-constant-P1}. Consider the morphism $f \colon \PP^{1}_C \to \Spec C$. Let $\PP^{1}_C = \AA^1_C(0) \cup \AA^1_C(\infty)$ be the open affine cover of $\PP^1_C$ by the affine lines around $0$ and $\infty$ and denote the restricted structure morphisms by $f_i \colon \AA^1_C(i) \to \Spa(C,\cO_C)$ for $i \in \{0,\infty\}$. Since both algebraic and analytic trace maps are compatible with open immersions, and the morphism $\Hh^2_c\bigl(\AA^{1, \an}_C(0), \Lambda(1)\bigr) \oplus \Hh^2_c\bigl(\AA^{1, \an}_C(\infty), \Lambda(1)\bigr) \to \Hh^2\bigl(\PP^{1, \an}_C, \Lambda(1)\bigr)$ is surjective, it suffices to show that the following diagrams 
    \[
    \begin{tikzcd}
        \Hh^2_c(\bf{A}^1_C(i), \Lambda(1)) \arrow{d}{\Hh^0(\ttr_{f_i})} \arrow{r}{\sim} & \Hh^2_c(\bf{A}^{1, \an}_C(i), \Lambda(1)) \arrow{ld}{\Hh^0(\ttr_{f_i^{\an}})} \\
        \Lambda & 
    \end{tikzcd}
    \]
    commute for $i\in \{0, \infty\}$. This follows directly from the definition of the analytic trace map on the analytic affine line and \cref{etale-trace-compatibility}. %
\end{proof}

\subsection{Properties of the smooth trace}

In this subsection, we establish some properties of the smooth trace with constant coefficients constructed in \cref{smooth-trace-constant}.
We begin with the compatibility with the analytic trace for affinoid curves constructed in \cref{defn:analytic-trace-curves}.
\begin{lemma}
    Let $f \colon X \to \Spa(C,\cO_C)$ be a smooth affinoid curve over an algebraically closed complete nonarchimedean field $C$ and $n \in C^\times$.
    Then the analytic trace $t_X \colon \Hh^2_c(X,\mu_n) \to \ZZ/n$ from \cref{defn:analytic-trace-curves} agrees with $\Hh^0(\ttr_f)$ for the smooth trace $\ttr_f$ constructed in \cref{smooth-trace-constant}. 
\end{lemma}
\begin{proof}
    Recall the construction of the smooth trace $\ttr_f$:
    we pick an affinoid open cover $X = \bigcup_{i \in I} U_i$ and diagrams
    \[ \begin{tikzcd}
        U_i \arrow[r,"g_i"] \arrow[rd,"f_i"'] & \AA^{1,\an}_C \arrow[d,"\pi_C"] \\
        & \Spa(C,\cO_C)
    \end{tikzcd} \]
    with $g_i$ \'etale, set $\ttr_{f_i} \colonequals \ttr_{\pi_C} \circ \rR\pi_{C,!}(\ttr^{\et}_{g_i}(1)[2])$ (using \cref{trace-relative-disk} and the \'etale traces $\ttr^\et_{g_i}$ from \cref{etale-trace}), and descend $\bigoplus_i \ttr_{f_i}$ to a morphism $\ttr_f \colon \rR f_!\mu_n[2] \to \ZZ/n$ via \cref{smooth trace lives in discrete space} and the epimorphism $\bigoplus_i \Hh^2_c(U_i,\mu_n) \twoheadrightarrow \Hh^2_c(X,\mu_n)$ from the proof of \cref{trace local source}.
    
    By \cref{compatibility-smooth-curve-trace}, $\Hh^0\bigl(\ttr_{\pi_C} \circ \rR\pi_{C,!}(\ttr^\et_{g_i}(1)[2]) \bigr)$ is given by the analytic trace $t_{U_i} \colon \Hh^2_c(U_i,\mu_n) \to \ZZ/n$ from \cref{defn:analytic-trace-curves}.
    On the other hand, \cref{curve-trace-etale-compatibility} applied to the open immersions $U_i \hookrightarrow X$ guarantees that the composition $\bigoplus_i \Hh^2_c(U_i,\mu_n) \twoheadrightarrow \Hh^2_c(X,\mu_n) \xrightarrow{t_X} \ZZ/n$ is given by $\bigoplus_i t_{U_i}$.
    Since the first map is an epimorphism, we can conclude the desired identity $t_X = \Hh^0(\ttr_f)$.
\end{proof}
\begin{lemma}\label{tr-vs-cl}
Let $f \colon X \to Y$ be a separated taut smooth of equidimension $d$ morphism of locally noetherian analytic adic spaces with $n \in \cO^\times_Y$.
Assume that $f$ has a section $s \colon Y \hookrightarrow X$.
Then $s$ is an lci immersion of pure codimension $d$ and the composition
\[ \LLambda_Y = \rR f_! \circ s_* \LLambda_Y \xrightarrow{\rR f_!(\cl_s)} \rR f_!\LLambda_X(d)[2d] \xrightarrow{\ttr_f} \LLambda_Y, \]
of the induced cycle class map from \cref{variant:cycle-class-map} with the smooth trace of $f$ from \cref{smooth-trace-constant} is the identity.
\end{lemma}
\begin{proof}
    The first statement that $s$ is an lci immersion of pure codimension $d$ is proven in \cite[Cor.~5.11]{adic-notes}.
    To verify the second statement, we proceed in two steps:
    \begin{enumerate}[wide,label={\textit{Step~\arabic*}.},ref={Step~\arabic*}]
        \item\label{tr-vs-cl-AAd} \textit{Case when $f$ is the structure morphism $X=\AA^{d,\an}_Y \to Y$ and $s$ is the zero section.}
        We argue by induction on $d$. 
        If $d=0$, the claim is trivial. 
        If $d=1$, the claim essentially follows from the construction of the trace map in \cref{trace-relative-disk}.
        
        Now we fix $d>1$ and assume that the claim has been proven in dimensions $<d$.
        Consider the commutative diagram
        \[
        \begin{tikzcd}[column sep=huge]
        Y \arrow[r, hook, "i"] \arrow[rd, hook, "i"] \arrow[rdd,equals] \arrow[rr, hook, bend left=20, "s"] &\AA^{d-1, \an}_Y \arrow[d, equals] \arrow[r, hook, "j"] & \AA^{d, \an}_Y \arrow[dl, swap, "g"] \arrow[ldd, "f"] \\
        & \AA^{d-1, \an}_Y \arrow[d, near start, "h"] & \\
        & Y,&  
        \end{tikzcd}
        \]
        where $i \colon Y \hookrightarrow \AA^{d-1, \an}_Y$ is the zero section, $j\colon \AA^{d-1, \an}_Y \hookrightarrow \AA^{d, \an}_Y$ is the natural inclusion as the vanishing locus of the last coordinate, $g$ is the projection onto the first $d-1$ factors, and $h$ is the structure morphism. 
        Then \cref{smooth-trace-constant}\cref{smooth-trace-constant-composition}, \cref{cor:gysin-composition}, and the induction hypothesis guarantee that 
        \begin{multline*}
        \ttr_f \circ \rR f_!(\cl_s) = \ttr_h \circ \rR h_!\bigl(\ttr_g(d-1)[2d-2]\bigr) \circ \rR f_!\bigl(\cl_j(d-1)[2d-2]\bigr) \circ \rR h_!(\cl_i)= \\
        = \ttr_h \circ \rR h_!(\id) \circ \rR h_!(\cl_i) = \ttr_h \circ \rR h_!(\cl_i) = \id. 
        \end{multline*}
        
        \item \textit{General case.}
        Since the sheaf $\LLambda_Y$ is overconvergent, the equality of two endomorphisms may be checked on stalks at geometric points attached to rank-$1$ points of $\abs{Y}$ (\cref{smooth trace lives in discrete space} and \cref{overconvergent-target}).
        Moreover, the formation of $\rR f_!$ and $\cl_X(Y)$ is compatible with pullbacks to these geometric points by \cite[Cor.~5.4.8]{Huber-etale} and \cref{lemma:base-change-lci-classes}, respectively.
        Thus, we are reduced to the case where $Y = \Spa(C,\cO_C)$ for some algebraically closed complete nonarchimedean field $C$ and $s$ is given by a rational point $y \in X(C)$.
        
        Since the morphism $\cl_s \colon s_*\LLambda_Y \to \LLambda_X$ in \cref{variant:cycle-class-map} is constructed from the cycle class $c\ell_s$ via the adjunction $(s_*,Rs^!)$, the image of $1 \in \Lambda$ under the compactly supported pushforward
        \[ \rR\Gamma_c\bigl(\cl_X(y)\bigr) \colon \Lambda \simeq \rR\Gamma_c(X,s_*\Lambda_y) \to \rR\Gamma_c\bigl(X,\Lambda(d)[2d]\bigr) \to \Hh^{2d}_c\bigl(X,\Lambda(d)\bigr) \]
        is the compactly supported cohomology cycle class $c\ell_X(y)$ from \cref{defn:cycle-clas-divisors-compact}, which similarly arises from composing with the counit of adjunction.
        In conclusion, it suffices to prove that for any separated taut smooth rigid space $f \colon X \to \Spa(C,\cO_C)$ and any rational point $y \in X(C)$, we have $\ttr_f\bigl(c\ell_X(y)\bigr) = 1 \in \Lambda$ for the compactly supported cycle class $c\ell_X(y) \in \Hh^{2d}_c\bigl(X,\Lambda(d)\bigr)$.
    
        By \cite[Cor.~1.6.10]{Huber-etale}, the point $y$ has a quasicompact open neighborhood $U \subseteq X$ such that $\restr{f}{U}$ factors through an \'etale map $U \to \DD^d_C$;
        we may assume that it sends $y$ to the origin $0 \in \DD^d_C(C)$. \Cref{lemma: cycle class of points compatible with etale morphism}, \cref{smooth-trace-constant}\cref{smooth-trace-constant-composition} and \cref{smooth-trace-constant}\cref{smooth-trace-constant-etale} applied to the diagram of pointed \'etale maps 
        \[ (X,y) \longleftarrow (U,y) \longrightarrow (\DD^d_C,0) \longrightarrow \bigl(\AA^{d,\an}_C,0\bigr) \]
        then allow us to deduce the general statement from that for $0 \in \AA^{d,\an}_C$. 
        This case was already treated in \cref{tr-vs-cl-AAd}. \qedhere
    \end{enumerate}
\end{proof}
\begin{lemma}\label{tr-epi}
    Let $f \colon X \to Y$ be a separated taut smooth of equidimension $d$ morphism of locally noetherian analytic adic spaces with $n \in \cO^\times_Y$ such that the fibers of $f$ are nonempty.
    Then the (truncation of the) trace map $\ttr_f \colon \rR^{2d}f_!\LLambda_X(d) \to \LLambda_Y$ is an epimorphism of \'etale sheaves on $Y$.
\end{lemma}
\begin{proof}
    By \cite[Prop.~2.5.5]{Huber-etale}, it suffices to check that for any $y \in \abs{Y}$, the induced maps on stalks $\bigl(\ttr_f\bigr)_{\overline{y}} \colon \bigl(\rR^{2d}f_!\LLambda_X(d)\bigr)_{\overline{y}} \to \LLambda_{Y, \overline{y}}$ at the geometric point $\overline{y} \colon \Bigl( \Spa\bigl(\widehat{\overline{k(y)}},\widehat{\overline{k(y)}}^+\bigr),\{y\}\Bigr) \to Y$ attached to $y$ is an epimorphism.
    This would follow from the assertion that $\ttr_f$ becomes an epimorphism after pullback along the natural map $\Spa\Bigl(\widehat{\overline{k(y)}},\widehat{\overline{k(y)}}^+\Bigr) \to Y$ through which $\overline{y}$ factors.
    Thanks to the proper base change isomorphism for morphisms of transcendence dimension $0$ \cite[Th.~5.4.6]{Huber-etale} and the compatibility of the smooth trace under pullbacks (\cref{smooth-trace-constant}\cref{smooth-trace-constant-pullback}), we may therefore assume that $Y = \Spa(C,C^+)$ for some algebraically closed nonarchimedean field $C$ with $n \in C^\times$ and some open and bounded valuation subring $C^+ \subset C$.

    By \cite[Cor.~1.6.10]{Huber-etale} applied to a point lying over the closed point of $Y = \Spa(C,C^+)$, we can pick an open $U \subseteq X$ for which $\restr{f}{U}$ is still surjective and factors as 
    \[ \begin{tikzcd}
        U \arrow[r,"g"] \arrow[rd,"\restr{f}{U}"'] & \DD^d_Y \arrow[d,"\pi_Y"] \\
        & Y
    \end{tikzcd} \]
    with $g$ \'etale.
    Since $g(U) \subseteq \DD^d_Y$ is open, $U \to g(U)$ is surjective, $\ttr^\et_g$ is given by summing over fibers and $\rR^{2d}\pi_{Y,!}$ is right exact, it suffices to prove the statement for $\restr{\pi_Y}{g(U)} \colon g(U) \to Y$.
    But $\restr{\pi_Y}{g(U)}$ has a section by \cite[Lem.~9.5]{Diamonds} because $g(U)$ still surjects onto $Y$;
    an application of \cref{tr-vs-cl} then finishes the proof.
\end{proof}
Next, we formulate and prove the comparison between our smooth trace and the ``usual'' one coming from algebraic geometry.
First, we fix some notation. Let $S=\Spa(A, A^+)$ be a strongly noetherian Tate affinoid, let $f \colon X \to Y$ be a separated finite type morphism between locally finite type $A$-schemes, and let relative analytification $f^{\an/S} \colon X^{\an/S} \to Y^{\an/S}$ be its relative analytification (see \cref{construction:relative-analytification}).
Then by \cite[(3.2.8)]{Huber-etale}, we have a commutative diagram of \'etale topoi
\[
\begin{tikzcd}
X^{\an/S}_\et \arrow{r}{c_{X/S}} \arrow{d}{f_\et^{\an/S}} & X_\et \arrow{d}{f_\et} \\
Y^{\an/S}_\et \arrow{d} \arrow{r}{c_{Y/S}} & Y_\et \arrow{d} \\
\Spa(A,A^+)_\et \arrow{r}{c_{S}} & (\Spec A)_\et.
\end{tikzcd}
\]
The algebraic pushforward with compact support and the analytic pushforward with compact support are related via the natural isomorphism of functors
$c_{Y/S}^*\bigl(\rR f_!(\blank)\bigr) \xrightarrow{\sim} \rR f^{\an/S}_!\bigl(c_{X/S}^*(\blank)\bigr)$,
see \cite[Th.~5.7.2]{Huber-etale}.
If the map $f$ is in addition smooth of equidimension $d$, it comes equipped with a trace map $\rm{tr}_f \colon \rR f_!\LLambda_X(d)[2d] \to \LLambda_Y$ (see \cite[Exp.~XVIII, Th.~2.9]{SGA4}). 
\begin{proposition}[Compatibility with algebraic geometry]
\label{Compatibility with algebraic geometry} Let $S=\Spa(A, A^+)$ be a strongly noetherian Tate affinoid, let $X$ and $Y$ be locally finite type $A$-schemes, and let $f \colon X \to Y$ be a finite type, smooth, separated morphism of equidimension $d$. Then we have $c_{Y/S}^*\bigl(\rm{tr}_f\bigr) = \rm{tr}_{f^{\an/S}}$.
\end{proposition}
\begin{proof}
First, we note that \cref{etale-trace-compatibility} proves the claim when $f$ is \'etale. Then we recall that Zariski-open immersions are taut (see \cite[Lem.~5.1.4]{Huber-etale}). Therefore, the established above case of \'etale maps and \cref{trace local source} imply that the statement is local on $X$, so we can further assume that the morphism $f\colon X \to Y$ factors as a composition of an \'etale map $g \colon X \to \AA^d_Y$ followed by the projection $\pi_Y \colon \AA^d_Y \to Y$. Using that both the algebraic trace maps are compatible with compositions and the established above case of an \'etale map, we conclude that it suffices to show the claim when $f$ is the relative affine line $\AA^1_Y \to Y$. 

In either case, we note that \cref{smooth trace lives in discrete space} and \cref{overconvergent-target} ensure that it suffices to check equality $c_{Y/S}^*\bigl(\rm{tr}_f\bigr) = \rm{tr}_{f^{\an/S}}$ on stalks at rank-$1$ geometric points. Using algebraic proper base change and weak analytic proper base change (see \cite[Th.~5.4.6]{Huber-etale}), we can assume that $A=C$, $A^+=\O_C$ for an algebraically closed non-archimedean field $C$, and $Y=\Spec C$. 

In this case, we can use the case of \'etale morphisms again to reduce the question to the case of the projective line $f\colon \PP^1_C\to \Spec C$. Then the result follows from \cref{smooth-trace-constant}\cref{smooth-trace-constant-P1}.
\end{proof}

Lastly, we show that our trace map is compatible with other constructions in the context of rigid geometry.
We begin with a comparison with Berkovich's trace, whenever both are defined.
Fix a complete nonarchimedean field $K$ with a valuation of rank $1$.
Recall that Huber constructed an equivalence of categories
\begin{equation}\label{adic-Berkovich}
    u \colon (A)' \colonequals \left\{ \parbox{1.95in}{\centering taut adic spaces locally of finite type over $\Spa(K,\cO_K)$} \right\} \longrightarrow \left\{ \parbox{1.8in}{\centering Hausdorff strictly $K$-analytic Berkovich spaces} \right\} \equalscolon (An) 
\end{equation}
which roughly sends an adic space $X$ in $(A)'$ to its maximal Hausdorff quotient;
see \cite[Rmk.~8.3.2]{Huber-etale}.
By \cite[p.~427,~(a)]{Huber-etale}, a morphism $f$ in $(A)'$ is partially proper and \'etale if and only if $u(f)$ is \'etale in the sense of Berkovich \cite[Def.~3.3.4]{Berkovich}.
Thus, for any $X \in (A)'$, the equivalence $u$ induces a morphism of topoi
\[ \theta_X \colon X_\et \longrightarrow u(X)_\et. \]
where $X_\et$ denotes Huber's \'etale topos used in our paper and $u(X)_\et$ the \'etale topos defined by Berkovich \cite[\S~4.1]{Berkovich};
cf.\ \cite[p.~426]{Huber-etale}.\footnote{The \'etale site from \cite[\S~4.1]{Berkovich} uses \'etale covers by arbitrary (not neccessarily \emph{stricly} $K$-analytic) Berkovich spaces, which are not included in the equivalence $u$.
Therefore, Huber first considers the morphism $\theta_X \colon \Et_{/X} \to s.\Et_{/u(X)}$ to the strict \'etale site of $u(X)$, a slightly modified version of the \'etale site $\Et_{/u(X)}$ that only includes covers in $(An)$.
Since the natural morphism $\Et_{/u(X)} \to s.\Et_{/u(X)}$ induces an equivalence of topoi \cite[Cor.~A.5]{Huber-etale}, this yields the mentioned comparison of the adic \'etale and the Berkovich \'etale topos.}
This morphism is fully faithful, with essential image the overconvergent sheaves \cite[Th.~8.3.5]{Huber-etale}.
Moreover, it is functorial in the following sense:
for any partially proper morphism $f \colon X \to Y$ in $(A)'$, there is a natural isomorphism of functors on $D^+_\et(u(X),\ZZ)$
\[ \alpha_f \colon \rR f_! \circ \theta^*_X \xlongrightarrow{\sim} \theta^*_Y \circ \rR u(f)_!, \]
where $\rR u(f)_!$ is Berkovich's derived pushforward with compact support from \cite[\S~5.1]{Berkovich};
cf.\ \cite[Prop.~8.3.6]{Huber-etale} and \cite[Th.~A.15]{Z-thesis}.

\begin{proposition}[Compatibility with the Berkovich trace]\label{prop:compatibility-berkovich-trace-1}
    Let $f \colon X \to Y$ be a partially proper, smooth of equidimension $d$ morphism of taut adic spaces that are locally of finite type over $\Spa(K,\cO_K)$.
    Let $u(f) \colon u(X) \to u(Y)$ be the separated smooth of equidimension $d$ morphism of Hausdorff Berkovich spaces that is associated with $f$ under the equivalence \cref{adic-Berkovich} and denote by $\ttr_{u(f)} \colon \rR u(f)_!\LLambda_{u(X)}(d)[2d] \to \rR^{2d}u(f)_!\LLambda_{u(X)}(d) \to \LLambda_{u(Y)}$ the trace morphism defined in \cite[Th.~7.2.1]{Berkovich}.
    Then the following natural diagram commutes:
    \[ \begin{tikzcd}[column sep=-2em]
        \rR f_!\LLambda_X(d)[2d] \simeq \rR f_!\theta^*_X\LLambda_{u(X)}(d)[2d] \arrow[rr,"\alpha_f","\sim"'] \arrow[rd,start anchor=230,"\ttr_f"'] && \theta^*_Y\rR u(f)_!\LLambda_{u(X)}(d)[2d] \arrow[ld,"\theta^*_Y\ttr_{u(f)}"] \\
        & \LLambda_Y \simeq \theta^*_Y\LLambda_{u(Y)} &
    \end{tikzcd} \]
\end{proposition}
Note that partially proper morphisms of locally noetherian analytic adic spaces are automatically separated and taut \cite[Def.~1.3.3~ii), Lem.~5.1.10~i)]{Huber-etale}, so $\ttr_f$ is indeed defined. 
Moreover, we use that the induced morphism $u(f)$ of Berkovich spaces is smooth in the sense of Berkovich because $f$ is partially proper and smooth \cite[Cor.~5.4.8]{Ducros} (cf.\ also \cite[Cor.~A.11]{Z-thesis}).
\begin{proof}
    First, we note that \cite[Lem.~A.19]{Z-thesis} proves the claim when $g$ is partially proper and \'etale. Thus \cref{trace local source} suffices that we can argue locally on $X^B$. Since the induced morphism $u(f)$ of Berkovich spaces is smooth in the sense of Berkovich, so any $x \in u(X)$ has an open neighborhood $V \subset u(X)$ such that $\restr{u(f)}{V}$ factors as
    \[ \begin{tikzcd}
        V \arrow[r,"h"] \arrow[rd,"\restr{u(f)}{V}"'] & \AA^{d,B}_Y \arrow[d,"u(\pi)"] \\
        & Y
    \end{tikzcd} \]
    such that $h$ is \'etale (cf.\ \cite[Def.~3.5.1]{Berkovich}).
    Applying $u^{-1}$, one obtains a factorization
    \[ \begin{tikzcd}
        U \arrow[r,"g"] \arrow[rd,"\restr{f}{U}"'] & \AA^{d,\an}_Y \arrow[d,"\pi"] \\
        & Y
    \end{tikzcd} \]
    with $g$ being partially proper and \'etale (see \cite[p.~427,~(a)]{Huber-etale}). Moreover, $U = u^{-1}(V) \subseteq X$ is open and partially proper inside $X$ because $u$ is defined by passing to the maximal Hausdorff quotient. Since both the analytic and the Berkovich trace are compatible with compositions, we conclude that it suffices to prove the claim when $f$ is partially proper \'etale or the relative affine line $\AA^{1, \an}_Y \to Y$. The case of a partially proper \'etale morphism was solved above, so we only need to consider the case of the relative affine line. 
    
    Since $\LLambda_Y$ is overconvergent, we may check the equality of the two trace morphisms after passing to stalks at geometric points of rank $1$ (\cref{smooth trace lives in discrete space} and \cref{overconvergent-target}).
    Thanks to \cite[Cor.~5.4.8]{Huber-etale} and \cite[Th.~5.3.1]{Berkovich}, we can therefore assume that $Y = \Spa(C,\cO_C)$ for some algebraically closed nonarchimedean field $C$, and $X=\AA^{1, \an}_C$. In this case, both the analytic trace and the Berkovich trace comes as an analytification of the algebraic trace for the schematic affine line $\AA^1_C \to \Spec C$. This finishes the proof. %
\end{proof}

\begin{corollary}\label{cor:compatibility-berkovich-trace-2} Let $f \colon X \to Y$ be a partially proper smooth of equidimension $d$ morphism between rigid-analytic spaces over $\Spa(K, \O_K)$. Then $\cal{H}^{2d}(\ttr_f) \colon \rR^{2d} f_! \ud{\Lambda}_X(d) \to \ud{\Lambda}_Y$ coincides with the Berkovich trace $t_f$ from \cite[Th.~5.3.3]{Z-thesis}.
\end{corollary}
\begin{proof}
    The question is local on $Y$, so we can assume that $Y$ is an affinoid. Then the result follows directly from \cref{prop:compatibility-berkovich-trace-1}.
\end{proof}

Assume now that $K$ is a $p$-adic field (i.e., a complete discrete valuation field of mixed characteristic $(0,p)$ whose residue field is perfect).
Set $C \colonequals \widehat{\overline{K}}$.
Recall that for every Zariski-compactifiable smooth rigid-analytic space $U$ of equidimension $d$ over $K$, Lan--Liu--Zhu also constructed in \cite[Th.~1.3]{LLZ} a $p$-adic rational trace
\[
t_{U, \et} \colon \rm{H}^{2d}_c\bigl(U_C, \mathbf{Q}_p(d)\bigr) \coloneqq \Bigl(\lim_r \rm{H}^{2d}_c\bigl(U_C,\Z/p^r\Z(d)\bigr)\Bigr)\Bigl[\frac{1}{p}\Bigr] \longrightarrow \mathbf{Q}_p.
\]
The trace maps from \cref{smooth-trace-constant} gives rise to another $p$-adic rational trace map $\rm{tr}_U \colon \rm{H}^{2d}_c\bigl(U_C, \mathbf{Q}_p(d)\bigr) \to \mathbf{Q}_p$ via the formula $\rm{tr}_U\coloneqq \bigl(\lim_r \Hh^{2d}(\rm{tr}_{\Z/p^r\Z})\bigr)\bigl[\frac{1}{p}\bigr]$. We claim that these two maps coincide:

\begin{lemma}\label{lemma:comparison-LLZ} 
In the situtation described above, we have $t_{U, \et} = \ttr_U \colon \rm{H}^{2d}_c\bigl(U_C, \mathbf{Q}_p(d)\bigr) \to \mathbf{Q}_p$.
\end{lemma}

\begin{proof}[Sketch of the proof.]
Using resolution of singularities (see \cite[Th.~5.2.2]{Tem12}), we can assume that $U$ admits a smooth proper (Zariski-)compactification $X$. 
By virtue of \cite[Th.~4.4.1~(1)]{LLZ} and \cref{smooth-trace-constant}\cref{smooth-trace-constant-etale}, it then suffices to prove the statement for the smooth and proper $X$.
After passing to a finite extension of $K$ and a component of $X$, we may further assume that $X$ is geometrically connected and admits a rational point $x \in X(K)$. 

Now we argue by induction on $d=\dim X$. If $d=0$, the claim is obvious.
Therefore, we assume that $d = \dim X>0$ and that the claim is known in dimensions $<d$.
\Cref{Poincare dualizability theorem} below (whose proof does not use \cref{lemma:comparison-LLZ}) and \cref{tr-vs-cl} imply that $\Hh^{2d}\bigl(X_C, \QQ_p(d)\bigr)$ is one-dimensional and that $\ttr_X\bigl(c\ell_X(x)\bigr)=1$.
Thus, it suffices to prove that $t_{X, \et}\bigl(c\ell_X(x)\bigr)=1$. 

An inspection of the proof of \cite[Th.~4.4.1]{LLZ} shows that we can further assume that there is an effective Cartier divisor $x\in E\xhookrightarrow{i} X$.
In this case, the proof of \textit{loc.\ cit.} guarantees that the composition\footnote{
We implicitly use that the Gysin map $i_* \ud{\QQ}_p \to \ud{\QQ}_p(1)[2]$ defined in \cite[Rmk.~4.3.12]{LLZ} coincides with $\cl_i$. For this, we note that \cite[Th.~3.21]{BH} implies that $\rR\cHom\bigl(i_* \ud{\QQ}_p, \ud{\QQ}_p(1)[2]\bigr) = i_*\ud{\QQ}_p$, so the question whether two maps between these complexes coincide is \'etale local on $X$.
Thus, it suffices to prove the claim for $\AA_k^{d-1,\an} \hookrightarrow \AA_k^{d,\an}$. In this case, both maps coincide with the analytification of the algebraic cycle class map due to \cref{lemma:analytification-cycle-classes} and the claim that the construction in \cite[Rmk.~4.3.12]{LLZ} coincides with the construction in \cite[\S~4]{Faltings-almost-extensions}.}
$\Hh^{2d-2}\bigl(E_C, \QQ_p(d-1)\bigr) \xr{\Hh^{2d}(X_C, \cl_i(d-1))} \Hh^{2d}\bigl(X_C, \QQ_p(d)\bigr) \xr{t_{X, \et}} \QQ_p$ is equal to $t_{E, \et}$. 
Combined with \cref{lemma:class-of-point-composes-well}, the statement is therefore reduced to proving that $t_{E, \et}\bigl(c\ell_E(x)\bigr)=1$.
This follows immediately from the induction hypothesis. 
\end{proof}

\begin{remark}
    In \cite[Cor.~3.10.22]{Mann-thesis}, Mann proves Poincar\'e duality for smooth proper morphisms of rigid-analytic spaces over a nonarchimedean field extension $K$ of $\QQ_p$ along the following lines:
    \begin{enumerate}[label={\upshape{(\arabic*)}},leftmargin=*]
        \item For any analytic adic space $X$ over $\QQ_p$ (or more generally small v-stack), he defines an $\infty$-category $\mathcal{D}^\ra_\square(\cO^+_X/p)^\varphi$ of ``almost quasicoherent solid $\varphi$-modules over $\cO^+_X/p$'' \cite[Th.~3.9.10~(b)]{Mann-thesis} together with a fully faithful ``Riemann--Hilbert functor'' from overconvergent \'etale $\FF_p$-sheaves
        \begin{equation}\label{Mann-RH}
            \blank \otimes^L \cO^{+,\ra}_X/p \colon \mathcal{D}_\et(X,\FF_p)^{\mathrm{oc}} \hookrightarrow \mathcal{D}^\ra_\square(\cO^+_X/p)^\varphi.
        \end{equation}
        This functor admits a right adjoint, which is (locally in the v-topology) roughly given by taking $\varphi$-invariants and induces an equivalence on perfect objects \cite[Th.~1.2.7]{Mann-thesis}.
        \item\label{Mann-6functor} He develops a $6$-functor formalism for $\mathcal{D}^\ra_\square(\cO^+_X/p)^\varphi$ \cite[Th.~1.2.4]{Mann-thesis}.
        \item For any smooth morphism $f \colon X \to Y$ of equidimension $d$ between analytic adic spaces over $\QQ_p$, he proves that $\rR f^!\bigl(\cO^{+,\ra}_Y/p\bigr) \simeq \cO^{+,\ra}_X/p(d)[2d]$ \cite[Th.~1.2.8]{Mann-thesis}.
        When $f$ is in addition proper, this leads, under the equivalence from \cref{Mann-RH}, to a Poincar\'e duality for perfect complexes of overconvergent \'etale $\FF_p$-sheaves;
        the resulting trace map corresponds to the counit $\rR f_! \rR f^! \bigl(\cO^{+,\ra}_Y/p\bigr) \simeq \rR f_*\bigl(\cO^{+,\ra}_X/p(d)[2d]\bigr) \to \cO^{+,\ra}_Y/p$.
    \end{enumerate}

    However, it does not seem straightforward to compare Mann's trace map and Poincar\'e duality isomorphism with the one from this paper.
    In fact, the 6-functor formalism in \cref{Mann-6functor} is not compatible with Huber's functors from \cite{Huber-etale} when the latter are defined.
    On the one hand, Huber's $\rR f_!$ need not preserve overconvergent sheaves unless $f$ is partially proper, hence cannot be given as the $\varphi$-invariants of Mann's compactly supported pushforward functor.
    
    On the other hand, Mann's compactly supported pushforward functor is not the image of Huber's compactly supported pushforward under \cref{Mann-RH}:
    For instance, when $f \colon X\coloneqq \accentset{\circ}{\DD}^1_C \to \Spa(C, \O_C)$ is the structure morphism of the open unit disk over an algebraically closed nonarchimedean field $C$ of mixed characteristic $(0,p)$, Mann proves that in his formalism the natural transformation of functors $\rR f^!\bigl(\cO^{+,\ra}_X/p\bigr) \otimes^L Lf^*(\blank) \to \rR f^!(\blank)$ is an equivalence \cite[Th.~3.10.17, Prop.~3.8.4~(i)]{Mann-thesis}.
    In particular, $\rR f^!(\blank)$ preserves filtered colimits, so its left adjoint $\rR f_!(\blank)$ preserves almost compact objects.
    If Mann's $\rR f_!\bigl(\cO^{+,\ra}_X/p\bigr)$ was isomorphic to the image of Huber's $\rR f_!(\ud{\FF}_p)$ under the fully faithful functor \cref{Mann-RH}, it would also be discrete and thus a perfect object of $\mathcal{D}^\ra_\square(\cO_C/p)^\varphi$;
    cf.\ \cite[Prop.~3.7.5, Def.~3.9.15]{Mann-thesis}.
    However, since \cref{Mann-RH} induces an equivalence of perfect objects, this would imply that Huber's $\Hh^2_c\bigl(\accentset{\circ}{\DD}^1_C,\FF_p\bigr)$ is finite-dimensional, contradicting \cref{example:no-proper-base-change}\cref{example:no-proper-base-change-2}.
\end{remark}

\subsection{Digression: K\"unneth Formula}

In this subsection, we establish a version of the K\"unneth Formula that we will crucially use in our proof of Poincar\'e duality in \cref{section:smooth-poincare-duality}. 
We recall that we fix a positive integer $n$ and put $\Lambda \colonequals \Z/n\Z$.

\begin{lemma}[Proper Base Change]\label{lemma:proper-base-change}
Consider a cartesian diagram of locally noetherian analytic adic spaces
\[
\begin{tikzcd}
    X' \arrow{r}{g'} \arrow{d}{f'} & X \arrow{d}{f} \\
    Y' \arrow{r}{g} & Y
\end{tikzcd}
\]
with $n\in \O_Y^\times$ and $f$ proper.
Let $\F\in D(X_\et; \Lambda)$ be a complex of \'etale sheaves such that, for each geometric point $\ov{y} \to Y$ of rank $1$, the restriction $\restr{\F}{X_{\ov{y}}}$ lies in $D_{zc}(X_{\ov{y}, \et}; \Lambda)$. Then the natural morphism
\[
\rm{BC}_{f, g}\colon g^* \rR f_* \F \to \rR f'_* g'^* \F
\]
is an isomorphism.
\end{lemma}
\begin{proof}
    First, the claim is \'etale local on $Y$ and $Y'$, so we can assume that they are both affinoid. Then \cite[Lem.~9.1~(1)]{adic-notes} implies that $f$ has bounded cohomological dimension. 
    Thus, a standard argument allows us to reduce to the case $\F\in \Shv(X_\et; \Lambda)$. In this case, \cite[Prop.~8.2.3]{Huber-etale} ensures that all cohomology sheaves of $g^* \rR f_* \F$ and $\rR f'_* g'^* \F$ are overconvergent. Therefore, it suffices to show that $\rm{BC}_{f, g}$ is an isomorphism at geometric rank-$1$ points.
    Thanks to \cite[Prop.~2.6.1]{Huber-etale}, we may thus assume that $Y'=\Spa(C', \O_{C'})$ and $Y=\Spa(C, \O_C)$ are both geometric points of rank $1$.
    Now if $\charac C >0$, the result follows from \cite[Th.~4.4.1~(a)]{Huber-etale}. If $\charac C = 0$, then the result follows from \cite[Th.~3.15]{BH}.
\end{proof}

\begin{remark}\label{use-perfectoid}
We note that \cite[Th.~3.15]{BH} crucially uses perfectoid spaces in its proof. However, this is the only instance where we need to use perfectoid spaces in this paper.
\end{remark}

We note that proper base change fails for more general coefficients: 

\begin{example}\label{example:no-proper-base-change} Let $C$ be an algebraically closed nonarchimedean field of mixed characteristic $(0, p)$. 
\begin{enumerate}[leftmargin=*]
    \item\label{example:no-proper-base-change-1} (constructible example\footnote{Here, we use the word ``constructible'' in the sense of \cite[Def.~2.7.2]{Huber-etale}. In particular, a sheaf which is both constructible and Zariski-constructible must be a local system due to \cite[Lem.~2.7.10]{Huber-etale}.}) Let $j\colon \DD^1_C \hookrightarrow \PP^{1, \an}_C$ be the natural inclusion of the closed unit disk into the analytic projective line. Then proper base change for the sheaf $\F = j_!\mu_p$ would, in particular, imply that $\rR\Gamma_c(\DD^1_C, \mu_p)$ does not depend on the choice of the algebraically closed ground field $C$.
    However, $\Hh^2_c(\DD^1_C, \mu_p)$ does depend on $C$ due to \cref{distance and cycle class}:
    one can take $C\subset C'$ such that the cardinality of the residue field of $C'$ is bigger than $\Hh^2_c(\DD^1_C, \mu_p)$.
    \item\label{example:no-proper-base-change-2} (overconvergent example) Let $j\colon \accentset{\circ}{\DD}^1_C \hookrightarrow \PP^{1, \an}_C$ be the natural inclusion of the open unit disk into the analytic projective line.
    Similarly, proper base change for the sheaf $\F = j_!\mu_p$ would, in particular, imply that $\rR\Gamma_c(\accentset{\circ}{\DD}^1_C, \mu_p)$ does not depend on the choice of the algebraically closed ground field $C$.
    The closed complement to $j$ is equal to $\DD^{1, c}_C$, the universal compactification of the closed unit disc. Therefore, the excision sequence and \cref{cohomology-affine-curve} imply that we have the following exact triangle:
    \[
    \rR \Gamma_c(\accentset{\circ}{\DD}^1_C, \mu_p) \to \rR \Gamma(\PP^{1, \an}_C,\mu_p) \to \rR \Gamma(\DD^1_C, \mu_p)
    \]
    Hence, proper base change for $j_!\mu_p$ would, in particular, imply that $\rR\Gamma(\DD^1_C, \mu_p)$ is independent of the choice of algebraically closed $C$. However, $\Hh^1(\DD^1_C, \mu_p)$ does depend on $C$ due to \cref{rmk:infinite-cohomology}:
    one can take $C\subset C'$ such that the cardinality of $\m_{C'}/p\m_{C'}$ is bigger than $\Hh^1(\DD^1_C, \mu_p)$).
\end{enumerate}
\end{example}

\begin{construction}[K\"unneth map]\label{construction:kunneth-map}
Consider the following commutative diagram of locally noetherian analytic adic spaces:
\begin{equation}\label{eqn:kunneth-data}
\begin{tikzcd}
W \arrow[r,"g"] \arrow[d,"g'"] \arrow[rd, "h"] & X \arrow[d,"f"] \\
X' \arrow[r,"f'"] & Y
\end{tikzcd}
\end{equation}
Let $\cal{E}\in D(X_\et; \Lambda)$ and $\cal{E}'\in D(X'_\et; \Lambda)$. We define the \emph{K\"unneth map}
\[
\rm{KM}\colon \rR f_*\cal{E} \otimes^L \rR f'_*\cal{E}' \to \rR h_*(g^*\cal{E} \otimes^L g'^*\cal{E}')
\]
as the adjoint to the map
\[
h^*\bigl(\rR f_*(\cal{E}) \otimes^L \rR f'_*(\cal{E}')\bigr) \simeq (g^* f^* \rR f_*\cal{E}) \otimes^L (g'^* f'^* \rR f'_* \cal{E}') 
\xrightarrow{g^*(\epsilon_f) \otimes^L g'^*(\epsilon_{f'})} g^*\cal{E} \otimes^L g'^*\cal{E}',
\]
where $\epsilon_f$ (resp.\ $\epsilon_{f'}$) denotes the counit of the $(f^*, \rR f_*)$-adjunction (resp.\ the $(f'^*, \rR f'_*)$-adjunction).
\end{construction}

The K\"{u}nneth map is functorial in both variables $\cal{E}$ and $\cal{E}'$, and in Diagram~\cref{eqn:kunneth-data}. 

\begin{remark}\label{cup-product}
We note that the K\"unneth map boils down to the cup-product map (see \cite[\href{https://stacks.math.columbia.edu/tag/0B6C}{Tag~0B6C}]{stacks-project}) when $W=X=X'$, $f=f'$, and $g=g'=\rm{id}_X$. 
\end{remark}

\begin{remark}
Now consider the case $X=X'$, $f=f'$, and $W=X\times_Y X$ with $g$ and $g'$ being the natural projection maps. 
In this situation, the functoriality in $W$ (with respect to the morphism $\Delta \colon X \to X\times_Y X$) implies that the composition
\begin{multline*}
\rR f_*\cal{E} \otimes^L \rR f_* \cal{E}'
\xrightarrow{\text{KM}}
\rR h_*(g^*\cal{E} \otimes^L g'^*\cal{E}') 
\xrightarrow{\rR h_*(\eta_{\Delta})}  \rR h_*\bigl(\rR \Delta_*\Delta^*(g^*\cal{E} \otimes^L g'^*\cal{E}') \bigr) = \\ = \rR f_*(\Delta^*\left(g^*\cal{E} \otimes^L g'^*\cal{E}'\right)) \simeq \rR f_*(\cal{E} \otimes^L \cal{E}')
\end{multline*}
is equal to the cup-product morphism, where $\eta_{\Delta}$ denotes the unit of the $(\Delta^*, \rR\Delta_*)$-adjunction. 
\end{remark}

Now we wish to prove that under some assumptions, the K\"unneth map is an isomorphism. For this, we will need the following very general lemma:

\begin{lemma}
\label{claim factorizing Kunneth map}
Consider a commutative diagram of locally noetherian analytic adic spaces
\[
\begin{tikzcd}
W \arrow[r,"g"] \arrow[d,"g'"] \arrow[rd, "h"] & X \arrow[d,"f"] \\
X' \arrow[r,"f'"] & Y.
\end{tikzcd}
\]
Let $\cal{E} \in D(X_\et; \Lambda)$ and $\cal{E}' \in D(X'_\et; \Lambda)$.
Then the following diagram commutes:
\[
\begin{tikzcd}[column sep =5em]
\rR f_* \cal{E} \otimes^L \rR f'_* \cal{E}' \arrow[d,"\rm{KM}"] \arrow[r,"\PF_f"] & \rR f_*\bigl(\cE \otimes^L (f^*\rR f'_* \cE')\bigr) \arrow[d,"\rR f_*(\id \otimes \BC_{f',f})"] \\
\rR h_*(g^*\cE \otimes^L g'^*\cE') \simeq \rR f_* \rR g_*(g^*\cE \otimes^L g'^*\cE')& \rR f_*\bigl(\cE \otimes^L (\rR g_* g^{\prime *}\cE')\bigr) \arrow[l,"\rR f_*(\PF_g)"].
\end{tikzcd}
\]
Here, $\rm{BC}$ stands for the base change morphism, and $\rm{PF}$ stands for the projection formula morphism (see \cite[\href{https://stacks.math.columbia.edu/tag/07A7}{Tag 07A7}]{stacks-project} and \cite[\href{https://stacks.math.columbia.edu/tag/0B56}{Tag 0B56}]{stacks-project}). 
\end{lemma}

\begin{proof}
In this proof, we denote the counit of the $(f^*, \rR f_*)$-adjunction by $\epsilon_f$. Then we consider the following diagram:
\[
\begin{tikzcd}[scale cd=.8,center picture,column sep =5em]
h^*(\rR f_*\cal{E} \otimes^L \rR f'_*\cal{E}') \arrow[rd,"g^*(\epsilon_f \otimes \id)"'] \arrow[r,"h^*(\PF_f)"] & h^*{\rR f_*}(\cE \otimes^L f^*\rR f'_*\cE') \arrow[d,"g^*(\epsilon_f)"] \arrow[r,"h^*\rR f_*(\id \otimes \BC_{f',f})"] & h^*\rR f_*(\cE \otimes^L \rR g_*g^{\prime *}\cE') \arrow[d,"g^*(\epsilon_f)"] \arrow[r,"h^*(\rR f_*(\PF_g))"] & h^*\rR f_*\rR g_*(g^*\cE \otimes^L g'^*\cE') \arrow[d,"g^*(\epsilon_f)"] \\
& g^*(\cE \otimes^L f^*\rR f'_* \cE') \arrow[r,"g^*(\id \otimes \rm{BC}_{f',f})"] \arrow[dr,sloped,"\sim"'] & g^*(\cE \otimes^L \rR g_*g^{\prime *}\cE') \arrow[dr,"\id \otimes \epsilon_g"] \arrow[r,"g^*(\PF_g)"] & g^*\rR g_*(g^*\cal{E} \otimes^L g^{\prime *}\cal{E}') \arrow[d,"\epsilon_g"] \\
& & g^*\cE \otimes^L g'^* f^{\prime *}\rR f'_* \cE'  \arrow[r,"\id \otimes g^{\prime *}(\epsilon_{f'})"] & g^*\cE \otimes^L g'^*\cE',
\end{tikzcd}
\]
where we implicitly identify $g^*\circ f^* \simeq h^* \simeq g'^*\circ f'^*$.
Note that the diagram above commutes:
indeed, the top left and bottom right triangles commute due to the definition of the projection formula morphism, the bottom parallelogram commutes due to the definition of the base change morphism, and the top two squares commute due to the functoriality of $\epsilon_f$.
Now it only remains to observe that if we go from the top left to the bottom right corner in the clockwise direction in the above diagram, we get the $(h^*, \rR h_*)$-adjoint of the composition $\rR f_*(\PF_g) \circ \rR f_*(\id \otimes \BC_{f',f}) \circ \PF_f$, whereas if we go counterclockwise, we get the morphism $g^*(\epsilon_f) \otimes g'^*(\epsilon_{f'})$. By definition, the latter is the $(h^*, \rR h_*)$-adjoint of the K\"unneth map from \cref{construction:kunneth-map}.
\end{proof}

\begin{corollary}[K\"{u}nneth formula]
\label{Kunneth formula}
Consider a cartesian diagram of locally noetherian analytic adic spaces
\[
\begin{tikzcd}
    X\times_Y X' \arrow{d}{g'} \arrow{rd}{h} \arrow{r}{g} & X \arrow{d}{f} \\
    X' \arrow{r}{f'} & Y
\end{tikzcd}
\]
with $n \in \cO^\times_Y$.
Let $\cal{E} \in D(X_\et; \Lambda)$ and $\cal{E}' \in D(X'_\et; \Lambda)$ such that, for each geometric point geometric point $\ov{y} \to Y$ of rank $1$, the restriction $\restr{\cal{E}}{X_{\ov{y}}}$ lies in $D_{zc}(X_{\ov{y}, \et}; \Lambda)$.
If $f$ and $f'$ are proper, then the K\"unneth map 
\[
\rm{KM} \colon \rR f_*\cal{E} \otimes^L \rR f'_*\cal{E}' \to
\rR h_*(g^*\cal{E} \otimes^L g'^*\cal{E'})
\]
is an isomorphism.
\end{corollary}
\begin{proof}
    This follows directly from \cref{claim factorizing Kunneth map}, proper base change (see \cref{lemma:proper-base-change}), and the projection formula (see \cite[Prop.~9.3~(2)]{adic-notes}). 
\end{proof}

\subsection{Poincar\'{e} duality for locally constant coefficients, revisited}\label{section:smooth-poincare-duality}

In this subsection, we prove Poincar\'e duality for smooth proper morphisms and lisse sheaves. Our proof will simultaneously show that higher direct images along smooth proper morphisms preserve lisse sheaves. Somewhat surprisingly, our proofs will be essentially formal and diagrammatic in nature.
As before, we fix a positive integer $n$ and set $\Lambda \colonequals \ZZ/n$.

We first verify both claims for the subclass of perfect lisse complexes, and then extend the results to the general case. For this, we recall that perfect lisse complexes have the following categorical description: for a locally noetherian analytic adic space $X$, the category of perfect complexes in $D(X_\et; \Lambda)$ coincides with the category dualizable\footnote{
For the notion of dualizable objects (or objects having ``left dual''),
we refer the reader to \cite[\href{https://stacks.math.columbia.edu/tag/0FFP}{Tag 0FFP}]{stacks-project}.} objects and this category is contained in the full subcategory $D^{(b)}_\lisse(X_\et; \Lambda)$ of locally bounded complexes of \'etale sheaves with lisse cohomology sheaves; see e.g.\ \cite[\href{https://stacks.math.columbia.edu/tag/0FPU}{Tag~0FPU}, \href{https://stacks.math.columbia.edu/tag/0FPV}{Tag~0FPV}]{stacks-project}, and \cite[Lem.~11.1]{adic-notes}. 

\begin{theorem}[{cf.\ \cite[Th.~1.1.2]{Z-thesis} and \cite[Cor.~1.2.9]{Mann-thesis}}]
\label{Poincare dualizability theorem}
Let $X$ and $Y$ be locally noetherian analytic adic spaces such that $n \in \cO^\times_Y$, let $f \colon X \to Y$ be a smooth proper morphism of equidimension $d$, and let $\cal{E} \in D(X_\et; \Lambda)$ be a dualizable object with left dual $\cal{E}^{\vee} \coloneqq \rR\cHom_X(\cal{E}, \LLambda_X)$.
Then the evaluation and coevaluation maps from \cref{evaluation and coevaluation maps} below make 
$\rR f_*(\cal{E}^{\vee}(d)[2d])$ into a left dual of $\rR f_*\cal{E}$.
In particular $\rR f_*\cal{E}$ is a dualizable object of $D(Y_\et; \Lambda)$ and there is a natural isomorphism 
\[
\PD_f(\cal{E}) \colon \rR f_*(\cal{E}^{\vee}(d)[2d]) \xrightarrow{\sim} \rR\cHom_Y(\rR f_*\cal{E}, \LLambda_Y).
\]
\end{theorem}
We first explain how to construct the evaluation map, the coevalution map, and the duality map $\PD_f(\EE)$ in the statement of \cref{Poincare dualizability theorem}. 
Later, we check that these maps indeed define the structure of a dualizable object on $\rR f_*\cal{E}$. 

\begin{construction}\label{construction:evaluation-coevaluation-poincare-duality}
\label{evaluation and coevaluation maps}
Let $f \colon X \to Y$ be as in \cref{Poincare dualizability theorem}. 
Let $\cal{E}\in D_\lisse(X; \Lambda)$ and let $\cal{E}^{\vee} \coloneqq \rR\cHom_X(\cal{E}, \LLambda_X)$ be its (naive) dual. We denote by $ev_{\cal{E}} \colon \cal{E}^{\vee} \otimes \cal{E} \to \LLambda_X$ the natural evaluation map. 
If $\cal{E}$ is dualizable, we also denote by $coev_{\cal{E}}\colon \LLambda_X \to \cal{E}^{\vee} \otimes \cal{E}$ its coevaluation map. 
Moreover, we use the notation of the following commutative diagram:
\[ \begin{tikzcd}
X \arrow[rd,"\Delta"]\arrow[rdd, swap, bend right, "\rm{id}"] \arrow[rrd, bend left, "\rm{id}"]& & \\
& X\times_Y X \arrow{r}{\pi_2} \arrow{d}{\pi_1} \arrow{rd}{h} & X \arrow{d}{f} \\
& X \arrow[r, swap, "f"] & Y
\end{tikzcd} \]

\begin{enumerate}[label=\upshape{(\roman*)},leftmargin=*]
\item\label{evaluation map} (\emph{Evaluation map}) We define the \emph{evaluation map} $e(f, \cal{E}) \colon \rR f_*(\cal{E}^{\vee}(d)[2d]) \otimes^L \rR f_*(\cal{E}) \to \LLambda_Y$ as the composition
\[ \rR f_*(\cal{E}^{\vee}(d)[2d]) \otimes^L \rR f_*(\cal{E})
\xlongrightarrow{\cup} \rR f_*(\cal{E}^{\vee}(d)[2d] \otimes^L \cal{E})
\xrightarrow{\rR f_*(ev_{\cal{E}}(d)[2d])} \rR f_*(\LLambda_X(d)[2d])
\xlongrightarrow{\ttr_f} \LLambda_Y, \]
where $\cup$ is the cup-product map from \cite[\href{https://stacks.math.columbia.edu/tag/0B6C}{Tag~0B6C}]{stacks-project} (\cref{cup-product}) and $\ttr_f$ is the trace morphism from \cref{smooth-trace-constant}.
\item\label{PD map} (\emph{Duality map}) We define the \emph{duality map} $\rm{PD}_f(\cal{E}) \colon \rR f_*(\cal{E}^{\vee}(d)[2d]) \to \rR \cHom_Y(\rR f_*\cal{E}, \LLambda_Y)$ as the map adjoint to $e(f, \cal{E})$ under the tensor-hom adjunction.
In other words, $\PD_f(\cal{E})$ is the composition
\[
\rR f_*(\cal{E}^{\vee}(d)[2d]) \longrightarrow \rR \cHom_Y(\rR f_*\cal{E}, \rR f_*\LLambda_X(d)[2d]) \xrightarrow{\ttr_f \circ -} \rR \cHom_Y(\rR f_*\cal{E}, \LLambda_X),
\]
where the first map comes from \cite[\href{https://stacks.math.columbia.edu/tag/0B6D}{Tag~0B6D}]{stacks-project}.
\item\label{coevaluation map} (\emph{Coevaluation map}) Now we also assume that $\cal{E}$ is dualizable. We define the \emph{coevaluation map} 
\[
c(f, \cal{E}) \colon \LLambda_Y \to \rR f_*\cal{E} \otimes^L \rR f_*(\cal{E}^{\vee}(d)[2d])
\]
as the composition
\[ \begin{tikzcd}[nodes in empty cells,center picture,scale cd=.9,column sep=large]
\LLambda_Y \arrow[r,"\eta_f"] & \rR f_*\LLambda_X \arrow[r,"\rR f_*(coev_{\cal{E}})"] &[1em] \rR f_*(\cal{E}\otimes^L \cal{E}^{\vee}) \arrow[r,"\sim"] & \rR h_*(\rR \Delta_* \Delta^* (\pi_1^*\cal{E} \otimes^L \pi_2^*\cal{E}^{\vee})) \arrow[d,"\rR h_*(\rm{PF}^{-1}_{\Delta})","\sim"'{sloped}] \\
\rR f_*(\cal{E}) \otimes^L \rR f_*(\cal{E}^{\vee}(d)[2d]) & \multi{r}{{\rR h_*\left(\pi_1^*\cE \otimes^L \pi_2^*\cE^\vee(d)[2d]\right)}} && \rR h_*\left((\pi_1^*\cE \otimes^L \pi_2^*\cE^\vee) \otimes^L \Delta_*(\LLambda_X)\right),
\arrow[from=A,to=2-1, "\sim"',"\KM^{-1}"] 
\arrow[from=2-4,to=A,"\rR h_*(\rm{id} \otimes^L \cl_{\Delta})"]
\end{tikzcd} \]
where $\rm{cl}_{\Delta}\coloneqq \cl_{X \times_Y X}(X)$ is the cycle class map\footnote{Note that $\Delta$ is an lci immersion of pure codimension $d$ due to \cite[Cor.~5.11]{adic-notes}.} introduced in \cref{variant:cycle-class-map}, $\rm{PF}$ is the projection formula map, and $\rm{KM}$ is the K\"unneth map introduced in \cref{construction:kunneth-map}.
We crucially use \cite[Prop.~9.3~(2)]{adic-notes} and \cref{Kunneth formula} to invert the projection formula map and the K\"unneth map, respectively. 
\end{enumerate}
\end{construction}

Given the evaluation and coevaluation maps, the proof of \cref{Poincare dualizability theorem} essentially boils down to verifying that some diagrams commute. To do this, we need some preliminary lemmas: 

\begin{lemma}\label{lemma:base-change-trace} Keep the notation of \cref{construction:evaluation-coevaluation-poincare-duality}. Then the diagram
\[
\begin{tikzcd}[column sep = huge]
    \rR f_* \cal{E} \otimes^L \rR f_*\LLambda_X(d)[2d] \arrow{d}{\rm{KM}} \arrow{r}{\rm{id} \otimes^L \ttr_f} & \rR f_* \cal{E} \\
    \rR h_*(\pi_1^*\cal{E} \otimes^L \LLambda_{X\times_Y X}(d)[2d]) \arrow{r}{\rR f_*(\rm{PF}^{-1}_{\pi_1})} &\rR f_*(\cal{E} \otimes^L \rR \pi_{1, *} \Lambda_{X\times_Y X}(d)[2d]) \arrow{u}{\rR f_*(\rm{id} \otimes^L \ttr_{\pi_1})} \\
\end{tikzcd}
\]
commutes.
\end{lemma}
\begin{proof}
    First, we note that it suffices to show that the following diagram commutes:
    \[ \begin{tikzcd}[column sep = 8em]
        \rR f_* \cal{E} \arrow[d, equals] & \arrow[l, swap, "\rm{id} \otimes^L \ttr_f"] \rR f_*\cal{E} \otimes^L \rR f_* \LLambda_X(d)[2d] \arrow{d}{\rm{PF}_f} \arrow{r}{\rm{KM}} & \rR h_*(\pi_1^*\cal{E} \otimes^L \LLambda_{X\times_Y X}(d)[2d]) \\
        \rR f_* \cal{E} & \arrow[l, "\rR f_*(\rm{id}\otimes^L f^*\ttr_f)"'] \rR f_*(\cal{E} \otimes f^*\rR f_*\LLambda_{X}(d)[2d]) \arrow[r,"\rR f_*(\id \otimes^L \rm{BC}_{f, f})"] & \rR f_*(\cal{E} \otimes^L \rR \pi_{1, *} \LLambda_{X\times_Y X}(d)[2d]) \arrow[ll, bend left=10, "\rR f_*(\rm{id} \otimes^L \ttr_{\pi_1})"] \arrow[u,"\rR f_*(\rm{PF}_{\pi_1})"]
    \end{tikzcd} \]
    For this, one can easily check that the left square commutes using the very definition of the projection formula morphism (see \cite[\href{https://stacks.math.columbia.edu/tag/0B56}{Tag 0B56}]{stacks-project}).
    \Cref{claim factorizing Kunneth map} ensures that the right square commutes and \cref{smooth-trace-constant}\cref{smooth-trace-constant-pullback} guarantees that the bottom triangle commutes.
\end{proof}

For the next lemma, we need to introduce a new construction:

\begin{construction}\label{complex-cycle-class}
Let $i\colon X \hookrightarrow Y$ be an lci closed immersion of pure codimension $c$ and let $\F$ be an object of $D(Y_\et; \Lambda)$. 
Then we define the \emph{cycle class morphism} $\rm{cl}_{i}(\F) \colon i_* i^* \F \to \F \otimes \LLambda_Y(c)[2c]$ as the composition
\[ i_* i^* \F \xlongrightarrow[\sim]{\rm{PF}^{-1}_i} \F \otimes^L  i_* \LLambda_X \xlongrightarrow{\rm{id}\otimes^L \rm{cl}_i} \F \otimes^L \LLambda_Y(c)[2c], \]
where $\rm{cl}_i$ is the cycle class map from \cref{variant:cycle-class-map}.
\end{construction}
\begin{lemma}\label{lemma:preliminary-duality}
Let $X$ and $Y$ be locally noetherian analytic adic spaces with $n \in \cO^\times_Y$, let $f \colon X \to Y$ be a smooth proper morphism of equidimension $d$.
Assume that $f$ has a section $s\colon Y \to X$.
Then the compositions
\begin{gather*}
\LLambda_Y \simeq \rR f_*(s_* \LLambda_Y) \xrightarrow{\rm{cl}_s} \rR f_*(\LLambda_X(d)[2d]) \xrightarrow{\ttr_f} \LLambda_Y \quad \text{and} \\
\begin{split}
\LLambda_Y(d)[2d] \simeq \rR f_*\bigl(s_* \LLambda_Y(d)[2d]\bigr) \xrightarrow{\rR f_*(\rm{cl}_s(\LLambda(d)[2d]))} & \rR f_*\bigl(\LLambda_X(d)[2d] \otimes^L \LLambda_X(d)[2d]\bigr) \\ 
& \xrightarrow[\sim]{\PF^{-1}_f} \rR f_*\bigl(\LLambda_X(d)[2d]\bigr) \otimes^L \LLambda_Y(d)[2d] \xrightarrow{\ttr_f\otimes^L \rm{id}} \LLambda_Y(d)[2d]
\end{split}
\end{gather*}
are equal to the identity morphisms.
\end{lemma}
\begin{proof}
    For brevity, we denote the first composition by $\alpha$ and the second composition by $\beta$.
    For any objects $\F, \G\in D(X_\et; \Lambda)$, we denote by $\sigma\colon \F \otimes^L \G \xr{\sim} \G\otimes^L \F$ the morphism that ``swaps factors''.
    Then \cref{tr-vs-cl} directly implies that $\alpha=\id$.
    
    To see that $\beta = \id$, we now prove the stronger assertion that $\beta = \alpha(d)[2d]$.
    Since $\LLambda_Y(d)[2d]$ is a locally free sheaf concentrated in degree $2d$, our sign conventions for the commutativity constraint (see \cite[\href{https://stacks.math.columbia.edu/tag/0GWN}{Tag~0GWN}]{stacks-project}) imply that the morphism $\sigma \colon \ud{\Lambda}_{X}(d)[2d] \otimes^L \ud{\Lambda}_{X}(d)[2d]  \to \ud{\Lambda}_{X}(d)[2d] \otimes^L \ud{\Lambda}_{X}(d)[2d]$ is given by multiplication with $(-1)^{(2d)\cdot (2d)}=1$;
    that is, it is the identity morphism.
    This observation implies that the following diagram commutes: 
    \[ \begin{tikzcd}[scale cd=.8,center picture,column sep=large]
        \rR f_*s_* \ud{\Lambda}_Y(d)[2d] \arrow[r,blue,"\rR f_*(\PF^{-1}_s)"] & \rR f_*\bigl(\ud{\Lambda}_X(d)[2d] \otimes^L s_*\ud{\Lambda}_Y\bigr)  \arrow[r,blue,"\rR f_*(\id \otimes^L \cl_s)"] \arrow[d,"\PF^{-1}_f"] &[1em] \rR f_*\bigl(\ud{\Lambda}_X(d)[2d] \otimes^L \ud{\Lambda}_X(d)[2d]\bigr) \arrow[r,blue,"\rR f_*(\sigma)=\rm{id}","\sim"'] \arrow[d,"\PF^{-1}_f"]  & \rR f_*\bigl(\ud{\Lambda}_X(d)[2d] \otimes^L \ud{\Lambda}_X(d)[2d]\bigr) \arrow[d,blue,"\PF^{-1}_f"] \\
        \ud{\Lambda}_Y(d)[2d] \arrow[u,blue,sloped,"\sim"] \arrow[r,red,"\sim"] & \ud{\Lambda}_Y(d)[2d] \otimes^L \rR f_* s_* \ud{\Lambda}_Y \arrow[r,red,"\id \otimes^L \cl_s"]  & \ud{\Lambda}_Y(d)[2d] \otimes^L \rR f_* \ud{\Lambda}_X(d)[2d] \arrow{r}{\sigma} \arrow[d,red,"\id \otimes^L \ttr_f"] & \rR f_* \ud{\Lambda}_X(d)[2d] \otimes^L \ud{\Lambda}_Y(d)[2d] \arrow[ld,blue,"\ttr_f \otimes^L \id"] \\
        & & \ud{\Lambda}_Y(d)[2d]. & 
    \end{tikzcd} \]
    Therefore, the map $\beta \colon \ud{\Lambda}_Y(d)[2d] \to \ud{\Lambda}(d)[2d]$ (given by the blue arrows in the diagram above) is equal to $\id\otimes \alpha = \alpha(d)[2d]$ (given by the red arrows), as desired.
\end{proof}

\begin{proof}[Proof of \cref{Poincare dualizability theorem}]
    The only thing we need to check is that the compositions\footnote{In these formulas, we implicitly make the usual identifications $\rR f_* \cal{E} \otimes^L \LLambda_Y \simeq \rR f_* \cal{E}$, $\LLambda_Y \otimes^L \rR f_*\cal{E} \simeq \rR f_*\cal{E}$, etc.
    To indicate this subtlety, we use quotation marks in the maps which implicitly use these identifications.}
    \[
    \begin{tikzcd}[column sep = huge]
    \rR f_* \cal{E} \arrow{r}{``c(f, \cal{E}) \otimes^L \rm{id}\text{''}} & \rR f_* \cal{E} \otimes^L \rR f_* \cal{E}^{\vee}(d)[2d] \otimes^L \rR f_* \cal{E} \arrow{r}{``\rm{id}\otimes^L e(f, \cal{E})\text{''}}  & \rR f_* \cal{E},
    \end{tikzcd}
    \]
    \[
    \begin{tikzcd}[column sep = huge]
    \rR f_* \cal{E}^{\vee}(d)[2d] \arrow{r}{``\rm{id} \otimes^L c(f, \cal{E})\text{''}} & \rR f_* \cal{E}^{\vee}(d)[2d] \otimes^L \rR f_* \cal{E} \otimes^L \rR f_* \cal{E}^{\vee}(d)[2d] \arrow{r}{``e(f, \cal{E}) \otimes^L \rm{id}\text{''}} & \rR f_* \cal{E}^{\vee}(d)[2d]
    \end{tikzcd}
    \]
    are equal to the identity morphisms. For brevity, we denote the first composition by $\varphi(f, \cal{E})$ and the second composition by $\psi(f, \cal{E})$. We give a full justification for why $\varphi(f, \cal{E}) = \rm{id}$, and then only describe the necessary changes to justify that $\psi(f, \cal{E}) = \rm{id}$.

    Recall that we denote by $\eta$ the unit of the (derived) pullback-pushforward adjunction, by $\KM$  the K\"unneth map from \cref{construction:kunneth-map}, by $\PF$ the projection formula map, and, for any objects $\F, \G\in D(X_\et; \Lambda)$, by $\sigma\colon \F \otimes^L \G \xr{\sim} \G\otimes^L \F$ the morphism that ``swaps factors''.
    In what follows, we also freely use \cref{Kunneth formula} and \cite[Prop.~9.3~(2)]{adic-notes} which guarantee that the K\"unneth map and the projection formula map are isomorphisms under some assumptions that are always satisfied in this proof. 

    That being said, we resume the notation of \cref{construction:evaluation-coevaluation-poincare-duality} and consider the diagram in \cref{gigantic-diagram}. 
    Using the definitions of $\PF$, $\KM$, and basic properties of adjunctions, one can check that this diagram commutes. 
    For the most part, the verification is very similar to that in the proof of \cref{claim factorizing Kunneth map} with the following two exceptions: 
    triangle~\cref{gigantic-diagram-evcoev} commutes due to the assumption $ev_{\cal{E}}$ and $coev_{\cal{E}}$ define a duality datum on $\cal{E}$, and trapezoid~\cref{gigantic-diagram-base-change-trace} commutes due to \cref{lemma:base-change-trace}. 
    
    \begin{sidewaysfigure}

    \thisfloatpagestyle{empty}
    \tiny

    \vspace{60em}
    \[ \begin{tikzcd}[center picture,row sep=huge, column sep = 8em]
    \LLambda \otimes \rR f_*\cal{E} \arrow[rr, "\PF_f"] \arrow[d,"\eta_f \otimes \rm{id}"] & & \rR f_*\bigl(\LLambda \otimes \cal{E}\bigr) \arrow[d,"\rR f_*(\eta_{\pi_2})"] \arrow[rd, start anchor=east,bend left=15,"\rR h_*(\PF^{-1}_\Delta)"] && \\
    \rR f_* \LLambda \otimes \rR f_* \cal{E} \arrow[rr,"\KM"] \arrow[d,"\rR f_*(coev_{\cal{E}}) \otimes \id"] & & \rR h_*\bigl(\LLambda \otimes \pi^*_2 \cal{E} \bigr) \arrow[d,"\rR h_*\bigl(\pi^*_1(coev_{\cal{E}}) \otimes \id\bigr)"] \arrow{r}{\rR h_*(\eta_{\Delta} \otimes \rm{id})} & \rR h_*\bigl(\Delta_*\LLambda \otimes \pi_2^* \cal{E}\bigr)  \arrow{d}{\rR h_*\bigl(\Delta_*(coev_{\cal{E}})\otimes \rm{id}\bigr)} \arrow[dr, start anchor=east, bend left=15, "\rR h_*(\PF_\Delta)"]& \\
    \rR f_* \bigl(\cal{E} \otimes \cal{E}^\vee\bigr) \otimes \rR f_* \cal{E} \arrow[rr,"\KM"] \arrow[d,sloped, equals] & & \rR h_*\bigl(\pi^*_1\cal{E} \otimes \pi^*_1\cal{E}^\vee \otimes \pi^*_2\cal{E}\bigr) \arrow{r}{\rR h_*(\eta_{\Delta}\otimes \rm{id})} \arrow[d,sloped, equals] & \rR h_*\Bigl(\Delta_*\bigl(\cal{E}\otimes \cal{E}^{\vee}\bigr)\otimes \pi_2^*\cal{E}\Bigr) \arrow{dd}{\rR h_*(\PF_\Delta)}&  \rR f_*\bigl(\LLambda \otimes \cal{E}\bigr) \arrow{ldd}{\rR f_*(coev_{\cal{E}}\otimes \rm{id})}  \arrow{dd}{\rR f_*(\sigma)} \arrow[dddl, eq=gigantic-diagram-evcoev] \\
    \rR f_* \Delta^* \bigl(\pi^*_1 \cal{E} \otimes \pi^*_2\cal{E}^\vee\bigr) \otimes \rR f_* \cal{E} \arrow[r,"\rR f_*(\eta_{\pi_1}) \otimes \rR f_*(\eta_{\pi_2})"] \arrow[d,equals] & \rR h_*\pi_1^*\Delta^*\bigl(\pi_1^*\cal{E} \otimes \pi_2^*\cal{E}^\vee\bigr) \otimes \rR h_*\pi_2^*\cal{E} \arrow{r}{\cup} \arrow{d}{\rR h_*(\eta_{\Delta})\otimes \rm{id}} & \rR h_*\Bigl(\pi^*_1\bigl( \Delta^* \pi^*_1\cal{E} \otimes \Delta^* \pi^*_2\cal{E}^\vee\bigr) \otimes \pi^*_2\cal{E}\Bigr) \arrow[d, "\rR h_*(\eta_\Delta \otimes \rm{id})"] & &   \\
    \rR h_* \Delta_* \Delta^* \bigl(\pi^*_1 \cal{E} \otimes \pi^*_2\cal{E}^\vee\bigr) \otimes \rR f_*\cal{E} \arrow[r,"\rm{id} \otimes \rR f_*(\eta_{\pi_2})"] \arrow[d,"\rR h_*(\PF^{-1}_\Delta) \otimes \id"]&  \rR h_*\Bigl(\Delta_*\Delta^*\bigl(\pi_1^*\cal{E} \otimes \pi_2^*\cal{E}^{\vee}\bigr) \Bigr) \otimes \rR h_*\pi_2^*\cal{E} \arrow{r}{\cup} \arrow{d}{\rR h_*(\PF^{-1}_\Delta) \otimes \rm{id}} & \rR h_*\Bigl( \Delta_*\Delta^*\bigl(\pi_1^*\cal{E}\otimes \pi_2^* \cal{E}^\vee\bigr) \otimes \pi_2^*\cal{E}\Bigr) \arrow{r}{\rR h_*(\PF_{\Delta})} \arrow{d}{\rR h_*(\PF^{-1}_{\Delta}\otimes \rm{id})} & \rR f_*\bigl(\cal{E} \otimes \cal{E}^\vee \otimes \cal{E}\bigr) \arrow{r}{\rR f_*(\rm{id}\otimes ev_{\cal{E}})} \arrow{d}{\rR h_*(\PF^{-1}_\Delta)} &   \rR f_*\bigl(\cal{E} \otimes \LLambda\bigr) \arrow{d}{\rR h_*(\PF^{-1}_\Delta)} \\
    \rR h_*\bigl(\pi^*_1\cal{E} \otimes \pi^*_2\cal{E}^\vee \otimes \Delta_*\LLambda \bigr) \otimes \rR f_*\cal{E}  \arrow[d,"\rR h_*(\rm{id}\otimes \cl_\Delta) \otimes \rm{id}"] \arrow[r,"\rm{id}\otimes  \rR f_*(\eta_{\pi_2})"] & \rR h_*(\pi_1^* \cal{E} \otimes \pi_2^*\cal{E}^\vee \otimes \Delta_*\LLambda) \otimes \rR h_*\pi_2^*\cal{E}  \arrow{r}{\cup} \arrow[d,"\rR h_*(\rm{id}\otimes \cl_\Delta) \otimes \rm{id}"] & \rR h_*(\pi_1^*\cal{E} \otimes \pi_2^*\cal{E}^\vee \otimes\Delta_* \LLambda \otimes \pi_2^* \cal{E} ) \arrow[rr, bend right=20, "\rR h_*\bigl(\rm{id} \otimes \pi_2^*(ev_{\cal{E}}) \otimes \rm{id}\bigr)"] \arrow{r}{\rR h_*(\rm{id} \otimes \PF_{\Delta})}  \arrow[d,swap, "\rR h_*(\rm{id}\otimes \cl_\Delta \otimes \rm{id})"] & \rR h_*\bigl(\pi_1^* \cal{E} \otimes \Delta_*(\cal{E}^\vee \otimes \cal{E})\bigr)  \arrow{r}{\rR h_*\bigl(\rm{id}\otimes \Delta_*(ev_{\cal{E}})\bigr)}& \rR h_*(\pi_1^*\cal{E} \otimes \Delta_* \LLambda)  \arrow[dddl, end anchor=east, bend left, swap, "\rR h_*(\rm{id}\otimes \cl_\Delta)"] \arrow[ddd, "\rR f_*(\PF^{-1}_{\pi_1})"]  \\
    \rR h_* \bigl(\pi^*_1 \cal{E} \otimes \pi^*_2\cal{E}^\vee(d)[2d]\bigr) \otimes \rR f_* \cal{E} \arrow[r,"\rm{id} \otimes \rR f_*(\eta_{\pi_2})"] \arrow[d,"\KM^{-1} \otimes \id"] & \rR h_*(\pi_1^* \cal{E} \otimes \pi_2^* \cal{E}^\vee (d)[2d]) \otimes \rR h_*\pi_2^*\cal{E} \arrow{r}{\cup} & \rR h_*(\pi_1^*\cal{E} \otimes \pi_2^*\cal{E}^\vee(d)[2d] \otimes \pi_2^*\cal{E}) \arrow[dd,equals]  &   &  \\
    \rR f_*\cal{E} \otimes \rR f_*(\cal{E}^\vee(d)[2d]) \otimes \rR f_*\cal{E}  \arrow[d,"\id \otimes \cup"] & & & \\
    \rR f_* \cal{E} \otimes \rR f_*(\cal{E}^\vee(d)[2d] \otimes \cal{E}) \arrow{d}{\rm{id} \otimes \rR f_*\bigl(ev_{\cal{E}}(d)[2d]\bigr)} \arrow{rr}{\KM}  &  & \rR h_*(\pi_1^*\cal{E} \otimes \pi_2^*\cal{E}^\vee(d)[2d] \otimes \pi_2^*\cal{E}) \arrow{r}{\rR h_*\bigl(\rm{id}\otimes \pi_2^*(ev_{\cal{E}}(d)[2d])\bigr)} \arrow{d}{\rR h_*\bigl(\rm{id} \otimes \pi_2^*(ev_{\cal{E}}(d)[2d])\bigr)}& \rR h_*(\pi_1^*\cal{E} \otimes \LLambda(d)[2d]) \arrow{d}{\rR f_*(\PF^{-1}_{\pi_1})} & \rR f_*(\cal{E} \otimes \LLambda)  \arrow{dl}{\rR f_*\bigl(\rm{id} \otimes \rR \pi_{1, *}(\cl_\Delta)\bigr)} \\
    \rR f_* \cal{E} \otimes \rR f_* \LLambda(d)[2d] \arrow[d,"\id \otimes \ttr_f"] \arrow{rr}{\KM} \arrow[rrd,eq=gigantic-diagram-base-change-trace] & & \rR h_*(\pi^*_1\cal{E} \otimes \LLambda(d)[2d]) \arrow{r}{\rR f_*(\PF^{-1}_{\pi_1})} &  \rR f_*(\cal{E}\otimes \rR \pi_{1, *}\LLambda(d)[2d])   \arrow{ld}{\rR f_*(\rm{id} \otimes \ttr_{\pi_1})} &  \\
    \rR f_* \cal{E} \otimes \LLambda & & \arrow[ll, "\PF_f^{-1}"] \rR f_*(\cal{E}\otimes \LLambda) & &
    \end{tikzcd} \]
    \caption{The gigantic diagram. Beware that $\otimes$ denotes the derived tensor product so that the figure fits into the page.}\label{gigantic-diagram}
    \end{sidewaysfigure}   
    
    The map $\Lambda \otimes^L \rR f_*\cal{E} \to \rR f_* \cE \otimes^L \Lambda$ obtained by going down the entire left column is equal to
    \[
    \bigl(\rm{id}\otimes^L e(f, \cal{E})\bigr) \circ \bigl(c(f, \cal{E}) \otimes^L \rm{id}\bigr)
    \]
    by its very construction. The commutativity of the diagram in \cref{gigantic-diagram} implies that this composition can be computed by going around the outer diagram from the top left corner to the bottom left corner in a clockwise direction.
    Furthermore, we see that \cref{lemma:preliminary-duality} and the formula $\rR f_*(\PF_{\pi_1}^{-1}) \circ \rR h_*(\PF_{\Delta}^{-1}) = \rm{id}$ imply that 
    \[
    \bigl(\rm{id}\otimes^L e(f, \cal{E})\bigr) \circ \bigl(c(f, \cal{E}) \otimes^L \rm{id}\bigr) = \PF_{f}^{-1} \circ \rR f_*(\sigma)\circ \PF_{f} =\sigma \colon \Lambda \otimes^L \rR f_*\cal{E} \to \rR f_*\cal{E} \otimes^L \Lambda.
    \]
    This formally implies that $\varphi(f, \cal{E}) = \rm{id} \colon \rR f_* \cal{E} \to \rR f_*\cal{E}$. 

    To see that $\psi(f, \cal{E})=\rm{id}$, we need to use a diagram similar to that of \cref{gigantic-diagram}; we leave it to the reader to figure out the exact shape of the diagram. We only mention that every instance of $\pi_1$ should be replaced with $\pi_2$ (and vice versa) and one needs to use the second part of \cref{lemma:preliminary-duality} (as opposed to the first part used in the proof above).
\end{proof}

As the first application of \cref{Poincare dualizability theorem}, we show that derived pushforwards along smooth and proper morphisms preserve lisse sheaves. 

\begin{corollary}\label{cor:preservation-lisse-sheaves}
Let $f \colon X \to Y$ be a smooth proper morphism of locally noetherian analytic adic spaces with $n \in \cO^\times_Y$. 
Let $\mathcal{E} \in D_\lisse(X_\et; \Lambda)$ be a lisse complex. Then $\rR f_*\mathcal{E}$ lies in $D_\lisse(Y_\et; \Lambda)$.
If $\mathcal{E}$ is locally bounded (resp.\ perfect), then so is $\rR f_*\mathcal{E}$.
\end{corollary}
\begin{proof}
The statement is local on $Y$, so we may assume that $Y$ is affinoid. 
As the property of being smooth of equidimension $d$ is open on the source, we have a finite disjoint decomposition $X = \bigsqcup_{i = 0}^n X_i$ such that $X_i \to Y$ is smooth of equidimension $i$.
Since each $X_i$ is clopen in $X$, we conclude that each $X_i$ is also proper over $Y$.
Therefore, we can replace $X$ with each $X_i$ separately to assume that $f$ is smooth proper of equidimension $d$ for some $d \in \ZZ_{\ge 0}$.
Then \cref{Poincare dualizability theorem} and  \cite[Lem.~11.1]{adic-notes} show that $\rR f_*$ preserves perfect complexes. Furthermore, the cohomological dimension of $\rR f_*$ is bounded by $2d$ due to \cite[Prop.~5.3.11]{Huber-etale}, so it only remains to show that $\rR f_*$ preserves lisse complexes.

Using again the finite cohomological dimension of $\rR f_*$, we may assume that $\cE$ is a bounded below lisse complex.
Then a standard argument using \cite[\href{https://stacks.math.columbia.edu/tag/093U}{Tag~093U}]{stacks-project} and the Leray spectral sequence from \cite[\href{https://stacks.math.columbia.edu/tag/0732}{Tag 0732}]{stacks-project} implies that we can assume that $\cal{E}$ is a lisse $\Lambda$-module on $X_\et$.
The Chinese remainder theorem implies that we can assume that $\Lambda = \ZZ/p^m$ for some prime number $p\in \O_Y^\times$.
By considering the $p$-adic filtration on $\mathcal{E}$ and arguing one graded piece at a time, we reduce to the case when $\Lambda = \bf{F}_p$ and $\cal{E}$ is a lisse $\bf{F}_p$-module on $X_\et$. In this case, $\cal{E}$ is a dualizable object of $D(X_\et; \Lambda)$ by virtue of \cite[Lem.~11.1]{adic-notes}.
Then \cref{Poincare dualizability theorem} and another application of \textit{loc.\ cit.} imply that $\rR f_*\cal{E}\in D^b_\lisse(Y_\et; \Lambda)$.
\end{proof}

\begin{remark}
\label{remark comparing with previous results}
It seems that \cref{cor:preservation-lisse-sheaves} is a new result in this level of generality. However, it was certainly known under some additional assumptions. 
If $Y$ admits a map to $\Spa(K, \O_K)$ for a nonarchimedean field $K$ and $n\in (\O_Y^+)^\times$, this result was shown in \cite[Cor.~6.2.3]{Huber-etale} by an extremely elaborate argument. The assumption that $Y$ admits a map to $\Spa(K, \O_K)$ was recently removed in \cite[App.~1.3.4(4)]{Z-revised}. Now if $Y$ is a rigid-analytic space over $\Spa(K, \O_K)$ and $p$ is equal to the characteristic of the residue field of $\O_K$, this result was shown in \cite[Th.~10.5.1]{Berkeley} using the full strength of the perfectoid and diamond machinery. 
In contrast to these two proofs, our proof is uniform in $n$, is essentially formal, and remains largely in the world of locally noetherian analytic adic spaces.\footnote{The only exception to this occurs in the proof of \cref{lemma:proper-base-change}; see \cref{use-perfectoid}.}
\end{remark}

The main goal of the rest of this subsection is to extend Poincar\'e duality to general (not necessarily dualizable) lisse sheaves.
The essential difficulty comes from the fact that the constant sheaf $\ud{\Z/p \Z}$ is not dualizable in $D(X_\et; \Z/p^2\Z)$ for any prime $p$.
Nevertheless, our extension of Poincar\'e duality to this kind of coefficients will be essentially formal. 

\begin{theorem}
\label{general smooth Poincare duality}
Let $X$ and $Y$ be locally noetherian analytic adic spaces such that $n \in \cO^\times_Y$, let $f\colon X \to Y$ be a smooth proper morphism of equidimension $d$, and let $\cal{E}\in D_\lisse(X_\et; \Lambda)$ be a complex with lisse cohomology sheaves.
Then the duality morphism 
\[ \PD_f(\mathcal{E}) \colon \rR f_*(\rR \cHom_{\Lambda}(\mathcal{E}, \LLambda_X(d)[2d])) \to 
\rR \cHom_{\Lambda}(\rR f_*\mathcal{E}, \LLambda_Y) \]
from \cref{evaluation and coevaluation maps}\cref{PD map} is an isomorphism.
\end{theorem}

\begin{proof}
The strategy of this proof is to reduce to the case where $\Lambda = \Z/p^r\Z$ for some prime number $p\in \O_Y^\times$ and $\cal{E} = \ud{\bf{F}}_p$. In this case, we deduce the result from \cref{Poincare dualizability theorem}. 

\begin{enumerate}[wide,label={\textit{Step~\arabic*}.},ref={Step~\arabic*}]
    \setcounter{enumi}{-1}
    \item \textit{We reduce to the case when $X$ is qcqs and connected and $\Lambda=\Z/p^r\Z$ for a prime number $p\in \O_Y^\times$.}
    First, the question is clearly local on $Y$, so we can assume that $Y$ is an affinoid. 
    Furthermore, \cite[Cor.~2.3]{adic-notes} shows that connected components of $X$ are clopen, so we may and do assume that $X$ is qcqs and connected.
    Then the Chinese Remainder Theorem implies that we can assume that $\Lambda=\Z/p^r\Z$ for some prime number $p\in \O^\times_Y$ and some integer $r > 0$. 

    In the rest of the proof, we will freely use the following two basic ``reduction principles'':
    \begin{enumerate}
    \item (``two-out-of-three'') If we have a triangle $\mathcal{E}_1 \to \mathcal{E}_2 \to \mathcal{E}_3$ in $D_\lisse(X_\et; \Lambda)$ and $\PD_f(\mathcal{E}_i)$ is an isomorphism for two of the three $\cE_i$, then $\PD_f(\cE_i)$ is an isomorphism for all three $\cE_i$.
    \item (``closure under retracts'') If $\mathcal{E}$ is a direct summand of $\mathcal{G}$ and
    $\PD_f(\mathcal{G})$ is an isomorphism, then $\PD_f(\mathcal{E})$ is an isomorphism as well.
    \end{enumerate}

    \item\label{general smooth Poincare duality - sheaf reduction} \textit{We reduce to the case when $\cal{E}$ is a lisse sheaf of $\Lambda$-modules.}
    First, we note that $\rR f_*$ commutes with sequential homotopy colimits (e.g., as defined in \cite[\href{https://stacks.math.columbia.edu/tag/0A5K}{Tag~0A5K}]{stacks-project}) due to \cite[Lem.~9.1]{adic-notes}.
    This implies that both the source and target of $\PD_f$ (viewed as functors in $\cal{E}$) transform sequential homotopy colimits into sequential homotopy limits (e.g., as defined in \cite[\href{https://stacks.math.columbia.edu/tag/08TB}{Tag~08TB}]{stacks-project}).
    Since the natural morphism $\hocolim_N \tau^{\le N} \cal{E} \to \cal{E}$ is an isomorphism (this can be deduced from \cite[\href{https://stacks.math.columbia.edu/tag/0CRK}{Tag 0CRK}]{stacks-project}), we reduce to the case when $\cal{E}\in D^-_\lisse(X_\et; \Lambda)$.
    In this case, we consider the exact triangle
    \[
    \tau^{\leq -N}(\mathcal{E}) \to \mathcal{E} \to \tau^{> -N} \mathcal{E}.
    \]
    Recall that $\rR f_*$ has cohomological dimension $2d$ by virtue of \cite[Prop.~5.3.11]{Huber-etale}.
    As a consequence, both $\rR f_*\Bigl(\rR \cHom_{\Lambda}\bigl(\tau^{\leq -N}(\mathcal{E}), \LLambda_X(d)[2d]\bigr)\Bigr)$ and $\rR \cHom_{\Lambda}\bigl(\rR f_*\tau^{\leq -N}(\mathcal{E}), \LLambda_Y\bigr)$ lie in $D^{\geq N-2d}(Y_\et; \Lambda)$. 
    Given an integer $q$, the map on cohomology sheaves $\cal{H}^q\bigl(\PD_f(\cal{E})\bigr)$ is therefore an isomorphism if and only if $\cal{H}^q\bigl(\PD_f(\tau^{> - (q+1+2d)} \cal{E})\bigr)$ is so.
    In particular,  if $\PD_f(\tau^{> - N} \cal{E})$ is an isomorphism for all $N$, then $\PD_f(\cal{E})$ is an isomorphism as well. 
    Thus, we reduce to the case when $\cal{E}$ is bounded.
    In this case, the two-out-of-three reduction principle reduces the question further to the case when $\mathcal{E}$ is a lisse sheaf of $\Lambda$-modules on $X_\et$.

    \item \textit{We reduce to the case $\cal{E}=\ud{\bf{F}}_p$.}
    First, the two-out-of-three reduction principle implies that it suffices to prove the claim for $p^k\cal{E}/p^{k+1}\cal{E}$ for each $0\leq k\leq r-1$.
    Therefore, we can assume that $\cal{E}$ is an $\bf{F}_p$-lisse sheaf (considered as a $\Z/p^r\Z$-lisse sheaf).
    The ``m\'ethode de la trace''\footnote{To make this precise, one can argue as in the proof of \cite[\href{https://stacks.math.columbia.edu/tag/0A3R}{Tag 0A3R}]{stacks-project}.} then implies that there is a finite \'etale morphism $\pi\colon X' \to X$ of constant degree prime to $p$ such that $\cal{E}'\coloneqq \restr{\cal{E}}{X'}$ is a finite successive extension of constant sheaves $\ud{\bf{F}}_p$.
    The composition
    \[
    \cal{E} \to \pi_* \cal{E}' \xr{\ttr_{\pi, \cal{E}}} \cal{E}
    \]
    is equal to $\deg(\pi)$ (see \cref{thm:flat-trace}\cref{thm:flat-trace-4}).
    Since $\deg(\pi)$ is coprime to $p$, we conclude that $\cal{E}$ is a direct summand of $\pi_*\cal{E}'$. 
    By the closure under retracts reduction principle, it suffices to show that $\PD_f\bigl(\pi_*\cal{E}'\bigr)$ is an isomorphism.
    Since the smooth trace is compatible with compositions
    (\cref{smooth-trace-constant}\cref{smooth-trace-constant-composition}),
    we see that the composition
    \begin{multline*}
    \rR f_*\pi_*\bigl(\rR\cHom_{\Lambda}(\cal{E}', \LLambda_{X'}(d)[2d])\bigr)
    \xrightarrow{\rR f_*(\PD_{\pi}(\cal{E}')(d)[2d])}
    \rR f_*\bigl(\rR\cHom_{\Lambda}(\pi_*\cal{E}', \LLambda_X(d)[2d])\bigr) 
    \longrightarrow \\
    \xrightarrow{\PD_f(\pi_*\cal{E}')}
    \rR\cHom_{\Lambda}(\rR f_*\pi_*\cal{E}', \LLambda_Y)
    \end{multline*}
    is given by $\PD_{f \circ \pi}(\cal{E}')$.
    Hence, we are reduced to showing that both $\PD_{\pi}(\cal{E}')$
    and $\PD_{f \circ \pi}(\cal{E}')$ are isomorphisms. In other words, we can assume\footnote{We note that we replace $X$ with $X'$, which might be disconnected. This will not be important for the rest of the argument. At any rate, we can further replace $X'$ with its connected component to preserve the assumption that $X$ is connected.} 
    that $\cE$ is a finite successive extension of constant sheaves $\bf{F}_p$. 
    The two-out-of-three reduction principle then allows us to reduce to $\cal{E}=\ud{\bf{F}}_p$.
 
    \item \textit{End of proof.}
    Now we prove the claim for $\Lambda=\Z/p^r\Z$ and $\cal{E}=\bf{F}_p$.
    In this case, the lisse sheaf of $\Lambda$-modules $\ud{\bf{F}}_{p, X}$ has the following free resolution:
    \[
    C\coloneqq \Bigl( \ldots \xrightarrow{p} \LLambda_X \xrightarrow{p^{r-1}} \LLambda_X \xrightarrow{p} \LLambda_X \Bigr) \xr{\sim} \ud{\bf{F}}_{p, X}.
    \]
    For any integer $i$, denote the naive truncation of $C$ by $C_i \coloneqq \sigma^{\geq -i} C$.
    Then $C_i$ fits into the exact sequence
    \[
    \ud{\bf{F}}_{p, X}[i] \to C_i \to \ud{\FF}_{p, X}. 
    \]
    This induces the following morphism of exact triangles
    \[ \begin{tikzcd}[column sep=small]
    \rR f_*\bigl(\rR \cHom_{\Lambda}(\ud{\bf{F}}_{p, X}, \LLambda_X(d)[2d])\bigr) \arrow[d,"{\PD_f(\ud{\mathbf{F}}_{p, X})}"] \arrow[r] & \rR f_*\bigl(\rR \cHom_{\Lambda}(C_i, \LLambda_X(d)[2d])\bigr) \arrow[d,"\PD_f(C_i)"] \arrow[r] & \rR f_*\bigl(\rR \cHom_{\Lambda}(\ud{\mathbf{F}}_{p, X}[i], \LLambda_X(d)[2d])\bigr) \arrow[d,"{\PD_f(\ud{\bf{F}}_{p, X}[i])}"] \\
    \rR\cHom_{\Lambda}(\rR f_*\ud{\mathbf{F}}_{p, X}, \LLambda_Y) \arrow[r] & \rR\cHom_{\Lambda}(\rR f_*C_i, \LLambda_Y) \arrow[r] & \rR\cHom_{\Lambda}(\rR f_*\ud{\mathbf{F}}_{p, X}[i], \LLambda_Y).
    \end{tikzcd} \]
    By construction, $C_i$ is a perfect (hence dualizable) object in $D(X_\et; \Lambda)$ for each $i$, so \cref{Poincare dualizability theorem} implies that $\PD_f(C_i)$ is an isomorphism for every integer $i$.
    Now as in \cref{general smooth Poincare duality - sheaf reduction}, both $\rR f_*\bigl(\rR \cHom_{\Lambda}(\ud{\mathbf{F}}_{p, X}[i], \LLambda_X(d)[2d])\bigr)$ and $\rR\cHom_{\Lambda}(\rR f_*\ud{\mathbf{F}}_{p, X}[i], \LLambda_Y)$ lie in $D^{\geq i-2d}(Y_\et; \Lambda)$. Therefore, we conclude that $\cal{H}^q\bigl(\PD_f(\ud{\mathbf{F}}_{p, X})\bigr)$ is an isomorphism for $q< i-2d$. Since $i$ was an arbitrary integer, $\PD_f(\ud{\mathbf{F}}_{p, X})$ is an isomorphism, finishing the proof. \qedhere
\end{enumerate}
\end{proof}

\begin{remark}\label{rmk:no-Poincare-duality-closed-unit-disk}
We point out that the example of the closed unit disk prevents any na\"{i}ve form of ``weak'' Poincar\'{e} duality to hold.
Namely, let $C$ be an algebraically closed nonarchimedean field of mixed characteristic $(0, p)$, let $X=\DD^1_C$, and let $n=p>0$. Then \cref{smooth-trace-constant} induces a pairing
\[
\Hh^i(X, \bf{F}_p) \otimes \Hh^{2-i}_c(X, \mu_p) \xlongrightarrow{\cup} \Hh^2_c(X, \mu_p) \xlongrightarrow{\Hh^2(\ttr_X)} \bf{F}_p
\]
for each integer $0\leq i\leq 2$. One may wonder whether, for a fixed $i$, the duality map from one of these two $\bf{F}_p$-vector spaces to the dual of the other can be an isomorphism (or at least injective or surjective).
It turns out that none of these options holds: 
\begin{enumerate}[leftmargin=*]
    \item \Cref{distance and cycle class} guarantees that $\Hh^2_c(X, \mu_p)$ is infinite, while \cref{cohomology-affine-curve} guarantees that $\Hh^0(X, \bf{F}_p)\simeq \bf{F}_p$. Thus, the map
    $\Hh^{2}_c(X, \mu_p) \to \Hh^{0}(X, \bf{F}_p)^{\vee}$
    cannot be injective, and the map 
    $\Hh^{0}(X, \bf{F}_p) \to \Hh^{2}_c(X, \mu_p)^{\vee}$
    cannot be surjective.
    \item On the other hand, \cref{rmk:infinite-cohomology} implies that $\Hh^1(X, \ZZ/p) \cong \Hh^1(X, \mu_p)$ is infinite, while \cref{cor:first-cohomology-affinoids} implies that $\Hh^1_c(X, \mu_p) =0$. Thus, the map $\Hh^{1}_c(X, \mu_p) \to \Hh^{1}(X, \bf{F}_p)^{\vee}$
    is not surjective,
    whereas the map $\Hh^{1}(X, \bf{F}_p) \to \Hh^1_{c}(X, \mu_p)^{\vee}$ is not injective.
\end{enumerate}
A similar computation can be adapted to the \emph{open} unit disk $X=\accentset{\circ}{\DD}^1$ showing that no form of ``weak'' Poincar\'e duality could hold in the partially proper case as well. 
\end{remark}

\begin{remark}
It seems plausible that there could be a more sophisticated version of Poincar\'e duality in the style of \cite{CGN}.
For example, in the case of smooth connected affinoid curves, 
\cref{cohomology-affine-curve}, \cref{comp-supp-cohomology-affine-curve}, and \cref{cor:first-cohomology-affinoids} imply that, among the cohomology groups involved in Poincar\'e duality, only $\Hh^1(X, \bf{F}_p)$ and $\Hh^2_c(X, \mu_p)$ could be infinite.
It is believable that the hugeness of $\Hh^1(X, \bf{F}_p)$ is ``dual'' to the hugeness of $\Hh^2_c(X, \mu_p)$, or rather to the hugeness of $\ker (\ttr_{X}) \subset \Hh^2_c(X, \mu_p)$.
Even though the numerology of the usual Poincar\'e duality does not allow this, a more elaborate form of duality (involving higher Ext groups) might ``mix'' degrees appropriately.
A similar phenomenon occurs in \cite{CGN}. Unfortunately, we do not know how to make this precise. 
\end{remark}

\section{The trace map for proper morphisms}\label{proper-trace}

In this section, we discuss the construction of a trace map for an arbitrary proper morphism of rigid-analytic spaces over a non-archimedean field of characteristic $0$. Then we prove a version of Poincar\'e duality for an arbitrary proper morphism; this positively answers the question raised in \cite[Rmk.~3.23]{BH}.
In order to even formulate the notion of a trace map and of Poincar\'e duality for proper morphisms that are not necessarily smooth, we need to use the notions of Zariski-constructible sheaves and of dualizing complexes developed in \cite[\S~3.1-3.2]{BH} and \cite[\S~3.4]{BH}, respectively.
Since the theory has only been worked out for rigid-analytic spaces over a nonarchimedean field of $0$, we always work in this setup in this section (in contrast to \cref{smooth-traces}, where we considered general locally noetherian analytic adic spaces). 

Throughout this section, we fix a non-archimedean field $K$ of characteristic $0$, an integer $n>0$, and put $\Lambda \colonequals \Z/n\Z$. 

\subsection{Preliminaries on dualizing complexes}

We recall that \cite[Th.~3.21]{BH} constructs a dualizing complex $\omega_X$ for any rigid-analytic space $X$ over $K$.
The main goal of this subsection is to record some basic facts about these dualizing complexes that are not addressed in \cite{BH}. 
Namely, we show that the formation of $\omega_X$ behaves well with respect to smooth morphisms and relative analytifications. 

First, we start with the following lemma:

\begin{lemma}\label{lemma:geometry-to-algebra} Let $f\colon X=\Spa(B, B^+) \to Y=\Spa(A, A^+)$ be a smooth morphism of affinoid rigid-analytic spaces over $K$. Then the morphism $f^\sharp \colon A \to B$ is regular. If $f$ is of equidimension $d$ and $\m\subset A$ is a maximal ideal such that $B\otimes_A k(\m) \neq 0$, then $B\otimes_A k(\m)$ has pure Krull dimension $d$.
\end{lemma}

This lemma holds without the assumption that $\charac K=0$. 

\begin{proof}
    First, we show that $f^\sharp$ is regular. For this, we note that $f^\sharp \colon A \to B$ is flat due to flatness of $f$ and \cite[Lem.~B.4.3]{Z-quotients}. 
    Furthermore, \cite[Satz~3.3.3]{Differentialrechnung} and \cite[Th.~3.3]{Kiehl-excellence} imply that the rings $A$ and $B$ are excellent.
    Therefore, \cite[\href{https://stacks.math.columbia.edu/tag/07NQ}{Tag 07NQ}]{stacks-project} and \cite[Th.~p.1]{Andre} ensure that it suffices to show that $k(\m)\otimes_A B$ is either zero or geometrically regular for any maximal ideal $\m\subset A$. We put $C_\m\coloneqq \wdh{\ov{k(\m)}}$. Then \cite[Lem.~1.1.5(i)]{Conrad99} implies that it suffices to show that the $C_\m$-algebra $C_\m \wdh{\otimes}_A B$ is either regular or zero. For this, the maximal ideal $\m$ uniquely defines a (classical) point $y\in Y$. Then the geometric fiber $X_{\ov{y}}$ of $f$ over $y$ is given by $\Spa\Bigl(C_\m \wdh{\otimes}_A B, \bigl(C_\m \wdh{\otimes}_A B\bigr)^\circ\Bigr)$. Since $f$ is smooth, we conclude that $X_{\ov{y}}$ is smooth over $\Spa(C_\m, C_\m^\circ)$, so \cite[Th.~3.6.3]{FvdP04} implies that $C_\m \wdh{\otimes}_A B$ is regular or zero. 

    We are left to show that $B\otimes_A k(\m)$ has pure dimension $d$ if it is non-zero and $f$ is of equidimension $d$. We keep the notation of the previous paragraph and observe that the fiber $X_y$ is given by $\Spa\Bigl(B\otimes_A k(\m), \bigl(B\otimes_A k(\m)\bigr)^\circ\Bigr)$. Then \cite[Lem.~1.8.6~(ii)]{Huber-etale} implies that each connected component of $\Spec \bigl(B\otimes_A k(\m)\bigr)$ is of Krull dimension $d$. Since we already know that $B\otimes_A k(\m)$ is regular, we conclude that it is of pure Krull dimension $d$. 
\end{proof}

\begin{corollary}\label{cor:smooth-pullback}
Let $f\colon X \to Y$ be a smooth morphism of equidimension $d$ between  rigid-analytic spaces over $K$.
Then there is a canonical isomorphism 
\[
\alpha_f\colon f^*\omega_Y(d)[2d] \xlongrightarrow{\sim} \omega_X.
\]
\end{corollary}
\begin{proof}
    After unraveling the definition of the dualizing complex in \cite[Th.~3.21]{BH}, we reduce the question to showing that for a smooth morphism $f\colon \Spa(B, B^\circ) \to \Spa(A, A^\circ)$ of equidimension $d$ with associated morphism $f^\alg\colon \Spec B \to \Spec A$, there is a unique isomorphism (compatible with the pinnings) of potential dualizing complexes\footnote{See \cite[Exp.~XVII, \S~2]{deGabber} for the detailed and self-contained discussion of potential dualizing complexes.}
    $\alpha_{f^\alg} \colon f^{\alg, *} \omega_A(d)[2d] \xr{\sim} \omega_B$. This follows from \cref{lemma:geometry-to-algebra} and \cite[Lem.~3.22]{BH}.
\end{proof}

\begin{remark}[Algebraic version of \cref{cor:smooth-pullback}]\label{rmk:smooth-pullback-in-algebraic-geometry}
\begin{enumerate}[leftmargin=*]
    \item A proof similar to that of \cref{cor:smooth-pullback} (in fact, easier), shows that for any $K$-affinoid algebra $A$ and any smooth morphism $f\colon X \to Y$ of equidimension $d$ between locally finite type $A$-schemes, there is a canonical isomorphism
    \[
    \alpha_{f}^\alg\colon f^* \omega_Y(d)[2d] \xlongrightarrow{\sim} \omega_X.
    \]
    \item In the proof of \cref{cor:smooth-pullback}, we also used the following fact: for a smooth morphism $f\colon \Spa(B, B^\circ) \to \Spa(A, A^\circ)$ of equidimension $d$ between rigid-analytic spaces over $K$ and the corresponding morphism $f^\alg \colon \Spec B \to \Spec A$ of affine schemes (which is not necessarily of finite type), there is a canonical isomorphism
    \[
    \alpha_{f^\alg} \colon f^{\alg, *}\omega_{\Spec A}(d)[2d] \xlongrightarrow{\sim} \omega_{\Spec B}.
    \]
    \end{enumerate}
Both of these isomorphisms essentially come from \cite[Lem.~3.22]{BH} (or \cite[Exp.~XVII, Prop.~4.1.1]{deGabber}).
\end{remark}

Now we discuss the behavior of dualizing complexes with respect to relative analytifications. Again, we first need to verify an algebra result: 

\begin{lemma}\label{lemma:relative-analytification-regular} Let $A$ be a $K$-affinoid algebra, let $B$ be a finite type $A$-algebra, let $X=\Spec B$, and let $U=\Spa(R, R^\circ) \subset X^{\an/A}$ be an open affinoid in the relative analytification of $X$ (see \cref{construction:relative-analytification}). Then the natural morphism $B \to R$ is regular, and $\dim R \otimes_B k(\m)=0$ for any maximal ideal $\m\subset B$ such that $R\otimes_B k(\m) \neq 0$.
\end{lemma}

This lemma holds without the assumption that $\charac K=0$. 
\begin{proof}
    First, we note that $A$ is a Jacobson ring by virtue of \cite[Prop.~3.1/3]{B}.
    Therefore, \cite[\href{https://stacks.math.columbia.edu/tag/00GB}{Tag 00GB}]{stacks-project} implies that $\Spec B \to \Spec A$ sends closed points to closed points.
    Since $X^{\an/A} \to \Spa(A, A^\circ)$ sends classical points to classical points, \cite[Lem.~5.1.2]{Conrad99} implies that $\abs{c_{X/A}} \colon \abs{X^{\an/A}} \to \abs{X}$ defines a bijection between classical points of $X^{\an/A}$ and closed points of $X$.
    As a consequence, the natural morphism $r\colon \Spec R \to \Spec B$ sends closed points to closed points and is injective on closed points. 
    Since $R$ is Jacobson, we conclude that $r^{-1}(\{s\})$ consists of at most one closed point for any closed point $s\in \Spec B$. Combining these results with \cite[Prop.~4.1/2]{B} and \cite[Lem.~5.1.2(2)]{Conrad99}, we conclude that, for every maximal ideal $\m\subset B$ such that $k(\m)\otimes_B R\neq 0$, the ideal $\m R \subset R$ is maximal, and the morphism $B_\m \to R_\m$ induces an isomorphism on residue fields. 

    Now we recall that $B$ is excellent due to \cite[Th.~3.3]{Kiehl-excellence} and \cite[Scholie 7.8.3(ii)]{EGA4_2}. Therefore, \cite[\href{https://stacks.math.columbia.edu/tag/07NQ}{Tag 07NQ}]{stacks-project}, \cite[Th.~p.1]{Andre}, and the conclusion of the previous paragraph imply that, in order to obtain both claims of the lemma, it suffices to show that, for every maximal ideal $\m\subset B$ such that $k(\m)\otimes_B R\neq 0$, the natural morphism 
    \[
    B_\m \to R_\m
    \]
    is flat and induces an isomorphism on residue fields. The latter claim was already verified in the previous paragraph.
    The first claim follows from \cite[Lem.~6.4]{adic-notes}, \cite[\href{https://stacks.math.columbia.edu/tag/0523}{Tag 0523}]{stacks-project}, and \cite[Prop.~4.1/2]{B}.
\end{proof}

\begin{corollary}\label{cor:analytification-pullback} Let $A$ be a $K$-affinoid algebra, let $X$ be a locally finite type $A$-scheme with the relative analytification $c_{X/A} \colon X^{\an/A}_\et \to X_\et$. Then there is a canonical isomorphism
\[
\beta_{X/A} \colon c_{X/A}^* \omega_{X} \xlongrightarrow{\sim} \omega_{X^{\an/A}},
\]
where $\omega_X$ is a potential dualizing complex on $X$ \emph{(}see \cite[Th.~3.19]{BH} and \cite[Exp.~XVII, Th.~5.1.1]{deGabber}\emph{)}.
\end{corollary}
\begin{proof}
    The proof is completely analogous to that of \cite[Th.~3.21~(7)]{BH} using \cref{lemma:relative-analytification-regular} and \cite[Lem.~3.22]{BH}.
\end{proof}

\begin{lemma}\label{lemma:analytification-of-smooth-iso} Let $A$ be a $K$-affinoid algebra, let $f^\alg\colon X \to \Spec A$ be a smooth morphism of equidimension $d$, and let $f\colon X^{\an/A} \to \Spa(A, A^\circ)$ be its relative analytification.
Then the diagram
\begin{equation}\label{eqn:compatibility-of-alphas}
\begin{tikzcd}[column sep = huge]
    c_{X/A}^* f^{\alg, *} \omega_{\Spec A} \simeq f^* c_A^* \omega_{\Spec A} \arrow[d, "c_{X/A}^*\alpha_{f^\alg}","\sim"'{sloped}] \arrow[r, "f^*(\beta_{\Spec A/A})", "\sim"'] & f^* \omega_{\Spa (A, A^\circ)} \arrow[d, "\alpha_f", "\sim"'{sloped}]   \\
    c^*_{X/A} \omega_X \arrow[r, "\beta_{X/A}", "\sim"'] & \omega_{X^{\an/A}}
\end{tikzcd}
\end{equation}
commutes, where the $\beta$'s are the isomorphisms from \cref{cor:analytification-pullback}, $\alpha_f$ is the isomorphism from \cref{cor:smooth-pullback}, and $\alpha_{f^{\alg}}$ is the isomorphism from \cref{rmk:smooth-pullback-in-algebraic-geometry}.
\end{lemma}
\begin{proof}
We note that $c_{X/A}^*f^{\alg, *} \omega_{\Spec A}$ is isomorphic to $\omega_{X^{\an/A}}$ via the composition $\alpha_f \circ f^*(\beta_{\Spec A/A})$. Therefore, \cite[Th.~3.21~(3)]{BH} implies that $\rR\cHom(c_{X/A}^*f^{\alg, *} \omega_{\Spec A}, \omega_{X^{\an}/A}) \simeq \ud{\Lambda}_{X^{\an/A}}$ lies in $D^{\geq 0}(X^{\an/A}_\et; \Lambda)$.
As a consequence, it suffices to check that Diagram~\cref{eqn:compatibility-of-alphas} commutes \'etale locally on $X^{\an/A}$.
After unraveling the definitions, the result then follows from \cref{lemma:geometry-to-algebra}, \cref{lemma:relative-analytification-regular}, and (most importantly) \cite[Exp.~XVII, Rmq.~4.1.3]{deGabber}.
\end{proof}

\subsection{Smooth and closed traces}

The main goal of this subsection is to define versions of trace maps for closed immersions and smooth morphisms (with coefficients in dualizing sheaves). The first construction will essentially come from adjunction, while the second construction will essentially come from the smooth trace map of \cref{smooth-trace-constant}. 

We start with the case of closed immersions. For this, we recall that \cite[Th.~3.21~(1)]{BH} provides us with a canonical isomorphism $c_i \colon \omega_X \xr{\sim} \rR i^! \omega_Y$ for any closed immersion $i\colon X\hookrightarrow Y$.  This gives us the desired trace via the following construction: 

\begin{construction}[Closed trace]
\label{trace for closed immersion}
Let $i\colon X \hookrightarrow Y$ be a closed immersion of rigid-analytic spaces over $K$.
The \emph{closed trace map} is the morphism $\tr_i \colon i_* \omega_X \to \omega_Y$ defined as the composition
\[
i_* \omega_X \xr{i_*(c_i)} i_* \rR i^! \omega_Y \xr{\epsilon_i} \omega_Y
\]
where $\epsilon_i$ is the counit of the $(i_*, \rR i^!)$-adjunction.
In other words, $\tr_i$ is adjoint to the isomorphism $c_i$. 
\end{construction}
The closed trace has its obvious analog in algebraic geometry:
\begin{construction}[Closed trace in algebraic geometry]
\label{trace for closed immersion in algebraic geometry}
Let $A$ be a $K$-affinoid algebra, and let $i\colon X \hookrightarrow Y$ be a closed immersion of locally finite type $A$-schemes. The \emph{closed trace map} is the morphism $\tr_i^{\alg} \colon i_* \omega_X \to \omega_Y$ defined as the composition
\[
i_* \omega_X \xr[\sim]{i_*(c^\alg_i)} i_* \rR i^! \omega_Y \xr{\epsilon_i} \omega_Y
\]
where $c^\alg_i$ is the isomorphism induced by \cite[Exp.~XVII, Prop.~4.1.2]{deGabber} and $\epsilon_i$ is the counit of the $(i_*, \rR i^!)$-adjunction.
\end{construction}

\begin{remark}\label{rmk:basic-properties-closed-trace}
The closed trace map satisfies the following basic properties:
\begin{enumerate}[label=(\arabic*),leftmargin=*]
    \item\label{rmk:basic-properties-closed-trace-composition} the closed trace is compatible with compositions, i.e., for a pair of Zariski-closed immersions $i_1\colon X \hookrightarrow Y$ and $i_2\colon Y \hookrightarrow Z$, we have the following equality:
    \[
    \tr_{i_2} \circ i_{2, *}(\tr_{i_1}) = \tr_{i_2 \circ i_1} \colon (i_2 \circ i_1)_* \omega_X \to \omega_Z;
    \]
    \item\label{rmk:basic-properties-closed-trace-local} the closed trace is \'etale local, i.e., for a Zariski-closed immersion $i\colon X \hookrightarrow Y$, an \'etale morphism $g\colon Y' \to Y$, and the fiber product $X' \coloneqq Y' \times_Y X$ with the two projections $i' \colon X' \hookrightarrow Y'$ and $g'\colon X' \to X$, the diagram
    \[ \begin{tikzcd}[column sep=large]
        i'_{*} \omega_{X'} \arrow[rr,"\tr_{i'}"]  & & \omega_{Y'}  \\
        i'_{*}g^{\prime *} \omega_X \arrow[u,"i'_{*}(\alpha_{g'})","\sim"'{sloped}] & \arrow[l,"\BC"',"\sim"] g^*i_{*} \omega_X \arrow[r,"g^*\tr_{i}"] & g^*\omega_Y \arrow[u,"\alpha_g","\sim"'{sloped}],
    \end{tikzcd} 
    \]
    commutes, where the $\alpha$'s are the isomorphism from \cref{cor:smooth-pullback};
    \item\label{rmk:basic-properties-closed-trace-analytification-of-closed trace} the closed trace is compatible with relative analytification, i.e., for a $K$-affinoid algebra $A$ and a closed immersion $i\colon X \to Y$ of locally finite type $A$-schemes, the diagram
    \[
        \begin{tikzcd}[column sep = huge]
        c_{Y/A}^* i_* \omega_X \arrow{r}{\sim} \arrow{d}{c_{Y/A}^*(\tr_i^\alg)}& i_*^{\an/A}c_{X/A}^* \omega_X \arrow[r, "i^{\an/A}_*(\beta_{X/A})", "\sim"'] & i^{\an/A}_* \omega_{X^{\an/A}} \arrow{d}{\tr_{i^{\an/A}}} \\
        c_{Y/A}^* \omega_Y \arrow[rr, "\beta_{Y/A}", "\sim"'] & & \omega_{Y^{\an}/A},
        \end{tikzcd}
    \]
    commutes, where the top left arrow is the isomorphism from \cite[Th.~5.7.2]{Huber-etale} and the $\beta$'s are the isomorphisms from \cref{cor:analytification-pullback}.
    \end{enumerate}
    We do not justify these facts fully.
    Instead, we only mention the main ingredients and leave the details to the interested reader.
    Using the constructions of $\alpha$ and of the trace map (and the construction of $c_i$ in \cite[Th.~3.21~(1)]{BH}), one reduces the first two claims to the analogous claims in algebraic geometry,\footnote{
    For the second claim, note that the ring map $A \to B$ induced by an \'etale morphism $g \colon \Spa(B,B^\circ) \to \Spa(A,A^\circ)$ is in general not an \'etale ring morphism, but only a regular morphism for which all nonempty fibers over the closed points have pure dimension $0$. Fortunately, \cite[Exp.~XVII, \S~4]{deGabber} is written in the generality of regular morphisms and can thus still be applied.}
    then using the $(g_!, g^*)$- and the $(i_*,\rR i^!)$-adjunctions, one reduces the first claim to \cite[Exp.~XVII, Rmq.~4.1.3]{deGabber} and the second claim to \cite[Exp.~XVII, Lem.~4.3.2.3]{deGabber}.\footnote{
    The statement \cite[Exp.~XVII, Lem.~4.3.2.3]{deGabber} imposes the additional assumption that the morphism $g$ is surjective, but the proof does not use it.}
    The last claim follows from the construction of closed traces, \cref{lemma:relative-analytification-regular}, and \cite[Exp.~XVII, Lem.~4.3.2.3]{deGabber}.
\end{remark}

For future reference, we also record the following basic result:

\begin{lemma}\label{dualizing-nilimmersion} Let $i\colon X \to Y$ be a nil-immersion of rigid-analytic spaces over $K$. Then $\tr_i\colon i_* \omega_X \to \omega_Y$ is an isomorphism.
\end{lemma}
\begin{proof}
    After unraveling the definition, we see that it suffices to show that the counit morphism $i_* \rR i^! \omega_Y \xr{\epsilon_i} \omega_Y$ is an isomorphism. For this, it suffices to show that $i_*\colon \Shv(X_\et; \Lambda) \to \Shv(Y_\et; \Lambda)$ is an equivalence. This follows directly from \cite[Prop.~2.3.7]{Huber-etale}.
\end{proof}

Now we wish to explicate this construction in some cases. For this, we assume that $i\colon X \to Y$ is a closed immersion, $Y$ is smooth of equidimension $d_Y$, and $X$ is smooth of equidimension $d_X$. 
Then \cite[Th.~3.21~(1)]{BH} ensures that there are canonical isomorphisms $\omega_X \simeq \ud{\Lambda}_X(d_X)[2d_X]$ and $\omega_Y \simeq \ud{\Lambda}_Y(d_Y)[2d_Y]$. Furthermore, \cite[Lem.~5.9, Lem.~5.6]{adic-notes} ensure that $i$ is an lci closed immersion of pure codimension $d_Y-d_X$. %
\begin{construction}[Cycle class]\label{construction:cycle-class-dualizing-complexes}
With the notation above, we define the \emph{cycle class} $\cl_i \colon i_*\omega_X \to \omega_Y$ as the composition
\[
i_* \omega_X \simeq i_* \ud{\Lambda}_X(d_X)[2d_X] \xr{\cl_i(\ud{\Lambda}_Y(d_X)[2d_X])} \ud{\Lambda}_Y(d_Y)[2d_Y] \simeq \omega_Y,
\]
where $\cl_i\bigl(\ud{\Lambda}_Y(d_X)[2d_X]\bigr)$ is the cycle class from \cref{complex-cycle-class}.
\end{construction}
Note that the notation in \cref{construction:cycle-class-dualizing-complexes} leads to a slight ambiguity because $\cl_i$ already denoted the cycle class map $i_* \ud{\Lambda}_X \to \ud{\Lambda}_Y(d_Y-d_X)[2(d_Y-d_X)]$ in \cref{variant:cycle-class-map}.
However, as we always consider cycle class morphisms for dualizing complexes in this section, this should not cause any confusion.
Now we show that \cref{trace for closed immersion} and \cref{construction:cycle-class-dualizing-complexes} agree with one another:
\begin{lemma}
\label{closed trace and cycle class for smooth things}
Let $X$ and $Y$ be smooth rigid-analytic spaces over $K$ of equidimension $d_X$ and $d_Y$, respectively, and let $i\colon X \to Y$ be a closed immersion. Then 
\[
\cl_i=\tr_i \colon i_* \omega_X \to \omega_Y.
\]
\end{lemma}
\begin{proof}
First, \cite[Th.~3.21~(1),(3)]{BH} implies that $\rR \cHom(i_*\omega_X, \omega_Y) \simeq i_* \rR \cHom(\omega_X, \omega_X) \simeq i_* \ud{\Lambda}_X \in D^{\geq 0}(Y_\et; \Lambda)$. 
Therefore, we can check locally on $Y$ that $\cl_i$ and $\tr_i$ coincide. So we may and do assume that $Y=\Spa(A, A^\circ)$ is affinoid, and then $X=\Spa(A/I, (A/I)^\circ)$ for some ideal $I\subset A$. In this case, \cite[Th.~3.6.3]{FvdP04} implies that both $A$ and $A/I$ are regular.
The dualizing complexes of $Y$ and $X$ are constructed as an analytification of potential dualizing complexes on $\Spec A$ and $\Spec A/I$ respectively. Tracing through the proof of \cite[Th.~3.21~(1)]{BH} and using \cite[Exp. XVII, Lem.~2.4.3.4]{deGabber}, we see that $\tr_i$ is the analytification of the (appropriately twisted and shifted) algebraic cycle class map for $i^{\alg} \colon \Spec A/I \to \Spec A$. Thus, we reduce the question to showing that the analytic cycle class map is equal to the analytification of the algebraic one. This was already proven in  \cref{lemma:analytification-cycle-classes}.
\end{proof}
Next, we define a version of smooth trace maps with coefficients in dualizing complexes.
\begin{construction}[Smooth trace]\label{constructing smooth trace for omega}
Let $f\colon X \to Y$ be a separated taut smooth morphism of rigid-analytic spaces over $K$.
\begin{enumerate}[leftmargin=*,label=(\arabic*)]
    \item Assume that $f$ is of equidimension $d$.
    Then we define the \textit{smooth trace map} $\tr_f\colon \rR f_! \omega_X \to \omega_Y$ as the composition
    \[
    \rR f_! \,\omega_X \xr{\rR f_!(\alpha^{-1}_f)} \rR f_! \bigl(f^*\omega_Y(d)[2d]\bigr) \xr{\PF_f^{-1}} \omega_Y \otimes^L \rR f_! \bigl( \ud{\Lambda}_X(d)[2d] \bigr) \xr{\rm{id} \otimes^L \ttr_f} \omega_Y,
    \]
    where $\alpha_f$ is the isomorphism from \cref{cor:smooth-pullback}, $\PF_f$ is the projection formula isomorphism from \cite[Th.~5.5.9~(ii)]{Huber-etale}, and $\ttr_f$ is the trace morphism from \cref{smooth-trace-constant}. 
    \item\label{constructing smooth trace for omega-2} Now let $f$ be a general separated taut smooth morphism.
    Then there exists a clopen decomposition $X=\bigsqcup_{d\in \NN} X_d$ such that $f_d \coloneqq \restr{f}{X_d} \colon X_d \to Y$ is of equidimension $d$.
    We define
    \[
    \tr_f \colon \rR f_! \, \omega_X \simeq \bigoplus_{d\in \NN} \rR f_{d, !}\, \omega_{X_d} \xlongrightarrow{\sum \tr_{f_d}} \omega_Y. 
    \]
\end{enumerate}
\end{construction}

\begin{construction}[Smooth trace in algebraic geometry]\label{constructing smooth trace for omega in algebraic geometry}
Let $A$ be a $K$-affinoid algebra and let $f\colon X \to Y$ be a separated smooth morphism of locally finite type $A$-schemes. Then one can define the trace map
\[
\tr^{\alg}_f \colon \rR f_! \omega_X \to \omega_Y
\]
similarly to \cref{constructing smooth trace for omega} (using \cite[Exp.~XVIII, Th.~2.9]{SGA4} in place of the analytic trace map and the isomorphism $\alpha_f^\alg$ from \cref{rmk:smooth-pullback-in-algebraic-geometry} in place of $\alpha_f$). 
\end{construction}

\begin{remark}\label{rmk:etale-trace-dualizing} 
If $f$ is separated taut \'etale, then \cref{smooth-trace-constant}\cref{smooth-trace-constant-etale} implies that $\tr_f$ is given by the composition $f_!\, \omega_{X'} \xrightarrow{f_!(\alpha_f^{-1})}f_! f^*\omega_X \xrightarrow{\epsilon_f} \omega_X$, where $\epsilon_f$ is the counit of the $(f_!, f^*)$-adjunction.
\end{remark}

\begin{remark}
\label{rmk:basic-properties-smooth-trace}
The smooth trace map satisfies the following properties:
\begin{enumerate}[label=(\arabic*),leftmargin=*]
    \item\label{rmk:basic-properties-smooth-trace-composition} the smooth trace is compatible with compositions, i.e., for a pair of smooth separated taut morphisms $f_1\colon X \to Y$ and $f_2\colon Y \to Z$, we have the following equality:
    \[
    \tr_{f_2} \circ \rR f_{2, !}(\tr_{f_1}) = \tr_{f_2 \circ f_1} \colon \rR (f_2 \circ f_1)_! \omega_X \to \omega_Z;
    \]
    \item\label{rmk:basic-properties-smooth-trace-local} the smooth trace is \'etale local, i.e., for a smooth separated taut morphism $f\colon X \to Y$, an \'etale morphism $g\colon Y' \to Y$, and the fiber product $X' \coloneqq Y' \times_Y X$ with the two projections $f' \colon X' \to Y'$ and $g'\colon X' \to X$, the diagram
    \[ \begin{tikzcd}[column sep=large]
        \rR f'_{!} \omega_{X'} \arrow[rr,"\tr_{f'}"]  & & \omega_{Y'}  \\
        \rR f'_{!}g^{\prime *} \omega_X \arrow[u,"\rR f'_{!}(\alpha_{g'})","\sim"'{sloped}] & \arrow[l,"\BC_!"',"\sim"] g^*\rR f_{!} \omega_X \arrow[r,"g^*\tr_{f}"] & g^*\omega_Y \arrow[u,"\alpha_g","\sim"'{sloped}]
    \end{tikzcd} 
    \]
    commutes, where the $\alpha$'s are the isomorphisms from \cref{cor:smooth-pullback} and $\BC_!$ is the base change map for compactly supported pushforward from \cite[Th.~5.4.6]{Huber-etale};
    \item\label{rmk:basic-properties-smooth-trace-compatible-with-algebraization} the smooth trace is compatible with relative analytifications, i.e., for a $K$-affinoid algebra $A$ and a smooth separated morphism $f\colon X \to Y$ of locally finite type $A$-schemes, the diagram
    \[
        \begin{tikzcd}[column sep = huge]
        c_{Y/A}^* \rR f_! \omega_X \arrow{r}{\sim} \arrow{d}{c_{Y/A}^*(\tr_f^\alg)}& \rR f_!^{\an/A}c_{X/A}^* \omega_X \arrow[r, "\rR f_!^{\an/A}(\beta_{X/A})", "\sim"'] &[1em] \rR f_!^{\an/A} \omega_{X^{\an/A}} \arrow{d}{\tr_{f^{\an/A}}} \\
        c_{Y/A}^* \omega_Y \arrow[rr, "\beta_{Y/A}", "\sim"'] & & \omega_{Y^{\an}/A},
        \end{tikzcd}
    \]
    commutes, where the top left arrow is the isomorphism from \cite[Th.~5.7.2]{Huber-etale} and the $\beta$'s are the isomorphism from \cref{cor:analytification-pullback}.
    \end{enumerate}
    The first two claims follow immediately from \cref{smooth-trace-constant}\cref{smooth-trace-constant-composition} and \cref{smooth-trace-constant}\cref{smooth-trace-constant-pullback}, respectively.
    The last claim follows from \cref{lemma:analytification-of-smooth-iso} and \cref{Compatibility with algebraic geometry}.
\end{remark}

The rest of this subsection is devoted to showing that closed and smooth trace maps are compatible with each other in a precise way. We start with the following basic lemma: 
\begin{lemma}
\label{closmooth trace lives in discrete space}
Let $X'$, $X$, $Y$, and $Z$ be rigid-analytic spaces over $K$, let $h\colon X' \to X$ be a surjective separated taut \'etale morphism, let $g\colon X \to Y$ be a separated taut smooth morphism, and let $i \colon Y \to Z$ be a closed immersion.
Set $f \colonequals i \circ g \colon X \to Z$ and $f' \colonequals i \circ g \circ h \colon X' \to Z$.
Then:
\begin{enumerate}[itemsep=0.2\baselineskip,leftmargin=*,label=\upshape{(\roman*)}]
    \item\label{closmooth trace lives in discrete space-1} $\rR\cHom(\rR f_!\, \omega_X, \omega_Z)$ lies in $D^{\geq 0}(Z_\et; \Lambda)$;
    \item\label{closmooth trace lives in discrete space-2} $\Hom(\rR f_! \omega_X, \omega_Z) \simeq \rm{H}^0\bigl(Y,\cHom( \rR^{2d} g_! \ud{\Lambda}_X(d), \ud{\Lambda}_Y)\bigr)$ if $g$ is of equidimension $d$;
    \item\label{closmooth trace lives in discrete space-3} the morphism $\Hom(\rR f_! \omega_X, \omega_Z) \to \Hom(\rR f'_!\, \omega_{X'}, \omega_Z)$ induced by $\tr_h$ is injective. 
\end{enumerate}

\end{lemma}
\begin{proof}
Using a similar clopen decomposition as in \cref{constructing smooth trace for omega}\cref{constructing smooth trace for omega-2}, we can reduce all three parts to the case where $g$ is of relative equidimension $d$ for some integer $d\geq 0$.
Then \cref{cor:smooth-pullback} and the projection formula imply that $\rR g_!\, \omega_X \simeq \rR g_! \ud{\Lambda}_X(d)[2d] \otimes^L \omega_Y$. Therefore, we have
\begin{multline}\label{eqn:formula-closmoth-shriek-push}
\rR\cHom\bigl(\rR f_!\, \omega_{X}, \omega_Z\bigr) \simeq  \rR \cHom(i_* \rR g_!\,\omega_X, \omega_Z) \simeq i_* \rR\cHom\bigl(\rR g_!\,\ud{\Lambda}_X(d)[2d] \otimes^L \omega_{Y}, \omega_Y\bigr) \\
\simeq  i_* \rR\cHom\bigl(\rR g_!\,\ud{\Lambda}_X(d)[2d], \rR\cHom(\omega_Y, \omega_Y)\bigr) \simeq  i_* \rR\cHom\bigl(\rR g_!\,\ud{\Lambda}_X(d)[2d], \ud{\Lambda}_Y\bigr) \in D^{\geq 0}(Z_\et; \Lambda),
\end{multline}
where the first isomorphism follows from $\rR i_! \simeq i_*$, the second isomorphism follows from \cite[Th.~3.21~(1)]{BH} and $\rR g_!\, \omega_X \simeq \rR g_! \ud{\Lambda}_X(d)[2d] \otimes^L \omega_Y$, the third isomorphism follows from the (derived) tensor-hom adjunction, the fourth isomorphism follows from \cite[Th.~3.21~(3)]{BH}, and the last containment follows from \cref{smooth trace lives in discrete space}.
This finishes the proof of \cref{closmooth trace lives in discrete space-1}.

Now \cref{eqn:formula-closmoth-shriek-push} and \cite[Prop.~5.5.8]{Huber-etale} directly imply that
\[
\Hom(\rR f_! \omega_X, \omega_Z) \simeq \rm{H}^0\bigl(Y, \cHom(\rR^{2d} g_! \ud{\Lambda}_X(d), \ud{\Lambda}_Y)\bigr).
\]
This proves \cref{closmooth trace lives in discrete space-2}. Now we deal with \cref{closmooth trace lives in discrete space-3}.
Set $g' \colonequals g \circ h \colon X' \to Y$;
this is also a separated smooth taut morphism of equidimension $d$.
Thus, \cref{closmooth trace lives in discrete space-2} implies that
\[
\Hom(\rR f_! \omega_X, \omega_Z) \simeq \Hh^0\bigl(Y, \cHom(\rR^{2d} g_! \ud{\Lambda}_X(d), \ud{\Lambda}_Y)\bigr) \text{ and } \Hom(\rR f'_! \omega_{X'}, \omega_Y) \simeq \Hh^0\bigl(Y, \cHom(\rR^{2d} g'_! \ud{\Lambda}_{X'}(d), \ud{\Lambda}_Y)\bigr).
\]
Therefore, it suffices to show injectivity of the morphism $\cHom(\rR^{2d} g_! \ud{\Lambda}_X(d), \ud{\Lambda}_Y) \to \cHom(\rR^{2d} g'_! \ud{\Lambda}_{X'}(d), \ud{\Lambda}_Y)$, which is induced by 
\[
\rR^{2d} g_!\bigl(\ttr_h^\et(d)\bigr) \colon \rR^{2d} g'_! \ud{\Lambda}_{X'}(d) \to \rR^{2d} g_! \ud{\Lambda}_{X}(d)
\]
thanks to \cref{rmk:etale-trace-dualizing}.
For this, it suffices to prove that $\rR^{2d} g_!\bigl(\ttr_h^\et(d)\bigr)$ is surjective.
This follows from \cite[Prop.~5.5.8]{Huber-etale} and (an easy case of) \cref{tr-epi}.
\end{proof}

The following technical lemma comes in handy later.
\begin{lemma}
\label{Lemma: Identifying natural transformations}
Let     
\[ 
\begin{tikzcd}
X' \arrow{d}{f'} \arrow{rd}{h} \arrow[r, "g'"] & X \arrow{d}{f} \\
Y' \arrow[r, "g"] & Y
\end{tikzcd} 
\]
be a Cartesian diagram of locally noetherian analytic adic spaces
with $f$ \'etale and $g$ proper.
Then:
\begin{enumerate}[leftmargin=*,label=\upshape{(\roman*)}]
    \item\label{Lemma: Identifying natural transformations-1} the base change natural transformation 
    $\BC \colon f^*\rR g_* \to \rR g'_*f^{\prime,*}$
    and the reverse direction natural transformation
    $\BC'\colon \rR g'_*f^{\prime,*} \simeq \rR g'_!f^{\prime,!} \to f^!\rR g_! \simeq f^*\rR g_*$
    are inverse to each other;
    \item\label{Lemma: Identifying natural transformations-2} the two natural transformations
    \[ \rR g'_* \xlongrightarrow{\rR g'_*(\eta_{f'})} \rR g'_* f^{\prime,*} f'_! \xlongrightarrow{\BC^{-1}} f^* \rR g_* f'_! \simeq f^* f_! \rR g'_* \quad \text{and} \quad \rR g'_* \xlongrightarrow{\eta_f} f^* f_! \rR g'_* \]
    agree, where $\eta_{f}$ \emph{(}resp.~$\eta_{f'}$\emph{)} denotes the unit of $(f_!, f^*)$ \emph{(}resp.\ $(f'_!, f^{\prime,*})$\emph{)}.
\end{enumerate}
\end{lemma}

\begin{proof}
To see \cref{Lemma: Identifying natural transformations-1}, note that the functors involved are derived functors of functors defined at sheaf level.
Since $f$ is \'etale, the functors $f^*$, $f^{\prime,*}$, $g_*$, and $g'_*$ preserve $K$-injective complexes.
For each $\F\in D(Y_\et; \Lambda)$, both $\BC(\F)$ and $\BC'(\F)$ are hence computed by applying them termwise to a fixed $K$-injective complex representing $\F$.
Therefore, it suffices to check the claim for injective sheaves,
and at sheaf level one can argue on stalks.
Now the question is \'{e}tale local on $Y$,
so we may even assume that the \'{e}tale map is an open immersion, in which case it follows directly from 
the definition that the two natural transformations are inverse to each other.

To deduce \cref{Lemma: Identifying natural transformations-2} from \cref{Lemma: Identifying natural transformations-1}, we consider the reverse base change maps $\BC'$ for the two diagrams
\[ 
\begin{tikzcd}
X' \arrow[r,"\id"] \arrow[d,"\id"] & X' \arrow[r,"g'"] \arrow[d,"f'"] & X \arrow[d,"f"] \\
X' \arrow[r,"f'"] & Y' \arrow[r,"g"] & Y
\end{tikzcd} 
\quad \text{and} \quad 
\begin{tikzcd}
X' \arrow[r,"g'"] \arrow[d,"\id"] & X \arrow[r,"\id"] \arrow[d,"\id"] & X \arrow[d,"f"] \\
X' \arrow[r,"g'"] & X \arrow[r,"f"] & Y.
\end{tikzcd} 
\]
Our claim then follows from the fact that
the two outer diagrams are the same and the reverse base change maps are compatible with ``concatenation'' of diagrams.
\end{proof}

We can now study the compatibility of closed and smooth trace maps in cartesian diagrams:
\begin{lemma}[Compatibility for Cartesian diagrams]
\label{compatibility lemma: Cartesian diagram case}
Let $f\colon X \to Y$ be a separated smooth taut morphism of rigid-analytic spaces over $K$, let $i\colon Y' \to Y$ be a closed immersion of rigid-analytic spaces over $K$, and let 
\[ \begin{tikzcd}
X' \arrow{d}{f'} \arrow{rd}{h} \arrow[r, hook, "i'"] & X \arrow{d}{f} \\
Y' \arrow[r, hook, "i"] & Y
\end{tikzcd} \]
be the resulting pullback square.
Then the following diagram commutes:
\begin{equation}\label{eqn:cartesian-square-compatibility-lemma-2}
\begin{tikzcd}[column sep = huge]
\rR h_! \,\omega_{X'} \arrow{r}{\rR f_!(\tr_{i'})} \arrow{d}{i_*(\tr_{f'})} &
\rR f_! \,\omega_X \arrow{d}{\tr_f}  \\
i_* \omega_{Y'} \arrow{r}{\tr_i} & \omega_Y.
\end{tikzcd}
\end{equation}
\end{lemma}
\begin{proof}
\begin{enumerate}[wide,label={\textit{Step~\arabic*}.},ref={Step~\arabic*}]
    \item\label{compatibility lemma: Cartesian diagram case etale} \textit{Proof for \'etale $f$.}
    We first establish the claim when $f$ is additionally assumed to be \'etale.
    To do so, we verify the commutativity of the $(f_!, f^*)$-adjoint of \cref{eqn:cartesian-square-compatibility-lemma-2}.
    Explicitly, this adjoint is given by the red rectangle in the following diagram:
    \[ \begin{tikzcd}[row sep=tiny,column sep=huge]
        & i'_*f'_* \omega_{Y'} \arrow[lddd,"i'_*(\eta_{f'})"'] \arrow[rddd,pos=.67,font=\scriptsize,eq=cartesian-square-compatibility-lemma-identifying] & i'_* \omega_{X'} \arrow[l,"i'_*(\alpha^{-1}_{f'})"',"\sim"] \arrow[r,red,"\tr_{i'}"] \arrow[dd,red,"\eta_f"] \arrow[lddd,"i'_*(\eta_{f'})"'] & \omega_X \arrow[rd,"\alpha^{-1}_f","\sim"'{sloped}] \arrow[dd,red,"\eta_f"] & \\
        &&&& f^*\omega_Y \arrow[dd,"\eta_f"] \\
        && f^*f_!i'_* \omega_{X'} \arrow[r,red,"f^*f_!(\tr_{i'})"] \arrow[d,phantom,red,sloped,"\simeq"] & f^*f_! \omega_X \arrow[rd,near start,"f^*f_!(\alpha^{-1}_f)","\sim"'{sloped,midway}] \arrow[dd,red,"f^*(\tr_f)"] & \\
        i'_*f^{\prime,*}f'_!f^{\prime,*} \omega_{Y'} \arrow[rd,"i'_*f^{\prime,*}(\epsilon_{f'})"'] & i'_*f^{\prime,*}f'_! \omega_{X'} \arrow[l,"i'_*f^{\prime,*}f'_!(\alpha^{-1}_{f'})"',"\sim"] \arrow[r,"\BC^{-1}","\sim"'] \arrow[d,"i'_*f^{\prime,*}(\tr_{f'})"] & f^*i_*f'_! \omega_{X'} \arrow[d,red,"f^*i_*(\tr_{f'})"] && f^*f_!f^* \omega_Y \arrow[ld,"f^*(\epsilon_f)"] \\[2em]
        & i'_*f^{\prime,*} \omega_{Y'} \arrow[r,"\BC^{-1}","\sim"'] & f^*i_* \omega_{Y'} \arrow[r,red,"f^*(\tr_i)"] & f^* \omega_Y &
    \end{tikzcd} \]
    Note that except for the lower red rectangle, every other part of this diagram commutes, thanks to the naturality of the base change maps $\BC^{-1}$ and the adjunction units $\eta_f$ and $\eta_{f'}$, the description of \'etale trace maps (\cref{rmk:etale-trace-dualizing}), as well as \cref{Lemma: Identifying natural transformations}\cref{Lemma: Identifying natural transformations-2} in the case of \cref{cartesian-square-compatibility-lemma-identifying}.
    Thus, it suffices to show that the outer diagram commutes.
    Since $f^*(\epsilon_f) \circ \eta_f = \id$ and $f^{\prime,*}(\epsilon_{f'}) \circ \eta_{f'} = \id$, this amounts to the commutativity of  
    \[  \begin{tikzcd}[column sep=large]
            i'_* \omega_{X'} \arrow[rr,"\tr_{i'}"] \arrow[d,"i'_*(\alpha^{-1}_{f'})","\sim"'{sloped}] & & \omega_X \arrow[d,"\alpha^{-1}_f","\sim"'{sloped}] \\
            i'_* f^{\prime,*} \omega_{Y'} \arrow[r,"\BC^{-1}","\sim"'] & f^*i_* \omega_{Y'} \arrow[r,"f^*(\tr_i)"] & f^*\omega_Y,
    \end{tikzcd} \]
    which follows directly from the observation that the closed trace map is \'etale-local on the target (\cref{rmk:basic-properties-closed-trace}).
        
    \item\label{compatibility lemma: Cartesian diagram case local} \textit{The statement is \'etale local on $X$ and $Y$.}
    Next, we prove that if $g' \colon U \twoheadrightarrow X$ and $g \colon V \twoheadrightarrow Y$ are separated taut \'etale covers and $\tilde{f} \colon U \to V$ is a separated smooth taut morphism fitting into a diagram
    \[ \begin{tikzcd}[column sep=large]
        U \arrow[d,two heads,"g'"] \arrow[r,"\tilde{f}"] & V \arrow[d,two heads,"g"] \\
        X \arrow[r,"f"] & Y,
    \end{tikzcd}\]
    then it suffices to show the assertion for $\tilde{f} \colon U \to V$ and $\tilde{i} \colonequals (i,\id) \colon Y' \times_Y V \to V$ instead of $f$ and $i$.
    To see this, note that we can extend \cref{eqn:cartesian-square-compatibility-lemma-2} to the following diagram, in which we set $\tilde{f}' \colonequals (\id,\tilde{f}) \colon Y' \times_Y U \to Y' \times_Y V$, $\tilde{i}' \colonequals (i,\id) \colon Y' \times_Y U \to U$, and $\tilde{h} \colonequals g \circ \tilde{f} \circ \tilde{i}' = g \circ \tilde{i} \circ \tilde{f}'$ and all the diagonal arrows are \'etale trace maps:
    \[ \begin{tikzcd}
        & \rR \tilde{h}_!\omega_{Y' \times_Y U} \arrow[rr,"g_!\rR \tilde{f}_!(\tr_{\tilde{i}'})"] \arrow[dd,near start,"g_!\tilde{i}_*(\tr_{\tilde{f}'})"] \arrow[ld] && g_!\rR \tilde{f}_! \omega_U \arrow[dd,"g_!(\tr_{\tilde{f}})"] \arrow[ld] \\
        \rR h_! \omega_{X'} \arrow[rr,crossing over,near end,"\rR f_!(\tr_{i'})"'] \arrow[dd,"i_*(\tr_{f'})"] && \rR f_! \omega_X & \\
        & g_!\tilde{i}_* \omega_{Y' \times_Y V} \arrow[rr,near start,"g_!(\tr_{\tilde{i}})"] \arrow[ld] && g_!\omega_V \arrow[ld] \\
        i_* \omega_{Y'} \arrow[rr,"\tr_i"] && \omega_Y \arrow[from=uu,crossing over,near start,"\tr_f"] &
    \end{tikzcd}\]
    The top and bottom side of the diagram are induced by cartesian squares and commute thanks to \cref{compatibility lemma: Cartesian diagram case etale}.
    The left and right sides commute by the compatibility of smooth traces with compositions (\cref{rmk:basic-properties-smooth-trace}).
    Since the morphism
    \[ \Hom(\rR h_!\, \omega_{X'},\omega_Y) \to \Hom(\rR \tilde{h}_!\, \omega_{Y' \times_Y U} ,\omega_Y) \]
    induced by $\tr_{(\id,g')}$ is injective by virtue of \cref{closmooth trace lives in discrete space}\cref{closmooth trace lives in discrete space-3}, a simple diagram chase shows that the commutativity of the back side of the diagram implies the commutativity of the front side.

    \item \textit{Reduce to the case of $X=\PP^{d,\an}_Y$ and $Y$ is affinoid.}
    By \cref{compatibility lemma: Cartesian diagram case local} and \cite[Cor.~1.6.10, Lem.~5.1.3]{Huber-etale}, we may assume that $Y$ is affinoid, that $f$ is of equidimension $d$ for some integer $d \ge 0$, and that $f$ factors as $X \xrightarrow{g} \AA^{d,\an}_Y \to Y$ for some separated taut \'etale morphism $g \colon X \to \AA^{d, \an}_Y$.
    Since the commutativity of \cref{eqn:cartesian-square-compatibility-lemma-2} can be checked for the two morphisms separately and was verified for $g$ in \cref{compatibility lemma: Cartesian diagram case etale}, we may assume that $X = \AA^{d,\an}_Y$.
    Another application of \cref{compatibility lemma: Cartesian diagram case etale} to the open immersion $\AA^{d,\an}_Y \hookrightarrow \PP^{d,\an}_Y$ reduces us to the case when $X = \PP^{d,\an}_Y$ and $f$ is the structure map $\PP^{d,\an}_Y \to Y$.

    \item \textit{End of proof.}
    Now we are in the situation where $Y$ is affinoid and $X= \PP^{d, \an}_Y$. We put $Y=\Spa(A, A^\circ)$ and $Y'=\Spa(A/I, (A/I)^\circ)$. 
    Thanks to \cref{rmk:basic-properties-closed-trace}\cref{rmk:basic-properties-closed-trace-analytification-of-closed trace} and \cref{rmk:basic-properties-smooth-trace}\cref{rmk:basic-properties-smooth-trace-compatible-with-algebraization}, it suffices to show that for the corresponding Cartesian diagram of finite type $A$-schemes
    \[
    \begin{tikzcd}
        \PP^d_{A/I} \arrow[rd, "h^\alg"] \arrow[r, hook, "i^{\prime,\alg}"] \arrow[d, "f^{\prime,\alg}"] & \PP^{d}_A \arrow{d}{f^\alg} \\
        \Spec A/I \arrow[r, hook, "i^\alg"] & \Spec A,
    \end{tikzcd}
    \]
    the induced diagram 
    \begin{equation}\label{eqn:many-traces}
    \begin{tikzcd}[column sep = huge]
        \rR h^\alg_* \,\omega_{\PP^d_{A/I}} \arrow{r}{\rR f^\alg_*(\tr^\alg_{i^{\prime,\alg}})} \arrow{d}{i^\alg_*(\tr^\alg_{f^{\prime,\alg}})} &
        \rR f^\alg_* \,\omega_{\PP^d_A} \arrow{d}{\tr^\alg_{f^\alg}}  \\
        i^\alg_* \omega_{\Spec A/I} \arrow{r}{\tr^\alg_{i^\alg}} & \omega_{\Spec A}
    \end{tikzcd}
    \end{equation}
    commutes.
    Recall that the adjoint to the algebraic trace morphism $\rR f^\alg_* f^{\alg, *} \omega_{\Spec A}(d)[2d] \to \omega_{\Spec A}$ (see \cite[Exp.~XVIII, Th.~2.9]{SGA4}) defines the Poincar\'e duality isomorphism $\PD_{f^\alg} \colon f^{\alg, *} \omega_{\Spec A}(d)[2d] \xr{\sim} \rR f^{\alg, !} \omega_{\Spec A}$, and similarly for $f^{\prime,\alg}$.
    We denote by $c_{f^\alg} \colon \omega_{\PP^d_A} \xr{\sim} \rR f^{\alg, !} \omega_{\Spec A}$ the isomorphism coming from \cite[Exp.~XVII, Prop.~4.1.2]{deGabber} (and similarly for $f^{\prime,\alg}$, $i^\alg$, and $i^{\prime,\alg}$) and remind the reader of the isomorphism $\alpha_{f^\alg} \colon f^{\alg, *} \omega_{\Spec A}(d)[2d] \xr{\sim} \omega_{\PP^d_A}$ from \cref{rmk:smooth-pullback-in-algebraic-geometry}.
    Then the second paragraph after \cite[Exp.~XVII, Lem.~4.4.1]{deGabber} implies that the composition
    \[
    f^{\alg, *}\omega_{\Spec A}(d)[2d] \xr[\sim]{\alpha_{f^\alg}} \omega_{\PP^d_A} \xr[\sim]{c_{f^\alg}} \rR f^{\alg, !} \omega_{\Spec A}
    \]
    is equal to the Poincar\'e duality isomorphism $\PD_{f^\alg}$ defined above;
    the same claim holds for $f'$.
    Therefore, after unraveling the definition of $\tr^\alg_{f^\alg}$, we conclude that $\tr_{f^\alg}^\alg$ is given by the composition
    \[
    \rR f_*^\alg \omega_{\PP^d_A} \xr[\sim]{\rR f^\alg_*(c_{f^\alg})} \rR f^\alg_* \rR f^{\alg, !} \omega_{\Spec A} \xr{\epsilon_{f^\alg}} \omega_{\Spec A},
    \]
    where $\epsilon_{f^\alg}$ is the counit of the $(\rR f^\alg_!, \rR f^{\alg, !})$-adjunction;
    a similar formula holds for $f^{\prime,\alg}$.
    After unraveling the definition of $\tr_{i^\alg}$, we see that it is also given by the composition
    \[
    \rR i_*^\alg \omega_{\Spec A/I} \xr[\sim]{i^\alg_*(c_{i^\alg})} i^\alg_* \rR i^{\alg, !} \omega_{\Spec A} \xr{\epsilon_{i^\alg}} \omega_{\Spec A}. 
    \]
    In conclusion, for the purpose of proving the commutativity of Diagram~\cref{eqn:many-traces}, it suffices to show that the following diagram commutes: 
    \[
    \begin{tikzcd}[column sep = 9em]
    \rR h^\alg_* \omega_{\PP^d_{A/I}} \arrow[r, "\rR h^\alg_*(c_{i^{\prime,\alg}})"] \arrow[d, "\rR h^\alg_*(c_{f^{\prime,\alg}})"] & \rR h^\alg_* \rR i^{\prime,\alg,!} \omega_{\PP^d_A} \arrow[r, "\rR f_*^\alg(\epsilon_{i^\alg})"] \arrow[d, "\rR h^\alg_* \rR i^{\prime,\alg,!}(c_{f^\alg})"] & \rR f^\alg_* \omega_{\PP^d_A} \arrow[d, "\rR f^\alg_*(c_{f^\alg})"] \\
    \rR h^\alg_* \rR f^{\prime,\alg, !} \omega_{\Spec A/I} \arrow[r, "\rR h^\alg_* \rR f^{\prime,\alg,!}(c_{i^\alg})"] \arrow[d, "i^\alg_*(\epsilon_{f^{\prime,\alg}})"] & \rR h^\alg_* \rR h^{\alg, !} \omega_{\Spec A}\arrow[d, "i^\alg_*(\epsilon_{f^{\prime,\alg}} i^{\alg,!})"]  \arrow[r, "\rR f^\alg_*(\epsilon_{i^{\prime,\alg}}\rR f^{\alg,!})"] & \rR f^\alg_* \rR f^{\alg,!} \omega_{\Spec A} \arrow[d, "\epsilon_{f^\alg}"] \\
    i^\alg_* \omega_{\Spec A/I} \arrow[r, "c_{i^\alg}"] & i^\alg_* \rR i^{\alg,!} \omega_{\Spec A} \arrow[r, "\epsilon_{i^\alg}"] & \omega_{\Spec A}
    \end{tikzcd}
    \] 
    The top right square and the bottom left square commute due to functoriality of $\epsilon_{i^\alg}$ and $\epsilon_{f^{\prime,\alg}}$, respectively, the bottom right square commutes due to the compatibility of adjunction counits with compositions, and the top left square commutes due to \cite[Exp.~XVII, Rmq.~4.1.3]{deGabber}. \qedhere
\end{enumerate}
\end{proof}

Now we are ready to show a version of \cref{tr-vs-cl} for traces with coefficients in dualizing complexes: 

\begin{lemma}
\label{compatibility lemma: trace of section class is 1}
Let $f \colon X \to Y$ be a separated taut smooth morphism of rigid-analytic spaces over $K$ and let $s \colon Y \hookrightarrow X$ be a section.
Then the composition of trace maps
\[
\omega_Y \simeq (\rR f_! \circ s_*) (\omega_Y) \xrightarrow{\rR f_!(\tr_{s})} \rR f_!\, \omega_X
\xrightarrow{\tr_f} \omega_Y
\]
is the identity.
\end{lemma}
\begin{proof}
    Without loss of generality, we may and do assume that $Y$ is connected.
    Pick a classical point $y \in Y$. The canonical inclusions fit into a cartesian diagram
    \[ \begin{tikzcd}
        X_y \arrow[r,hook,"i'"] \arrow[d,"f'"] & X \arrow[d,"f"] \\
        y \arrow[r,hook,"i"] \arrow[u,hook',bend left=40,"s'"] & Y. \arrow[u,hook',bend left=40,"s"]
    \end{tikzcd} \]
    Thanks to \cref{rmk:basic-properties-closed-trace}\cref{rmk:basic-properties-closed-trace-composition} and \cref{compatibility lemma: Cartesian diagram case}, respectively, the left and right square in the following diagram commute:
    \[ \begin{tikzcd}[column sep=huge]
        i_*\omega_y \arrow[r,phantom,"\simeq"] \arrow[d,"\tr_i"] &[-4.4em] i_* \rR f'_! s'_* \omega_y \arrow[r,"i_*\rR f'_!(\tr_{s'})"] &[1em] i_* \rR f'_! \omega_{X_y} \arrow[r,"i_*(\tr_{f'})"] \arrow[d,"\rR f_!(\tr_{i'})"] & i_* \omega_y \arrow[d,"\tr_i"] \\
        \omega_Y \arrow[r,phantom,"\simeq"] & \rR f_! s_* \omega_Y \arrow[r,"\rR f_!(\tr_s)"] & \rR f_! \omega_X \arrow[r,"\tr_f"] & \omega_Y.
    \end{tikzcd} \]
    It follows that it suffices to prove the assertion for $f'$ instead of $f$:
    indeed, by \cite[Th.~3.21~(3)]{BH} and the full faithfulness of $i_*$, the two horizontal compositions are given by scalar multiplication with an element of $\Lambda$ and the commutativity of the diagram guarantees that both classes map to the same element in 
    \[ \begin{tikzcd}[column sep=huge]
        \Lambda \arrow[r,"\sim"] \arrow[d,sloped,"\sim"] & \Hom(\omega_Y,\omega_Y) \arrow[d,"\blank \circ \tr_i", "\sim"'{sloped}] \\
        \Hom(\omega_y,\omega_y) \arrow[r,"\tr_i \circ i_*(\blank)", "\sim"'] & \Hom(i_*\omega_y,\omega_Y),
    \end{tikzcd} \]
    hence they must be equal.
    In conclusion, it is enough to prove the statement when $Y$ is a smooth rigid space over $K$ (in fact, $Y = \Spa(K',\cO_{K'})$ for some finite extension $K'/K$) and $X$ is a separated taut smooth rigid-analytic space over $K$.
    In this case, the result follows directly from \cref{tr-vs-cl} and \cref{closed trace and cycle class for smooth things}. 
\end{proof}

Lastly, we extend \cref{compatibility lemma: Cartesian diagram case} to commutative diagrams that are not necessarily cartesian.
\begin{theorem}[Compatibility for commutative diagrams]\label{compatibility proposition}
Consider a commutative diagram of rigid-analytic spaces over $K$
\[
\begin{tikzcd}
X' \arrow{d}{f'} \arrow{rd}{h} \arrow[r, hook, "i'"]& X \arrow{d}{f} \\
Y' \arrow[r, hook, "i"] & Y.
\end{tikzcd}
\]
Suppose that $f$ and $f'$ are separated, taut and smooth, and that $i$ and $i'$ are closed immersions.
Then the following diagram in $D(Y_\et; \Lambda)$ commutes:
\[
\begin{tikzcd}[column sep = huge]
\rR h_!\, \omega_{X'} \arrow{r}{\rR f_!(\tr_{i'})} \arrow{d}{i_*(\tr_{f'})} &
 \rR f_!\, \omega_X \arrow{d}{\tr_f} \\
i_* \omega_{Y'} \arrow{r}{\tr_i} & \omega_Y.
\end{tikzcd}
\]
\end{theorem}
\begin{proof}
    We first deal with two special cases in which one of the morphisms is the identity and then use them to deduce the general version.
    \begin{enumerate}[wide,label={\textit{Step~\arabic*}.},ref={Step~\arabic*}]
        \item\label{compatibility proposition-closed-closed-smooth} \textit{Proof when $X' = Y'$ and $f' = \id$.}
        The fiber product $W \colonequals X \times_Y Y'$ comes with natural projections $g \colon W \to Y'$ and $j \colon W \hookrightarrow X$.
        Moreover, $i'$ induces a natural section $s \colon Y' \to W$ of $g$.
        These maps fit into the commutative diagram
        \[ \begin{tikzcd}
        W \arrow[r,hook,"j"] \arrow[d,"g"] & X \arrow[d,"f"] \\
        Y' \arrow[r,hook,"i"] \arrow[u,hook',bend left=40,"s"] \arrow[ru,hook,"i'"] & Y.
        \end{tikzcd} \]
        An application of \cref{rmk:basic-properties-closed-trace}\cref{rmk:basic-properties-closed-trace-composition}, \cref{compatibility lemma: Cartesian diagram case},
        and \cref{compatibility lemma: trace of section class is 1}, respectively, then gives the desired identity
        \[ \tr_f \circ \rR f_!(\tr_{i'}) = \tr_f \circ \rR f_!(\tr_j) \circ \rR(f \circ j)_!(\tr_s) = \tr_i \circ i_*(\tr_g) \circ \rR(i \circ g)_!(\tr_s) = \tr_i. \]

        \item\label{compatibility proposition-smooth-closed-smooth} \textit{Proof when $Y' = Y$ and $i = \id$.}
        The fiber product $W \colonequals X' \times_Y X$ comes with natural projections $g \colon W \to X'$ and $g' \colon W \to X$, which are again separated, taut and smooth.
        Moreover, $i'$ induces a natural section $s \colon X' \to W$ of $g$.
        These maps fits into the commutative diagram
        \[ \begin{tikzcd}
        W \arrow[r,"g'"] \arrow[d,"g"] & X \arrow[d,"f"] \\
        X' \arrow[r,"f'"] \arrow[u,hook',bend left=40,"s"] \arrow[ru,hook,"i'"] & Y.
        \end{tikzcd} \]
        An application of \cref{compatibility proposition-closed-closed-smooth}, \cref{rmk:basic-properties-smooth-trace}\cref{rmk:basic-properties-smooth-trace-composition}, and \cref{compatibility lemma: trace of section class is 1}, respectively, then gives the desired identity
        \[ \tr_f \circ \rR f_!(\tr_{i'}) = \tr_f \circ \rR f_!(\tr_{g'}) \circ \rR (f \circ g')_!(\tr_s) = \tr_{f'} \circ \rR f'_!(\tr_g) \circ \rR (f' \circ g)_!(\tr_s) = \tr_{f'}. \]

        \item \textit{Proof in the general case.}
        The fiber product  $W \coloneqq X' \times_Y Y$ fits into the commutative diagram
        \[ \begin{tikzcd}
        X' \arrow[rd,hook,"j'"] \arrow[rrd,hook,bend left=20,"i'"] \arrow[rdd,bend right=25,"f'"] && \\
        & W \arrow[r,hook,"j"] \arrow[d,"g"] & X \arrow[d,"f"] \\
        & Y' \arrow[r,hook,"i"] & Y,
        \end{tikzcd} \]
        where $j$ and $j'$ are closed immersions.
        An application of \cref{rmk:basic-properties-closed-trace}\cref{rmk:basic-properties-closed-trace-composition}, \cref{compatibility lemma: Cartesian diagram case}, and \cref{compatibility proposition-smooth-closed-smooth} respectively, then gives the desired identity
        \[ \tr_f \circ \rR f_!(\tr_{i'}) = \tr_f \circ \rR f_!(\tr_j) \circ \rR (f \circ j)_!(\tr_{j'}) = \tr_i \circ i_*(\tr_g) \circ \rR (i \circ g)_!(\tr_{j'}) = \tr_i \circ i_*(\tr_{f'}). \qedhere \]
    \end{enumerate}
\end{proof}

\subsection{Smooth-source trace}

The main goal of this subsection is to construct a trace map for any separated taut morphism $f\colon X \to Y$ of rigid-analytic spaces over $K$ with $X$ smooth and $Y$ separated and taut.
In the next subsection, we will drop the assumptions on $X$ and $Y$ at the expense of assuming properness of $f$.

For the next construction, we fix a separated taut morphism $f\colon X \to Y$ as above.
This automatically implies that $X$ is separated and taut as well.
Then we factor $f$ as the composition
\[
X \xhookrightarrow{\Gamma_f} X\times Y \xrightarrow{\pi_Y} Y
\]
of the graph morphism $\Gamma_f$ and the natural projection $\pi_Y$.
Note that $\Gamma_f$ is a closed immersion since $Y$ is separated (see \cite[Cor.~B.6.10 and B.7.4]{Z-quotients}) and that $\pi_Y$ is a separated taut smooth morphism because $X$ is separated taut and smooth over $K$. 

\begin{construction}[Smooth-source trace]\label{smooth source trace}
For $f$ as above, the \emph{smooth-source trace map} $\tr_f\colon \rR f_!\, \omega_X \to \omega_Y$ is the composition
\[ \rR f_!\,\omega_X \simeq (\rR \pi_{Y,!} \circ \Gamma_{f,*}) \omega_X \xrightarrow{\rR \pi_{Y, !}(\tr_{\Gamma_f})} \rR \pi_{Y, !}\, \omega_{X\times Y} \xrightarrow{\tr_{\pi_{Y}}} \omega_Y, \]
where $\tr_{\Gamma_f}$ is the closed trace from \cref{trace for closed immersion} and $\tr_{\pi_Y}$ is the smooth trace from \cref{constructing smooth trace for omega}. 
\end{construction}
We now verify some basic properties of smooth-source trace maps. 
\begin{proposition}\label{compatibility for smooth source trace}
Let $f\colon X \to Y$ and $g\colon Y \to Z$ be separated taut morphisms of rigid-analytic spaces over $K$.
Assume that $X$ is smooth over $K$ and that $Z$ is separated and taut over $K$. 
\begin{enumerate}[leftmargin=*,label=\upshape{(\roman*)}]
\item\label{compatibility for smooth source trace-1}(Compatibility with smooth trace) If $f$ is smooth, then $\tr_f$ is equal to the smooth trace map from \cref{constructing smooth trace for omega}. 
In particular, $\tr_f=\rm{id}$ when $f=\rm{id}$. 
\item\label{compatibility for smooth source trace-2}(Compatibility with closed trace) If $f$ is a closed immersion, then $\tr_f$ is equal to the closed trace map from \cref{trace for closed immersion}.
\item\label{compatibility for smooth source trace-3}(Compatibility with smooth maps) If $h\colon Y' \to Y$ is a separated taut smooth morphism and 
\[ \begin{tikzcd}
X' \arrow[r,"h'"] \arrow[d,"f'"] & X \arrow[d,"f"] \\
Y' \arrow[r,"h"] & Y
\end{tikzcd} \]
is a cartesian diagram, then $\tr_h \circ \rR h_!(\tr_{f'}) = \tr_f \circ \rR f_!(\tr_{h'})$, where $\tr_h$ and $\tr_{h'}$ denote the smooth trace maps from \cref{constructing smooth trace for omega}.
\item\label{compatibility for smooth source trace-4}(Compatibility with compositions I) If $Y$ is smooth, then $\tr_{g\circ f}= \tr_g \circ \rR g_!(\tr_f)$.
\item\label{compatibility for smooth source trace-5}(Compatibility with compositions II) If $g$ is a closed immersion (resp.\ smooth), then $\tr_{g\circ f} = \tr_g \circ \rR g_!(\tr_f)$ where $\tr_g$ is the closed trace from \cref{trace for closed immersion} (resp.\ the smooth trace from \cref{constructing smooth trace for omega}). 
\end{enumerate}
\end{proposition}
\begin{proof}
Parts \cref{compatibility for smooth source trace-1} and \cref{compatibility for smooth source trace-2} follow directly from \cref{compatibility proposition} (more specifically, from the special cases treated in \cref{compatibility proposition-smooth-closed-smooth} and \cref{compatibility proposition-closed-closed-smooth} of its proof, respectively).

Now we deal with \cref{compatibility for smooth source trace-3}. For this, we consider the following commutative diagram:\footnote{We warn the reader that the inner squares in the diagram below are not cartesian.}
\[
\begin{tikzcd}[column sep=large]
X' \arrow[rr, bend left, "f'"] \arrow[r, hook, "\Gamma_{f'}"] \arrow{d}{h'} & X' \times Y' \arrow{r}{\pi_{Y'}}\arrow{d}{h'\times h} & Y' \arrow{d}{h}\\
X \arrow[rr, bend right, "f"] \arrow[r, hook, "\Gamma_f"]  & X\times Y \arrow{r}{\pi_Y} & Y
\end{tikzcd}
\]
Note that $h$, $h'$, $\pi_Y$, and $\pi_{Y'}$ are separated, taut and smooth, and that $\Gamma_f$ and $\Gamma_{f'}$ are closed immersions.
Therefore,  the assertion results from
\begin{align*}
\tr_h \circ \rR h_!(\tr_{f'}) & = \tr_h \circ \rR h_!(\tr_{\pi_{Y'}}) \circ \rR (h\circ \pi_{Y'})_!(\tr_{\Gamma_{f'}}) \\ 
& = \tr_{\pi_Y} \circ \rR \pi_{Y, !} (\tr_{h'\times h}) \circ \rR \bigl(\pi_Y \circ (h'\times h)\bigr)_!(\tr_{\Gamma_{f'}}) \\
& = \tr_{\pi_Y} \circ \rR \pi_{Y, !} \bigl(\tr_{\Gamma_f} \circ \rR\Gamma_{f,!}(\tr_{h'})\bigr) \\
& = \tr_{\pi_Y} \circ \rR \pi_{Y, !} (\tr_{\Gamma_f}) \circ \rR f_!(\tr_{h'}) \\
& = \tr_{f} \circ \rR f_!(\tr_{h'}), 
\end{align*}
where the first equality follows from \cref{smooth source trace}, the second equality follows from \cref{rmk:basic-properties-smooth-trace}\cref{rmk:basic-properties-smooth-trace-composition}, third equality follows from \cref{compatibility proposition}, the fourth equality follows from $f=\pi_Y \circ \Gamma_f$, and the last equality follows again from \cref{smooth source trace}. 

To prove \cref{compatibility for smooth source trace-4}, we first treat the case when $f$ is a closed immersion.
In this case, we consider the following commutative diagram:
\[
\begin{tikzcd}
X \arrow[rr, hook, "\Gamma_{g\circ f}"] \arrow[rd, hook, swap, "f"] & & X\times Z \arrow[d, hook, "f\times \id"] \arrow[dd, bend left = 80, "\pi_Z^X"]\\
& Y \arrow[dr, swap, "g"] \arrow[r, hook, "\Gamma_g"] & Y\times Z \arrow[d, "\pi_Z^Y"] \\
& &  Z
\end{tikzcd}
\]
Note that $f$, $\Gamma_{g\circ f}$, $f\times \id$, and $\Gamma_g$ are closed immersions. Therefore, the assertion results from
\begin{align*}
    \tr_g \circ \rR g_!(\tr_f) & = \tr_{\pi^Y_Z} \circ \rR \pi^Y_{Z, !}(\tr_{\Gamma_g}) \circ \rR\pi^Y_{Z, !}\bigl(\rR\Gamma_{g,!}(\tr_{f})\bigr) \\
    & = \tr_{\pi_Z^Y} \circ \rR \pi^Y_{Z, !}(\tr_{f\times \id}) \circ \rR \pi_{Z, !}^X(\tr_{\Gamma_{g\circ f}}) \\ 
    & = \tr_{\pi_Z^X} \circ \rR \pi_{Z, !}^X(\tr_{\Gamma_{g\circ f}}) \\
    & = \tr_{g\circ f},
\end{align*}
where the first equality follows from \cref{smooth source trace}, the second equality follows from part \cref{compatibility for smooth source trace-2} and \cref{rmk:basic-properties-closed-trace}\cref{rmk:basic-properties-closed-trace-composition}, the third equality follows from \cref{compatibility proposition}, and the last equality follows again from \cref{smooth source trace}.

Now we prove \cref{compatibility for smooth source trace-4} for a general $f$.
For this, we consider the following commutative diagram:
\begin{equation*}\label{eqn:nice-diagram-compatibility-trace}
\begin{tikzcd}[column sep =3em, row sep = 3em]
X \arrow[r, hook, "\Gamma_f"] \arrow{dr}{f}  \arrow[rr, hook, bend left, "\Gamma_{g\circ f}"] & X\times Y \arrow{d}{\pi_{Y}} \arrow{r}{\id \times g} & X\times Z  \arrow{d}{\pi_Z} \\
& Y \arrow{r}{g} & Z 
\end{tikzcd}
\end{equation*}
Then the assertion results from
\begin{align*}
    \tr_g \circ \rR g_!(\tr_f) & =  \tr_g \circ \rR g_!(\tr_{\pi_Y}) \circ \rR g_!\bigl(\rR \pi_{Y, !} (\tr_{\Gamma_f})\bigr) \\
    & = \tr_{\pi_Z} \circ \rR \pi_{Z, !}(\tr_{\id\times g}) \circ \rR \pi_{Z, !}\bigl(\rR (\id\times g)_!(\tr_{\Gamma_f})\bigr) \\
    & = \tr_{\pi_Z} \circ \rR \pi_{Z, !}(\tr_{\Gamma_{g\circ f}}) \\
    & = \tr_{g\circ f},
\end{align*}
where the first equality follows from \cref{smooth source trace}, the second equality follows from \cref{compatibility proposition}, the third equality follows from the observation that $X\times Y$ is smooth and the case of closed immersions established above, and the last equality follows again from \cref{smooth source trace}.

The proof of \cref{compatibility for smooth source trace-5} is similar to that of \cref{compatibility for smooth source trace-4}.
We leave the details to the interested reader.
\end{proof}

\Cref{compatibility for smooth source trace}\cref{compatibility for smooth source trace-3} formally implies that the smooth-source trace is \'etale local, yielding the following variant of \cref{rmk:basic-properties-closed-trace}\cref{rmk:basic-properties-closed-trace-local} and \cref{rmk:basic-properties-smooth-trace}\cref{rmk:basic-properties-smooth-trace-local}: 

\begin{remark}[Smooth-source trace is \'etale-local on the target]\label{rmk:smooth-source-trace-etale-local}
Let $f \colon X \to Y$ be a separated taut morphism and $h \colon Y' \to Y$ be a separated taut \'etale morphism of rigid-analytic spaces over $K$.
Assume that $X$ is smooth over $K$ and that $Y$ is separated and taut over $K$.
Let $X' \colonequals Y'\times_Y X$ be the fiber product and $f' \colon X' \to Y'$ and $h' \colon X' \to X$ the two natural projections.
Then the $(h_!, h^*)$-adjoint of the equality $\tr_h \circ \rR h_!(\tr_{f'}) = \tr_f \circ \rR f_!(\tr_{h'})$ from \cref{compatibility for smooth source trace}\cref{compatibility for smooth source trace-3} amounts by virtue of \cref{rmk:etale-trace-dualizing} to the commutativity of the following diagram:
\[ \begin{tikzcd}[column sep=large]
    \rR f'_! \omega_{X'} \arrow[rr,"\tr_{f'}"]  &[-3em] & \omega_{Y'}  \\
    \rR f'_!h^{\prime *} \omega_X \arrow[u,"\rR f'_!(\alpha_{h'})","\sim"'{sloped}] \arrow[r,phantom,"="] & h^*\rR f_! \omega_X \arrow[r,"h^*(\tr_f)"] & h^*\omega_Y \arrow[u,"\alpha_h","\sim"'{sloped}]. 
\end{tikzcd} \]
\end{remark}

Our first application of the smooth-source trace will be a vanishing result for the Verdier dual of the derived pushforward of a dualizing sheaf (see \cref{proper trace lives in discrete space}). This vanishing result will play a crucial role in our construction of general proper trace in next subsection. 

Before we start proving this vanishing result, we need a number of preliminary lemmas. 

\begin{lemma}\label{lemma:shriek-star-base-change}
Let $f\colon X \to Y$ be a morphism of rigid-analytic spaces over $K$, let $i \colon Y' \hookrightarrow Y$ be a closed immersion of rigid-analytic spaces over $K$, and let 
\[ \begin{tikzcd}
X' \arrow{d}{f'} \arrow[r, hook, "i'"] & X \arrow{d}{f} \\
Y' \arrow[r, hook, "i"] & Y
\end{tikzcd} \]
be the resulting pullback square. Then the natural transformation of functors
\begin{equation}\label{eqn:natural-transformation-!-*}
\rR f'_* \rR i^{\prime,!}(\blank) \to \rR i^! \rR f_*(\blank) \colon D(X_\et; \Lambda) \longrightarrow D(Y'_\et; \Lambda),
\end{equation}
given by the $(i_*,\rR i^!)$-adjoint to $i_* \rR f'_* \rR i^{\prime,!}(\blank) \simeq \rR f_* i'_* \rR i^{\prime,!}(\blank) \xr{\rR f_*(\epsilon_{i'})} \rR f_*(\blank)$, is an equivalence. 
\end{lemma}
\begin{proof}
    Since both $\rR f_*$ and $\rR i^!$ are right adjoints, we may show instead that the natural transformation \cref{eqn:natural-transformation-!-*} is an equivalence after passing to left adjoints.
    In other words, it suffices to prove that the natural transformation $f^* i_*(\blank) \to i'_*f'^*(\blank)$ is an equivalence.
    This can be checked easily by arguing on stalks.
\end{proof}

\begin{corollary}\label{cor:shriek-star-base-change-dualizing-complex}
Under the assumptions of \cref{lemma:shriek-star-base-change}, there is a canonical isomorphism
\[
c_{f, i}\colon \rR f'_* \omega_{X'} \xrightarrow{\sim} \rR i^! \rR f_*\omega_X.
\]
\end{corollary}
\begin{proof}
    This follows directly from \cref{lemma:shriek-star-base-change} and \cite[Th.~3.21~(1)]{BH}.
\end{proof}

\begin{notation}\label{tr-!-pullback}
Let $f\colon X \to Y$ and $i\colon Y' \hookrightarrow Y$ be as in \cref{lemma:shriek-star-base-change}. Assume that $f$ is proper, $X$ is smooth over $K$, and $Y$ is taut and separated. Then we denote by 
\[
\rR i^!(\tr_f) \colon \rR f'_* \omega_{X'} \to \omega_{Y'}
\]
the following composition
\[
\rR f'_* \omega_{X'} \xrightarrow[\sim]{c_{f, i}} \rR i^! \rR f_* \omega_X \xrightarrow{\rR i^!(\tr_f)} \rR i^! \omega_X \xrightarrow[\sim]{c_i^{-1}} \omega_Z,
\]
where $\tr_f$ is the smooth-source trace from \cref{smooth source trace}.
\end{notation}

For future reference, it will be convenient to introduce the following definition: 

\begin{definition}\label{defn:modification} Let $X$ be a rigid-analytic space over $K$, $i\colon Z \hookrightarrow X$ a Zariski-closed immersion, and $U\subset X$ be its (Zariski) open complement. 
\begin{enumerate}[leftmargin=*,label=\upshape{(\roman*)}]
    \item A \textit{$U$-modification} $\pi \colon X' \to X$ is a proper morphism of rigid-analytic spaces over $K$ such that $\restr{\pi}{\pi^{-1}(U)} \colon \pi^{-1}(U) \to U$ is an isomorphism.
    \item A \textit{$U$-admissible modification} $\pi \colon X' \to X$ is a $U$-modification such that $U\subset X$ and $U\simeq \pi^{-1}(U) \subset X'$ are dense.
    \item A \textit{regular $U$-modification} $\pi \colon X' \to X$ is a $U$-modification such that $X'$ is smooth over $K$.
    \item A \textit{regular $U$-admissible modification} is a regular $U$-modification which is $U$-admissible.
\end{enumerate}
For a $U$-modification $\pi\colon X' \to X$, we will often denote by $i'\colon Z' \coloneqq \pi^{-1}(Z)=X'\times_X Z \hookrightarrow X'$ the preimage of $Z$ along $X$. 
\end{definition}

There is an abundance of regular $U$-admissible modifications:
\begin{proposition}\label{modifications-exist}
    Let $X$ be a quasicompact reduced rigid-analytic space over $K$.
    Then:
    \begin{enumerate}[leftmargin=*,label=\upshape{(\roman*)}]
        \item\label{modification-exist-smooth-locus} The smooth locus of $X$ is Zariski-open and dense.
        \item\label{modifications-exist-Temkin} (Temkin) For any Zariski-open and dense subspace $U \subseteq X$ that is contained in the smooth locus of $X$, there exists a regular $U$-admissible modification $\pi \colon X' \to X$.
    \end{enumerate}
\end{proposition}
We remind the reader that we always (implicitly) assume that $\charac K =0$ in this section.
\begin{proof}
    Part \cref{modification-exist-smooth-locus} follows from the fact that $K$-affinoid algebras are excellent;
    see e.g.\ the discussion after \cite[Lem.~3.3.1]{Conrad99}.
    Part \cref{modifications-exist-Temkin} is \cite[Th.~5.2.2]{Tem12} (cf.\ also \cite[Th.~1.2.1]{Tem12}).
\end{proof}

\begin{lemma}\label{lemma:dimension-of-fibers-modifications} Let $U\subset X$ be a dense Zariski-open subspace, and let $\pi \colon X' \to X$ be a $U$-admissible modification. Then, for any classical point $x\in X$, we have $\dim \pi^{-1}(x)< \max(\dim X, 1)$.
\end{lemma}
\begin{proof}
    We denote by $Z\subset X$ Zariski-closed complement to $U$ (with reduced adic space structure) and by $Z'\subset X'$ the fiber product $Z' \coloneqq Z\times_X X'$. 

    Now we start the proof. If $x\in U$, then $\pi^{-1}(x)$ is a singleton. In particular, $\dim \pi^{-1}(x)=0< \max(\dim X, 1)$. If $x\in Z$, then $\dim \pi^{-1}(x) \leq \dim Z' < \dim X' = \dim U = \dim X \leq \max(\dim X, 1)$, where the second inequality holds due to the assumption that $Z'\subset X'$ is nowhere dense.
\end{proof}

\begin{lemma}\label{lemma:exact-triangle} Let $X$ be a rigid-analytic space over $K$, let $i\colon Z\hookrightarrow X$ be a Zariski-closed subspace over $K$ with the open complement $U\subset X$, and let $\pi \colon X' \to X$ be a regular $U$-modification. Put $Z' \coloneqq X'\times_X Z$ and $i'\colon Z' \hookrightarrow X'$ and $\pi' \colon Z' \to Z$ be the natural projections and $h\colon Z' \to X$ the evident composition. Then there is an exact triangle
\[
\rR h_* \omega_{Z'} \xr{\bigl(-\rR \pi_*(\tr_{i'}), i_*\rR i^!(\tr_\pi)\bigr)} \rR \pi_*\omega_{X'} \oplus i_* \omega_{Z} \xr{\tr_\pi \oplus \tr_i} \omega_X \to \rR h_* \omega_{Z'}[1]
\]
in $D(Y_\et; \Lambda)$.
\end{lemma}
\begin{proof}
    For brevity, we denote by $\alpha \colon \rR h_* \omega_{Z'} \to \rR \pi_*\omega_{X'} \oplus i_* \omega_{Z}$ the morphism $\bigl(-\rR \pi_*(\tr_{i'}), i_*\rR i^!(\tr_\pi)\bigr)$ and by $\beta \colon \rR \pi_*\omega_{X'} \oplus  i_* \omega_{Z} \to \omega_X$ the morphism $\tr_\pi \oplus \tr_i$.
    We also set $C\coloneqq \rm{fib}(\beta) = \rm{cone}(\beta)[-1]$.
    It fits into an exact triangle
    \begin{equation*}\label{eqn:fiber}
        C \xr{\alpha'} \rR \pi_*\omega_{X'} \oplus i_* \omega_{Z} \xr{\beta} \omega_X \xr{\gamma} C[1].
    \end{equation*}
    Now, after unravelling all the definitions, we see that the following diagram commutes:
    \[
    \begin{tikzcd}[column sep = 5em]
        i_* \rR \pi'_* \omega_{Z'} \arrow[d, equals] \arrow[r, red, "i_*(c_{\pi, i})"] & i_* \rR i^! \pi_*\omega_{X'} \arrow[d, "\epsilon_i"] \arrow[r, red, "i_*\rR i^!(\tr_\pi)"] &  i_*\rR i^! \omega_X \arrow[dr, "\epsilon_i"] \arrow[r, red, "i_*(c_i^{-1})"] & i_* \omega_{Z} \arrow[d, red, "\tr_i"] \\
        \rR \pi_* i'_* \omega_{Z'} \arrow[r,"\rR \pi_*(\tr_{i'})"] & \rR \pi_* \omega_{X'}  \arrow[rr, "\tr_\pi"] & & \omega_X,
    \end{tikzcd}
    \]
    where $\epsilon_i$ is the counit of the $(i_*, \rR i^!)$-adjunction. Since the composition of red arrows is equal to $\tr_i\circ i_*\rR i^!(\tr_\pi)$, we conclude that $\beta \circ \alpha =0$. Then the axioms of triangulated categories imply that there is a morphism $A\colon \rR h_* \omega_{Z'} \to C$ such that the following diagram commutes
    \begin{equation}\label{eqn:maps-of-fibers}
    \begin{tikzcd}
        \rR h_* \omega_{Z'}  \arrow{d}{A} \arrow{r}{\alpha} & \rR \pi_*\omega_{X'} \oplus i_* \omega_{Z} \arrow[d, "\id"]\arrow[r, "\beta"] & \omega_X \arrow[d, "\id"] \\
        C \arrow{r}{\alpha'}& \rR \pi_*\omega_{X'} \oplus  i_* \omega_{Z} \arrow{r}{\beta}& \omega_X.
    \end{tikzcd}
    \end{equation}
    Therefore, it suffices to show that $A$ is an isomorphism. For this, it suffices to check that both $\restr{A}{U}$ and $\rR i^!A$ are isomorphisms. 

    We first show that $\restr{A}{U}$ is an isomorphism. Clearly, $\restr{\bigl(\rR h_* \omega_{Z'}\bigr)}{U} =0$, so it suffices to show that $\restr{C}{U}=0$. This follows from the observations that $\restr{\bigl(i_* \omega_{Z}\bigr)}{U}=0$ and $\restr{(\tr_\pi)}{U}=\id$ (see \cref{compatibility for smooth source trace}\cref{compatibility for smooth source trace-1}).
    
    Now we show that $\rR i^!A$ is an isomorphism. In this case, we note that, after applying $\rR i^!$ to \cref{eqn:maps-of-fibers}, it becomes isomorphic to the following commutative diagram:
    \begin{equation}\label{eqn:maps-of-fibers-2}
    \begin{tikzcd}[column sep = 10em]
        \rR \pi'_* \omega_{Z'}  \arrow{d}{\rR i^!A} \arrow{r}{\rR i^! \alpha = \bigl(-\id, \rR i^!(\tr_\pi)\bigr)} & \rR \pi'_*\omega_{Z'} \oplus \omega_{Z} \arrow[d, "\id"]\arrow[r, "\rR i^!\beta = \rR i^!(\tr_\pi)\oplus \id"] & \omega_{Z} \arrow[d, "\id"] \\
        \rR i^! C \arrow{r}{\rR i^! \alpha'}& \rR \pi'_*\omega_{Z'} \oplus  \omega_{Z} \arrow{r}{\rR i^!\beta = \rR i^!(\tr_\pi)\oplus \id}& \omega_{Z}.
    \end{tikzcd}
    \end{equation}
    Now we note that the map $\rR i^! \beta$ admits a section $(0, \id) \colon \omega_{Z} \to \rR \pi'_*\omega_{Z'} \oplus  \omega_{Z}$. Therefore, we conclude that the boundary map $\rR i^!(\gamma) =0$. Likewise, we use that $\rR i^! \beta$ admits a section to verify that the first row of \cref{eqn:maps-of-fibers-2} extends to a distinguished triangle $\rR \pi'_* \omega_{Z'} \xr{\rR i^!(\alpha)} \rR \pi'_*\omega_{Z'} \oplus \omega_{Z} \xr{\rR i^! \beta} \omega_{Z} \xr{0} \rR \pi'_*\omega_{Z'}[1]$. Therefore, we conclude that we can extend \cref{eqn:maps-of-fibers-2} to a morphism of distinguished triangles
    \[
    \begin{tikzcd}[column sep=9em]
        \rR \pi'_* \omega_{Z'}  \arrow{d}{\rR i^!A} \arrow{r}{\rR i^! \alpha = (-\id, \rR i^!(\tr_\pi))} & \rR \pi'_*\omega_{Z'} \oplus \omega_{Z} \arrow[d, "\id"]\arrow[r, "\rR i^!\beta = \rR i^!(\tr_\pi)\oplus \id"] & \omega_{Z} \arrow[d, "\id"] \arrow[r, "0"] & \rR \pi_* \omega_{Z'}[1] \arrow{d}{\rR i^!A[1]} \\
        \rR i^! C \arrow{r}{\rR i^! \alpha'}& \rR \pi'_*\omega_{Z'} \oplus  \omega_{Z} \arrow{r}{\rR i^!\beta = \rR i^!(\tr_\pi)\oplus \id}& \omega_{Z} \arrow[r, "\rR i^! \gamma = 0"] & \rR i^! C[1].
    \end{tikzcd}
    \]
    Since this is a morphism of distinguished triangles and two-out-of-three vertical arrows are isomorphisms, we conclude that $\rR i^! A$ must be an isomorphism as well. This finishes the proof.
\end{proof}

Finally, we are ready to prove the desired vanishing result: 

\begin{theorem}\label{proper trace lives in discrete space}
Let $f \colon X \to Y$ be a proper map of rigid-analytic spaces over $K$. Then $\rR \cHom (\rR f_* \omega_X, \omega_Y)$ lies in $D^{\geq 0}(Y_\et;\Lambda)$.
\end{theorem}
\begin{proof}
\begin{enumerate}[wide,label={\textit{Step~\arabic*}.},ref={Step~\arabic*}]
    \item \textit{Reduce to the case when $Y$ is a geometric point.}
    First, we note that $\rR \cHom(\rR f_* \omega_X, \omega_Y)$ lies in $D_\zc(Y_\et; \Lambda)$ \cite[Th.~3.10, Th.~3.21~(3)]{BH}.
    Therefore, it suffices to show that for every classical point $i_y \colon y = \Spa\bigl(k(y), k(y)^\circ\bigr)\hookrightarrow Y$, the pullback $i_y^* \rR \cHom(\rR f_* \omega_X, \omega_Y)$ lies in $D^{\geq 0}(y_\et; \Lambda)$.
    To simplify this complex even further, we consider the following cartesian square:
    \[ \begin{tikzcd}
    X_{y} \arrow[d,"f_y"] \arrow[r,"i'_y"] & X \arrow[d,"f"] \\
    y \arrow[r,"i_y"] & Y.
    \end{tikzcd} \]
    \Cref{cor:shriek-star-base-change-dualizing-complex} constructs a canonical isomorphism $\rR f_{y, *}\omega_{X_y} \simeq \rR i_y^! \rR f_* \omega_X$. Combining this with (the proof of) \cite[Th.~3.21~(4)]{BH}, we obtain 
    \begin{equation*}
    i_y^* \rR \cHom (\rR f_* \omega_X, \omega_Y) = i_y^* \DD_Y(\rR f_*\omega_X) \simeq \DD_y(\rR i_y^!\rR f_* \omega_X) \simeq \DD_y(\rR f_{y, *} \omega_{X_y}).
    \end{equation*}
    Therefore, it suffices to show that $\DD_y(\rR f_{y, *} \omega_{X_y})$ lies in $D^{\ge 0}(y_\et; \Lambda)$. Since the statement does not depend on $K$, we can replace $K$ with $k(y)$ to assume that $Y=\Spa(K, \O_K)$. Then \cite[Prop.~3.24]{BH} ensures that we can replace $K$ with $\wdh{\ov{K}}$ to assume that $K$ is algebraically closed.
    
    \item \textit{Proof when $Y$ is a geometric point.}
    In this case, we have $D(Y_\et; \Lambda)=D(\Lambda)$ and $\omega_Y = \Lambda$.
    Therefore, the question becomes equivalent to showing that $\rR \Gamma(X, \omega_X)$ lies in $D^{\leq 0}(\Lambda)$.
    Thanks to \cref{dualizing-nilimmersion}, we have $\rR \Gamma(X,\omega_X) \simeq \rR \Gamma(X_\red,\omega_{X_\red})$.
    After replacing $X$ with $X_\red$, we may thus assume that $X$ is reduced.
    
    Since $X$ is proper, it has finite dimension.
    We prove the claim by induction on $d=\dim X$.
    If $d \le 0$, then $X$ is either empty or a finite disjoint union of points, so the result is obvious.
    Therefore, we assume that $d\geq 1$ and that we know the result for all $\dim X \le d-1$. 

    Let $U \subseteq X$ the smooth locus of $X$.
    By \cref{modifications-exist}, $U$ is Zariski-open and dense and there exists a regular $U$-admissible modification $\pi\colon X' \to X$.
    Set $Z \colonequals X \smallsetminus U$ (with the reduced adic space structure) and $Z' \colonequals Z\times_X X'$;
    we have $\dim Z< \dim X$ and $\dim Z' < \dim X' = \dim X$ since $U$ is dense in $X$ and $U \simeq \pi^{-1}(U)$ is dense in $X'$. We denote by $h\colon Z' \to X$ the natural composition. Then \cref{lemma:exact-triangle} implies that we have the following exact triangle
    \[
    \rR h_* \omega_{Z'} \to i_* \omega_Z \oplus \rR \pi_* \omega_{X'} \to \omega_X.
    \]
    Therefore, it suffices to show that the complexes $\rR\Gamma(X, \rR h_* \omega_{Z'}) = \rR \Gamma(Z', \omega_{Z'})$, $\rR\Gamma(X, i_* \omega_Z)=\rR\Gamma(Z, \omega_Z)$, and $\rR\Gamma(X, \rR \pi_* \omega_{X'}) = \rR \Gamma(X', \omega_{X'})$ lie in $D^{\leq 0}(\Lambda)$.
    Since $\dim Z< \dim X=d$ and $\dim Z' < \dim X' = d$,
    we conclude that the first two complexes lie in $D^{\leq 0}(\Lambda)$ by the induction hypothesis. Finally, \cref{smooth trace lives in discrete space} and smoothness of $X'$ imply that $\rR\Gamma(X', \omega_{X'})$ lies in $D^{\leq 0}(\Lambda)$ as well. This finishes the proof. \qedhere
\end{enumerate}
\end{proof}

\begin{remark}\label{rmk:general-coconnectivity}
It seems reasonable to expect that $\rR\cHom(\rR f_!\omega_X, \omega_Y)$ lies in $D^{\geq 0}(Y_\et; \Lambda)$ for any taut separated morphism $f$. For instance, \cref{closmooth trace lives in discrete space} implies that this holds for smooth $f$, while \cref{proper trace lives in discrete space} 
ensures it for proper maps.
However, we cannot justify this more general expectation in full generality. 
\end{remark}

\begin{corollary}\label{cor:injective-homs-dualizing}
Keep the notation of \cref{lemma:exact-triangle}. Let $f\colon X \to Y$ be a proper morphism of rigid-analytic spaces over $K$, and let $i_Y \colon Z \to Y$, $\pi_Y \colon X' \to Y$, and $h_Y \colon Z' \to Y$ be the natural compositions.
Then the sequence
\begin{multline*}
0 \to \Hom\bigl(\rR f_*\omega_X, \omega_Y\bigr) \xr{(\blank\circ \rR f_*(\tr_\pi), \blank\circ \rR f_*(\tr_i))} \Hom\bigl(\rR \pi_{Y,*} \omega_{X'}, \omega_Y\bigr) \oplus \Hom\bigl(\rR i_{Y,*} \omega_{Z}, \omega_Y\bigr) \\ \xr{\blank\circ \rR \pi_{Y, *}((-1)\cdot\tr_{i'}) \oplus \blank \circ \rR i_{Y,*}(\rR i^!\tr_\pi)} \Hom\bigl(\rR h_* \omega_{Z'}, \omega_Y\bigr)
\end{multline*}
is exact. 
\end{corollary}
\begin{proof}
    \cref{lemma:exact-triangle} implies that there is an exact triangle
    \begin{multline*}
    \rR\Hom\bigl(\rR f_*\omega_X, \omega_Y\bigr) \xr{(-\circ \rR f_*(\tr_\pi), -\circ \rR f_*(\tr_i))} \rR\Hom\bigl(\rR \pi_{Y, *} \omega_{X'}, \omega_Y\bigr) \oplus \rR\Hom\bigl(\rR i_{Y, *} \omega_{Z}, \omega_Y\bigr) \\ \xr{-\circ\rR \pi_{Y, *}((-1)\cdot \tr_{i'}) \oplus -\circ \rR i_{Y, *}(\rR i^!\tr_\pi)} \rR\Hom\bigl(\rR h_* \omega_{Z'}, \omega_Y\bigr)
    \end{multline*}
    Therefore, it suffices to show that $\Ext^{-1}(\rR h_* \omega_{Z'}, \omega_Y)=0$. This follows directly from \cref{proper trace lives in discrete space}. 
\end{proof}

\subsection{Proper trace in general}
\label{general proper trace subsection}

It is time to bootstrap \cref{smooth source trace} to the case of general proper maps.
We begin by stating the main goal for this subsection.
\begin{theorem}\label{most general proper trace!!}
There is a unique way to assign to every proper morphism $f \colon X \to Y$ of rigid-analytic spaces over $K$ a trace map $\tr_f \in \Hom(\rR f_*\omega_X,\omega_Y)$ in $D(Y_\et;\Lambda)$, satisfying the following properties:
\begin{enumerate}[label=\upshape{(\arabic*)}]
    \item\label{most general proper trace!!-closed}(Compatibility with closed trace) When $f$ is a closed immersion, then $\tr_f$ equals the closed trace map from \cref{trace for closed immersion}.
    \item\label{most general proper trace!!-smooth}(Compatibility with smooth-source trace) When $X$ is smooth and $Y$ is separated and taut, then $\tr_f$ equals the smooth-source trace map from \cref{smooth source trace}.
    \item\label{most general proper trace!!-composition}(Compatibility with compositions) For any two proper morphisms $f \colon W \to X$ and $g \colon X \to Y$ of rigid-analytic spaces over $K$, we have $\tr_{g \circ f} = \tr_g \circ \rR g_*(\tr_f)$.
    \item\label{most general proper trace!!-localization}(\'Etale-local on target) For any pullback diagram of rigid-analytic spaces over $K$
    \[ \begin{tikzcd}
        \widetilde{X} \arrow[r,"\tilde{h}"] \arrow[d,"\tilde{f}"] & X \arrow[d,"f"] \\
        \widetilde{Y} \arrow[r,"h"] & Y
    \end{tikzcd} \]
    in which $f$ and $\tilde{f}$ are proper and $h$ and $\tilde{h}$ are \'etale, the following diagram commutes (with the vertical isomorphisms coming from \cref{cor:smooth-pullback}):
    \[ \begin{tikzcd}[column sep=large]
        \rR \tilde{f}_* \omega_{\widetilde{X}} \arrow[rr,"\tr_{\tilde{f}}"]  &[-3em] & \omega_{\widetilde{Y}}  \\
        \rR \tilde{f}_*\tilde{h}^* \omega_X \arrow[u,"\rR \tilde{f}_*(\alpha_{\tilde{h}})","\sim"'{sloped}] \arrow[r,phantom,"="] & h^*\rR f_* \omega_X \arrow[r,"h^*(\tr_f)"] & h^*\omega_Y \arrow[u,"\alpha_h","\sim"'{sloped}] 
    \end{tikzcd} \]
\end{enumerate}
\end{theorem}
Before embarking on the proof of \cref{most general proper trace!!}, we explain the main idea:
Assume that $f \colon X \to Y$ is a proper morphism of rigid-analytic spaces over $K$ with $X$ reduced, quasicompact, and separated.
Let $U \subseteq X$ be a dense, Zariski-open subspace which is contained in the smooth locus of $X$ (cf.\ \cref{modifications-exist}\cref{modification-exist-smooth-locus}).
Let $Z \colonequals X \smallsetminus U$ be its complement, endowed with the canonical reduced adic space structure.
\Cref{modifications-exist}\cref{modifications-exist-Temkin} yields a regular $U$-admissible modification $\pi \colon X' \to X$, fitting into a commutative diagram
\begin{equation}\label{modification-square} \begin{tikzcd}[/tikz/column 1/.append style={anchor=base east,column sep=0pt,inner xsep=0pt}]
    Z \times_X X' \equalscolon & Z' \arrow[r,hook,"i'"] \arrow[d,"\pi'"] \arrow[rd,"h"] & X' \arrow[d,"\pi"] \arrow[rdd,bend left=25,"\pi_Y"] & \\
    & Z \arrow[r,hook,"i"] \arrow[rrd,bend right=20,"i_Y"] & X \arrow[rd,"f"] & \\
    &&& Y.
\end{tikzcd} \end{equation}
Since $X'$ is smooth over $K$, \cref{smooth source trace} provides us with a trace map $\tr_{\pi_Y}$, and since $\dim Z < \dim X$, we may assume (by induction on the dimension of the source) that we also have a trace map $\tr_{i_Y}$.
Using \cref{cor:injective-homs-dualizing}, we will then check that these uniquely determine a map $\tr_f \colon \rR f_* \omega_X \to \omega_Y$.
Afterward, we will show that this definition does not depend on the choice of $U$ and $\pi$ and satisfies the desired compatibilities \cref{most general proper trace!!-closed}, \cref{most general proper trace!!-smooth}, \cref{most general proper trace!!-composition}, and \cref{most general proper trace!!-localization}.
Now the details:
\begin{proof}[{Proof of \cref{most general proper trace!!}}]
    \begin{enumerate}[wide,label={\textit{Step~\arabic*}.},ref={Step~\arabic*}]
        \item\label{most general proper trace!!-qcsep} \textit{It suffices to show the statement when all rigid-analytic spaces involved are quasicompact and separated.}
        Indeed, for any proper morphism $f \colon X \to Y$ of general rigid-analytic spaces over $K$, the coconnectivity result from \cref{proper trace lives in discrete space} combined with the BBDG gluing lemma \cite[Prop.~3.2.2]{BBDG} guarantees that $\cHom(\rR f_*\omega_X,\omega_Y)$ forms a sheaf on the \'etale site $Y_\et$.
        For any quasicompact and separated open $U \subseteq Y$, the preimage $f^{-1}(U)$ under the proper map $f$ is again quasicompact and separated.
        Assume that we have unique trace maps $\tr_{f|_U} \in \cHom(\rR f_*\omega_X,\omega_Y)(U)$ which satisfy the desired compatibilities. 
        Then the \'etale locality of traces from \cref{most general proper trace!!-localization} allow us to glue the $\tr_{f|_{U_i}}$ for some quasicompact separated (e.g., affinoid) covering $X = \bigcup_{i \in I} U_i$ to a map $\tr_f \colon \rR f_* \omega_X \to \omega_Y$;
        the uniqueness guarantees that this does not depend on the choice of the covering subspaces $U_i$.
        
        Compatibilities \cref{most general proper trace!!-closed}, \cref{most general proper trace!!-smooth}, and \cref{most general proper trace!!-composition} can be immediately reduced to the case of quasicompact separated rigid-analytic spaces because $\cHom(\rR f_*\omega_X,\omega_Y)$ is a sheaf (so we can check equality of two maps between $\rR f_*\omega_X$ and $\omega_Y$ \'etale locally on $Y$), preimages of quasicompact and separated subspaces under proper maps are quasicompact and separated, and the facts that closed trace and smooth-source traces are \'etale local on the target (see \cref{rmk:basic-properties-closed-trace}\cref{rmk:basic-properties-closed-trace-local} and \cref{rmk:smooth-source-trace-etale-local}).
        Lastly, the $\tr_f$ will still satisfy \cref{most general proper trace!!-localization} because it is \'etale local by its very construction as a section of the \'etale sheaf $\cHom(\rR f_*\omega_X,\omega_Y)$.

        \item \textit{Induction setup.}
        A quasicompact rigid-analytic space over $K$ has finite dimension (see, for example, \cite[Lem.~3.7]{adic-notes}).
        By \cref{most general proper trace!!-qcsep}, it thus remains to prove the following statement, which we show via induction on $d$:
        
        {\itshape Let $d \in \ZZ_{\ge 0} \cup \{-\infty\}$.
        There is a unique way to assign to every proper morphism $f \colon X \to Y$ of quasicompact and separated rigid-analytic spaces over $K$ with $\dim X \le d$ a trace map $\tr_f \colon \rR f_* \omega_X \to \omega_Y$ satisfying properties \cref{most general proper trace!!-closed}, \cref{most general proper trace!!-smooth}, \cref{most general proper trace!!-localization} (with $\widetilde{Y}$ also quasicompact and separated), as well as
        
        \cref{most general proper trace!!-composition}' (Compatibility with compositions II) For any two proper morphisms $f \colon W \to X$ and $g \colon X \to Y$ of quasicompact and separated rigid-analytic spaces over $K$ such that $\dim W \le d$ and either $X$ is smooth or $\dim X \le d$, we have $\tr_{g \circ f} = \tr_g \circ \rR g_*(\tr_f)$, where $\tr_g$ denotes the smooth-source trace from \cref{smooth source trace} when $X$ is smooth.
        }
    
        Note for \cref{most general proper trace!!-composition}' that in the induction step $\tr_g$ is not yet defined in general, so we need to assume that either $X$ is smooth or $\dim X \le d$ in order for the statement to have meaning.
        When $X$ satisfies both assumptions, there will be no ambiguity by \cref{most general proper trace!!-smooth}.
        On the other hand, it will be indispensable during the induction step that we allow for smooth $X$ with $\dim X > d$.
        
        In the base case $d = -\infty$, we have $X = \varnothing$, so there is nothing to show.
        We proceeed with the induction step.

        \item\label{most general proper trace!!-construction} \textit{Induction step: construction and uniqueness.}
        Assume that $d \ge 0$ and that the statement has been shown in dimensions $<d$.
        Let $f \colon X \to Y$ be a proper morphism of quasicompact and separated rigid-analytic spaces over $K$ with $\dim X = d$.

        First, consider the closed immersion $\iota \colon X_\red \hookrightarrow X$ from the maximal reduced closed subspace and set $f_\red \colonequals f \circ \iota$.
        By \cref{dualizing-nilimmersion}, precomposition with the closed trace map $\tr_\iota \colon \iota_*\omega_{X_\red} \to \omega_X$ induces an isomorphism $\Hom(\rR f_* \omega_X,\omega_Y) \xrightarrow{\sim} \Hom(\rR f_{\red,*}\omega_{X_\red},\omega_Y)$.
        In view of properties \cref{most general proper trace!!-closed} and \cref{most general proper trace!!-composition}, any trace map $\rR f_{\red,*} \omega_{X_\red} \to \omega_Y$ hence pins down a unique trace map $\rR f_* \omega_X \to \omega_Y$.
        As a consequence, it suffices to verify uniqueness of a proper trace satisfying \cref{most general proper trace!!-closed}, \cref{most general proper trace!!-smooth}, \cref{most general proper trace!!-composition}', and \cref{most general proper trace!!-localization} for reduced $X$. Furthermore, a construction of a proper trace for reduced $X$ canonically extends to all $X$. 
        We will construct a proper trace map in this Step and verify all the desired compatibilities in \cref{verification-of-properties}. Note, however, that it is a priori not clear that $\tr_f$ satisfies \cref{most general proper trace!!-closed}, \cref{most general proper trace!!-smooth}, \cref{most general proper trace!!-composition}', and \cref{most general proper trace!!-localization} for all $X$ if it satisfied these compatibilities for reduced $X$. 
           
        Now we implement the main idea described before the proof:
        We pick a dense, Zariski-open subspace $U \subseteq X$ which is contained in the smooth locus of $X$ and let $Z \colonequals X \smallsetminus U$ be its reduced complement.
        Thanks to \cref{modifications-exist}\cref{modifications-exist-Temkin}, we can find a regular $U$-admissible modification $\pi \colon X' \to X$ and consider the diagram \cref{modification-square}.
        Since $X'$ is smooth, we have smooth-source traces for all arrows in the diagram starting from $X'$.
        On the other hand, since $\dim Z < \dim X$ and $\dim Z' < \dim X' = \dim X$, the induction hypothesis provides us with unique traces subject to the desired properties for all (compositions of) arrows starting from either $Z$ or $Z'$.
        
        The traces $\tr_i$, $\tr_{i'}$, and $\tr_{\pi'}$ coming from induction and the smooth source trace $\tr_{\pi}$ satisfy
        \[ \tr_\pi \circ \tr_{i'} = \tr_h = \tr_i \circ \tr_{\pi'}; \]
        the first equality uses the modified compatibility \cref{most general proper trace!!-composition}' from the induction hypothesis.
        Under the $(i_*,\rR i^!)$-adjunction, this translates to
        \[ \rR i^! \tr_\pi = \tr_{\pi'}, \]
        where we use \cref{tr-!-pullback}.
        A combination with \cref{cor:injective-homs-dualizing} and property \cref{most general proper trace!!-closed} yields the exact sequence
        \begin{multline}\label{most general proper trace!!-sequence}
            0 \to \Hom\bigl(\rR f_*\omega_X, \omega_Y\bigr) \xrightarrow{(\blank\circ \rR f_*(\tr_\pi), \blank \circ \rR f_*(\tr_i))} \Hom\bigl(\rR \pi_{Y, *} \omega_{X'}, \omega_Y\bigr) \oplus \Hom\bigl(\rR i_{Y, *} \omega_{Z}, \omega_Y\bigr) \\ \xrightarrow{\blank\circ \rR \pi_{Y, *}((-1)\cdot\tr_{i'}) \oplus \blank\circ \rR i_{Y, *}(\tr_{\pi'})} \Hom\bigl(\rR h_* \omega_{Z'}, \omega_Y\bigr).
        \end{multline}
        If $\tr_f \in \Hom\bigl(\rR f_*\omega_X, \omega_Y\bigr)$ is to satisfy property \cref{most general proper trace!!-composition}', it needs to map to $\bigl(\tr_{\pi_Y},\tr_{i_Y}\bigr)$ in this sequence.
        Hence, $\tr_f$ is unique if it exists.
        Conversely, $\bigl(\tr_{\pi_Y},\tr_{i_Y}\bigr)$ must be induced by an element of $\Hom\bigl(\rR f_*\omega_X, \omega_Y\bigr)$ because
        \[ \tr_{\pi_Y} \circ \rR \pi_{Y,*}(\tr_{i'}) = \tr_{\pi_Y \circ i'} = \tr_{i_Y \circ \pi'} = \tr_{i_Y} \circ \rR i_{Y,*}(\tr_{\pi'}) \]
        by property \cref{most general proper trace!!-composition}' from the induction hypothesis.
        This produces a unique trace map $\tr_f \colon \rR f_* \omega_X \to \omega_Y$, which a priori depends on the choice of $U$ and $\pi$.

        To finish the construction, we explain why $\tr_f$ does in fact not depend on the choice of the dense, Zariski-open, smooth subspace $U \subseteq X$, nor on the choice of the regular $U$-admissible modification $\pi$.
        Let $U_j \subseteq X$ and $\pi_j \colon X'_j \to X$ for $j = 1,2$ be two such choices.
        Set $U \colonequals U_1 \cap U_2$ and $Z \colonequals X \smallsetminus U$ (again with the reduced adic space structure).
        Then $U \subset X$ is a dense, Zariski-open, smooth subspace and $(\pi_1,\pi_2) \colon X'_1 \times_X X'_2 \to X$ is a $U$-admissible modification.
        Applying \cref{modifications-exist}\cref{modifications-exist-Temkin} to $(\pi_1,\pi_2)^{-1}(U) \subseteq (X'_1 \times_X X'_2)_\red$, we obtain a regular $U$-admissible modification $\pi \colon X' \to X$ that factors through both $\pi_1$ and $\pi_2$.
        It suffices to compare the trace morphisms obtained from the $U_j$-admissible modifications $\pi_j$ for $j=1,2$ to the trace morphism obtained from the $U$-admissible modification $\pi$.
        Note that the $\pi_j$ can be considered both as $U_j$-modification and as $U$-modification;
        for clarity, we write $(\pi_j,U_j)$ and $(\pi_j,U)$.
        
        The modifications $(\pi,U)$, $(\pi_j,U_j)$, and $(\pi_j,U)$ form the red-purple rectangle, the blue-purple rectangle, and the red-blue-purple square, respectively, in the commutative diagram
        \[ \begin{tikzcd}
            & Z' \arrow[r,hook,red] \arrow[d,red] & X' \arrow[d,red,"\pi_{X'_j}"'] \arrow[rddd,bend left=25,"\pi_Y"] & \\
            Z'_j \arrow[r,hook,blue] \arrow[d,blue] & X'_j \times_X Z \arrow[r,hook,blue] \arrow[d,red] & X'_j \arrow[d,purple,"\pi_j"'] \arrow[rdd,bend left=20,"\pi_{j,Y}"'] & \\
            Z_j \arrow[r,hook,blue,"i_{j,Z}"] \arrow[rrrd,bend right=20,near start,"i_{j,Y}"'] & Z \arrow[r,hook,purple,"i"] \arrow[rrd,bend right=15,near start,"i_Y"'] & X \arrow[rd,near start,"f"'] & \\
            &&& Y.
        \end{tikzcd} \]
        With this notation, the injections in the exact sequence \cref{most general proper trace!!-sequence} for $(\pi,U)$, $(\pi_j,U_j)$, and $(\pi_j,U)$ fit into a square
        \begin{equation}\label{most general proper trace!!-independence} \begin{tikzcd}[column sep=7.5em] 
            \Hom\bigl(\rR f_*\omega_X, \omega_Y\bigr) \arrow[r,hook,"{(\blank\circ \rR f_*(\tr_{\pi_j}), \blank\circ \rR f_*(\tr_{i_j}))}"] \arrow[d,,shift left=3,hook',"{(\blank\circ \rR f_*(\tr_\pi), \blank\circ \rR f_*(\tr_i))}"'] \arrow[rd,hook,"{(\blank \circ \rR f_*(\tr_{\pi_j}),\blank \circ \rR f_*(\tr_i))}" description] & \Hom\bigl(\rR \pi_{j,Y,*} \omega_{X'_j}, \omega_Y\bigr) \oplus \Hom\bigl(\rR i_{j,Y,*} \omega_{Z_j}, \omega_Y\bigr) \\
            \Hom\bigl(\rR \pi_{Y,*} \omega_{X'}, \omega_Y\bigr) \oplus \Hom\bigl(\rR i_{Y,*} \omega_Z, \omega_Y\bigr) & \Hom\bigl(\rR \pi_{j,Y,*} \omega_{X'_j}, \omega_Y\bigr) \oplus \Hom\bigl(\rR i_{Y,*} \omega_Z, \omega_Y\bigr) \arrow[l,"{(\blank \circ \rR \pi_{j,Y,*}(\tr_{\pi_{X'_j}}),\id)}"] \arrow[u,"{(\id,\blank \circ \rR i_{Y,*}(\tr_{i_{j,Z}}))}"'] .
        \end{tikzcd} \end{equation}
        The two triangles commute by virtue of the equalities $\tr_\pi = \tr_{\pi_j} \circ \rR \pi_{j,*}(\tr_{\pi_{X'_j}})$ and $\tr_{i_j} = \tr_i \circ i_*(\tr_{i_{j,Z}})$ from \cref{compatibility for smooth source trace}\cref{compatibility for smooth source trace-4} and \cref{rmk:basic-properties-closed-trace}\cref{rmk:basic-properties-closed-trace-composition}, respectively.
        By construction, the three different choices of $\tr_f$ for $(\pi,U)$, $(\pi_j,U_j)$, and $(\pi_j,U)$ correspond to the elements $(\tr_{\pi_Y},\tr_{i_Y})$, $(\tr_{\pi_{j,Y}},\tr_{i_{j,Y}})$, and $(\tr_{\pi_{j,Y}},\tr_{i_Y})$ in the bottom left, top right, and bottom right corner in Diagram \cref{most general proper trace!!-independence}, respectively.
        Another application of \cref{compatibility for smooth source trace}\cref{compatibility for smooth source trace-4} and \cref{rmk:basic-properties-closed-trace}\cref{rmk:basic-properties-closed-trace-composition} shows that
        \[ \tr_{\pi_Y} = \tr_{\pi_{j,Y}} \circ \rR \pi_{j,Y,*}(\tr_{\pi_{X'_j}}) \quad \text{and} \quad \tr_{i_{j,Y}} = \tr_{i_Y} \circ \rR i_{Y,*}(\tr_{i_{j,Z}}), \]
        so that $(\tr_{\pi_{j,Y}},\tr_{i_Y})$ maps to $(\tr_{\pi_Y},\tr_{i_Y})$ and $(\tr_{\pi_{j,Y}},\tr_{i_{j,Y}})$ under the two maps $(\blank \circ \rR \pi_{j,Y,*}(\tr_{\pi_{X'_j}}),\id)$ and $(\id,\blank \circ \rR i_{Y,*}(\tr_{i_{j,Z}}))$ of Diagram \cref{most general proper trace!!-independence}, respectively.
        Since the three maps emanating from $\Hom\bigl(\rR f_*\omega_X, \omega_Y\bigr)$ are all injective, this shows that the three different choices of $\tr_f$ must all coincide.
        
        \item\label{verification-of-properties} \textit{Induction step: verification of properties.}
        To finish, we show that the trace maps $\tr_f$ constructed in \cref{most general proper trace!!-construction} satisfy the compatibilities \cref{most general proper trace!!-closed}, \cref{most general proper trace!!-smooth}, \cref{most general proper trace!!-composition}', and \cref{most general proper trace!!-localization}.
        When $X$ is smooth, it is in particular reduced and $\tr_f$ is constructed via the modification diagram \cref{modification-square}.
        Moreover, we may choose $U = X$ and $\pi = \id$, thanks to the independence of $\tr_f$ from the choice of $\pi$.
        This immediately yields \cref{most general proper trace!!-smooth}.
        
        When $f$ is a closed immersion, we consider again the modification diagram \cref{modification-square}, but with $f \colon X \to Y$ replaced by $f_\red \colon X_\red \xhookrightarrow{\iota} X \xrightarrow{f} Y$.
        Denote the closed trace map of $\iota$ by $\tr_\iota$.
        We defined $\tr_f$ as the unique element mapping to $(\tr_{\pi_Y},\tr_{i_Y})$ under the injection
        \begin{multline*}
            \Hom\bigl(\rR f_*\omega_X, \omega_Y\bigr) \xrightarrow[\sim]{(\blank \circ \rR f_*(\tr_\iota))} \Hom\bigl(\rR f_{\red,*}\omega_{X_\red}, \omega_Y\bigr) \\
            \xhookrightarrow{(\blank\circ \rR f_{\red,*}(\tr_\pi), \blank \circ \rR f_{\red,*}(\tr_i))} \Hom\bigl(\rR \pi_{Y, *} \omega_{X'}, \omega_Y\bigr) \oplus \Hom\bigl(\rR i_{Y, *} \omega_{Z}, \omega_Y\bigr).
        \end{multline*} 
        On the other hand, \cref{rmk:basic-properties-closed-trace}\cref{rmk:basic-properties-closed-trace-composition} and \cref{compatibility for smooth source trace}\cref{compatibility for smooth source trace-5} ensure that the image of the closed trace of $f$ under this injection has image $(\tr_{\pi_Y},\tr_{i_Y})$.
        This shows \cref{most general proper trace!!-closed}.

        Next, we verify \cref{most general proper trace!!-composition}'.
        Let $f \colon W \to X$ and $g \colon X \to Y$ be two proper morphisms of quasicompact and separated rigid-analytic spaces over $K$ such that $\dim W \le d$ and either $X$ is smooth or $\dim X \le d$.
        Leaving $d$ fixed, we show that $\tr_{g \circ f} = \tr_g \circ \rR g_*(\tr_f)$ by induction on $e \colonequals \dim W \le d$.
        When $e = -\infty$ (so $W = \varnothing$), the statement is clear.
        
        Now assume that $d \ge \dim W = e \ge 0$ and that \cref{most general proper trace!!-composition}' has been proven in dimensions $<e$.
        We assume first that $W$ and $X$ are both reduced.
        Under this assumption, we may choose a dense, Zariski-open, smooth $V \subseteq W$ with reduced complement $A \colonequals W \smallsetminus V$ and a regular $V$-admissible modification $\rho \colon W' \to W$;
        see \cref{modifications-exist}.
        Then the definition of $\tr_f$ via the injection in \cref{most general proper trace!!-sequence} allows us again to check \cref{most general proper trace!!-composition}' for $f \colon W \to X$ replaced by $\restr{f}{A} \colon A \to X$ and by $f \circ \rho \colon W' \to X$.
        For the former morphism, we may apply the induction hypothesis because $\dim A < \dim W = e$.
        For the latter morphism, we may check \cref{most general proper trace!!-composition}' separately on each connected (and hence irreducible) component of $W'$.
        The statement for connected components of dimension $<c$ is again covered by the induction hypothesis.
        Thus, after replacing $W$ by one of the remaining connected components of $W'$, it suffices to verify \cref{most general proper trace!!-composition}' when $W$ is smooth, irreducible and of equidimension $e \le d$. 

        In this case, $\tr_f$ and $\tr_{g \circ f}$ are given by the smooth-source trace thanks to \cref{most general proper trace!!-smooth}, which has already been verified.
        If $X$ is smooth, then the compatibility $\tr_{g \circ f} = \tr_g \circ \rR g_*(\tr_f)$ follows from \cref{compatibility for smooth source trace}\cref{compatibility for smooth source trace-4}.
        Thus, we may assume that $\dim X \le d$.
        Recall that $f(W) \subseteq X$ is a Zariski-closed subset \cite[Prop.~9.6.3/3]{BGR}.
        Since $W$ is irreducible, so is $f(W)$.
        Thus, we can pick an irreducible component $X_0 \subseteq X$ such that $f(W) \subseteq X_0$.
        If $\dim f(W) < \dim X_0$, there exists a dense, Zariski-open, smooth $U \subset X$ whose reduced Zariski-closed complement $i \colon Z \hookrightarrow X$ contains $f(W)$.
        Then $\tr_g$ is computed using a modification square as in \cref{modification-square} for some regular $U$-admissible modification $\pi \colon X' \to X$.
        Denoting by $f_Z \colon W \to Z$ the map factoring $f \colon W \to X$ and setting $i_Y \colon Z \xhookrightarrow{i} X \xrightarrow{g} Y$, we conclude
        \[ \tr_g \circ \rR g_*(\tr_f) = \tr_g \circ \rR g_*(\tr_i) \circ \rR i_{Y,*}(\tr_{f_Z}) = \tr_{i_Y} \circ \rR i_{Y,*}(\tr_{f_Z}); \]
        here, the first equality follows from the fact that $\tr_f$ is given by the smooth-source trace and \cref{compatibility for smooth source trace}\cref{compatibility for smooth source trace-5} and the second equality follows from the construction of $\tr_g$ and property \cref{most general proper trace!!-closed} of the induction hypothesis.
        Replacing $f$ by $f_Z$, $g$ by $i_Y$, and $X_0$ by the irreducible component of $Z$ containing $f_Z(W)$, we may therefore reduce $\dim X_0$ by at least $1$;
        repeating the same process finitely many times, we finally arrive at a situation where $f(W) = X_0$.
        
        Choose once more a dense, Zariski-open, smooth subspace $U \subseteq X$ with reduced complement $Z \subset X$ and a regular $U$-admissible modification $\pi \colon X' \to X$ fitting into a diagram of the form \cref{modification-square}.
        Since $W$ is irreducible and $f(W) = X_0$, the preimage $V \colonequals f^{-1}(U) \subseteq W$ is still dense and Zariski-open.
        Let $W_0 \subseteq (W \times_X X')_\red$ be an irreducible component containing $V \times_X X'$.
        The induced map $\rho \colon W_0 \to W$ is a $V$-admissible modification.
        \Cref{modifications-exist}\cref{modifications-exist-Temkin} yields a regular $\rho^{-1}(V)$-admissible modification $\rho' \colon W' \to W_0$.
        Then $\pi' \colonequals \rho \circ \rho' \colon W' \to W$ is a regular $V$-admissible modification, fitting into a commutative diagram
        \[ \begin{tikzcd}
           W' \arrow[d,"\pi'"] \arrow[r,"f'"] & X' \arrow[d,"\pi"]  \arrow[dr, bend left, "\pi_Y"] & \\
           W \arrow[r,"f"] & X \arrow[r, "g"] & Y.
        \end{tikzcd} \]
        Let $A \subset W$ be the reduced complement of $V$.
        As before, we may check \cref{most general proper trace!!-composition}' for $f \colon W \to X$ replaced by $\restr{f}{A} \colon A \to X$ and by $f \circ \pi' \colon W' \to X$.
        Since $\dim A < \dim W \le e$, the former follows from the induction hypothesis.
        The latter comes from the identity
        \[ \tr_g \circ \rR g_*(\tr_{f \circ \pi'}) = \tr_g \circ \rR g_*(\tr_{\pi \circ f'}) = \tr_g \circ \rR g_*(\tr_\pi) \circ \rR \pi_{Y,*}(\tr_{f'}) = \tr_{\pi_Y} \circ \rR \pi_{Y,*}(\tr_{f'}) = \tr_{\pi'_Y}, \]
        in which every trace map except $\tr_g$ is the smooth-source trace and hence the second and last equality follow from \cref{compatibility for smooth source trace}\cref{compatibility for smooth source trace-4}, whereas the third equality follows from the construction of $\tr_g$.
        This finishes the verification of \cref{most general proper trace!!-composition}' when $W$ and $X$ are both reduced.

        In order to prove \cref{most general proper trace!!-composition}' for $d \ge \dim W = e \ge 0$ in general, consider the commutative diagram
        \[ \begin{tikzcd}
            W_\red \arrow[r,"f_\red"] \arrow[d,hook,"\iota'"] & X_\red \arrow[rd,bend left,"g_\red"] \arrow[d,hook,"\iota"] & \\
            W \arrow[r,"f"] & X \arrow[r,"g"] & Y
        \end{tikzcd} \]
        in which the top row consists of the maximal reduced closed subspaces and the vertical arrows of the canonical closed immersions.
        We still have $\dim X_\red = \dim X \le d$ and $\dim W_\red = \dim W = e \le d$.
        By the construction in \cref{most general proper trace!!-construction} and property \cref{most general proper trace!!-closed}, the trace map $\tr_{g \circ f}$ is uniquely determined by the property
        \[ \tr_{g \circ f} \circ \rR (g \circ f)_*(\tr_{\iota'}) = \tr_{g \circ f \circ \iota'} = \tr_{g_\red \circ f_\red}. \]
        On the other hand, we have 
        \begin{multline*}
            \tr_g \circ \rR g_*(\tr_f) \circ \rR(g \circ f)_*(\tr_{\iota'}) = \tr_g \circ \rR g_*(\tr_{f \circ \iota'}) = \tr_g \circ \rR g_*(\tr_{\iota \circ f_\red}) \\
            = \tr_g \circ \rR g_*(\tr_\iota) \circ \rR g_{\red,*}(\tr_{f_\red}) = \tr_{g_\red} \circ \rR g_{\red,*}(\tr_{f_\red}) = \tr_{g_\red \circ f_\red} 
        \end{multline*}
        where the first and fourth equality follow from the construction of $\tr_f$ and $\tr_g$ via $\tr_{f_\red}$ and $\tr_{g_\red}$, respectively, and the third and fifth equality hold because $W_\red$ and $X_\red$ are reduced of the right dimensions.
        As a consequence, $\tr_{g \circ f} = \tr_g \circ \rR g_*(\tr_f)$, yielding the general case of \cref{most general proper trace!!-composition}'.

        It remains to deal with property \cref{most general proper trace!!-localization}.
        Let $h \colon \widetilde{Y} \to Y$ be an \'etale map from a quasicompact and separated rigid space $\widetilde{Y}$ over $K$;
        in particular, $h$ is in addition separated and taut \cite[Lem.~5.1.3]{Huber-etale}.
        Once more, we assume first that $X$ is reduced.
        In that case, we can choose a dense, Zariski-open, smooth subspace $U \subseteq X$ with reduced complement $Z \subset X$ and a regular $U$-admissible modification $\pi \colon X' \to X$.
        
        Set $\widetilde{U} \colonequals U \times_Y \widetilde{Y} \subseteq X \times_Y \widetilde{Y} \equalscolon \widetilde{X}$.
        The pullback $\tilde{\pi} \colon \widetilde{X}' \to \widetilde{X}$ of $\pi$ along $\tilde{h} \colon \widetilde{X} \to X$ is again a regular $\widetilde{U}$-admissible\footnote{\label{pullback-admissible}To see the $\widetilde{U}$-admissibility, we observe that \'{e}tale maps are open, hence the preimage of a dense open is again dense open.} modification, giving rise to a commutative diagram 
        \[ \begin{tikzcd}[row sep=small,column sep=small]
            & \widetilde{X}' \arrow[rr,"\tilde{h}'_\pi"] \arrow[dd,near start,"\tilde{\pi}"'] \arrow[dddd,bend left,very near start,"\tilde{\pi}_{\widetilde{Y}}"] && X' \arrow[dd,"\pi"'] \arrow[dddd,bend left,"\pi_Y"] \\
            \widetilde{Z}' \arrow[rr,crossing over,near start,"\tilde{h}'_i"'] \arrow[dd] \arrow[ru,hook] && Z' \arrow[ru,hook] & \\
            & \widetilde{X} \arrow[rr,crossing over,near end,"\tilde{h}"] \arrow[dd,near start,"\tilde{f}"'] && X \arrow[dd,"f"'] \\
            \widetilde{Z} \arrow[rr,crossing over] \arrow[rd,"\tilde{i}_{\widetilde{Y}}"'] \arrow[ru,hook,"\tilde{i}"] && Z \arrow[rd,"i_Y"] \arrow[ru,hook,"i"] \arrow[from=2-3,crossing over] & \\
            & \widetilde{Y} \arrow[rr,"h"] && Y,
        \end{tikzcd}\]
        in which all squares are pullback squares.
        Since the closed trace and the smooth-source trace are \'etale local on the target by \cref{rmk:basic-properties-closed-trace}\cref{rmk:basic-properties-closed-trace-local} and \cref{rmk:smooth-source-trace-etale-local}, respectively, the same is true for the short exact sequences \cref{most general proper trace!!-sequence} coming from \cref{cor:injective-homs-dualizing}.
        This yields the following commutative diagram, in which the upper vertical maps are given by pullback along $h$:
        \[ \begin{tikzcd}[column sep=12em]
            \Hom\bigl(\rR f_*\omega_X, \omega_Y\bigr) \arrow[r,hook,"{(\blank\circ \rR f_*(\tr_\pi), \blank \circ \rR f_*(\tr_i))}"] \arrow[d] & \Hom\bigl(\rR \pi_{Y, *} \omega_{X'}, \omega_Y\bigr) \oplus \Hom\bigl(\rR i_{Y, *} \omega_{Z}, \omega_Y\bigr) \arrow[d] \\
            \Hom\bigl(h^*\rR f_*\omega_X, h^*\omega_Y\bigr) \arrow[r,hook,"{(\blank\circ h^*\rR f_*(\tr_\pi), \blank \circ h^*\rR f_*(\tr_i))}"] \arrow[d,"{\alpha_h \circ \blank \circ \rR \tilde{f}_*\bigl(\alpha^{-1}_{\tilde{h}}\bigr)}","\sim"'{sloped}] & \Hom\bigl(h^*\rR \pi_{Y, *} \omega_{X'}, h^*\omega_Y\bigr) \oplus \Hom\bigl(h^*\rR i_{Y, *} \omega_{Z}, h^*\omega_Y\bigr) \arrow[d,"\sim"{sloped},"{\bigl(\alpha_h \circ \blank \circ \rR \tilde{\pi}_{\widetilde{Y},*}\bigl(\alpha^{-1}_{\tilde{h}'_\pi}\bigr)\bigr) \oplus \bigl(\alpha_h \circ \blank \circ \rR \tilde{i}_{\widetilde{Y},*}\bigl(\alpha^{-1}_{\tilde{h}'_i}\bigr)\bigr)}"'] \\
            \Hom\bigl(\rR \tilde{f}_*\omega_{\widetilde{X}}, \omega_{\widetilde{Y}}\bigr) \arrow[r,hook,"{(\blank\circ \rR \tilde{f}_*(\tr_{\tilde{\pi}}), \blank \circ \rR \tilde{f}_*(\tr_{\tilde{i}}))}"] & \Hom\bigl(\rR \tilde{\pi}_{\widetilde{Y},*} \omega_{\widetilde{X}'}, \omega_{\widetilde{Y}}\bigr) \oplus \Hom\bigl(\rR \tilde{i}_{\widetilde{Y}, *} \omega_{\widetilde{Z}}, \omega_{\widetilde{Y}}\bigr)
        \end{tikzcd} \]
        We need to show that $\tr_f \in \Hom\bigl(\rR f_*\omega_X, \omega_Y\bigr)$ maps to $\tr_{\tilde{f}} \in \Hom\bigl(\rR \tilde{f}_*\omega_{\widetilde{X}}, \omega_{\widetilde{Y}}\bigr)$ under the composition of the left vertical maps.
        By the injectivity of the horizontal maps and the construction of $\tr_f$ and $\tr_{\tilde{f}}$ from \cref{most general proper trace!!-construction}, it suffices to see that $\bigl(\tr_{\pi_Y},\tr_{i_Y}\bigr)$ maps to $\bigl(\tr_{\tilde{\pi}_{\widetilde{Y}}},\tr_{\tilde{i}_{\widetilde{Y}}}\bigr)$ under the composition of the right vertical maps.
        For the first component, this follows from \cref{rmk:smooth-source-trace-etale-local}.
        For the second component, one can use property \cref{most general proper trace!!-localization} of the induction hypothesis.
        This shows property \cref{most general proper trace!!-localization} when $X$ is reduced.

        For general $X$, the extended pullback diagram of rigid-analytic spaces over $K$
        \[ \begin{tikzcd}
            \widetilde{X}_\red \arrow[r,hook,"\iota'"] \arrow[rr,bend left,"\tilde{f}_\red"] \arrow[d,"\tilde{h}_\red"] & \widetilde{X} \arrow[r,"\tilde{f}"] \arrow[d,"\tilde{h}"] & \widetilde{Y} \arrow[d,"h"] \\
            X_\red \arrow[r,hook,"\iota"] \arrow[rr,bend right,"f_\red"] & X \arrow[r,"f"] & Y,
        \end{tikzcd}\]
        where as before $\iota$ and $\iota'$ denote the closed immersions from the maximal reduced closed subspaces, induces the following extended diagram:
        \[ \begin{tikzcd}[row sep=tiny,column sep=huge]
            \rR \tilde{f}_{\red,*} \omega_{\widetilde{X}_\red} \arrow[r,"\rR f_*(\tr_{\iota'})","\sim"'] \arrow[rr,bend left=15,"\tr_{\tilde{f}_\red}"] & \rR \tilde{f}_* \omega_{\widetilde{X}} \arrow[r,"\tr_{\tilde{f}}"] & \omega_{\widetilde{Y}} \\[1em]
            \rR \tilde{f}_{\red,*} \tilde{h}^*_\red \omega_{X_\red} \arrow[u,"\rR \tilde{f}_{\red,*}(\alpha_{\tilde{h}_\red})","\sim"'{sloped}] \arrow[d,equals] && \\
            \rR \tilde{f}_* \tilde{h}^* \iota_* \omega_{X_\red} \arrow[r,"\rR \tilde{f}_*\tilde{h}^*(\tr_\iota)","\sim"'] \arrow[d,equals] & \rR \tilde{f}_* \tilde{h}^* \omega_X \arrow[uu,"\rR \tilde{f}_*(\alpha_{\tilde{h}})","\sim"'{sloped}] \arrow[d,equals] & \\
            h^* \rR f_{\red,*} \omega_{X_\red} \arrow[r,"h^*\rR f_*(\tr_\iota)","\sim"'] \arrow[rr,bend right=15,"h^*(\tr_{f_\red})"] & h^* \rR f_* \omega_X \arrow[r,"h^*(\tr_f)"] & h^* \omega_Y \arrow[uuu,"\alpha_h","\sim"'{sloped}]
        \end{tikzcd}\]
        In this diagram, the upper and lower ``triangles'' commute by the construction of $\tr_{\tilde{f}}$ and $\tr_f$ via $\tr_{\tilde{f}_\red}$ and $\tr_{f_\red}$, the upper left rectangle commutes because the closed trace map is \'etale local (\cref{rmk:basic-properties-closed-trace}\cref{rmk:basic-properties-closed-trace-local}), the lower left rectangle commutes by the naturality of the base change isomorphism, and the outer diagram commutes thanks to the special case treated above when the source is reduced.
        On the other hand, \cref{dualizing-nilimmersion} and \cref{cor:smooth-pullback} guarantee that all arrows on the left side of the diagram are isomorphisms, so the right square must commute as well.
        This establishes property \cref{most general proper trace!!-localization} in general.
        We can finally declare victory in the proof of \cref{most general proper trace!!}! \qedhere
    \end{enumerate}
\end{proof}
We now check some compatibilities of the proper trace map beyond those mentioned in statement of \cref{most general proper trace!!}.
\begin{proposition}\label{smooth trace equals proper trace for smooth proper morphism}
Let $f \colon X \to Y$ be a smooth, proper morphism of rigid-analytic spaces over $K$.
Then the smooth trace map of $f$ from \cref{constructing smooth trace for omega} agrees with the proper trace map of $f$ from \cref{most general proper trace!!}.
\end{proposition}
\begin{proof}
Since $\rR \cHom(\rR f_*\omega_X,\omega_Y)$ lies in $D^{\geq 0}(Y_\et; \Lambda)$ (see \cref{proper trace lives in discrete space}), we may check the statement \'{e}tale locally on $Y$ and thus assume that $Y$ is quasicompact and separated.
The closed immersions of the maximal reduced closed subspaces fit into a commutative diagram
\[ \begin{tikzcd}
    X_\red \arrow[r,hook,"\iota'"] \arrow[rd,"h"] \arrow[d,"f_\red"] & X \arrow[d,"f"] \\
    Y_\red \arrow[r,hook,"\iota"] & Y.
\end{tikzcd} \]
Since $f$ is smooth, this diagram is even Cartesian;
in particular, $f_\red$ is also smooth and proper.
By virtue of \cref{compatibility proposition} (or even \cref{compatibility lemma: Cartesian diagram case}) and a combination of \cref{most general proper trace!!}\cref{most general proper trace!!-closed} and \cref{most general proper trace!!}\cref{most general proper trace!!-composition}, respectively, the induced diagram
\[ \begin{tikzcd}[column sep=huge]
    \rR h_* \omega_{X_\red} \arrow[r,"\iota_*(\tr_{f_\red})"] \arrow[d,"\rR f_*(\tr_{\iota'})","\sim"'{sloped}] & \iota_* \omega_{Y_\red} \arrow[d,"\tr_\iota","\sim"'{sloped}] \\
    \rR f_* \omega_X \arrow[r,"\tr_f"] & \omega_Y
\end{tikzcd} \]
in $D(Y_\et;\Lambda)$ commutes for both the smooth and the proper traces of $f_\red$ and $f$, where in both cases $\tr_{\iota'}$ and $\tr_\iota$ denote the closed trace maps.
Moreover, these closed trace maps are isomorphisms by \cref{dualizing-nilimmersion}, so the smooth and proper trace for $f$ agree if and only they agree for $f_\red$.
As a consequence, we may also assume that both $X$ and $Y$ are reduced.

After these reductions, we proceed by induction on $d \colonequals \dim(Y)$.
When $d = 0$, the reducedness of $Y$ together with the standing assumption that $\charac K = 0$ means that $Y$ is smooth.
Since $f$ is smooth, $X$ is also smooth and we win thanks to \cref{compatibility for smooth source trace}\cref{compatibility for smooth source trace-1} and \cref{most general proper trace!!}\cref{most general proper trace!!-smooth}.

Now assume that $\dim Y = d > 0$ and the statement has been proven in dimensions $\le d-1$.
\Cref{modifications-exist} allows us to pick a dense, Zariski-open subspace $V \subseteq Y$ which is contained in the smooth locus of $Y$ and a regular $V$-admissible modification $Y' \to Y$.
Denote by $A \subset Y$ the complement of $V$ endowed with the canonical reduced adic space structure;
we have $\dim A < \dim Y$.

Let $U \colonequals f^{-1}(V)$ be the preimage, which is automatically Zariski-open.
Consider the following diagram with Cartesian squares:
\[ \begin{tikzcd}
    Z \arrow[r,hook,"i"] \arrow[d,"\bar{f}"] \arrow[rd,"i_Y"] & X \arrow[d,"f"] & X' \arrow[l,"\pi"'] \arrow[d,"f'"] \arrow[ld,"\pi_Y"] \\
    A \arrow[r,hook,"j"] & Y & Y' \arrow[l,"\rho"']
\end{tikzcd} \]
Since the pullback $\pi \colon X' \to X$ of $\rho$ is still a regular $U$-modification,\footnote{In fact, as in \cref{pullback-admissible}, the $U$-modification $\pi$ is $U$-admissible because $f$ is smooth, hence open.}
\cref{cor:injective-homs-dualizing} gives an injection
\[ \Hom\bigl(\rR f_*\omega_X, \omega_Y\bigr) \xhookrightarrow{(\blank\circ \rR f_*(\tr_\pi), \blank\circ \rR f_*(\tr_i))} \Hom\bigl(\rR \pi_{Y,*} \omega_{X'}, \omega_Y\bigr) \oplus \Hom\bigl(\rR i_{Y,*} \omega_{Z}, \omega_Y\bigr), \]
where $\tr_\pi$ and $\tr_i$ denote the smooth-source trace and the closed trace, respectively.
Thus, we only need to show that the images of the smooth trace of $f$ and the proper trace of $f$ under this injection agree with another.
The first components in $\Hom\bigl(\rR \pi_{Y,*} \omega_{X'}, \omega_Y\bigr)$ both agree with the smooth-source trace for $f \circ \pi \colon X' \to Y$ by \cref{compatibility for smooth source trace}\cref{compatibility for smooth source trace-5} and \cref{most general proper trace!!}\cref{most general proper trace!!-smooth} and \cref{most general proper trace!!-composition}.
For the second components in $\Hom\bigl(\rR i_{Y,*} \omega_{Z}, \omega_Y\bigr)$, we can use \cref{compatibility lemma: Cartesian diagram case}, \cref{most general proper trace!!}\cref{most general proper trace!!-closed} and \cref{most general proper trace!!-composition}, and the fact that the smooth trace and the proper trace of $\bar{f} \colon Z \to A$ agree thanks to the induction hypothesis.
This finishes the induction step.
\end{proof}
\begin{remark}
    \cref{smooth trace equals proper trace for smooth proper morphism} implies a version of \cref{compatibility proposition} when two maps are smooth and proper and the other two maps are proper. Unfortunately, we cannot establish an analogue of \cref{compatibility proposition} for smooth and proper traces in general. The essential difficulty comes from the fact that we cannot prove an analogue of \cref{cor:injective-homs-dualizing} for non-proper $f$ (the key coconnectivity claim \cref{proper trace lives in discrete space} used in the proof is unavailable in the non-proper case; see also \cref{rmk:general-coconnectivity}). It seems that the correct approach should be to construct trace maps for arbitrary taut separated maps compatible with compositions. The main obstacle to do this lies, again, in the fact that we cannot verify \cref{cor:injective-homs-dualizing} beyond the proper case. %
\end{remark}
Next, we record the compatibility of proper traces under change of base field.
\begin{definition}\label{base-field-compatibility}
    Let $f \colon X \to Y$ be a separated taut map of rigid-analytic spaces over $K$, which is either smooth or proper.
    Let $K\subset L$ be an extension of nonarchimedean fields, inducing the Cartesian diagram
    \[ \begin{tikzcd}
        X_L \arrow[d,"f_L"] \arrow[r,"a_X"] & X \arrow[d,"f"] \\
        Y_L \arrow[r,"a_Y"] & Y.
    \end{tikzcd} \]
    The smooth (resp.\ proper) trace map $\tr_f \colon \rR f_!\,\omega_X \to \omega_Y$ in the sense of \cref{constructing smooth trace for omega} (resp.\ \cref{most general proper trace!!}) is \emph{compatible with the base field extension $K\subset L$} if the diagram
    \[ \begin{tikzcd}[column sep=huge]
        a^*_Y\bigl(\rR f_!\,\omega_X\bigr) 
        \arrow[d,"a_Y^*(\tr_f)"] \arrow[r, "\BC_!"] & \rR f_{L,!} \bigl(a^*_X \omega_X\bigr)  \arrow[r, "\sim"', "\rR f_{L, !}(\gamma_{X, L})"] & \rR f_{L, !}\, \omega_{X_L} \arrow[d, "\tr_{f_L}"] \\
        a^*_Y\bigl(\omega_Y\bigr) \arrow[rr, "\sim"', "\gamma_{X, L}"] & & \omega_{Y_L}.
    \end{tikzcd} \]
    commutes, where $\gamma_{X, L}$ is the base change isomorphism from \cite[Th.~3.21~(6)]{BH} and $\BC_!$ is the base change map for compactly supported pushforward from \cite[Th.~5.4.6]{Huber-etale}.
\end{definition}
\begin{proposition}\label{change-base-field}
Let $f \colon X \to Y$ be a separated taut map of rigid-analytic spaces over $K$, which is either smooth or proper.
Then $\tr_f$ is compatible with every nonarchimedean base field extension $K\subset L$ in the sense of \cref{base-field-compatibility}.
\end{proposition}
Recall that by \cref{smooth trace equals proper trace for smooth proper morphism}, there is no ambiguity in the notation $\tr_f$ when $f$ is both smooth and proper.
\begin{proof}
    \begin{enumerate}[wide,label={\textit{Step~\arabic*}.},ref={Step~\arabic*}]
        \item\label{change-base-field-smooth} \textit{Proof for smooth $f$.}
        It suffices to prove the statement when $f$ is smooth of equidimension $d$.
        Using the definition of the smooth trace map from \cref{constructing smooth trace for omega}, the diagram in \cref{base-field-compatibility} can then be broken up as follows (using the canonical isomorphisms $\alpha_f\colon f^*\omega_Y(d)[2d] \xrightarrow{\sim} \omega_X$ provided by \cref{cor:smooth-pullback}):
        \[ \begin{tikzcd}[column sep=huge]
            a^*_Y(\rR f_!\omega_X) \arrow[d, "\BC_!"] & a^*_Y \rR f_! f^* \omega_Y(d)[2d] \arrow[l,"a^*_Y\rR f_!(\alpha_f)"',"\sim"] \arrow[d, "\BC_!"] & a^*_Y(\omega_Y \otimes^L \rR f_! \LLambda_X(d)[2d]) \arrow[l,"a^*_Y(\PF_f)"',"\sim"] \arrow[d,equals] \arrow[r,"a^*_Y(\id \otimes \ttr_f)"] & a^*_Y(\omega_Y) \arrow[ddd,"\sim"{sloped}, "\gamma_{Y, L}"'] \\
            \rR f_{L,!} (a^*_X \omega_X) \arrow[dd,"\sim"'{sloped}, "\rR f_{L, !}(\gamma_{X, L})"] & \rR f_{L,!} a^*_X f^* \omega_Y(d)[2d] \arrow[l,"\rR f_{L,!} a^*_X(\alpha_f)"',"\sim"] \arrow[d,equals] & a^*_Y\omega_Y \otimes^L a^*_Y\rR f_! \LLambda_X(d)[2d] \arrow[d, "\id \otimes \BC_!"] \arrow[ru,near start,"\id \otimes a^*_Y(\ttr_f)"'] & \\
            & \rR f_{L,!} f^*_L a^*_Y \omega_Y(d)[2d] \arrow[d,"\sim"'{sloped}, "{\rR f_{L, !}f_L^*(\gamma_{Y, L})(d)[2d]}"] & a^*_Y \omega_Y \otimes^L \rR f_{L,!} a^*_X \LLambda_X(d)[2d] \arrow[l,"\PF_{f_L}"',"\sim"] \arrow[d,"\sim"'{sloped}, "\gamma_{Y, L}\otimes \id"] & \\
            \rR f_{L,!}\omega_{X_L} & \rR f_{L,!} f^*_L \omega_{Y_L}(d)[2d] \arrow[l,"\rR f_{L,!}(\alpha_{f_L})"',"\sim"] & \omega_{Y_L} \otimes^L \rR f_{L,!} \LLambda_{X_L}(d)[2d] \arrow[l,"\PF_{f_L}"',"\sim"] \arrow[r,"\id \otimes \ttr_{f_L}"] & \omega_{Y_L}
        \end{tikzcd} \]
        The upper left square commutes by the naturality of the base change map, the upper middle rectangle by the compatibility of the projection formula with base change (see e.g.\ \cite[\href{https://stacks.math.columbia.edu/tag/0E48}{Tag~0E48}]{stacks-project}), the lower middle square by the naturality of the projection formula, and the right trapezoid thanks to the compatibility of the smooth trace map for constant coefficients with pullbacks (\cref{smooth-trace-constant}\cref{smooth-trace-constant-pullback}).
        
        To finish the argument for smooth $f$, it remains to see that the lower left rectangle commutes.
        We will show that the rectangle commutes even before applying $\rR f_{L,!}$.
        In order to check this stronger claim, we may work locally on $X$ and thus assume that $f$ is of the form $f\colon \Spa(B, B^\circ) \to \Spa(A, A^\circ)$ with associated regular morphism $f^\alg\colon \Spec B \to \Spec A$.
        In this case, $\alpha_f$ and $\alpha_{f_L}$ (up to a twist) are defined as the analytifications of the isomorphisms from \cite[Exp.~XVII, Prop.~4.1.1]{deGabber} applied to the regular morphisms $f$ and $f_L$ (see \cref{rmk:smooth-pullback-in-algebraic-geometry}), while $\gamma_{X, L}$ and $\gamma_{Y, L}$ are defined as the analytifications of the isomorphisms from \cite[Exp.~XVII, Prop.~4.1.1]{deGabber} applied to regular morphisms $\Spec \bigl(B\wdh{\otimes}_K L \bigr) \to \Spec B$ and $\Spec \bigl(A\wdh{\otimes}_K L \bigr) \to \Spec A$ (see the last paragraph of the proof of \cite[Prop.~3.24]{BH} for the fact that $A \to A\wdh{\otimes}_K L$ is regular and the proof of \cite[Th.~3.21~(6)]{BH} for the fact that the dimension function on $\Spec \bigl(A\wdh{\otimes}_K L\bigr)$ introduced in \cite[Exp.~XVII, Prop.~4.1.1]{deGabber} coincides with the dimension function introduced in \cite[Prop.~3.18]{BH}). Therefore, the desired diagram commutes due to the compatibility of the isomorphisms from \cite[Exp.~XVII, Prop.~4.1.1]{deGabber} with compositions (see \cite[Exp.~XVII, Rmk.~4.3.1.3]{deGabber}).

        \item\label{change-base-field-affinoid} \textit{Proof for proper $f$: reduction to affinoid $Y$.}
        Since the rest of the proof deals with proper $f$, we use $\rR f_*$ instead of $\rR f_!$.
        First, we show that the statement can be checked locally on $Y$;
        in particular, we may assume that $Y$ is $K$-affinoid.
        To this end, it suffices to verify that $\rR\cHom\bigl(a_Y^*\rR f_*\,\omega_X, \omega_{Y_L}\bigr) \in D^{\geq 0}(Y_{L, \et})$ thanks to the BBDG gluing lemma \cite[Prop.~3.2.2]{BBDG}.
        Now we observe that \cref{lemma:proper-base-change} implies that the base change morphism $a_Y^*(\rR f_*\omega_X) \to \rR f_{L,*}\omega_{X_L}$ is an isomorphism. Hence, the desired result follows directly from \cref{proper trace lives in discrete space}.
    
        \item\label{change-base-field-climm} \textit{Proof when $f$ is a closed immersion.}
        By \cref{change-base-field-affinoid}, we may assume that $Y=\Spa(A, A^\circ)$.
        Since $f$ is a closed immersion, $X$ is a $K$-affinoid adic space of the form $X=\Spa(A/I, (A/I)^\circ)$ for some ideal $I\subset A$. We consider the following cartesian diagram: 
        \[
        \begin{tikzcd}
            X^\alg_L= \Spec \bigl((A/I) \wdh{\otimes}_K L\bigr) \arrow{r}{f_L^\alg} \arrow{d}{a_{X^\alg}} & Y^\alg_L=\Spec \bigl( A\wdh{\otimes}_K L\bigr) \arrow{d}{a_{Y^\alg}} \\
            X^\alg = \Spec A/I \arrow{r}{f^\alg} & Y^\alg = \Spec A.
        \end{tikzcd}
        \]
        After unraveling all the definitions and using the $(f_{L, *}, \rR f_{L}^!)$-adjunction, we reduce the question to showing that the following diagram commutes:
        \[
        \begin{tikzcd}[column sep = 6em]
        a_{X^\alg}^* \omega_{X^\alg} \arrow[d, "\sim"'{sloped}, "a_{X^\alg}^*(c_{f^{\alg}})"] \arrow[rr, "\sim"', "\gamma_{a_{X^\alg}}"] & & \omega_{X_L^\alg} \arrow[d, "\sim"'{sloped}, "c_{f_L^\alg}"] \\
        a_{X^\alg}^*\rR f^{\alg, !} \omega_{Y^\alg} \arrow[r, "\sim"', "\BC^{*, !}"] & \rR f_{L}^{\alg, !} a_{Y^\alg}^* \omega_{Y^\alg} \arrow[r, "\sim"', "\rR f_L^{\alg, !}(\gamma_{a_{Y^\alg}})"] & \rR f_L^{\alg, !} \omega_{Y_L^\alg},
        \end{tikzcd}
        \]
        where the $\gamma$-isomorphisms come from \cite[Exp.~XVII, Prop.~4.1.1]{deGabber}, the $c$-isomorphisms come from \cite[Exp.~XVII, Prop.~4.1.2]{deGabber}, and the isomorphism $\BC^{*, !}$ comes from \cite[Exp.~XVII, Cor.~4.2.3]{deGabber}. Now we note that the composition 
        \[
        F\coloneqq c_{f_L^\alg}^{-1} \circ \rR f_L^{\alg, !}(\gamma_{a_{Y^\alg}}) \circ \BC^{*, !} \circ a_{X^\alg}^*(c_{f^{\alg}})\circ \gamma^{-1}_{a_{X^\alg}} \colon \omega_{X_L^\alg} \to \omega_{X_L^\alg}
        \]
        is an automorphism of the potential dualizing morphism on $X_L^\alg$. Therefore, \cite[Exp.~XVII, Th.~5.1.1]{deGabber} implies, in order to show that $F=\id$, it suffices to show that $F$ is compatible with pinnings. This, however, follows directly from \cite[Exp.~XVII, Lemme 4.3.2.3]{deGabber} (whose proof does not use surjectivity of the map $g$).

        \item \textit{Proof for general proper $f$.}
        Lastly, we show the statement for general proper morphisms by induction on $d \colonequals \dim X$.
        In the base case $d = -\infty$, there is nothing to prove.
        Now assume that $\dim X = d \ge 0$ and that the statement has been proven in dimensions $<d$.
        At first, we assume moreover that $X$ is reduced.
        Then \cref{modifications-exist} allows us to choose a regular admissible modification $\pi \colon X' \to X$ as in the beginning of proof of \cref{most general proper trace!!}, so we arrive again at the following commutative diagram from \cref{modification-square}:
        \[ \begin{tikzcd}
            Z' \arrow[r,hook,"i'"] \arrow[d,"\pi'"'] & X' \arrow[d,"\pi"] \arrow[rdd,bend left=25,"\pi_Y"] & \\
            Z \ar[r,hook,"i"] \arrow[rrd,bend right=20,"i_Y"] & X \ar[rd,"f"] & \\
            & & Y
        \end{tikzcd} \]
        By the uniqueness from \cref{most general proper trace!!-sequence} in the proof of \cref{most general proper trace!!},
        it suffices to check that the traces $\tr_i$, $\tr_{i_Y}$,
        $\tr_{\pi}$, and $\tr_{\pi_Y}$ are all compatible with change of base field.
        For $\tr_i$ and $\tr_{i_Y}$, this follows from the induction hypothesis.
        Therefore, we may assume that $X$ is smooth, in which case the proper trace agrees with the smooth-source trace in \cref{smooth source trace}, thanks to \cref{most general proper trace!!}\cref{most general proper trace!!-smooth}.
        Our claim for reduced $X$ is now a consequence of the compatibility of both smooth and closed traces with change of base field, as established in \cref{change-base-field-smooth} and \cref{change-base-field-climm}.

        For general (not necessarily reduced) $X$, we consider the diagram
        \[ \begin{tikzcd}
            X_{L,\red} \arrow[d,hook,"\iota_L"] \arrow[dd,bend right,"f_{L,\red}"'] \arrow[r,phantom,"\simeq"] &[-2em] X_{\red,L} \arrow[r,"a_{X_\red}"] &[.5em] X_\red \arrow[d,hook,"\iota"'] \arrow[dd,bend left,"f_\red"] \\
            X_L \arrow[rr,"a_X"] \arrow[d,"f_L"] && X \arrow[d,"f"'] \\
            Y_L \arrow[rr,"a_Y"] && Y,
        \end{tikzcd} \]
        where $\iota$ and $\iota_L$ denote the inclusions of the maximal reduced closed subspaces.
        Since the upper square and the outer rectangle are cartesian, the top right rectangle and the outer rectangle in the following diagram commute thanks to the case of reduced source treated before:
        \[ \begin{tikzcd}[scale cd=.95,center picture,column sep=huge]
            a^*_Y \rR f_{\red, *} (\omega_{X_\red}) \arrow[r, "\BC"] \arrow[d, "\sim"{sloped}, "a_Y^*\rR f_*(\tr_\iota)"'] &[-1em] \rR f_{L, *} a^*_X(\iota_* \omega_{X_\red})\arrow[r, "\rR f_{L, *}(\BC)"] \arrow[d, "\sim"{sloped}, "\rR f_{L, *}a_X^*(\tr_\iota)"']   & \rR f_{L, \red, *} a_{X_\red}^*(\omega_{X_\red}) \arrow[r, "\sim"', "\rR f_{L, \red, *}(\gamma_{X_\red, L})"]  &[2em] \rR f_{L, \red, *}(\omega_{X_{\red, L}}) \arrow[d, "\sim"{sloped}, "\rR f_{L, *}(\tr_{\iota_L})"']  \\ 
            a_Y^* \rR f_* (\omega_X) \arrow[r, "\BC"] \arrow[d, "a_Y^*(\tr_f)"'] & \rR f_{L, *} a_X^*(\omega_X) \arrow[rr, "\sim"', "\rR f_{L, *}(\gamma_{X, L})"]& & \rR f_{L, *}(\omega_{X_L}) \arrow[d, "\tr_{f_L}"'] \\ 
            a_Y^*(\omega_Y) \arrow[rrr, "\gamma_{Y, L}", "\sim"']  & & & \omega_{Y_L}
        \end{tikzcd} \]       
        On the other hand, the upper left rectangle in the diagram commutes by the naturality of the base change map and the left horizontal trace maps are isomorphisms (see \cref{dualizing-nilimmersion}).
        Thus, the right rectangle must also commute.
        This yields the statement for general $X$. \qedhere
    \end{enumerate}
\end{proof}

We recall that \cite[Th.~3.21~(1)]{BH} provides us with a canonical isomorphism $c_f \colon \omega_X \xr{\sim} \rR i^! \omega_Y$ for any finite morphism $f\colon X\to Y$ of rigid-analytic spaces over $K$.
Therefore, for such $f$, we can give an alternative construction of a trace map $\tr_f^{\BH} \colon f_* \omega_X \to \omega_Y$ as the composition
\[
f_* \omega_X \xr[\sim]{f_*(c_f)} f_*\rR f^! \omega_Y \xr{\epsilon_f} \omega_Y.
\]

\begin{proposition}\label{finite trace equals proper trace for finite morphism}
Let $f \colon X \to Y$ be a finite morphism of rigid-analytic spaces over $K$.
Then
\[ \tr_f^\BH = \tr_f \colon f_* \omega_X \longrightarrow \omega_Y. \]
\end{proposition}
\begin{proof}
First, note that when $f$ is a closed immersion, the statement follows from \cref{most general proper trace!!}\cref{most general proper trace!!-closed}.
Now we can use \cref{dualizing-nilimmersion} and the fact that both $\tr_f$ and $\tr_f^\BH$ respect compositions to reduce the question to the case when $X$ and $Y$ are reduced.
Furthermore, we can assume that $X$ is connected by arguing one clopen connected component of $X$ at a time;
cf.\ \cite[Cor.~2.3]{adic-notes}.
Finally, $f(X)\subset Y$ is Zariski-closed by virtue of \cite[Prop.~9.6.3/3]{BGR}, so we can replace $Y$ with $f(X)$ to assume that $f$ is surjective. 

Thanks to \cite[Th.~3.21]{BH} and our assumption that $X$ is connected,
\[
\rR\cHom(f_*\omega_X, \omega_Y) = \rR f_* \rR\cHom(\omega_X, \omega_X) = \rR f_* \ud{\Lambda}_X \in D^{\geq 0}(Y_\et; \Lambda)
\]
and $\cHom(f_*\omega_X, \omega_Y)(Y) \to \cHom(f_*\omega_X, \omega_Y)(U)$ is injective for any nonempty open subspace $U\subset Y$.
Since both traces are \'etale local on $Y$, we may thus prove the statement after replacing $Y$ with any such nonempty open $U\subset Y$ and $X$ with $X_U\coloneqq X\times_Y U$ (after this procedure, $X$ might be disconnected but it is not important for the rest of the proof).
Recall that we assume that $X$ and $Y$ are reduced and that $f$ is surjective, so there is a nonempty open $U \subset Y$ such that $\restr{f}{f^{-1}(U)}\colon f^{-1}(U)\to U$ is finite \'etale. 
Hence, we can assume that $f$ is finite \'etale.

Finally, both traces are \'etale local on the target and any finite \'etale morphism is \'etale locally a disjoint union of isomorphisms (or $X$ is empty). 
Therefore, we reduce to the case when $f$ is an isomorphism (or $X$ is empty).
In this case, the claim becomes trivial.
\end{proof}

\subsection{Poincar\'{e} duality for Zariski constructible coefficients}
We recall that, throughout this section, we always assume that $K$ is a nonarchimedean field of characteristic $0$, $n$ is a positive integer, and $\Lambda = \Z/n\Z$.

The goal of this subsection is to prove a version of Poincar\'e duality for general proper morphisms of rigid-analytic spaces and Zariski-constructible coefficients.
More precisely, we will show that given a proper morphism $f \colon X \to Y$ of rigid-analytic spaces over $K$, the functor $\rR f_* \colon D_\zc(X_\et; \Lambda) \to D_\zc(Y_\et; \Lambda)$ commutes with Verdier duality;
this confirms an expectation of Bhatt--Hansen (see \cite[Rmk.~3.23]{BH}).
As an application of our main result, we deduce duality for intersection cohomology on certain non-smooth and non-proper rigid-analytic spaces. In particular, this confirms the expectation raised in the comment after  \cite[Th.~4.13]{BH}. 

We begin by setting up some notation:
\begin{notation}\label{morphism-transformations}
Let $f \colon X \to Y$ be a morphism of rigid-analytic spaces over $K$.
\begin{enumerate}[leftmargin=*,label=\upshape{(\roman*)}]
\item\label{morphism-transformations-evaluation} (see e.g.\ \cite[\href{https://stacks.math.columbia.edu/tag/0B6D}{Tag~0B6D}]{stacks-project})
The \emph{evaluation transformation}
\[ \Ev_f \colon \rR f_* \rR\cHom(\blank,\blank) \to \rR\cHom\bigl(\rR f_*(\blank),\rR f_*(\blank)\bigr) \]
is the natural transformation of functors given on objects $\cE,\cE' \in D(X_\et;\Lambda)$ as the tensor-hom adjoint to the composition
\[ \rR f_*\rR\cHom(\cE,\cE') \otimes^L \rR f_* \cE \xr{\cup} \rR f_*\bigl(\rR\cHom(\cE,\cE') \otimes^L \cE\bigr) \xr{\rR f_*(\eval)} \rR f_* \cE' \]
of the relative cup product from \cite[\href{https://stacks.math.columbia.edu/tag/0B6C}{Tag~0B6C}]{stacks-project} (or \cref{cup-product}) and the derived pushforward of the evaluation map.
\item\label{morphism-transformations-adjunction} The adjunction between $f^*$ and $\rR f_*$ upgrades to an isomorphism
\[ \Adj_f \colon \rR f_* \rR\cHom(f^*\cE,\cE') \xlongrightarrow{\Ev_f} \rR\cHom(\rR f_*f^*\cE,\rR f_*\cE') \xlongrightarrow{\blank \circ \eta_f} \rR\cHom(\cE,\rR f_*\cE') \]
for any $\cE \in D(Y_\et;\Lambda)$ and $\cE' \in D(X_\et;\Lambda)$, where $\eta_f$ denotes the unit of the $(f^*,\rR f_*)$-adjunction.
This isomorphism is functorial in $\cE$ and $\cE'$ and hence gives rise to a natural equivalence of functors.
\end{enumerate}
\end{notation}
\begin{remark}\label{morphism-transformations-composition}
    The evaluation and adjunction transformations are compatible with compositions.
    That is, given morphisms $f \colon X \to Y$ and $g \colon Y \to Z$ of rigid-analytic spaces over $K$, we have $\Ev_{g \circ f} = \Ev_g \circ \rR g_*(\Ev_f)$ and $\Adj_{g \circ f} = \Adj_g \circ \rR g_* \Adj_f$.
    Unwinding definitions, the first identity concerning $\Ev$ amounts to the commutativity of the diagram
    \[ \begin{tikzcd}[scale cd=.88,center picture]
        \rR g_* \rR f_* \rR\cHom(\cE, \cE') \otimes^L \rR g_* \rR f_* \cE' \arrow[r,"\cup"] \arrow[rr,bend left=10,"\cup"] \arrow[d,"\rR g_*(\Ev_f) \otimes^L \id"] & \rR g_* (\rR f_* \rR\cHom(\cE, \cE') \otimes^L \rR f_* \cE') \arrow[r,"\rR g_*(\cup)"] \arrow[d,"\rR g_*(\Ev_f \otimes^L \id)"] & \rR g_* \rR f_*(\rR\cHom(\cE, \cE') \otimes^L \cE') \arrow[d,"\rR g_* \rR f_*(\eval)"] \\
        \rR g_* \rR\cHom(\rR f_* \cE, \rR f_* \cE') \otimes^L \rR g_* \rR f_* \cE' \arrow[r,"\cup"] & \rR g_* (\rR\cHom(\rR f_* \cE, \rR f_* \cE') \otimes^L \rR f_* \cE') \arrow[r,"\rR g_*(\eval)"] & \rR g_* \rR f_*\cE';
    \end{tikzcd} \]
    note that the right square commutes already before applying the derived pushforward $\rR g_*$ because the composition with the counit of the derived tensor-hom adjunction in the bottom horizontal map produces the adjoint of $\Ev_f \otimes^L \id$, which is by definition the composition in the clockwise direction.
    The case of $\Adj$ then quickly reduces to the case of $\Ev$ via the commutative diagram
    \[ \begin{tikzcd}
        \rR g_* \rR f_* \rR\cHom(f^* g^* \cE, \cE') \arrow[r,"\Ev_{g \circ f}"] \arrow[d,"\rR g_*(\Ev_f)"] & \rR \cHom(\rR g_* \rR f_* f^* g^* \cE,\rR g_* \rR f_* \cE') \arrow[rd,near start,"\rR g_*(\blank \circ \eta_f)"'] \arrow[r,"\blank \circ \eta_{g \circ f}"] & \rR \cHom(\cE,\rR g_* \rR f_* \cE') \\
        \rR g_* \rR \cHom(\rR f_* f^* g^* \cE, \rR f_* \cE') \arrow[r,"\rR g_*(\blank \circ \eta_f)"] \arrow[ru,"\Ev_g"] & \rR g_* \rR \cHom(g^* \cE,\rR f_* \cE') \arrow[r,"\Ev_g"] & \rR \cHom(\rR g_* g^* \cE,\rR g_* \rR f_* \cE'). \arrow[u,"\blank \circ \eta_g"]
    \end{tikzcd} \]
\end{remark}
The evaluation and adjunction isomorphism satisfy moreover the following compatibility:
\begin{lemma}\label{evaluation-adjunction-compatibility}
    Let $f \colon X \to Y$ be a morphism of rigid-analytic spaces over $K$.
    Let $\cE,\cE' \in D(X_\et;\Lambda)$ and denote the counit of the $(f^*,\rR f_*)$-adjunction by $\epsilon_f \colon f^* \circ \rR f_* \to \id$.
    Then the following diagram commutes:
    \[ \begin{tikzcd}[column sep=huge]
        \rR f_* \rR\cHom(\cE,\cE') \arrow[r,"\rR f_* \circ (\blank \circ \epsilon_f)"] \arrow[rd,"\Ev_f"'] & \rR f_* \rR\cHom(f^*\rR f_*\cE,\cE') \arrow[d,"\Adj_f"] \\
        & \rR\cHom(\rR f_*\cE,\rR f_*\cE')
    \end{tikzcd} \]
\end{lemma}
\begin{proof}
    Using \cref{morphism-transformations}\cref{morphism-transformations-adjunction}, we can expand the diagram in the statement as follows:
    \[ \begin{tikzcd}[column sep=huge]
        \rR f_* \rR\cHom(\cE,\cE') \arrow[r,"\rR f_* \circ (\blank \circ \epsilon_f)"] \arrow[d,"\Ev_f"'] & \rR f_* \rR\cHom(f^*\rR f_*\cE,\cE') \arrow[d,"\Ev_f"] \\
        \rR\cHom(\rR f_*\cE,\rR f_*\cE') \arrow[r,"\blank \circ \rR f_*(\epsilon_f)"] \arrow[rd,equals] & \rR\cHom(\rR f_*f^*\rR f_*\cE, \rR f_*\cE') \arrow[d,"\blank \circ \eta_f"] \\
        & \rR\cHom(\rR f_*\cE,\rR f_*\cE');
    \end{tikzcd} \]
    here, $\eta_f$ denotes the unit for the $(f^*,\rR f_*)$-adjunction.
    The upper square commutes by the naturality of the evaluation transformation.
    The commutativity of the lower triangle is a standard exercise about the relationship of units and counits of adjunctions (see \cite[\href{https://stacks.math.columbia.edu/tag/0GLL}{Tag~0GLL}]{stacks-project}).
    This yields the assertion.
\end{proof}
Now we define the duality morphism.
\begin{definition}\label{duality}
\begin{enumerate}[leftmargin=*,label=\upshape{(\roman*)}]
    \item (\emph{Duality functor})
    Let $X$ be a rigid-analytic space over $K$.
    The \emph{Verdier duality functor} is given by 
    \[ \DD_X(\blank) \colonequals \rR \cHom(\blank,\omega_X) \colon D(X_\et; \Lambda)^\op \to D(X_\et; \Lambda). \]
    
    \item\label{duality-PD-transformation} (\emph{Duality map})
    Let $f \colon X \to Y$ be a proper morphism of rigid-analytic spaces over $K$.
    The \emph{Poincar\'e duality transformation} is given on an object $\cE \in D(X_\et;\Lambda)$ by the composition 
    \begin{multline*}
        \PD_f(\cE) \colon \rR f_* \DD_X(\cE) = \rR f_*\rR \cHom (\cE, \omega_X) \xlongrightarrow{\Ev_f} \rR\cHom(\rR f_*\cE,\rR f_* \omega_X) \\ \xlongrightarrow{\tr_f \circ \blank} \rR\cHom(\rR f_*\cE, \omega_Y) = \DD_Y(\rR f_*\cE),
    \end{multline*}
    where the first morphism comes from \cref{morphism-transformations}\cref{morphism-transformations-evaluation} and the second morphism is given by postcomposition with the proper trace map $\tr_f$ from \cref{most general proper trace!!}.
    
    This composition is functorial in $\cE$ and hence defines a natural transformation of functors 
    \[ \PD_f \colon \rR f_* \circ \DD_X \to \DD_Y \circ \rR f_* \colon D(X_\et; \Lambda)^\op \to D(Y_\et; \Lambda). \]
\end{enumerate}
\end{definition}
The main goal of this subsection is to show that $\PD_f(\cF)$ is an isomorphism for proper $f$ and Zariski-constructible $\cF \in D^{(b)}_\zc(X_\et;\Lambda)$.
Before we embark on the proof, we need to establish some preliminary results. 
The first thing we discuss is the behavior of the Poincar\'e duality transformation with respect to the upper shriek functors. 

\begin{notation}\label{notation:duality-shriek-pullback} Consider a cartesian diagram of rigid-analytic spaces over $K$ 
\[ \begin{tikzcd}
    X' \arrow{d}{f'} \arrow{r}{i'} & X \arrow{d}{f} \\
    Y' \arrow{r}{i} & Y
\end{tikzcd} \]
such that $f$ and $f'$ are proper and $i$ and $i'$ are closed immersions.
Then we have the following natural transformations:
\begin{enumerate}[label=\upshape{(\roman*)}]
    \item the \emph{exchange transformation} $\Ex_i \colon \DD_{Y'} \circ i^* \to \rR i^! \circ \DD_Y$, defined as the $(i_*, \rR i^!)$-adjoint of the composition
    \[ i_* \circ \DD_{Y'} \circ i^* = i_* \rR\cHom\bigl(i^*(\blank),\omega_{Y'}\bigr) \xlongrightarrow{\Adj_i} \rR\cHom(\blank,i_*\omega_{Y'}) \xlongrightarrow{\tr_i \circ \blank} \rR\cHom(\blank,\omega_Y) = \DD_Y, \]
    where $\eta_i$ is the unit of the $(i^*, i_*)$-adjunction;\footnote{Note that the Verdier duality functor $\DD_X$ is contravariant, so it reverses the direction of morphisms.}
    \item the \emph{base change transformation} $\BC\colon i^* \circ \rR f_* \to \rR f'_* \circ i'^*$, defined as the $(f^{\prime,*},\rR f'_*)$-adjoint of
    \[ f^{\prime,*} \circ i^* \circ \rR f_* = i^{\prime,*} \circ f^* \circ \rR f_* \xlongrightarrow{i^{\prime,*}(\epsilon_f)} i^{\prime,*}; \]
    \item the \emph{shriek base change transformation} $\BC^! \colon \rR f'_* \circ \rR i'^! \to \rR i^! \circ \rR f_*$, defined as the $(i_*,\rR i^!)$-adjoint of 
    \[i_* \circ \rR f'_* \circ \rR i'^! = \rR f_* \circ i'_* \circ \rR i'^! \xr{\rR f_*(\epsilon_{i'})} \rR f_*.\]
\end{enumerate}
\end{notation}
These natural transformations interact with the Poincar\'e duality transformation from \cref{duality}\cref{duality-PD-transformation} in the following manner:
\begin{proposition}
\label{Poincare transformation and closed immersion compatibility}
With \cref{notation:duality-shriek-pullback}, the following diagram of natural transformations between contravariant functors from $D(X)$ to $D(Y')$ commutes:
\begin{equation}\label{eqn:dualuty-*-!} \begin{tikzcd}[column sep=huge]
    \rR f'_* \circ \DD_{X'} \circ i^{\prime,*} \arrow[r, "\rR f'_* \circ \Ex_{i'}"] \arrow[d, "\PD_{f'} \circ i^{\prime,*}"] & \rR f'_* \circ \rR i^{\prime,!} \circ \DD_{X} \arrow[r, "\BC^! \circ \DD_X"] & \rR i^! \circ \rR f_* \circ \DD_{X} \arrow[d, "\rR i^! \circ \PD_f"] \\
    \DD_{Y'} \circ \rR f'_* \circ i^{\prime,*} \arrow[r, "\DD_{Y'} \circ \BC"] &
    \DD_{Y'} \circ i^* \circ \rR f_* \arrow[r, "\Ex_{i} \circ \rR f_*"] &
    \rR i^! \circ \DD_{Y} \circ \rR f_*.
\end{tikzcd} 
\end{equation}
    Furthermore, all horizontal transformations become isomorphisms when evaluated on $\F\in D^{(b)}_\zc(X_\et; \Lambda)$. 
\end{proposition}
Here, the first bottom horizontal arrow has its direction seemingly reversed, which is due to the fact that the Verdier duality functor $\DD_{Y'}$ is contravariant.
\begin{proof}
    To lighten notation, we drop all the derived notation for pushforwards and exceptional inverse images for the duration of this proof; 
    for example, we write $f_*$ instead of $\rR f_*$ and $i^!$ instead of $\rR i^!$.
    Moreover, we omit all the ``$\circ$''-symbols between functors and natural transformations.
    As a further simplification, we denote the functor $\rR\cHom(\blank,f_* \omega_X) \colon D(Y_\et;\Lambda)^\op \to D(Y_\et; \Lambda)$ by $\DD_{X \to Y}(\blank)$, and similarly for other morphisms.
    With these shorthands, we have for example 
    \[ \PD_f \colon f_* \DD_X \xlongrightarrow{\Ev_f} \DD_{X \to Y} f_* \xlongrightarrow{\tr_f \circ \blank} \DD_Y f_* \quad \text{and} \quad \Adj_f \colon f_* \DD_X f^* \xlongrightarrow{\Ev_f} \DD_{X \to Y}f_*f^* \xlongrightarrow{\blank \circ \eta_f} \DD_{X \to Y}. \]

    We begin with the commutativity of the diagram \cref{eqn:dualuty-*-!}.
    By the $(i_*,i^!)$-adjunction, it suffices to show the commutativity of the corresponding diagram where we add $i_*$ to the left four terms and drop the $i^!$ from the right two terms.
    Plugging in the definitions of all the transformations, this adjoint diagram becomes the outer rim of the following diagram:
    \[ \begin{tikzcd}[scale cd=.9,center picture]
        i_* f'_* \DD_{X'} i^{\prime,*} \arrow[r,"i_* f'_* \eta_{i'}"] \arrow[rd,equals] \arrow[rdd,"i_* f'_* \DD_{X'} i^{\prime,*} \epsilon_f" description] \arrow[dddd,"i_* \Ev_{f'}"] & i_* f'_* i^{\prime,!} i'_* \DD_{X'} i^{\prime,*} \arrow[r,"i_* f'_* i^{\prime,!} \Adj_{i'}"] \arrow[rd,equals] &[.5em] i_* f'_* i^{\prime,!} \DD_{X' \to X} \arrow[r,"i_* f'_* i^{\prime,!} (\tr_{i'} \circ \blank)"] \arrow[rd,equals] &[.5em] i_* f'_* i^{\prime,!} \DD_X \arrow[r,equals] &[.5em] f_* i'_* i^{\prime,!} \DD_X \arrow[r,"f_* \epsilon_{i'}"] & f_* \DD_X \arrow[dddd,"\Ev_f"] \\
        & f_* i'_* \DD_{X'} i^{\prime,*} \arrow[r,"f_* i'_* \eta_{i'}"] \arrow[rd,equals] \arrow[rrd,pos=.35,font=\scriptsize,eq=Poincare transformation and closed immersion compatibility-unit-counit] & f_* i'_* i^{\prime,!} i'_* \DD_{X'} i^{\prime,*} \arrow[r,"f_* i'_* i^{\prime,!} \Adj_{i'}"] \arrow[d,"f_* \epsilon_{i'}"] & f_* i'_* i^{\prime,!} \DD_{X' \to X} \arrow[r,"f_* \epsilon_{i'}"] \arrow[ru,"f_* i'_* i'^!(\tr_{i'} \circ \blank)" description] & f_* \DD_{X' \to X} \arrow[ru,"f_* (\tr_{i'} \circ \blank)" description] \arrow[ddd,"\Ev_f"] \arrow[ld,"f_* \DD_{X' \to X} \epsilon_f" description] \arrow[lddd,pos=.3,font=\scriptsize,eq=Poincare transformation and closed immersion compatibility-evaluation-adjunction] & \\
        & i_* f'_* \DD_{X'} i^{\prime,*} f^* f_* \arrow[dd,"i_* \Ev_{f'}"] & f_* i'_* \DD_{X'} i^{\prime,*} \arrow[rru,"f_* \Adj_{i'}" description] \arrow[rd,"f_* i'_* \DD_{X'} i^{\prime,*} \epsilon_f" description] & f_* \DD_{X' \to X} f^* f_* \arrow[rdd,bend left=25,"\Adj_f" description] \arrow[rdd,pos=.6,font=\scriptsize,eq=Poincare transformation and closed immersion compatibility-adjunction-1] && \\[.5em]
        && i_* f'_* \DD_{X'} f^{\prime,*} i^* f_* \arrow[r,equals] \arrow[d,"i_* \Ev_{f'}"'] \arrow[rd,"i_* \Adj_{f'}" description] \arrow[rrd,bend left=2.5,"\Adj_{i \circ f'}" description] \arrow[rdd,pos=.29,font=\scriptsize,eq=defn-adj] \arrow[rrrdd,shift left=1ex,pos=.35,font=\scriptsize,eq=Poincare transformation and closed immersion compatibility-adjunction-2] & f_* i^{\prime,*} \DD_{X'} i^{\prime,*} f^* f_* \arrow[u,"f_* \Adj_{i'}" description] \arrow[rd,"\Adj_{f \circ i'}" description] && \\[.5em]
        i_* \DD_{X' \to Y'} f'_* i^{\prime,*} \arrow[r,outer sep=.7em,"i_* \DD_{X' \to Y'} f'_* i^{\prime,*} \epsilon_f"] \arrow[d,"i_* (\tr_{f'} \circ \blank)"] & i_* \DD_{X' \to Y'} f'_* i^{\prime,*} f^* f_* \arrow[r,equals] \arrow[d,"i_* (\tr_{f'} \circ \blank)"] & i_* \DD_{X' \to Y'} f'_* f^{\prime,*} i^* f_* \arrow[r,outer sep=.3em,"i_* \DD_{X' \to Y'} \eta_{f'}"'] \arrow[d,"i_* (\tr_{f'} \circ \blank)"] & i_* \DD_{X' \to Y'} i^* f_* \arrow[r,"\Adj_i"'] \arrow[d,"i_* (\tr_{f'} \circ \blank)"] & \DD_{X' \to Y} f_* \arrow[r,outer sep=.4em,"(f_*\tr_{i'} \circ \blank)"] \arrow[d,"(i_*\tr_{f'} \circ \blank)"'] \arrow[rd,font=\scriptsize,eq=Poincare transformation and closed immersion compatibility-trace-composition] & \DD_{X \to Y} f_* \arrow[d,"(\tr_f \circ \blank)"] \\[1em]
        i_* \DD_{Y'} f'_* i^{\prime,*} \arrow[r,"i_* \DD_{Y'} f'_* i^{\prime,*} \epsilon_f"] & i_* \DD_{Y'} f'_* i^{\prime,*} f^* f_* \arrow[r,equals] & i_* \DD_{Y'} f'_* f^{\prime,*} i^* f_* \arrow[r,"i_* \DD_{Y'} \eta_{f'}"] & i_* \DD_{Y'} i^* f_* \arrow[r,"\Adj_i"] & \DD_{Y' \to Y} f_* \arrow[r,"(\tr_i \circ \blank)"] & \DD_Y f_*
    \end{tikzcd} \]
    For the most part, the various cells in this diagram commute thanks to the naturality of the evaluation transformations $\Ev$, the adjunction transformations $\Adj$, the counit $\epsilon_{i'}$, and the transformations given by precomposition with traces, with the following exceptions:
    \begin{itemize}
        \item the triangle \cref{Poincare transformation and closed immersion compatibility-unit-counit} commutes by the unit-counit identity for adjunctions (see \cite[\href{https://stacks.math.columbia.edu/tag/0GLL}{Tag~0GLL}]{stacks-project});
        \item the triangle \cref{Poincare transformation and closed immersion compatibility-evaluation-adjunction} commutes by \cref{evaluation-adjunction-compatibility};
        \item the triangles \cref{Poincare transformation and closed immersion compatibility-adjunction-1} and \cref{Poincare transformation and closed immersion compatibility-adjunction-2} commute by \cref{morphism-transformations-composition}; 
        \item the triangle \cref{defn-adj} commutes due to the definition of $\Adj_f$ (see \cref{morphism-transformations}\cref{morphism-transformations-adjunction}); and
        \item the bottom right square \cref{Poincare transformation and closed immersion compatibility-trace-composition} commutes by \cref{most general proper trace!!}\cref{most general proper trace!!-composition}.
    \end{itemize}
    This finishes the proof of the first assertion.
    To prove the second assertion, we first note that \cite[Th.~3.10, Cor.~3.12, and Cor.~3.14]{BH} imply that all functors involved in \cref{eqn:dualuty-*-!} preserve $D_\zc^{(b)}$. Therefore, a combination of \cref{lemma:proper-base-change}, \cref{lemma:shriek-star-base-change}, and \cite[Th.~3.21~(4)]{BH} shows that all horizontal transformations become isomorphisms when evaluated on $\F\in D^{(b)}_\zc(X_\et; \Lambda)$.
\end{proof}

In order to be able to use \cref{Poincare transformation and closed immersion compatibility} effectively, we need the following general lemma: 

\begin{lemma}\label{lemma:upper-shrieks-conservative} Let $X$ be a rigid-analytic space over $K$ and let $\F \in D^{(b)}_\zc(X_\et; \Lambda)$ be a locally bounded complex with Zariski-constructible cohomology sheaves. Assume that $\rR i_x^! \F =0$ for every classical point $i_x\colon x \hookrightarrow X$. Then $\F=0$.
\end{lemma}
\begin{proof}
    We pick a classical point $i_x\colon x \hookrightarrow X$. Then \cite[Cor.~3.12]{BH} ensures that $\rR i_x^!$ carries $D^{(b)}_\zc(X_\et; \Lambda)$ to $D^{b}_\zc(x_\et; \Lambda)$. Furthermore, (the proof of) \cite[Th.~3.21~(4)]{BH} implies that $\DD_x\bigl(\rR i_x^! \F\bigr) \simeq i_x^*\DD_X\F$. Since $\DD_x$ induces an anti-equivalence of $D^b_\zc(x_\et; \Lambda)$ (see \cite[Th.~3.21~(3)]{BH}), we conclude that $i_x^* \DD_X(\F)=0$.
    Moreover, \textit{loc.\ cit.} implies that $\DD_X$ induces an anti-equivalence of $D^{(b)}_\zc(X_\et; \Lambda)$. In particular, $\DD_X(\F)\in D^{(b)}_\zc(X_\et; \Lambda)$. Since $x\in X$ was an arbitrary classical point, we conclude that $i_x^* \DD_X(\F) = 0$ for any classical point $x\in X$. This implies that $\DD_X(\F)=0$ because $\DD_X(\F)$ has Zariski-constructible cohomology. Finally, we use that $\DD_X$ induces an anti-equivalence of $D^{(b)}_\zc(X_\et; \Lambda)$ once again to conclude that $\F=0$.
\end{proof}

We also discuss a particularly nice set of generators in $D^b_\zc(X_\et; \Lambda)$ for a quasi-compact rigid-analytic space $X$ over a non-archimedean field $K$ of characteristic $0$. 

\begin{lemma}\label{lemma:generators-dbzc} Let $X$ be a quasi-compact rigid-analytic space over $K$. Then $D^b_\zc(X_\et; \Lambda)$ is the smallest thick triangulated subcategory of $D(X_\et; \Lambda)$ containing all objects of the form $\rR f_*\ud{M}_{X'}$ for a finitely generated $\Lambda$-module $M$ and a proper morphism $f\colon X' \to X$ such that $X'$ is smooth and $\dim f^{-1}(x)< \max(\dim X, 1)$ for any classical point $x\in X$. 
\end{lemma}
\begin{proof}
    We denote by $D'\subset D(X_\et; \Lambda)$ the smallest thick triangulated subcategory which contains $\rR f_* \ud{M}_{X'}$ for $f$ and $M$ as in the formulation of the lemma. Using \cite[Th.~3.10]{BH}, we conclude that $D'\subset D^b_\zc(X_\et; \Lambda)$. Therefore, it suffices to show that $D^b_\zc(X_\et;\Lambda) \subset D'$. We prove this by induction on $d=\dim X$ (note that $\dim X$ is finite since $X$ is quasi-compact). 
    
    If $d\leq 0$, then the claim is essentially obvious because either $X$ is empty or $X_\red = \sqcup_{i=1}^m \Spa(L_i, L_i^\circ)$ for some finite extensions $K\subset L_i$. Now we assume that $\dim X = d>0$ and the result is known for all spaces of dimension strictly less then $d$. Then \cite[Prop.~3.6]{BH} implies that it suffices to show that sheaves of the form $g_*\ud{M}_{X'}$ lie in $D'$, where $g\colon X' \to X$ is a finite morphism and $M$ is a finitely generated $\Lambda$-module. By the topological invariance of the \'etale site (see \cite[Prop.~2.3.7]{Huber-etale}), we may assume that both $X$ and $X'$ are reduced. Now denote by $U'$ the smooth locus of $X'$ and by $Z'$ its Zariski-closed complement (with reduced adic space structure).

    Then \cref{modifications-exist} implies that there is a regular $U'$-admissible modification $h\colon X'' \to X'$. We denote by $f\colon X'' \to X$ the composed morphism to $X'$ and by $Z''\coloneqq X''\times_X Z'$ the pre-image of $Z'$ in $X''$. 
   
    First, we show that, for every classical point $x\in X$, we have $\dim f^{-1}(x)<\max(\dim X, 1)$. To see this, we first note that $\dim X'\leq \dim X$ since $g$ is finite and, thus, \cref{lemma:dimension-of-fibers-modifications} implies that $\dim h^{-1}(x') < \max(\dim X', 1) \leq \max(\dim X, 1)$. Then we observe that, for every classical point $x\in X$, the fiber $\abs{f^{-1}(x)}$ is (topologically) equal to $\bigsqcup_{x'\in X'_{\rm{cl}}:f(x')=x} \abs{h^{-1}(x')}$. Therefore, we also have $\dim f^{-1}(x)< \max(\dim X, 1)$.

    Now we consider the following exact triangle:
    \begin{equation}\label{eqn:smooth-resolution-constant-sheaf}
    \ud{M}_{X'} \to \rR h_* \ud{M}_{X''} \to C.
    \end{equation}
    By construction, $\supp(C) \subset Z'$. After applying $g_*$ to \cref{eqn:smooth-resolution-constant-sheaf}, we get the following exact triangle
    \begin{equation}\label{eqn:smooth-resolution-constant-sheaf-2}
    g_* \ud{M}_{X'} \to \rR f_* \ud{M}_{X''} \to g_*(C).
    \end{equation}
    By construction, $g_*(C)$ is supported on $Z\coloneqq g(Z')$, which is a Zariski-closed subset of $X$ due to \cite[Th.~9.6.3/3]{BGR}. Since $Z' \to Z$ is surjective and $Z'$ is nowhere dense in $X'$, we conclude that $\dim Z \leq \dim Z' < \dim X' \leq \dim X =d$. Therefore, the induction hypothesis implies that $g_*(C) \in D'$. We also have $\rR f_* \ud{M}_{X''}$ by the very definition of $D'$. Thus, \cref{eqn:smooth-resolution-constant-sheaf-2} ensures that $g_*\ud{M}_{X'}\in D'$ finishing the proof. 
\end{proof}

Now we are ready to prove the general Poincar\'{e} duality result as 
expected by Bhatt and Hansen (see \cite[Rmk.~3.23]{BH}). 

\begin{theorem}
\label{general Poincare duality}
Let $f\colon X \to Y$ be a proper morphism of rigid-analytic spaces over $K$, and let $\F\in D_\zc(X_\et;\Lambda)$. Then the Poincar\'e duality transformation
\[
\PD_f(\F) \colon \rR f_* \bigl( \DD_X(\F) \bigr) \to \DD_Y \bigl( \rR f_* \F \bigr)
\]
is an isomorphism.
\end{theorem}
\begin{proof}

\begin{enumerate}[wide,label={\textit{Step~\arabic*}.},ref={Step~\arabic*}]
    \setcounter{enumi}{-1}
    \item \textit{We reduce to the case when $X$ and $Y$ are qcqs.}
    First, \cref{most general proper trace!!}\cref{most general proper trace!!-localization} implies that the question is local on $Y$. Therefore, we can assume that $Y$ is qcqs. In this case, $X$ is automatically qcqs as well.

    \item\label{general Poincare duality - bounded reduction} \textit{We reduce to the case when $\F$ lies in $D^b_\zc(X_\et; \Lambda)$.}
    First, we note that $\rR f_*$ commutes with sequential homotopy colimits (e.g., as defined in \cite[\href{https://stacks.math.columbia.edu/tag/0A5K}{Tag~0A5K}]{stacks-project}) due to \cite[Lem.~9.1]{adic-notes}.
    This implies that both the source and target of $\PD_f$ (viewed as functors in $\cal{E}$) transform sequential homotopy colimits into sequential homotopy limits (e.g., as defined in \cite[\href{https://stacks.math.columbia.edu/tag/08TB}{Tag~08TB}]{stacks-project}).
    Since the natural morphism $\hocolim_n \tau^{\le n} \F \to \F$ is an isomorphism (this can be deduced from \cite[\href{https://stacks.math.columbia.edu/tag/0CRK}{Tag 0CRK}]{stacks-project}), we reduce to the case when $\F\in D^-_\zc(X_\et; \Lambda)$.
    In this case, we consider the exact triangle
    \[
    \tau^{\leq -N}\F \to \F \to \tau^{> -N} \F.
    \]
    Recall that $\rR f_*$ has cohomological dimension $2\dim(f)$ by virtue of \cite[Prop.~5.3.11]{Huber-etale}. Furthermore, $\omega_X\in D^{[-2\dim X, 0]}(X_\et;\Lambda)$ and $\omega_Y\in D^{[-2\dim Y, 0]}(Y_\et; \Lambda)$ by virtue of \cite[Lem.~3.30]{BH}. Therefore, we conclude that
    \begin{gather*}
    \begin{split}
        \rR f_*\bigl( \DD_X(\tau^{\leq -N}\F)\bigr) = \rR f_*\bigl(\rR \cHom_{\Lambda}(\tau^{\leq -N}\F, \omega_X)\bigr) \in D^{\geq N-2\dim X}(X_\et; \Lambda), \\
        \DD_Y\bigl(\rR f_* \tau^{\leq -N} \F\bigr) = \rR \cHom_{\Lambda}\bigl(\rR f_*(\tau^{\leq -N}\F), \omega_Y\bigr) \in D^{\geq N-2\dim f-2\dim Y}(Y_\et; \Lambda). 
    \end{split}
    \end{gather*}    
    Given an integer $q$, the map on cohomology sheaves $\cal{H}^q\bigl(\PD_f(\F)\bigr)$ is therefore an isomorphism if and only if $\cal{H}^q\bigl(\PD_f(\tau^{> - M} \F)\bigr)$ is an isomorphism for any large $M\gg 0$.
    In particular,  if $\PD_f(\tau^{> - N} \F)$ is an isomorphism for all $N$, then $\PD_f(\F)$ is an isomorphism as well. 
    Thus, we reduce to the case when $\F$ is bounded.

    \item\label{step:fiberwise-reduction} \textit{We reduce to the case when $Y=\Spa(K, \O_K)$.} Pick a classical point $i_y\colon y \hookrightarrow Y$ and consider the fiber sequence
    \[
    \begin{tikzcd}
        X_y \arrow{d}{f_y} \arrow{r}{i'_y} & X\arrow{d}{f} \\
        y \arrow{r}{i_y} & Y.
    \end{tikzcd}
    \]
    Then \cite[Th.~3.10 and Th.~3.21~(3)]{BH} imply that both $\rR f_* \DD_X(\F)$ and $\DD_Y(\rR f_*\F)$ lie in $D^b_\zc(Y_\et; \Lambda)$. Therefore, \cref{lemma:upper-shrieks-conservative} (applied to $\rm{cone}(\PD_f(\F))$) implies that it suffices to show that $\rR i_y^!\bigl(\PD_f(\F)\bigr)$ is an isomorphism for any classical point $y\in Y$. Furthermore, \cref{Poincare transformation and closed immersion compatibility} then ensures that it suffices to show that $\PD_{f_y}(i_y^*\F)$ is an isomorphism for any classical point $y\in Y$. In other words, we can assume that $Y=\Spa(L, \O_L)$ for some finite extension $K\subset L$ (and $\F$ still lies in $D^b_\zc(X_\et; \Lambda)$). After replacing $K$ by $L$, we can even assume that $Y=\Spa(K, \O_K)$. 
    
    \item \textit{End of proof.} Finally, we complete the argument under the extra assumptions that $Y=\Spa(K, \O_K)$ and $\F\in D^b_\zc(X_\et; \Lambda)$. In this case, we argue by induction on $d=\dim X$. If $d\leq 0$, then the claim is essentially obvious because either $X$ is empty or $X_\red = \sqcup_{i=1}^m \Spa(L_i, L_i^\circ)$ for some finite extensions $K\subset L_i$. So we assume that $\dim X = d>0$ and that the result is known for all spaces of dimension strictly less then $d$.
    
    Let $D'\subset D(X_\et; \Lambda)$ be the full subcategory consisting of objects $\F$ such that $\PD_f(\F)$ is an isomorphism. We wish to show that $D'$ contains $D^b_\zc(X_\et; \Lambda)$. Now note that $D'$ is a thick triangulated subcategory of $D(X_\et; \Lambda)$. Therefore, \cref{lemma:generators-dbzc} implies that it suffices to show that $\rR g_* \ud{M}_{X'} \in D'$ for a finitely generated $\Lambda$-module $M$ and a proper morphism $g\colon X' \to X$ such that $X'$ is smooth and $\dim f^{-1}(x)< \max(\dim X, 1)$ for any classical point $x\in X$. Since proper trace is compatible with compositions
    (see \cref{most general proper trace!!}\cref{most general proper trace!!-composition}),
    we see that the composition
    \begin{multline*}
    \rR f_*\circ \rR g_*\circ \DD_{X'}(\ud{M}_{X'})
    \xrightarrow{\rR f_*(\PD_{g}(\ud{M}_{X'}))}
    \rR f_*\circ \DD_X\circ \rR g_*(\ud{M}_{X'}) \xrightarrow{\PD_f(\rR g_*\ud{M}_{X'})}
    \DD_Y\circ \rR f_* \circ \rR g_*(\ud{M}_{X'})
    \end{multline*}
    is given by $\PD_{f \circ g}(\ud{M}_{X'})$. 
    Hence, we are reduced to showing that both $\PD_{g}(\ud{M}_{X'})$
    and $\PD_{f \circ g}(\ud{M}_{X'})$ are isomorphisms. The latter map is an isomorphism due to \cref{general smooth Poincare duality} and the combination of \cref{most general proper trace!!}\cref{most general proper trace!!-smooth} and \cref{compatibility for smooth source trace}\cref{compatibility for smooth source trace-1} (see also \cref{constructing smooth trace for omega}). 
    
    Therefore, we reduce the question to showing that $\PD_g(\ud{M}_{X'})$ is an isomorphism. For any classical point $x\in X$, we denote by $X'_x \coloneqq X'\times_X x$ the fiber over $x$ and by $g_x\colon X'_x \to x$ the restriction of $g$. Arguing as in \cref{step:fiberwise-reduction}, we reduce this question to showing that $\PD_{g_x}(\ud{M}_{X'_x})$ is an isomorphism for any classical point  $x\in X$. Since we chose $g$ such that $\dim X'_x<\dim X$, the induction hypothesis implies that $\PD_{g_x}(\ud{M}_{X'_x})$ is indeed an isomorphism for any classical point  $x\in X$. This finishes the proof. \qedhere
\end{enumerate}
\end{proof}

\begin{remark} 
\label{potential different proof remark}
As pointed out to us by Scholze, it seems likely that \cite[Th.~3.9.23 and 3.10.20]{Mann-thesis} can be applied toward a different proof of \cref{general Poincare duality} when $K$ is of mixed characteristic $(0, p)$ and $\Lambda = \Z/p\Z$.
Namely, using the Primitive Comparison Theorem for Zariski-constructible complexes, one deduces that the functor $-\otimes \O_X^{+,a}/p \colon \cal{D}^{\rm{oc}}_\et(X; \Z/p\Z) \hookrightarrow \cal{D}^{\varphi}_{\Box}(X; \O^+/p)^a$ commutes with proper pushforwards. Therefore, the question essentially boils down to constructing an isomorphism $\omega_X \otimes \O_X^{+,a}/p \simeq f^! \O_{\Spa(K, \O_K)}^+/p$ for any morphism $f\colon X \to \Spa(K, \O_K)$ locally of finite type. If $X$ is smooth, this follows from \cite[Th.~3.10.20]{Mann-thesis} and \cite[Th.~3.21~(1)]{BH}. If $X$ admits a Zariski-closed immersion $i\colon X \hookrightarrow Y$ into a smooth $Y$ with the structure morphisms $f_X$ and $f_Y$, respectively, one can show that
\[ \omega_X \otimes \O_X^{+,a}/p\simeq \big( \rR i^! \omega_Y\big) \otimes \O_X^{+, a}/p \simeq i^! f_Y^! \O_{\Spa(K, \O_K)}^{+, a}/p \simeq f_X^! \O_{\Spa(K, \O_K)}^{+, a}/p. \]
In particular, these isomorphisms exist locally on $X$, so the only question is how to glue them globally.
However, gluing these maps seems somewhat subtle due to the inexplicit nature of Poincar\'e duality in \cite{Mann-thesis}. 
\end{remark}

As the main application of the general form of Poincar\'e duality, we show that a version of Poincar\'e duality holds for some non-smooth and non-proper spaces. For this, we need to recall some definitions. 

\begin{definition} A rigid-analytic space $X$ over $K$ is \textit{Zariski-compactifiable} if there is a Zariski-open immersion $j\colon X \hookrightarrow \overline{X}$ such that $\overline{X}$ is proper over $K$.
\end{definition}

In order to formulate this version of duality on non-smooth spaces, we also need to recall the definition of intersection cohomology of rigid-analytic varieties due to Bhatt and Hansen.

\begin{definition}
Let $X$ be a rigid-analytic space over $K$ of equidimension $d$, let $j\colon U\hookrightarrow X$ be a smooth Zariski-open subspace, and let $\bf{L}$ be a lisse sheaf of $\Lambda$-modules on $U_\et$.
\begin{enumerate}[leftmargin=*,label=\upshape{(\roman*)}]
\item The \textit{IC sheaf $\IC_X(\bf{L})$ attached to $\bf{L}$} is the intermediate extension $\IC_X(\bf{L}) \coloneqq j_{!*}\bigl(\bf{L}[d]\bigr)$ (see \cite[Th.~4.2~(5)]{BH}). 

\item The \emph{intersection cohomology complex $\IH(X, \bf{L})$  with coefficients in $\bf{L}$} is the complex $\IH(X, \bf{L})\coloneqq \rR\Gamma\bigl(X, \IC_X(\bf{L})\bigr)$. 

\item The \emph{compactly supported intersection cohomology complex $\IH_c(X, \bf{L})$  with coefficients in $\bf{L}$} is the complex $\IH_c(X, \bf{L})\coloneqq \rR\Gamma_c\bigl(X, \IC_X(\bf{L})\bigr)$. 

\item The \emph{$i$-th intersection cohomology $\IH^i(X, \bf{L})$ with coefficients in $\bf{L}$} is the $\Lambda$-module $\IH^i(X, \bf{L}) \coloneqq \Hh^i\bigl( \IH(X, \bf{L})\bigr)$. 

\item The \emph{$i$-th compactly supported intersection cohomology $\IH^i_c(X, \bf{L})$ with coefficients in $\bf{L}$} is the $\Lambda$-module $\IH^i_c(X, \bf{L}) \coloneqq \Hh^i_c\bigl( \IH(X, \bf{L})\bigr)$. 
\end{enumerate}
\end{definition}

In order to prove the version of Poincar\'e duality mentioned above, we need the following preliminary lemma: 

\begin{lemma}\label{lemma:preservation-zariski-constructible} Let $j\colon U \to X$ be a Zariski-open immersion of rigid-analytic spaces over $K$, let $\F \in D^{(b)}_\zc(U_\et; \Lambda)$. If there is $\ov{\F}\in D^{(b)}_\zc(X_\et; \Lambda)$ such that $j^* \ov{\F} = \F$, then $\rR j_*\F\in D^{(b)}_\zc(X_\et; \Lambda)$ and $j_!\F \in D^{(b)}_\zc(X_\et; \Lambda)$.
\end{lemma}
\begin{proof}
    It clearly suffices to show that $\rR j_* j^* \ov{\F}$ and $j_! j^*\ov{\F}$ lie in $D^{(b)}_\zc(X_\et; \Lambda)$. We denote by $i\colon Z \to X$ the closed complement of $U$ (with the reduced adic space structure on it). Then the exact triangles
    \[
    i_* \rR i^! \ov{\F} \to \ov{\F} \to \rR j_* j^* \ov{\F},
    \]
    \[
    j_! j^* \ov{\F} \to \ov{\F} \to i_* i^* \ov{\F}
    \]
    imply that it suffices to show that the functors $i_*$, $i^*$, and $\rR i^!$ preserve locally bounded Zariski-constructible complexes. The claim is evident for the first two functors, and \cite[Cor.~3.12]{BH} implies the claim for the third functor above.
\end{proof}

\begin{theorem}
\label{intersection cohomology duality} 
Let $\ov{X}$ be a proper rigid-analytic space over $K$, let $U\subset X\subset \ov{X}$ be two rigid-analytic subspaces which are both Zariski-open in $\ov{X}$,\footnote{
We note that this is stronger than requiring the inclusions $U \subset X$ and $X \subset \ov{X}$ to be Zariski-open. For example, $\bigl(\AA^{1,\an}_K \smallsetminus \{(1/p)^\NN\}\bigr) \subset \AA^{1,\an}_K$ and $\AA^{1,\an}_K \subset \PP^{1,\an}_K$ are Zariski-open, but $\bigl(\AA^{1,\an}_K \smallsetminus \{(1/p)^\NN\}\bigr) \subset \PP^{1,\an}_K$ is not because any Zariski-closed subset of $\PP^{1,\an}_K$ is either finite or $\PP^{1,\an}_K$ by rigid GAGA.}
let $\bf{L}$ be a local system of finite free $\Lambda$-modules on $U_\et$, and let $\bf{L}^{\vee}$ be its $\Lambda$-linear dual. Assume that $U$ is smooth of equidimension $d$ and set $C\coloneqq \wdh{\overline{K}}$. Then $\IH_c(X_C, \LL)$ and $\IH(X_C, \LL^\vee(d))$ lie in $D^b_{\coh}(\Lambda)$ and there is a Galois-equivariant isomorphism
\begin{equation}\label{eqn:duality-intersection-cohomology}
\rR \Hom_\Lambda\bigl(\IH_c(X_C, \bf{L}), \Lambda\bigr) \simeq \IH\bigl(X_C, \bf{L}^{\vee}(d)\bigr)
\end{equation}
which is functorial in $\bf{L}$. In particular, for any integer $i$, there is a Galois-equivariant isomorphism
\[ \IH^{-i}_c(X_C, \bf{L})^\vee \simeq \IH^{i}(X_C, \bf{L}^\vee)(d). \]
\end{theorem}
\begin{proof}
    First, we denote by $j \colon X \to \ov{X}$ the natural open immersion and by $f\colon X \to \Spa(K,\O_K)$ and $\ov{f} \colon \ov{X} \to \Spa(K, \O_K)$ the structure morphisms. Then we note that  \cite[Th.~4.2~(5)]{BH} implies that $\IC_X(\bf{L})$ lies in $D^b_\zc(X_\et; \Lambda)$, $\IC_{\ov{X}}(\bf{L})$ lies in $D^b_\zc(\ov{X}_\et; \Lambda)$, and $\ov{j}_X^*\IC_{\ov{X}}(\bf{L}) \simeq \IC_X(\bf{L})$. Furthermore, \cref{lemma:preservation-zariski-constructible} guarantees that $j_! \IC_X(\bf{L})$ lies in $D^b_{\zc}(\ov{X}_\et; \Lambda)$, so \cite[Th.~3.10]{BH} implies that $\IH_c(X, \bf{L}) \in D^b_{\coh}(\Lambda)$. Then we have the following sequence of isomorphisms:
    \begin{multline}\label{eqn:duality-intersection-over-K}
        \DD_{\Spa(K, \O_K)} \bigl(\rR f_! \IC_X(\bf{L})\bigr)  \simeq \DD_{\Spa(K, \O_K)}\bigl(\rR \ov{f}_* j_! \IC_X(\bf{L})\bigr) \\
         \simeq \rR \ov{f}_* \DD_{\ov{X}}\bigl( j_! \IC_X(\bf{L}) \bigr) 
         \simeq \rR \ov{f}_* \rR j_* \DD_{X}\bigl(\IC_X(\bf{L})\bigr)  
         \simeq \rR f_* \IC_X\bigl(\bf{L}^\vee(d) \bigr),
    \end{multline}
        where the first isomorphism follows from the formula $\rR f_! \simeq \rR \ov{f}_* \circ j_!$, the second isomorphism follows from \cref{general Poincare duality} and the observation that $j_! \IC_X(\bf{L}) \in D^b_\zc(\ov{X}; \Lambda)$, the third isomorphism follows from \cite[Th.~3.21~(5)]{BH}, and the last isomorphism follows from the formula $\rR f_* \simeq \rR \ov{f}_* \circ \rR j_*$, \cite[Th.~4.2~(5)]{BH} and the assumption that $U$ is smooth over $K$. Then the isomorphism \cref{eqn:duality-intersection-cohomology} follows directly \cref{eqn:duality-intersection-over-K} by passing to (derived) global section over $\Spa(C, \O_C)$. Likewise, we immediately conclude that $\IH(X_C, \bf{L}^\vee(d))$ lies in $D^b_\coh(\Lambda)$. The last assertion follows directly from \cref{eqn:duality-intersection-cohomology} and the observation that $\Lambda=\Z/n\Z$ is an injective $\Lambda$-module. 
\end{proof}

\begin{remark} The condition that $U \subset \ov{X}$ is Zariski-open is automatically satisfied if $U = X^\sm$ is the smooth locus of $X$. To justify this, we first observe that $\ov{X}^\sm \subset \ov{X}$ is Zariski-open because the smooth locus is always Zariski-open. Therefore, $X^\sm = \ov{X}^\sm \cap X \subset \ov{X}$ is Zariski-open as an intersection of two Zariski-open subspaces. In particular, \cref{intersection cohomology duality} proves Poincar\'e duality for intersection cohomology (with constant coefficients) for any Zariski-compactifiable $X$.  
\end{remark}

Finally, we deduce a usual version of Poincar\'e duality for local systems on smooth Zariski-compactifiable rigid-analytic spaces:

\begin{corollary}\label{cohomology duality} Let $U$ be a smooth Zariski-compactifiable rigid-analytic space over $K$ of equidimension $d$. Let $\bf{L}$ be a local system of finite free $\Lambda$-modules on $U_\et$, let $\bf{L}^{\vee}$ be its $\Lambda$-linear dual, and let $\mathrm{ev}\colon \bf{L} \otimes_{\Lambda} \bf{L}^{\vee} \to \ud{\Lambda}_U$ be the natural evaluation map.
Set $C\coloneqq \widehat{\overline{K}}$. 
Then $\rR \Gamma_c(U_C, \bf{L})$ and $\rR\Gamma(U_C, \bf{L}^\vee(d)[2d])$ lie in $D^b_{\coh}(\Lambda)$ and the Galois-equivariant pairing
\[
\rR \Gamma_c(U_C, \bf{L}) \otimes^L_{\Lambda} \rR \Gamma(U_C, \bf{L}^{\vee}(d)[2d]) \xr{\cup} \rR \Gamma_c(U_C, \bf{L} \otimes \bf{L}^{\vee}(d)[2d]) \xr{\rR \Gamma_c(U_C, \mathrm{ev}(d)[2d])} \rR \Gamma_c(U_C, \Lambda(d)[2d]) \xr{\ttr_{f}} \Lambda
\]
is perfect (in the derived sense), where $\ttr_f$ is the smooth trace from \cref{smooth-trace-constant}. 
\end{corollary}
Note that for $\ZZ_p$-local systems and discretely valued $K$, a rational duality statement was obtained in \cite[Th.~1.3]{LLZ}.
\begin{proof}
    The first conclusion of \cref{intersection cohomology duality} implies that both $\rR \Gamma_c(U_C, \bf{L})$ and $\rR \Gamma(U_C, \bf{L}^{\vee}(d)[2d])$ lie in $D^b_{\rm{coh}}(\Lambda)$. Therefore, it suffices to show that the natural morphism
    \begin{equation}\label{eqn:non-proper-duality}
    \rR\Gamma(U_C, \bf{L}^{\vee}(d)[2d]) \to \rR\Hom_\Lambda( \rR \Gamma_c(U_C, \bf{L}), \Lambda)
    \end{equation}
    is an isomorphism. This follows directly from \cref{intersection cohomology duality}. %
\end{proof}

\begin{appendices}

\section{Universal compactifications}
\label{univ-comp}

\begin{theorem}[{\cite[Th.~5.1.5, Cor.~5.1.6]{Huber-etale}}]
\label{compactifications exist}
    Let $X$ be a separated, $+$-weakly finite type adic space over $\Spa(C,\cO_C)$.
    Then there exists a proper adic space $X^c$ over $\Spa(C,\cO_C)$ and an open embedding $j \colon X \hookrightarrow X^c$ with the following universal property:
    if $Y$ is a proper adic space over $\Spa(C,\cO_C)$ and $j' \colon X \hookrightarrow Y$ an open embedding, then there exists a unique morphism $f \colon X^c \to Y$ such that $j' = f \circ j$.
    Moreover every point of $X^c$ is a specialization of a point of $j(X)$,
    and $\cO_{X^{c}} \to j_*(\cO_X)$ is an isomorphism of sheaves of topological rings.
\end{theorem}
\begin{definition}
We call an embedding $X \hookrightarrow X^c$ satisfying the conclusion of \cref{compactifications exist} the \emph{universal compactification} of $X$.
\end{definition}

In the affinoid case, \cref{compactifications exist} specializes to the following lemma.
While the statement is implicit in the proof of \cite[Th.~5.1.5]{Huber-etale}, we recall the argument here for the convenience of the reader.
\begin{lemma}\label{lemma:universal-compactification-affinoid}
Let $C$ be an algebraically closed nonarchimedean field, and $f\colon \Spa(A, A^+)\to \Spa(C, \cO_C)$ be a finite type morphism. Then its universal compactification is equal to
\[ \begin{tikzcd}[column sep=tiny]
\Spa(A,A^+) \arrow[rr,hook,"j"] \arrow[rd, swap, "f"] && \Spa(A,\cO_C[A^{\circ\circ}]^+) \arrow[ld,"f^c"] \\
& \Spa(C,\cO_C), &
\end{tikzcd} \]
where $\cO_C[A^{\circ\circ}]^+$ denotes the integral closure of the smallest subring of $A$ that contains $\cO_C$ and $A^{\circ\circ}$.
\end{lemma}
\begin{proof}
    Set $X \colonequals \Spa(A,A^+)$ and $X^c \colonequals \Spa(A, \cO_C[A^{\circ\circ}]^+)$.
    Let $j \colon X \to X^c$ be the map induced by the identity map on $A$.
    Since $(A,A^+)$ is of topologically finite type, hence $+$-weakly finite type (in the sense of \cite[Def.~1.2.1]{Huber-etale}) over $(C,\cO_C)$, there exists a finite set $E \subset A^+$ such that $A^+$ is the smallest integrally closed subring of $A$ which contains $\cO_C[A^{\circ\circ}]^+$ and $E$.
    Therefore, $j$ can be identified with the inclusion of the rational subset $\{ \abs{E} \le 1 \}$ and is in particular an open embedding.
    
    To show that $f^c \colon X^c \to \Spa(C, \cO_C)$ is proper, we first note that $f^c$ is of $+$-weakly finite type and $X^c$ is spectral, hence quasicompact and quasiseparated.
    Thus, we can invoke the valuative criterion for properness \cite[Lem.~1.3.10]{Huber-etale}:
    for any nonarchimedean field $K$ over $C$ and any open and bounded valuation subring $\cO_C \subset K^+ \subset K$, every diagram
    \[ \begin{tikzcd}
    \Spa(K,\cO_K) \arrow[r] \arrow[d] & X^c = \Spa(A, \cO_C[A^{\circ\circ}]^+) \arrow[d,"f^c"] \\
    \Spa(K,K^+) \arrow[r] \arrow[ur,dashed] & \Spa(C,\cO_C)
    \end{tikzcd} \]
    admits a lift as indicated by the dashed arrow because the map $A \to K$ which determines the top row is continuous and thus sends $A^{\circ\circ}$ into $K^{\circ\circ} \subset K^+$.
    
    Lastly, every point of $x \in X^c$ has a generization $y \in X^c$ corresponding to a valuation of rank $1$ \cite[Lem.~1.1.10~ii)]{Huber-etale}.
    Then necessarily $y(a) \le 1$ for all $a \in A^\circ$, so $y \in X$.
    Since the rational structure sheaf does not depend on the ring of integral elements, the natural map $\cO_{X^c} \to j_*\cO_X$ is an isomorphism of sheaves of topological rings.
    Therefore, $X^c$ is a universal compactification of $X$ by \cite[Lem.~5.1.7]{Huber-etale}.
\end{proof}

\begin{lemma}
\label{lemma:extra-points-higher-rank}
Let $X$ be a separated taut $C$-rigid space, then any $x \in \abs{X^c} \smallsetminus \abs{X}$
has rank $> 1$.
\end{lemma}

\begin{proof}
According to \cref{compactifications exist}, every point in $X^c$ is a specialization
of a point in $X$, in particular it admits generalization, hence of rank $> 1$.
\end{proof}

\section{Pseudo-adic spaces}
\label{section:pseudo-adic}

One of the major subtleties while working with adic spaces is that a (locally) closed subset of an analytic adic spaces is rarely an adic space itself. In fact, a higher rank closed point $x\in X$ of an analytic adic space $X$ \emph{never} admits a structure of an adic space.  

In order to circumvent this issue, Huber has defined the notion of a pseudo-adic space and its \'etale topos.
This theory has not been widely used beyond his book \cite{Huber-etale}, where this notion does play a crucial role to make many arguments work. This theory is also quite crucial for the results of this paper\footnote{Though, we need only a mild part of it.}, so we have decided to recall the main definitions and constructions from this theory in the Appendix. 

\subsection{Basic definitions}

\begin{definition}\label{defn:pseuado-adic-space} A \emph{(strongly) pseudo-adic space} $X$ is a pair $X=(\ud{X}, \abs{X})$ consisting of an adic space $\ud{X}$ and a closed subset $\abs{X} \subset \ud{X}$. A morphism of (strongly) pseudo-adic spaces $f\colon X=(\ud{X}, \abs{X}) \to Y=(\ud{Y}, \abs{Y})$ is a morphism of adic spaces $f\colon \ud{X}\to \ud{Y}$ such that $f(\abs{X}) \subset \abs{Y}$. 
\end{definition}

\begin{remark} Huber defines a more general notion of a pseudo-adic space in \cite[Def.~1.10.3]{Huber-etale}. This level generality is convenient to set up foundations. However, we will never need this level of generality in this paper, so we do not discuss it. We include the word ``strongly'' to emphasize that our definition of pseudo-adic spaces is stronger than the definition given by Huber. 
\end{remark}

\begin{lemma} Let $X=(\ud{X}, \abs{X})$ be a (strongly) pseudo-adic space. Then $X$ is a pseudo-adic space in the sense of \cite[Def.~1.10.3]{Huber-etale}.
\end{lemma}
\begin{proof}
    We need to show that $\abs{X}\subset \ud{X}$ is locally pro-constructible and convex (see \cite[(1.1.3)]{Huber-etale}). We note that closed subspaces of locally spectral spaces are closed under specialization, so convexity of $\abs{X}$ inside $\ud{X}$ is clear. Now we show a stronger claim that any closed subset of locally spectral space is pro-constructible. This is a local statement, so we can assume that $X$ is spectral, then this is \cite[Prop.~3.23(i)]{wedhorn}.
\end{proof}

We give two examples of pseudo-adic spaces that will be important for this paper:

\begin{example}\label{example:main-examples} 
\begin{enumerate}[label=(\arabic*),leftmargin=*]
\item (Closed points) Let $x\in X$ be a closed point of an adic space $X$. Then $(X, \{x\})$ is a (strongly) pseudo-adic space. We will usually denote it simply by $\{x\}$. 
\item (Closed pro-special subsets) Let $X=\Spa(A, A^+)$ be an affinoid adic space (with possibly non-complete $(A, A^+)$) and $\{f_i\}_{i\in I}$ a set of functions in $A^+$. Then
\[
X\left(\abs{f_i}<1\right) = \{x\in X \suchthat \abs{f_i(x)}<1 \text{ for } i\in I\}
\]
is a closed subspace of $X$. So $(X, X\left(\abs{f_i}<1\right))$ is a (strongly) pseudo-adic space.
\end{enumerate} 
\end{example}

\begin{remark}\label{rmk:closed-point-pro-special} Let $(k, k^+)$ be an affinoid field
and $s\in \Spa(k, k^+)$ a closed point. Then $\{s\}\in\Spa(k, k^+)$ is an example of a closed pro-special subset with the set of function $\{f_i\}_{i\in I}$ equal to the set of elements of the maximal ideal $\fm\subset k^+$.
\end{remark}

Now we wish to discuss the notion of an \'etale topos of a (strongly) pseudo-adic space $(X, Z)$. Let $U\coloneqq X\smallsetminus Z$ be the open complement of $Z$ in $X$. Then the \'etale topos $U_\et$ is identified with the slice topos $(X_{\et})_{/h_U}$ and the natural morphism $U_\et \to X_\et$ is fully faithful. Therefore, $U_\et$ is an open subtopos of $X_\et$ in the sense of \cite[\href{https://stacks.math.columbia.edu/tag/08LX}{Tag 08LX}]{stacks-project}. 

\begin{definition}\label{defn:etale-topos-pseudo-adic}
    The \emph{\'etale topos} $(X, Z)_\et$ of a (strongly) pseudo-adic space $(X, Z)$ is the closed subtopos of $X_\et$ obtained as the closed complement of an open subtopos $U_\et \subset X_\et$ (see \cite[\href{https://stacks.math.columbia.edu/tag/08LZ}{Tag 08LZ}]{stacks-project} and \cite[\href{https://stacks.math.columbia.edu/tag/08LZ}{Tag 08LZ}]{stacks-project}). 
\end{definition}

\begin{remark} Explicitly, $(X, Z)_\et$ is a full subcategory of $X_\et$ that consists of sheaves $\F\in X_\et$ such that the natural morphism $\F\times h_U\to h_U$ is an isomorphism. 
\end{remark}

\begin{remark} The \'etale topos of a (strongly) pseudo-adic space is functorial with respect to morphism of (strongly) pseudo-adic spaces. 
\end{remark}

\begin{remark}\label{rmk:excision} Let $j\colon U \to X$ be an open immersion with the complement $i\colon Z\to X$. The \'etale topos $(X, Z)_\et$ comes with a morphism of topoi $i\colon (X, Z)_\et \to X_\et$ by construction. It is essentially formal\footnote{For example, it can be deduced from \cite[Exp.~IV, Prop.~9.4.1]{SGA4}.} that the sequence
\[
0 \to j_!j^*\F \to \F \to i_*i^*\F \to 0
\]
is exact for any sheaf of abelian groups $\F\in X_\et$. 
\end{remark}

\begin{lemma}\label{lemma:Huber-definition-vs-ours} Let $(X, Z)$ be a (strongly) pseudo-adic space. Then $(X, Z)_\et$ defined in \cref{defn:etale-topos-pseudo-adic} coincides with the definition of the \'etale topos of a pseudo-adic space from \cite[Def.~2.3.1]{Huber-etale}.
\end{lemma}
\begin{proof}
    For the purpose of this proof, we denote by $(X, Z)_\et^{\rm{H}}$ the topos defined in \cite[Def.~2.3.1]{Huber-etale}. Then \cite[Rmk.~2.3.4~(i)]{Huber-etale} gives a particularly nice site $(X, Z)_{\et. \rm{w}}$ defining $(X, Z)_\et^{\rm{H}}$:
    \begin{enumerate}
        \item the underlying category of $(X, Z)_\et^{\rm{H}}$ is the category $\text{\'Et}/X$ of adic spaces \'etale over $X$;
        \item a family $\{f_i\colon Y_i \to Y\}_{i\in I}$ of morphisms in $\text{\'Et}/X$ is a covering if $h^{-1}(S)\subset \cup_{i\in I} f_i(Y_i)$ where $h\colon Y\to X$ is the structure morphism.
    \end{enumerate}
    Now the proof of \cite[\href{https://stacks.math.columbia.edu/tag/08LY}{Tag 08LY}]{stacks-project} implies that the same site defines the closed subtopos $(X, Z)_\et$. This finishes the proof.  
\end{proof}

\begin{remark} The definition of an \'etale topos of a (strongly) pseudo-adic space $(X, Z)$ depends on the ambient space $X$. However, \cite[Cor.~2.3.8]{Huber-etale} shows that it is independent of $X$ in some precise sense.
\end{remark} 

\begin{lemma}\label{lemma:pre-adic-dijoint-union}
Let $X=\bigl(\ud{X},\abs{X}\bigr)$ be a (strongly) pseudo-adic space, and let $\abs{X}=\abs{X_1}\sqcup \abs{X_2}$ is a disjoint union of two closed subsets $X_1$ and $X_2$. Then there is a canonical equivalence 
\[
\bigl(\ud{X}, \abs{X}\bigr)_\et\simeq \bigl(\ud{X}\sqcup \ud{X}, \abs{X_1} \sqcup \abs{X_2}\bigr)_\et\simeq \bigl(\ud{X}, \abs{X_1}\bigr)_\et \times \bigl(\ud{X}, \abs{X_2}\bigr)_\et.
\]
\end{lemma}
\begin{proof}
For the first isomorphism, it suffices to show that the natural morphism of pseudo-adic spaces $\bigl(\ud{X}\sqcup \ud{X}, \abs{X_1} \sqcup \abs{X_2}\bigr) \xrightarrow{\id \sqcup \id} \bigl(\ud{X}, \abs{X}\bigr)$ induces an equivalence on the associated \'etale topoi.
This follows from \cite[Prop.~2.3.7]{Huber-etale}. The second isomorphism follows from the following sequence of isomorphisms 
\[
\bigl(\ud{X}\sqcup \ud{X}, \abs{X_1} \sqcup \abs{X_2} \bigr)_\et \simeq \Bigl(\bigl(\ud{X}, \abs{X_1}\bigr) \sqcup \bigl(\ud{X}, \abs{X_2}\bigr)\Bigr)_\et \simeq \bigl(\ud{X}, \abs{X_1}\bigr)_\et \times \bigl(\ud{X}, \abs{X_2}\bigr)_\et. \qedhere
\]
\end{proof}

\subsection{\'Etale topos of a closed point}

The main goal of this subsection is to give an explicit characterization of the \'etale topos of a pseudo-adic space $(X, x)$ for $x\in X$ a closed point.  For the rest of the subsection, we fix a locally noetherian analytic adic space $X$. %

Let $x\in X$ be a closed point of an analytic locally noetherian adic space $X$; we wish to understand the cohomology groups of the pseudo-adic space $\{x\}=(X, x)$. For this, we define $(K_x, K_x^+)$ to be either $\Bigl(k(x)^\h, k(x)^{+, \h}\Bigr)$ or $\Bigl(\wdh{k(x)}^\h, \wdh{k(x)}^{+, \h}\Bigr)$ (see \cref{defn:residue-fields-2}), and $s$ to be the unique closed point of $\Spa(K_x, K_x^+)$. Then the morphism of pseudo-adic space $\left(\Spa(K_x, K_x^+), s\right) \to (X, x)$
induces a morphism of topoi
\[
b\colon \left(\Spa(K_x, K_x^+), \{s\}\right)_\et \to (X, x)_\et.
\]
The universal property of affine scheme (see \cite[\href{https://stacks.math.columbia.edu/tag/01I1}{Tag 01I1}]{stacks-project}) gives us a canonical morphism $\Spa(K_x, K_x^+) \to \Spec K_x^+$ that can be easily extended to a morphism of \'etale topoi
\[
a\colon \left(\Spa(K_x, K_x^+), \{s\}\right)_\et \to (\Spec K_x^+)_\et
\]

\begin{theorem}
\label{thm:topos-of-a-point} In the notation as above, both $a$ and $b$ are equivalences of topoi. In particular, 
\[
    \gamma=a\circ b^{-1}\colon (X, x)_\et \to (\Spec K_x)_\et
\]
is an equivalence of topoi. In particular, there are canonical isomorphisms
\[
\rm{R}\Gamma\left(\{x\}, \mu_n\right) \simeq \rm{R}\Gamma\bigl(\Spec k(x)^{\rm{h}}, \mu_n\bigr) \simeq \rm{R}\Gamma\Bigl(\Spec \wdh{k(x)}^{\rm{h}}, \mu_n\Bigr).
\]
for every integer $n$ invertible in $\O_X$.
\end{theorem}
\begin{proof}
    The first part follows from the proof of \cite[Prop.~2.3.10]{Huber-etale}. The second part is a formal consequence of the first part.
\end{proof}

\begin{warning}
    The result of \cref{thm:topos-of-a-point} is \emph{false} if we put $K_x=k(x)$ or $K_x=\wdh{k(x)}$. The (implicit) henselian assumption on $K_x$ is essential for the proof. 
\end{warning}

\subsection{Cohomology of closed pro-special subsets}

The main goal of this subsection is to understand cohomology groups of closed pro-special pseudo-adic spaces (see \cref{example:main-examples}). Unlike the case of a closed point, we will not be able to describe the whole \'etale topos of this pseudo-adic space.  

In what follows, we fix a (possibly non-complete) Tate--Huber pair $(A, A^+)$ with a pseudo-uniformizer $\varpi\in A^+$ and a set of elements $\{f_i\in A^+\}_{i\in I}$. We define $X\coloneqq \Spa(A, A^+)$ and a closed subspace $Z=X\left(\abs{f_i}<1\right) \subset X$. Then $(X, Z)$ is a pseudo-adic space. 

\begin{definition}
\label{defn:henselize-prospecial} 
We define the \emph{henselization of $A$ along $Z$} to be the ring 
\[
A(Z)\coloneqq (A^+)^{\rm{h}}_I\bigl[\tfrac{1}{\varpi}\bigr],
\]
where the henselization is taken with respect to the ideal $I=(f_i, \varpi)_{i\in I}\subset A^+$.
\end{definition}

\begin{remark} The ring $A(Z)$ is easily seen to be independent of the choice of a pseudo-uniformizer $\varpi$. A much harder result is that the ring $A(Z)$ is also independent of the choice of generators $\{f_i\}$ and is intrinsic to the pro-special set $Z$. We refer to \cite[Prop and Def.~3.1.12]{Huber-etale} for a proof of this result. 
\end{remark}

\begin{example}\label{example:different-henselizations} Let $(k, k^+)$ be an affinoid field, and $s\in X=\Spa(k, k^+)$ the closed point considered as a closed pro-special subset (see \cref{rmk:closed-point-pro-special}). 
Then $k(\{s\})$ from \cref{defn:henselize-prospecial} coincides with the henselized residue field $k^\h$ in the sense of \cref{defn:henselize-field}.  
\end{example}

\begin{example} If $Z=\varnothing$, we denote $A(\varnothing)$ by $A^\h$.
\end{example}

The main result of \cite[\S~3]{Huber-etale} says that the \emph{algebraic} cohomology of $\Spec A(Z)$ coincide with the \emph{analytic} cohomology of the pseudo-adic space $(X, Z)$. 

\begin{theorem}
\label{thm:cohomology-prospecial-subsets} 
Let $(A, A^+)$ be a strongly noetherian (possibly not complete) Tate--Huber pair, and $Z\subset X=\Spa(A, A^+)$ a closed pro-special subset (see \cref{example:main-examples}). Then there is a morphism of topoi 
\[
\gamma\colon (X, Z)_\et \to \bigl(\Spec A(Z)\bigr)_\et
\]
such that 
\begin{enumerate}[label=\upshape{(\arabic*)}]
    \item for each $n$ invertible in $A$, the natural morphism
    \[
    \rm{R}\Gamma(\Spec A(Z), \mu_n) \to \rm{R}\Gamma\bigl((X, Z)_\et, \mu_n\bigr)
    \]
    is an isomorphism;
    \item the morphism $\gamma$ is functorial in $(X, Z)$;
    \item if $(A, A^+)=(k, k^+)$ is an affinoid field and $Z\subset X$ is a closed point, then $\gamma$ coincides with the morphism $c$ constructed in \cref{thm:topos-of-a-point};
\end{enumerate}
\end{theorem}
In particular, one gets a functorial isomorphism $\rm{R}\Gamma(\Spa(A, A^+), \mu_n)\simeq \rm{R}\Gamma(\Spec A^\h, \mu_n)$ by putting $Z=\varnothing$. 
\begin{proof}
    This is essentially \cite[Th.~3.2.9]{Huber-etale}. Unfortunately, these properties (and the existence of a topos-theoretic morphism $\gamma$) is not explicitly stated in  \cite[Th.~3.2.9]{Huber-etale}, but it does follow from the proof. The reader willing to verify these properties should read \cite[Rmk.~3.2.10 and \S~3.3 and 3.4]{Huber-etale}. We especially refer to the proof of \cite[Th.~3.3.3]{Huber-etale} and the discussion on \cite[p.~194]{Huber-etale} for the construction of the morphism $\gamma$. 
\end{proof}

\end{appendices}

\providecommand{\bysame}{\leavevmode\hbox to3em{\hrulefill}\thinspace}
\providecommand{\MR}{\relax\ifhmode\unskip\space\fi MR }
\providecommand{\MRhref}[2]{
  \href{http://www.ams.org/mathscinet-getitem?mr=#1}{#2}
}
\providecommand{\href}[2]{#2}


\begin{thebibliography}{BKKN67}

\bibitem[AGV22]{rigid-motives}
Joseph Ayoub, Martin Gallauer, and Alberto Vezzani, \emph{The six-functor
  formalism for rigid analytic motives}, Forum Math. Sigma \textbf{10} (2022),
  Paper No. e61, 182.

\bibitem[ALY22]{ALY}
Piotr Achinger, Marcin Lara, and Alex Youcis, \emph{Specialization for the
  pro-\'{e}tale fundamental group}, Compos. Math. \textbf{158} (2022), no.~8,
  1713--1745. \MR{4490930}

\bibitem[And74]{Andre}
Michel Andr\'{e}, \emph{Localisation de la lissit\'{e} formelle}, Manuscripta
  Math. \textbf{13} (1974), 297--307.

\bibitem[BBDG18]{BBDG}
Alexander Beilinson, Joseph Bernstein, Pierre Deligne, and Ofer Gabber,
  \emph{Faisceaux pervers}, Analysis and topology on singular spaces, {I}
  ({L}uminy, 1981), Ast\'erisque, vol. 100, Soc. Math. France, Paris, 2nd ed.,
  2018, pp.~5--171. \MR{751966}

\bibitem[Ber93]{Berkovich}
Vladimir~G. Berkovich, \emph{\'{E}tale cohomology for non-{A}rchimedean
  analytic spaces}, Inst. Hautes \'{E}tudes Sci. Publ. Math. (1993), no.~78,
  5--161 (1994).

\bibitem[BGR84]{BGR}
Siegfried Bosch, Ulrich G\"{u}ntzer, and Reinhold Remmert,
  \emph{Non-{A}rchimedean analysis}, Grundlehren der mathematischen
  Wissenschaften, vol. 261, Springer-Verlag, Berlin, 1984, A systematic
  approach to rigid analytic geometry. \MR{746961}

\bibitem[BH22]{BH}
Bhargav Bhatt and David Hansen, \emph{The six functors for
  {Z}ariski-constructible sheaves in rigid geometry}, Compos. Math.
  \textbf{158} (2022), no.~2, 437--482. \MR{4413751}

\bibitem[Bha17]{Bhatt-notes}
Bhargav Bhatt, \emph{Lecture notes for a class on perfectoid spaces}, available
  at \url{https://www.math.ias.edu/~bhatt/teaching/mat679w17/lectures.pdf},
  2017.

\bibitem[BKKN67]{Differentialrechnung}
Robert Berger, Reinhardt Kiehl, Ernst Kunz, and Hans-Joachim Nastold,
  \emph{Differentialrechnung in der analytischen {G}eometrie}, Lecture Notes in
  Mathematics, No. 38, Springer-Verlag, Berlin-New York, 1967. \MR{224870}

\bibitem[BL93]{BL1}
Siegfried Bosch and Werner L{\"u}tkebohmert, \emph{Formal and rigid geometry.
  {I}. {R}igid spaces}, Math. Ann. \textbf{295} (1993), no.~2, 291--317.
  \MR{1202394}

\bibitem[Bos14]{B}
Siegfried Bosch, \emph{Lectures on formal and rigid geometry}, Lecture Notes in
  Mathematics, vol. 2105, Springer, Cham, 2014.

\bibitem[Bou98]{Bourbaki}
Nicolas Bourbaki, \emph{Commutative algebra. {C}hapters 1--7}, Elements of
  Mathematics (Berlin), Springer-Verlag, Berlin, 1998, Translated from the
  French, Reprint of the 1989 English translation. \MR{1727221}

\bibitem[Bou03]{Bourbaki-algebra-4-7}
\bysame, \emph{Algebra {II}. {C}hapters 4--7}, {E}nglish ed., Elements of
  Mathematics (Berlin), Springer-Verlag, Berlin, 2003. \MR{1994218}

\bibitem[BV18]{Buzzard-Verberkmoes}
Kevin Buzzard and Alain Verberkmoes, \emph{Stably uniform affinoids are
  sheafy}, J. Reine Angew. Math. \textbf{740} (2018), 25--39. \MR{3824781}

\bibitem[CGN23]{CGN}
Pierre Colmez, Sally Gilles, and Wies{\l}awa Nizio{\l}, \emph{Arithmetic
  duality for {$p$}-adic pro-{\'e}tale cohomology of analytic curves},
  Preprint, available at \url{https://arxiv.org/abs/2308.07712}, 2023.

\bibitem[Con99]{Conrad99}
Brian Conrad, \emph{Irreducible components of rigid spaces}, Ann. Inst. Fourier
  (Grenoble) \textbf{49} (1999), no.~2, 473--541. \MR{1697371}

\bibitem[Con15]{Seminar}
\bysame, \emph{Notes for a learning seminar on perfectoid spaces}, available at
  \url{http://virtualmath1.stanford.edu/~conrad/Perfseminar/}, 2014--2015.

\bibitem[dJvdP96]{dJ-vdP}
Johan de~Jong and Marius van~der Put, \emph{\'{E}tale cohomology of rigid
  analytic spaces}, Doc. Math. \textbf{1} (1996), No. 01, 1--56. \MR{1386046}

\bibitem[Duc18]{Ducros}
Antoine Ducros, \emph{Families of {B}erkovich spaces}, Ast\'{e}risque (2018),
  no.~400, vii+262. \MR{3826929}

\bibitem[Elk73]{Elkik}
Ren\'{e}e Elkik, \emph{Solutions d'\'{e}quations \`a coefficients dans un
  anneau hens\'{e}lien}, Ann. Sci. \'{E}cole Norm. Sup. (4) \textbf{6} (1973),
  553--603 (1974). \MR{345966}

\bibitem[EP05]{ValuedFields}
Antonio~J. Engler and Alexander Prestel, \emph{Valued fields}, Springer
  Monographs in Mathematics, Springer-Verlag, Berlin, 2005. \MR{2183496}

\bibitem[Fal02]{Faltings-almost-extensions}
Gerd Faltings, \emph{Almost \'etale extensions}, Cohomologies $p$-adiques et
  applications arithm\'etiques, II, no. 279, Soc. Math. France, Paris, 2002,
  pp.~185--270. \MR{1922831}

\bibitem[FGK11]{FGK}
Kazuhiro Fujiwara, Ofer Gabber, and Fumiharu Kato, \emph{On {H}ausdorff
  completions of commutative rings in rigid geometry}, J. Algebra \textbf{332}
  (2011), 293--321.

\bibitem[FK18]{FujKato}
Kazuhiro Fujiwara and Fumiharu Kato, \emph{Foundations of rigid geometry. {I}},
  EMS Monographs in Mathematics, European Mathematical Society (EMS),
  Z\"{u}rich, 2018. \MR{3752648}

\bibitem[FM86]{FM86}
Jean Fresnel and Michel Matignon, \emph{Sur les espaces analytiques
  quasi-compacts de dimension {$1$} sur un corps valu\'{e} complet
  ultram\'{e}trique}, Ann. Mat. Pura Appl. (4) \textbf{145} (1986), 159--210.
  \MR{886711}

\bibitem[Fuj02]{Fujiwara-purity}
Kazuhiro Fujiwara, \emph{A proof of the absolute purity conjecture (after
  {G}abber)}, Algebraic geometry 2000, {A}zumino ({H}otaka), Adv. Stud. Pure
  Math., vol.~36, Math. Soc. Japan, Tokyo, 2002, pp.~153--183. \MR{1971516}

\bibitem[FvdP04]{FvdP04}
Jean Fresnel and Marius van~der Put, \emph{Rigid analytic geometry and its
  applications}, Progress in Mathematics, vol. 218, Birkh\"{a}user Boston,
  Inc., Boston, MA, 2004. \MR{2014891}

\bibitem[GL21]{Guo-Li}
Haoyang Guo and Shizhang Li, \emph{Period sheaves via derived de {R}ham
  cohomology}, Compos. Math. \textbf{157} (2021), no.~11, 2377--2406.
  \MR{4323988}

\bibitem[GR03]{GR}
Ofer Gabber and Lorenzo Ramero, \emph{Almost ring theory}, Lecture Notes in
  Mathematics, vol. 1800, Springer-Verlag, Berlin, 2003. \MR{2004652}

\bibitem[Gro65]{EGA4_2}
Alexander Grothendieck, \emph{\'{E}l\'{e}ments de g\'{e}om\'{e}trie
  alg\'{e}brique. {IV}. \'{E}tude locale des sch\'{e}mas et des morphismes de
  sch\'{e}mas. {II}}, Inst. Hautes \'{E}tudes Sci. Publ. Math. (1965), no.~24,
  231.

\bibitem[Hub93a]{Huber-thesis}
Roland Huber, \emph{Bewertungsspektrum und rigide {G}eometrie}, Regensburger
  Mathematische Schriften, vol.~23, Universit\"{a}t Regensburg, Fachbereich
  Mathematik, Regensburg, 1993. \MR{1255978}

\bibitem[Hub93b]{H0}
\bysame, \emph{Continuous valuations}, Math. Z. \textbf{212} (1993), no.~3,
  455--477. \MR{1207303}

\bibitem[Hub94]{Huber-2}
\bysame, \emph{A generalization of formal schemes and rigid analytic
  varieties}, Math. Z. \textbf{217} (1994), no.~4, 513--551. \MR{1306024}

\bibitem[Hub96]{Huber-etale}
\bysame, \emph{\'{E}tale cohomology of rigid analytic varieties and adic
  spaces}, Aspects of Mathematics, E30, Friedr. Vieweg \& Sohn, Braunschweig,
  1996. \MR{1734903}

\bibitem[Hub01]{Swan}
\bysame, \emph{Swan representations associated with rigid analytic curves}, J.
  Reine Angew. Math. \textbf{537} (2001), 165--234. \MR{1856262}

\bibitem[ILO14]{deGabber}
Luc Illusie, Yves Laszlo, and Fabrice Orgogozo (eds.), \emph{Travaux de
  {G}abber sur l'uniformisation locale et la cohomologie \'{e}tale des
  sch\'{e}mas quasi-excellents}, Soc. Math. France, Paris, 2014, S\'{e}minaire
  \`a l'\'{E}cole Polytechnique 2006--2008. With the collaboration of
  Fr\'{e}d\'{e}ric D\'{e}glise, Alban Moreau, Vincent Pilloni, Michel Raynaud,
  Jo\"{e}l Riou, Beno\^{i}t Stroh, Michael Temkin and Weizhe Zheng,
  Ast\'{e}risque No. 363-364 (2014). \MR{3309086}

\bibitem[Ked05]{Kedlaya-covers}
Kiran~S. Kedlaya, \emph{More \'{e}tale covers of affine spaces in positive
  characteristic}, J. Algebraic Geom. \textbf{14} (2005), no.~1, 187--192.

\bibitem[Ked19]{KedAr}
\bysame, \emph{Sheaves, stacks, and shtukas}, Mathematical Surveys and
  Monographs, vol. 242, American Mathematical Society, Providence, RI, 2019,
  Lectures from the 2017 Arizona Winter School.

\bibitem[Kie69]{Kiehl-excellence}
Reinhardt Kiehl, \emph{Ausgezeichnete {R}inge in der nichtarchimedischen
  analytischen {G}eometrie}, J. Reine Angew. Math. \textbf{234} (1969), 89--98.

\bibitem[KM76]{KM76}
Finn Knudsen and David Mumford, \emph{The projectivity of the moduli space of
  stable curves. {I}. {P}reliminaries on ``det'' and ``{D}iv''}, Math. Scand.
  \textbf{39} (1976), no.~1, 19--55. \MR{437541}

\bibitem[Kob23]{Kobak}
Mateusz Kobak, \emph{Compactifications of rigid analytic spaces through formal
  models}, Preprint, available at \url{https://arxiv.org/abs/2306.09141}, 2023.

\bibitem[LLZ23]{LLZ}
Kai-Wen Lan, Ruochuan Liu, and Xinwen Zhu, \emph{De {R}ham comparison and
  {P}oincar\'{e} duality for rigid varieties}, Peking Math. J. \textbf{6}
  (2023), no.~1, 143--216. \MR{4552642}

\bibitem[Lur24]{kerodon}
Jacob Lurie, \emph{Kerodon}, \url{https://kerodon.net}, 2024.

\bibitem[L{\"{u}}t16]{Lut16}
Werner L{\"{u}}tkebohmert, \emph{Rigid geometry of curves and their
  {J}acobians}, Ergebnisse der Mathematik und ihrer Grenzgebiete. 3. Folge. A
  Series of Modern Surveys in Mathematics, vol.~61, Springer, Cham, 2016.
  \MR{3467043}

\bibitem[Man22]{Mann-thesis}
Lucas Mann, \emph{A {$p$}-{A}dic 6-{F}unctor {F}ormalism in {R}igid-{A}nalytic
  {G}eometry}, Preprint, available at \url{https://arxiv.org/abs/2206.02022},
  2022.

\bibitem[Mat89]{Matsumura}
Hideyuki Matsumura, \emph{Commutative ring theory}, second ed., Cambridge
  Studies in Advanced Mathematics, vol.~8, Cambridge University Press,
  Cambridge, 1989, Translated from the Japanese by M. Reid. \MR{1011461}

\bibitem[Ols15]{Olsson}
Martin Olsson, \emph{Borel-{M}oore homology, {R}iemann-{R}och transformations,
  and local terms}, Adv. Math. \textbf{273} (2015), 56--123. \MR{3311758}

\bibitem[Sch12]{Scholze-perfectoid}
Peter Scholze, \emph{Perfectoid spaces}, Publ. Math. Inst. Hautes \'Etudes Sci.
  \textbf{116} (2012), 245--313. \MR{3090258}

\bibitem[Sch13a]{Scholze-Hodge}
\bysame, \emph{{$p$}-adic {H}odge theory for rigid-analytic varieties}, Forum
  Math. Pi \textbf{1} (2013), e1, 77. \MR{3090230}

\bibitem[Sch13b]{Scholze-CDM}
\bysame, \emph{Perfectoid spaces: a survey}, Current developments in
  mathematics 2012, Int. Press, Somerville, MA, 2013, pp.~193--227.
  \MR{3204346}

\bibitem[Sch17]{Diamonds}
\bysame, \emph{{\'E}tale cohomology of diamonds}, Preprint, available at
  \url{https://arxiv.org/abs/1709.07343}, 2017.

\bibitem[SGA4]{SGA4}
Michael Artin, Alexander Grothendieck, and Jean-Louis Verdier (eds.),
  \emph{Th\'{e}orie des topos et cohomologie \'{e}tale des sch\'{e}mas {I},
  {II}, {III}}, Lecture Notes in Mathematics, vol. 269, 270, 305,
  Springer-Verlag, Berlin-New York, 1972, S\'{e}minaire de G\'{e}om\'{e}trie
  Alg\'{e}brique du Bois-Marie 1963--1964 (SGA 4). With the collaboration of
  Nicolas Bourbaki, Pierre Deligne et Bernard Saint-Donat.

\bibitem[SGA$4\tfrac{1}{2}$]{SGA41/2}
Pierre Deligne, \emph{Cohomologie \'{e}tale}, Lecture Notes in Mathematics,
  vol. 569, Springer-Verlag, Berlin, 1977, S\'{e}minaire de g\'{e}om\'{e}trie
  alg\'{e}brique du Bois-Marie (SGA $4\frac{1}{2}$).

\bibitem[SGA5]{SGA5}
Luc Illusie (ed.), \emph{Cohomologie {$l$}-adique et fonctions {$L$}}, Lecture
  Notes in Mathematics, vol. 589, Springer-Verlag, Berlin-New York, 1977,
  S\'{e}minaire de G\'{e}ometrie Alg\'{e}brique du Bois-Marie 1965--1966 (SGA
  5). \MR{0491704}

\bibitem[SP24]{stacks-project}
The Stacks~{Project Authors}, \emph{The {S}tacks {P}roject},
  \url{https://stacks.math.columbia.edu}, 2024.

\bibitem[SW20]{Berkeley}
Peter Scholze and Jared Weinstein, \emph{Berkeley lectures on {$p$}-adic
  geometry}, Annals of Mathematics Studies, vol. 207, Princeton University
  Press, Princeton, NJ, 2020. \MR{4446467}

\bibitem[Tem00]{Temkin2000}
Michael Temkin, \emph{On local properties of non-{A}rchimedean analytic
  spaces}, Math. Ann. \textbf{318} (2000), no.~3, 585--607.

\bibitem[Tem10]{Temkin-curve}
\bysame, \emph{Stable modification of relative curves}, J. Algebraic Geom.
  \textbf{19} (2010), no.~4, 603--677.

\bibitem[Tem12]{Tem12}
\bysame, \emph{Functorial desingularization of quasi-excellent schemes in
  characteristic zero: the nonembedded case}, Duke Math. J. \textbf{161}
  (2012), no.~11, 2207--2254. \MR{2957701}

\bibitem[vdP80]{vdPut80}
Marius van~der Put, \emph{The class group of a one-dimensional affinoid space},
  Ann. Inst. Fourier (Grenoble) \textbf{30} (1980), no.~4, 155--164.
  \MR{599628}

\bibitem[Wed19]{wedhorn}
Torsten Wedhorn, \emph{Adic {S}paces}, available at
  \url{https://arxiv.org/abs/1910.05934}, 2019.

\bibitem[Zav23]{Z-revised}
Bogdan Zavyalov, \emph{{P}oincar{\'e} {D}uality in abstract 6-functor
  formalisms}, arXiv e-prints (2023), arXiv:2301.03821.

\bibitem[Zav24a]{Z-alterations}
\bysame, \emph{Altered local uniformization of rigid-analytic spaces}, Israel
  J. Math. (to appear) (2024).

\bibitem[Zav24b]{Z-thesis}
\bysame, \emph{Mod-{$p$} {P}oincar{\'e} duality in {$p$}-adic analytic
  geometry}, Ann. of Math. (2) (to appear) (2024).

\bibitem[Zav24c]{Z-quotients}
\bysame, \emph{Quotients of admissible formal schemes and adic spaces by finite
  groups}, Algebra Number Theory \textbf{18} (2024), no.~3, 409--475.
  \MR{4705884}

\bibitem[Zav24d]{adic-notes}
\bysame, \emph{Some foundational results in adic geometry}, Preprint, available
  at \url{https://arxiv.org/abs/2409.15516}, 2024.

\end{thebibliography}
\end{document}